\documentclass{article}
\usepackage[utf8]{inputenc}
\usepackage[left=3cm,right=3cm,top=3cm,bottom=3cm]{geometry}
\usepackage{graphicx}
\usepackage{enumitem}
\usepackage{color}
\usepackage{amsmath,amsfonts,amssymb,amsthm,epsfig,epstopdf,titling,url,array}
\usepackage{mathrsfs}
\usepackage{caption}
\usepackage{subcaption}
\usepackage{tikz-cd}
\usepackage{MnSymbol}

\usepackage{fontawesome}

\usepackage{hyperref}      
\hypersetup{
    colorlinks,
    citecolor=black,
    filecolor=black,
    linkcolor=blue,
    urlcolor=black
}

\newcounter{ppart}

\newcommand{\ppart}[1]{%
    \refstepcounter{ppart}%
    \bigskip
    \noindent
    {\Large\bfseries Part~\Roman{ppart}. #1\par}
    \medskip
    \addcontentsline{toc}{part}{Part~\Roman{ppart}. #1}%
}

\DeclareMathAlphabet{\mathpzc}{OT1}{pzc}{m}{it}

\DeclareMathOperator{\conv}{Conv}
\DeclareMathOperator{\ham}{Ham}

\DeclareMathOperator{\osc}{Osc}
\DeclareMathOperator{\supp}{Supp}
\DeclareMathOperator{\leg}{Leg}

\DeclareMathOperator{\sympex}{\mathrm{Symp}_{\mathrm{ex}}}
\DeclareMathOperator{\Spec}{Spec}
\DeclareMathOperator{\diam}{diam}
\DeclareMathOperator{\leb}{leb}
\DeclareMathOperator{\im}{Im}
\DeclareMathOperator{\Span}{Span}
\DeclareMathOperator{\codim}{codim}

\newcommand{\R}{\mathbb{R}}

\newcommand{\N}{\mathbb{N}}
\newcommand{\Z}{\mathbb{Z}}
\newcommand{\C}{\mathbb{C}}

\newcommand{\T}{\mathbb{T}}
\newcommand{\cL}{\mathcal{L}}
\newcommand{\G}{\mathcal{G}}

\newcommand{\Mane}{Mañé }

\def\XXint#1#2#3{{\setbox0=\hbox{$#1{#2#3}{\int}$ }
\vcenter{\hbox{$#2#3$ }}\kern-.6\wd0}}

\newtheorem{theo}{Theorem}[section]
\newtheorem{prop}[theo]{Proposition}
\newtheorem{lem}[theo]{Lemma}
\newtheorem{cor}[theo]{Corollary}
\newtheorem{theor}{Theorem}
\newtheorem{coro}[theor]{Corollary}

\theoremstyle{definition}
\newtheorem{defi}[theo]{Definition}

\newtheorem{exmp}[theo]{Example}
\newtheorem{rem}[theo]{Remark}

\newtheorem{quest}{Question}

\theoremstyle{remark}

\numberwithin{equation}{section}

\title{A Multidimensional Birkhoff Theorem for \\ some $C^0$ Lagrangians}
\author{Skander Charfi \\ Université Paris-Saclay \and Ibrahim Trifa \\ ETH Z\"urich}
\date{}

\begin{document}

\maketitle

\begin{abstract}
    We prove a multidimensional Birkhoff theorem for a new class of ``$C^0$ Lagrangian subsets'' in cotangent bundles, obtained as Hausdorff limits of compact exact Lagrangian submanifolds with control on their Liouville primitives. If the successive images of such a subset $L$ under the flow of a Tonelli Hamiltonian admit convergent subsequences in both positive and negative time, then $L$ and all its images are Lipschitz graphs over the base. This extends Birkhoff's celebrated theorem for twist maps of the annulus and its known higher-dimensional generalizations, and provides a first version of such results for merely continuous objects. The proof combines Floer-theoretic graph selectors, variational solutions of the Hamilton--Jacobi equation, and weak KAM theory. We also investigate the rigidity and uniqueness of limiting primitives for this new class of singular Lagrangian subsets, which may be of independent interest in $C^0$ symplectic topology.
\end{abstract}

\tableofcontents

\hfill

\section{Introduction}

\subsection{Overview}

In his studies of twist maps of the cylinder, G.~D.~Birkhoff \cite{MR1555175} showed that any essential invariant curve is a Lipschitz graph over the circle. This result was a first step toward a non-perturbative study of invariant tori for twist maps. The extension of this phenomenon to higher dimensions has led to several generalizations in the settings of twist maps and convex Hamiltonian systems, often under additional assumptions on the topology or on the dynamics restricted to the invariant submanifold \cite{MR1067380,MR1027911,MR1001842,MR1043418,MR1159829,MR1157317,MR4597700}. One of the most notable generalizations is due to Arnaud \cite{MR2738994}, who proved that a Lagrangian submanifold of a cotangent bundle, Hamiltonian isotopic to the zero section and invariant under an autonomous Tonelli (convex) Hamiltonian flow, is a Lipschitz graph over the base. The non-autonomous case was established later by Arnaud and Venturelli \cite{MR3674224}.\\

While the original Birkhoff theorem applied to any continuous invariant circle, the higher-dimensional version of Arnaud required the invariant Lagrangian to be at least $C^1$. Therefore, one could ask the following question:
\begin{quest} \label{quest:BirkhoffC0}
    Are closed ``exact'' $C^0$ Lagrangian submanifolds\footnote{By a closed exact $C^0$ Lagrangian submanifold, we mean the image of a closed exact Lagrangian by a homeomorphism arising as a $C^0$-limit of exact symplectomorphisms.} that are invariant under the flow of a Tonelli Hamiltonian always graphs over the base manifold?
\end{quest}
This question has received significant attention and has been addressed, among others, by Marie-Claude Arnaud, Vincent Humilière, Rémi Leclercq, Sobhan Seyfaddini, Andrea Venturelli and Claude Viterbo. Important progress was made by Bernard--dos Santos \cite{MR2860422} and Amorim--Oh--dos Santos \cite{MR3860396}, who established the autonomous result for Lipschitz Lagrangian submanifolds, where the latter work also replaced the Hamiltonian isotopy assumption by an exactness condition.\\

In the present paper, we give a partial answer to Question \ref{quest:BirkhoffC0}: we prove a Birkhoff theorem for $C^0$ Lagrangians for which a notion of ``continuous Liouville primitive'' can be uniquely defined. We also show that this class of objects contains Lagrangians that are less regular than Lipschitz.\\

Another direction of generalization is to relax the invariance hypothesis. Fathi's weak KAM theory \cite{fathi2008weak,MR1650261} implies that in the autonomous setting a Lipschitz graph invariant under the time-one map is automatically invariant for all times, whereas this phenomenon fails in the non-autonomous case. In \cite{Recurrent}, the first author constructed a recurrent, non-periodic, smooth Lagrangian submanifold under a time-periodic Tonelli flow, and later established in \cite{Birkhoff} a Birkhoff-type theorem under a two-sided recurrence condition, more generally requiring only convergence along sequences of positive and negative times to Lagrangian submanifolds, with control on the Liouville primitives.\\

The ``$C^0$ Lagrangians'' we consider are inspired by this notion of convergence: they arise as Hausdorff limits of exact Lagrangian branes with control on their primitive. We call them \emph{bare Lagrangian branes}. As in \cite{Birkhoff}, the Birkhoff theorem we prove also holds under the two-sided convergence assumption, without requiring invariance.\\

The proof of our Birkhoff theorem combines Lagrangian Floer spectral invariants, graph selectors and weak KAM theory. Graph selectors were introduced by Sikorav and Chaperon \cite{MR2025275,MR1094198}, and were later constructed for exact Lagrangians using Floer homology in \cite{MR3860396}. They provide a natural way to construct variational solutions of the Hamilton--Jacobi equation, which have been studied in \cite{MR1604386,MR3151090,MR3957150}. We extend these constructions to bare Lagrangian branes. Using weak KAM techniques \cite{MR1451248,fathi2008weak}, we then show that a suitable asymptotic behaviour in both negative and positive time forces the Hamiltonian trajectories to be calibrated, ultimately yielding the Lipschitz graph property.\\

In the invariant case, our theorem requires both the Lagrangian object and its primitive to be fixed by the action of the Hamiltonian flow. When the primitive attached to a bare Lagrangian brane is unique (up to a constant shift), its invariance is automatically implied by the invariance of the underlying set. Therefore, to understand the range of application of our theorem, we need to study the uniqueness of primitives for bare Lagrangian branes. We prove uniqueness of the primitive for continuous Lagrangians of the form $\varphi(\operatorname{graph}(dv))$, where $v\in C^1(M)$ and $\varphi$ is an exact symplectic diffeomorphism, and even for images of closed exact Lagrangians by some merely continuous symplectic homeomorphisms $\varphi$. This provides new classes of $C^0$ Lagrangian subsets to which the Birkhoff theorem applies.\\

We also investigate some rigidity and flexibility properties of bare Lagrangian branes, which could be of independent interest in $C^0$ symplectic topology. We prove a semi-local rigidity theorem showing that if the support of a bare Lagrangian brane is a topological manifold of the same dimension as the base, then any $C^1$ open subset of it is Lagrangian, and the limiting function is a Liouville primitive on it. This is related to $C^0$-rigidity phenomena for Lagrangian and Legendrian submanifolds \cite{MR1266111,MR4334195,nakamura2020c0limitslegendriansubmanifolds,MR4876451,asano2025c0rigiditylegendrianscoisotropicssheaf}. Our proof uses microlocal sheaf theory and the relation between singular supports, $\gamma$-supports and the Humilière completion \cite{MR1299726,guillermou2015,viterbo2019sheafquantizationlagrangiansfloer,MR4768582,MR4683312}. If the base manifold has dimension at least two, we also prove a flexibility result: if a compact connected set contains the support of a bare Lagrangian brane, then any continuous extension of its primitive to the larger set can again be realized as the primitive of a bare Lagrangian brane. The main construction is inspired by convex integration \cite{MR864505,MR3024860,MR4381219,MR3619673}. This leads to examples of non-uniqueness of the primitives of bare Lagrangian branes, and to Birkhoff-type results for suitable stratified Lagrangian subsets for which uniqueness of the primitive need not hold.

\subsection{Main Results}

Fix a closed connected manifold $M$ of dimension $d$, endowed with a Riemannian metric $\langle \cdot , \cdot \rangle$, with induced norm $\Vert \cdot \Vert$ on the cotangent bundle and induced distance $d$ on the manifold $M$. The cotangent bundle $T^*M$ has a natural exact symplectic structure $(T^*M, \omega = -d\lambda)$ defined as follows. Let $\pi_M : T^*M \to M$ be the natural projection and denote by $(q,p)=(q_1,\ldots,q_d,p_1,\ldots,p_d)$ the coordinates in $T^*M$, where $q = (q_1,\ldots,q_d)$ are local coordinates of $M$ and $p = (p_1, \ldots, p_d)$ are fibrewise coordinates with respect to the cotangent vectors $dq_1,\ldots,dq_d$. The \textit{Liouville form} $\lambda$ is defined by $\lambda(q,p) = p \circ d\pi_M = p.dq$. The $1$-jet bundle of $M$ is denoted by $J^1M = T^*M \times \R$.

On $X = T^*M$ or $J^1M$, the \textit{Hausdorff distance} $d_H$ on the set of closed subsets of $X$ is defined for all such $\mathcal{K}$ and $\mathcal{K}'$ as
\begin{equation} \label{Hausdorff}
	d_H(\mathcal{K}, \mathcal{K}') = \max \Big\{ \sup_{x' \in \mathcal{K}'} d(x', \mathcal{K}) \; , \; \sup_{x \in \mathcal{K}} d(x, \mathcal{K}') \Big\} 
\end{equation}
where $d$ denotes the distance induced on $X$ by the Riemannian metric on $M$.

For $Y = M$ or $\R \times M$, the space $C^0(Y,\R)$ denotes the set of continuous scalar maps $u : Y \to \R$, endowed with the infinity norm $\Vert u \Vert_\infty = \sup_{q \in Y} |u(q)|$. Note that when $Y$ is not compact, this is not a norm since it may take infinite values. We say that a scalar map is $C^{1,1}$ regular if it is $C^1$ regular with Lipschitz differential. The graph of $du$ is denoted by $\mathcal{G}(du) \subset T^*M$. For any subset $A$ of $Y$, we will use the notation $\Vert u \Vert_\infty^{A} = \sup_{q \in A} |u(q)|$. The \textit{oscillation} of a scalar map $u$ is $\osc(u) = \sup u - \inf u$.

A \textit{symplectomorphism} $\varphi$ of $(T^*M,\omega)$ is a diffeomorphism of $T^*M$ such that $\varphi^* \omega = \omega$. A symplectomorphism is \textit{exact} if there exists a $C^1$ scalar map $f : T^*M \to \R$ such that $\varphi^* \lambda - \lambda = df$. We denote by $\sympex (T^*M,\lambda)$ the set of exact symplectomorphisms of $T^*M$.

Set $\mathbb{T}^1 = \mathbb{R}/ \mathbb{Z}$ and denote by $t$ the time coordinate in $\mathbb{T}^1$. A $C^2$ time-periodic \textit{Hamiltonian} is a map $H : \mathbb{T}^1 \times T^*M \to \mathbb{R}$. Given $H$, the \textit{Hamiltonian vector field} $X_H$ is the time-dependent vector field uniquely defined by the equation $\iota_{X^t_H} \omega = \omega( X^t_H, \cdot ) = dH_t$, where $H_t = H(t,\cdot,\cdot)$. The corresponding \textit{Hamiltonian flow} from time $s$ to time $t$ is denoted by $\phi_H^{s,t}$, with inverse $\phi_H^{t,s}$. We set $\phi_H^t := \phi_H^{0,t}$. The \textit{Hamiltonian diffeomorphism group} $\ham(T^*M, \omega)$ is the group of Hamiltonian diffeomorphisms, i.e., time-one maps of globally defined Hamiltonian flows. The \textit{Hamilton--Jacobi equation } associated to $H$ is given by 
\begin{equation} \label{HJ}
	\partial_t u + H(t,d_qu) = 0
\end{equation}

A Hamiltonian is \textit{Tonelli} if it is strictly convex and superlinear on fibres, and if it has a complete flow (Definition \ref{Tonelli}).\\

A \textit{Lagrangian submanifold} $\mathcal{L}$ of $T^*M$ is a submanifold such that $\dim \mathcal{L} =d$ and $\omega_{|T \mathcal{L}} =0$. The Lagrangian submanifold $\mathcal{L}$ is \textit{exact} if $\lambda_{|T \mathcal{L}}$ is an exact form, i.e., there exists a \textit{Liouville primitive} $h : \mathcal{L} \to \mathbb{R}$ such that $\lambda_{|\mathcal{L}} = dh$. The pair $(\cL,h)$ is called a \textit{Lagrangian brane}.\\

\subsubsection{The Multidimensional Birkhoff Result}

We aim to reach Lagrangian submanifolds of lower regularity while preserving a notion of primitive. To do so, we define a topology on the set of compact Lagrangian branes and subsequently consider the objects that can be approximated by Lagrangian branes with respect to this topology.

\begin{defi} \label{BraneDef}
	\begin{enumerate}
		\item A \textit{bare brane} $(\mathcal{K},h)$ is a pair consisting of a closed set $\mathcal{K}$ and a continuous map $h : \mathcal{K} \to \mathbb{R}$. We associate to it the partial graph $\Gamma_{(\mathcal{K},h)}=\{(x,h(x))\in T^*M\times \R\mid x\in \mathcal{K}\}$, which is a closed subset of the $1$-jet $J^1M = T^*M \times \R$ over the cotangent bundle $T^*M$. We denote by $\mathcal{B}_c(T^*M)$ the set of non-empty compact bare branes on $T^*M$.
		\item A \textit{Lagrangian brane} is a bare brane $(\mathcal{L},h)$ where $\mathcal{L}$ is an exact Lagrangian submanifold and $h$ is an associated Liouville primitive such that $\lambda_{|\mathcal{L}} = dh$. We denote by $\mathcal{LB}(T^*M,\omega)$ the set of compact Lagrangian branes on $(T^*M,\omega)$. Note that in this case, the graph $\Gamma_{(\cL,h)}$ is the Legendrian lift of $(\cL,h)$.
		\item We define the \textit{brane distance} $d_B$ between two compact bare branes to be the Hausdorff distance between their graphs in $J^1M$
        \[d_B((\mathcal{K},h),(\mathcal{K}',h'))=d_H(\Gamma_{(\mathcal{K},h)},\Gamma_{(\mathcal{K}',h')}),\]
        where $J^1M$ is endowed with the metric $d_{J^1M}$ defined by
		\begin{equation} \label{Formula:BraneDistanceHausdorff}
			d_{J^1M}\big( (x,z) ,(x',z') \big) = d(x,x') + |z-z'|.
		\end{equation}
		We call \textit{brane convergence} the convergence with respect to the brane distance $d_B$.
		\item We denote by $\overline{\mathcal{LB}(T^*M,\omega)}$ the closure of the set $\mathcal{LB}(T^*M,\omega)$ of Lagrangian branes in $\mathcal{B}_c(T^*M)$ endowed with the topology induced by the brane metric $d_B$. Its elements are called \textit{bare Lagrangian branes}.
		\item We denote by $\mathfrak{B}_c (T^*M)$ the quotient set $\mathcal{B}_c (T^*M)/ \mathbb{R}$ of bare branes $(\mathcal K,[h])$, where $[h]$ is the equivalence class of $h$ under the relation $h \sim h' \Leftrightarrow h' = h+c$ for some constant $c \in \mathbb{R}$. We define the metric $d_{\mathfrak{B}}$ on $\mathfrak{B}_c(T^*M)$ by
		\begin{equation}\label{distance on quotient}
			d_{\mathfrak{B}} \big((\mathcal K,[h]),(\mathcal K',[h'])\big) = \inf_{c \in \mathbb{R}} d_B\big((\mathcal{K},h+c),(\mathcal{K}',h')\big)
		\end{equation}
		We denote by $\mathfrak{LB} (T^*M,\omega)$ the quotient set $\mathcal{LB}(T^*M,\omega) / \mathbb{R}$ of Lagrangian branes. Since Liouville primitives are unique up to shift by a constant, $\mathfrak{LB} (T^*M,\omega)$ can be identified with the set $\mathscr{L}(T^*M)$ of closed exact Lagrangian submanifolds of $T^*M$.
		
		We denote by $\overline{\mathfrak{LB} (T^*M,\omega)}$ its closure in $\mathfrak{B}_c(T^*M)$ endowed with the topology induced by the brane metric $d_\mathfrak{B}$.
	\end{enumerate}
\end{defi}

\begin{rem} 
	\begin{enumerate}
		\item Since it is derived from the Hausdorff distance, the brane distance $d_B$ is indeed a metric.		
		\item The compact sets reached by this topology in $\overline{\mathfrak{LB} (T^*M,\omega)}$ can be extremely irregular, exhibiting, for instance, `filaments' (or hairs), or having Hausdorff dimension greater than $d = \dim M$. They may be as irregular as certain $\gamma$-supports, although the latter can display even more erratic behaviour. However, the primitive will contain additional information that will, in certain cases, allow for the control of the behaviour of the bare Lagrangian branes under consideration.
		\item We insist that the set of bare Lagrangian branes $\overline{\mathcal{LB}(T^*M,\omega)}$ is the closure, and not the completion of $\mathcal{LB}(T^*M,\omega)$ in $(\mathcal{B}_c (T^*M),d_B)$. Indeed, the set $\mathcal{B}_c (T^*M)$ is not complete for the brane metric $d_B$. As an example, consider the sequence of compact sets $\mathcal{K}_n = [0,1/n]$ together with continuous maps $h_n : \mathcal{K}_n \to \mathbb{R}$ defined by $h_n(x) = nx$. Then the sequence of graphs $\Gamma_{(\mathcal{K}_n,h_n)}$ converges in the Hausdorff topology to the vertical segment $\{0\}\times [0,1]$. Therefore, the sequence $(\mathcal{K}_n,h_n)$ is Cauchy for the brane distance, but does not converge in $\mathcal{B}_c (T^*M)$ since the vertical segment is not a partial graph.
		\item In the case of Lagrangian branes that are Hamiltonian isotopic to the zero section, brane convergence is equivalent to the “Reduced Complexity Convergence” introduced in \cite{Birkhoff}. This is shown in Appendix \ref{SectionRC}.
	\end{enumerate}
\end{rem}

We need to define a natural way to consider the image of a ``bare brane'' under the Hamiltonian flow. This is provided by the notion of a bare variational brane, defined as follows:

\begin{defi} \label{LagsImages} 
	Let $H : \mathbb{T}^1 \times T^*M \to \mathbb{R}$ be a Hamiltonian. To any bare brane $(\mathcal{K},h)$ in $\mathcal{B}_c(T^*M)$, we associate its \textit{variational bare brane} $(\mathcal{K}_t,h_t)_{t \in\mathbb{R}}$ defined by $\mathcal{K}_t = \phi_H^t(\mathcal{K})$ and
	\begin{equation} \label{LagsImagesFormula}
		h_t(x_t) = h (x) + \int_0^t \big( x_\tau^* \lambda - H(\tau, x_\tau) \big) \;d\tau \quad \text{where} \quad x_\tau = \phi_H^\tau (x).
	\end{equation}
	The variational bare brane associated to its class $(\mathcal{K},[h])$ in $\mathfrak{B}_c(T^*M)$ is the family $(\mathcal{K}_t,[h_t])_{t \in\mathbb{R}}$.
\end{defi}

We can now state the main theorem of this paper.

\begin{theor} \label{MainTheorem}
    	Let $H : \mathbb{T}^1 \times T^*M \to \mathbb{R}$ be a Tonelli Hamiltonian and let $(\mathcal{L},[h])$ be a bare Lagrangian brane in $\overline{\mathfrak{LB} (T^*M,\omega)}$. Let $(\mathcal{L}_t = \phi^t_H(\mathcal{L}),[h_t])$ be the associated variational brane. We assume that there exist two sequences of real times $t_k \to +\infty$ and $s_k \to -\infty$ such that $(\mathcal{L}_{t_k},[h_{t_k}])$ and $(\mathcal{L}_{s_k},[h_{s_k}])$ brane-converge, respectively, to two compact bare branes $(\mathcal{L}_\omega,[h_\omega])$ and $(\mathcal{L}_\alpha,[h_\alpha])$ in $\mathfrak{B}_c(T^*M)$.
	Then $\mathcal{L}$ and all its images $\mathcal{L}_t = \phi_H^t(\mathcal{L})$ are Lipschitz graphs over the zero section $0_M$.
\end{theor}

We state an important particular case:
\begin{theor}[Invariant case]\label{MainThmInvariant}
    Let $H : \mathbb{T}^1 \times T^*M \to \mathbb{R}$ be a Tonelli Hamiltonian and let $(\mathcal{L},[h])$ be a compact brane in $\overline{\mathfrak{LB} (T^*M,\omega)}$. Let $(\mathcal{L}_t = \phi^t_H(\mathcal{L}),[h_t])$ be the associated variational brane. We assume that $(\mathcal{L}_{1},[h_{1}])=(\mathcal{L},[h])$. 
	Then $\mathcal{L}$ is a Lipschitz graph over the zero section $0_M$.
\end{theor}

Let us discuss these results in two distinct points: the reduction of the invariance condition in Theorem \ref{MainTheorem}, and the nontrivial gain in regularity compared to the well-known $C^{1,1}$ version of Theorem \ref{MainThmInvariant}.

\paragraph{Optimality of Recurrence.}

We discuss a relaxation of the invariance assumption in Theorem \ref{MainTheorem}. Indeed, the theorem only requires the convergence, in positive and negative time, of subsequences of the variational bare Lagrangian brane, rather than full invariance under the Hamiltonian flow. This type of relaxation first appeared in \cite{Birkhoff}.\\

The graph property is therefore not a direct consequence of invariance, but rather of a specific asymptotic behaviour of the Lagrangian submanifolds at infinity. Moreover, we will see in this section that this behaviour is quite restrictive: more precisely, the considered bare Lagrangian branes turn out to be recurrent in both positive and negative time under the Hamiltonian flow, with respect to the brane topology.\\

The reason behind this fact comes from the proof of Theorem \ref{MainTheorem} itself. In order to show that $\mathcal{L}_t$ are Lipschitz graphs over the zero section, we prove that they are precisely the graphs of some $d u(t,\cdot)$, where $u : \R \times M \to \R$ is a $C^{1,1}$ solution of the Hamilton--Jacobi equation \eqref{HJ}. The behaviour of such solutions is quite rigid. We now make this precise.\\

It is known (\cite{fathi2008weak,MR2393423}) that there exists a unique real constant $\alpha_0$ known as the \textit{\Mane critical value} such that all (viscosity) solutions $u:[s,+\infty) \times M \to \mathbb{R}$ of the critical Hamilton--Jacobi equation
\begin{equation} \label{HJAlpha}
	\partial_t u + H(t,d_qu) = \alpha_0
\end{equation}
are bounded in $C^0([s,+\infty) \times M,\mathbb{R})$.\\ 

The notion of viscosity solution was introduced by Crandall and Lions \cite{MR0690039} and has since become a standard framework for the Hamilton--Jacobi equation. We will not need to recall its general definition here, since throughout this article we only consider sufficiently regular solutions. In this regular setting, viscosity solutions coincide with classical solutions.

\begin{theo}  [\cite{MR1650261,MR2041603,Representation, Birkhoff, MR2150356}]	 \label{TheoremViscosityAsymptotic}
	Let $H : \mathbb{T}^1 \times T^*M \to \mathbb{R}$ be a Tonelli Hamiltonian. Let $u : \mathbb{R} \times M \to \mathbb{R}$ be a global viscosity solution of the Hamilton--Jacobi equation \eqref{HJAlpha}. Then,
	\begin{enumerate}
		\item General Case \cite{Representation, Birkhoff, MR2150356}: $u$ is recurrent in both positive and negative times, and more precisely, there exist two increasing sequences $n_k$ and $m_k$ of positive integers such that
		\begin{equation}
			\begin{split}
				\lim_{k \to + \infty} \Vert u(n_k,\cdot) - u(0,\cdot) \Vert_\infty = 0 \quad &, \quad \lim_{k \to + \infty} \Vert u(-m_k,\cdot) - u(0,\cdot) \Vert_\infty = 0 \\
				\lim_{k \to + \infty} d_H \left( \overline{\mathcal{G}\big(d_qu(n_k,\cdot)\big)} , \overline{\mathcal{G}\big(d_qu(0,\cdot)\big)} \right) =0 \quad &, \quad \lim_{k \to + \infty} d_H \left( \overline{\mathcal{G}\big(d_qu(-m_k,\cdot)\big)} , \overline{\mathcal{G}\big(d_qu(0,\cdot)\big)} \right) =0
			\end{split}
		\end{equation}
		where $\overline{\mathcal{G}\big(d_qu(t,\cdot)\big)}$ denotes the closure of the graph of $d_qu(t, \cdot)$ in $T^*M$.
		\item One-Dimensional Case \cite{MR2041603}: If $M = \mathbb{T}^1$, then there exists an integer $N = N(H) \geq 1$ depending only on $H$ such that $u$ is $N$-periodic in time, i.e. $u(t+N,q) = u(t,q)$ for any $(t,q) \in \mathbb{R} \times M$.
		\item Autonomous Case \cite{MR1650261}: If $H : T^*M \to \mathbb{R}$ is autonomous, then $u(t,q) = u(q) : M \to \mathbb{R}$ is independent of time.
	\end{enumerate}
\end{theo}

\begin{coro} \label{OptimalityCor}
	Under the assumptions of Theorem \ref{MainTheorem}, we have the following:
	\begin{enumerate}
		\item General Case: There exist two increasing sequences $n_k$ and $m_k$ of positive integers such that both $(\mathcal{L}_{n_k},[h_{n_k}])$ and $(\mathcal{L}_{-m_k},[h_{-m_k}])$ brane-converge to $(\mathcal{L},[h])$ in $\mathfrak{B}_c(T^*M)$.
		\item One-Dimensional Case: If $M = \mathbb{T}^1$, then there exists $N = N(H) \geq 1$ depending only on $H$ such that $\phi_H^N(\mathcal{L}) = \mathcal{L}$. Additionally, $[h_N] = [h]$.
		\item Autonomous Case: If $H : T^*M \to \mathbb{R}$ is autonomous, then $\mathcal{L}$ is invariant under the Hamiltonian flow $\phi_H^t$, i.e. for all times $t \in \mathbb{R}$, $\phi_H^t(\mathcal{L}) = \mathcal{L}$. Additionally, for all times $t$ in $\R$, $[h_t] = [h]$.
	\end{enumerate}
\end{coro}

\begin{rem}
	We make some remarks about the weakening of the invariance assumption.
	\begin{enumerate}
			\item The proof relies heavily on the convergence of the bare Lagrangian branes in both positive and negative time. This naturally raises the following question:
			\begin{quest}
				Is it possible to construct an exact Lagrangian submanifold that is recurrent in positive time, but not in negative time?
			\end{quest}
			\item We will see in the next paragraph that weakening the regularity of the considered bare Lagrangian branes may produce counterexamples to the graph property, as in Example \ref{exp:Pendulum}, which provides an example of a $\gamma$-support that is invariant under a Tonelli Hamiltonian flow while not being a graph over the base. However, one may ask whether the topology of convergence of $\mathcal{L}_{n_k}$ and $\mathcal{L}_{-m_k}$ can also be weakened.
			\begin{quest}
				If $\cL$ is an exact Lagrangian submanifold that is recurrent, in the spectral or Hausdorff topology, under the flow of a Tonelli Hamiltonian, must it be a graph over the base?
			\end{quest}
	\end{enumerate}		 
\end{rem}

\paragraph{Gain in Regularity.} We discuss the gain in regularity of the considered Lagrangian submanifolds for which the Birkhoff Theorem \ref{MainThmInvariant} applies.\\

The issue is that, although brane convergence allows us to reach extremely irregular sets, this does not imply that Theorem \ref{MainThmInvariant} remains valid for such sets. More precisely, the theorem fails if we only assume that the set $\cL$ is invariant instead of the condition $(\mathcal{L}_1,[h_1])=(\mathcal{L},[h])$.\\

Indeed, we present the following counterexample:

\begin{exmp} \label{exp:Pendulum}
	We consider the pendulum $H(q,p) = p^2 + \cos (2\pi q)$. Then, the elliptic island of the phase portrait $\mathcal{L} = \{(q,p) \in T^* \mathbb{T}^1 \; | \; |p| \leq \sqrt{2}|\sin (\pi q ) | \}$ does correspond to a bare Lagrangian brane $(\cL,[0])\in \overline{\mathfrak{LB} (T^*M,\omega)}$. The approximating sequence of Lagrangian branes $(\cL_n,h_n)$ is represented in Figure \ref{fig:CounterExmpUniq}.

	The elliptic island $\mathcal{L}$ is invariant under the Tonelli Hamiltonian flow of $H$. However, it is not a graph over the base $\mathbb{T}^1$. Nevertheless, this does not contradict Theorem \ref{MainThmInvariant} since for all times $t \in \mathbb{R}\setminus\{0\}$, $[h_t] \neq [0]$. In particular, there is no uniqueness of the class $[h]$ for $(\mathcal{L},[h]) \in \overline{\mathfrak{LB} (T^*M,\omega)}$.
    \begin{figure}[h]
		\centering
		\begin{subfigure}{0.48\textwidth}
			\centering
			\includegraphics[width=\textwidth]{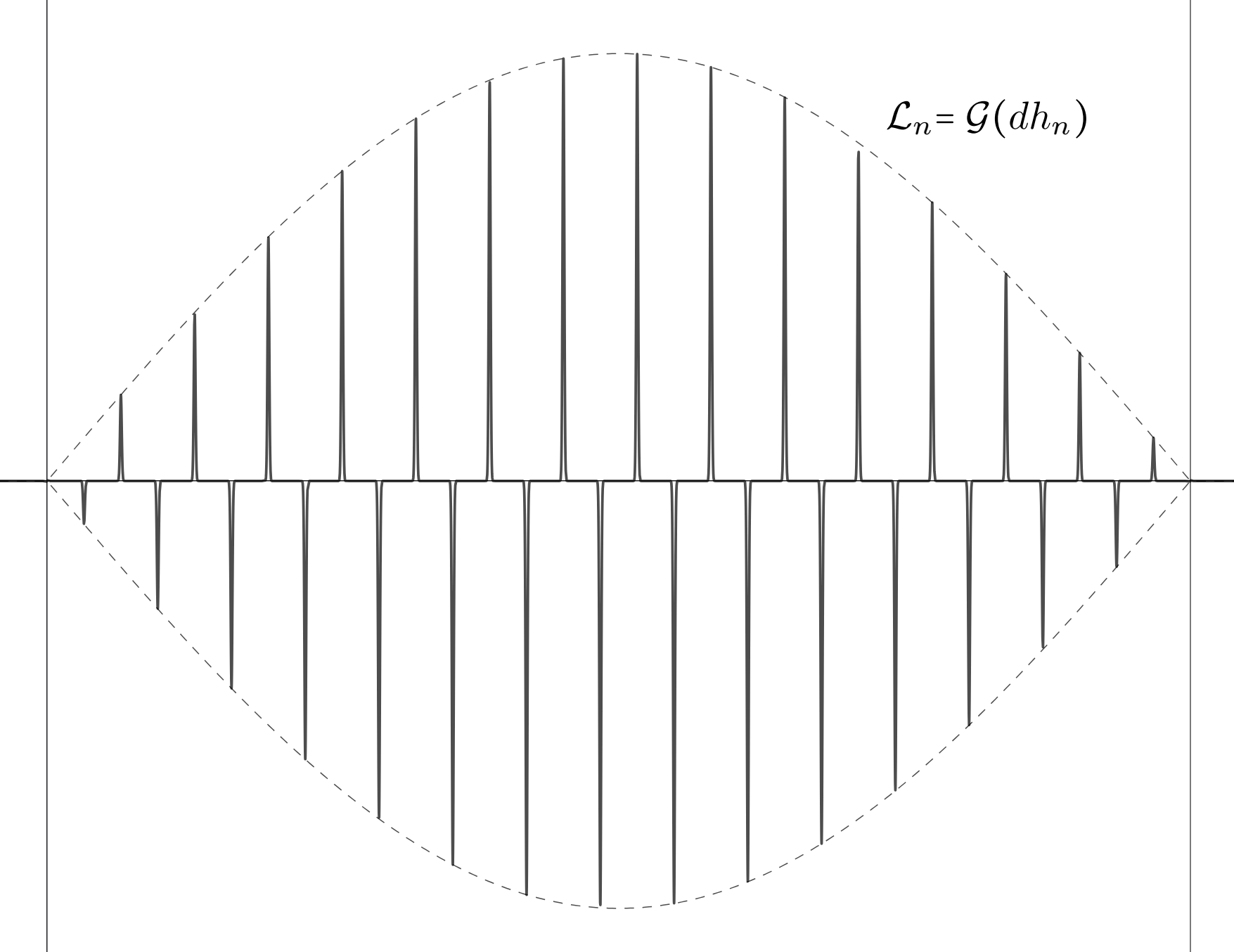}
			\caption{$(\mathcal{L}_n,h_n \approx 0)$}
		\end{subfigure}
		\begin{subfigure}{0.495\textwidth}
			\centering
			\includegraphics[width=\textwidth]{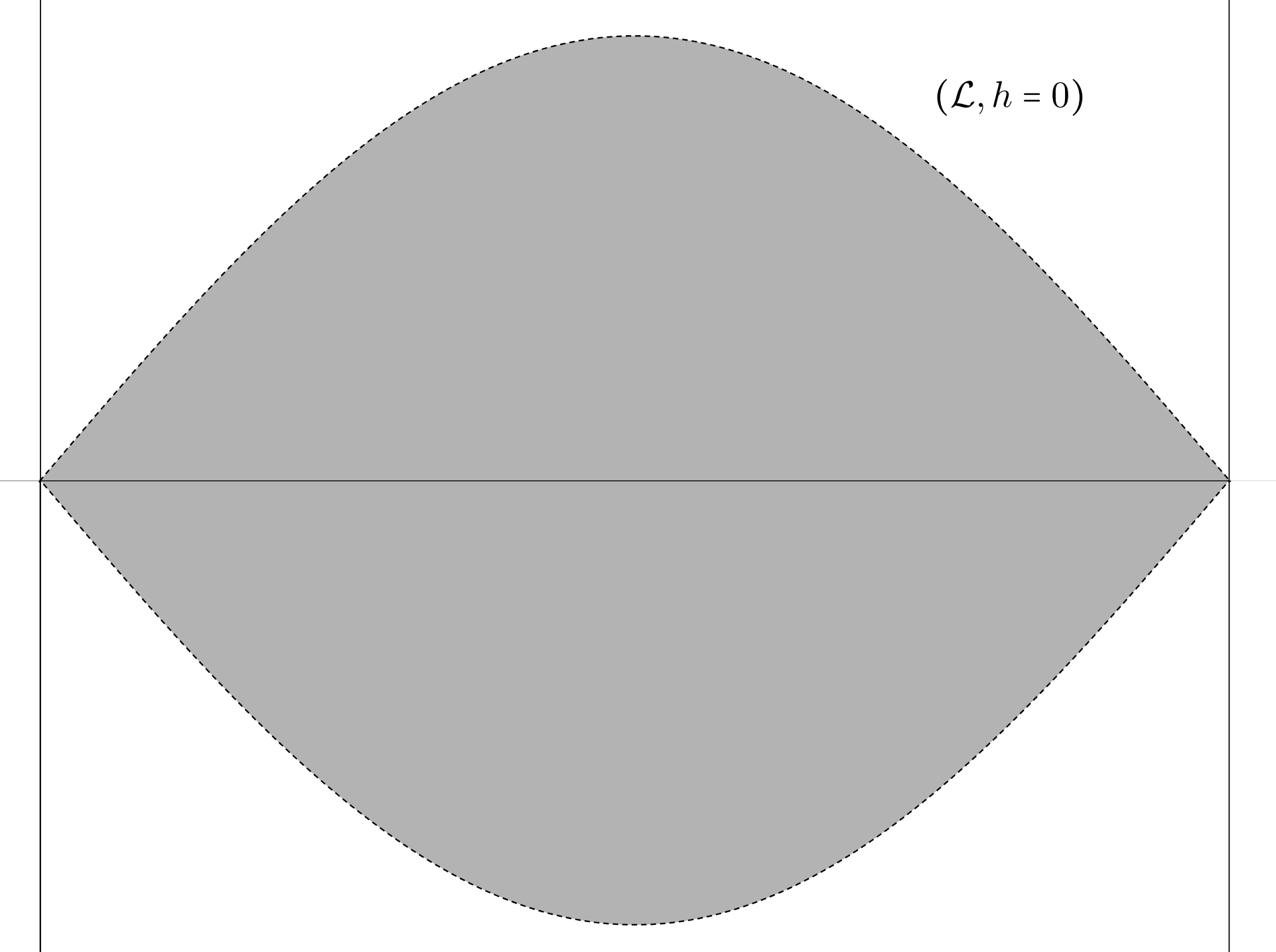}
			\caption{$(\mathcal{L},h\equiv 0)$}
		\end{subfigure}
		\caption{Counterexample to the uniqueness of $[h]$ in $\overline{\mathfrak{LB} (T^*M,\omega)}$.}
		\label{fig:CounterExmpUniq}
	\end{figure}
\end{exmp}

This example shows that the gain in regularity in Theorem \ref{MainThmInvariant} is not immediate and requires a better understanding of the set of bare Lagrangian branes $\overline{\mathcal{LB} (T^*M,\omega)}$. In particular, the Birkhoff theorem fails in this example due to the lack of uniqueness of the primitives $[h_t]$ associated with $\mathcal{L}$ in the quotient $\overline{\mathfrak{LB} (T^*M,\omega)}$. We therefore introduce the following definitions:

\begin{defi}
    \begin{enumerate}
        \item A compact subset $\mathcal{K}$ of $T^*M$ is said to be a \textit{liftable Lagrangian subset} if there exists a continuous map $h: \mathcal{K} \to \R$ in $C^0(\mathcal{K},\R)$ such that $(\mathcal{K},h)$ belongs to $\overline{\mathcal{LB} (T^*M,\omega)}$. We denote by $\mathscr{L}^\exists(T^*M,\omega)$ the set of liftable Lagrangian subsets of $T^*M$.
        \item A compact subset $\mathcal{K}$ of $T^*M$ is said to be a \textit{uniquely liftable Lagrangian subset} if there exists a unique class $[h]$ in $C^0(\mathcal{K},\R)/\R$ such that $(\mathcal{K},[h])$ belongs to $\overline{\mathfrak{LB} (T^*M,\omega)}$. We denote by $\mathscr{L}^!(T^*M,\omega)$ the set of uniquely liftable Lagrangian subsets of $T^*M$.
    \end{enumerate}
\end{defi}

With these new notions, Theorem \ref{MainThmInvariant} can be reformulated as follows, generalizing the results of \cite{MR2860422,MR3860396} to the non-autonomous setting and to weaker regularity assumptions.

\begin{theor}\label{thm:Invariance}
	 Let $H: \mathbb{T}^1 \times T^*M \to \mathbb{R}$ be a Tonelli Hamiltonian, and let $\mathcal{L} \in \mathscr{L}^!(T^*M,\omega)$ be a uniquely liftable Lagrangian subset. If $\phi_H^1(\mathcal{L}) = \mathcal{L}$, then $\mathcal{L}$ is a Lipschitz graph over the base $M$.
\end{theor}

\begin{proof} 
	Let $\mathcal{L}$ be in $\mathscr{L}^!(T^*M, \omega)$ such that $\phi_H^1(\cL) = \cL$. There exists a unique $[h]$ in $C^0(\cL,\R) / \R$ such that $(\mathcal{L},[h])$ belongs to $\overline{\mathfrak{LB}(T^*M, \omega)}$. Let $(\mathcal{L}_t,[h_t])$ be the corresponding variational bare brane. By Proposition \ref{ConvTimeProp}, $(\mathcal{L}_1 = \cL,[h_1])$ also belongs to $\overline{\mathfrak{LB}(T^*M, \omega)}$. Hence, by uniqueness of $[h]$, we deduce that $[h_1] = [h]$. Therefore, Theorem \ref{MainTheorem} applies for $(\mathcal{L},[h]) = (\mathcal{L}_\omega,[h_\omega]) = (\mathcal{L}_\alpha,[h_\alpha])$.
\end{proof}

\begin{rem}
    It was pointed out to us by Claude Viterbo that the theorem would still hold if, instead of considering uniquely liftable Lagrangian subsets, we were considering elements $\cL$ of $\mathscr{L}^\exists(T^*M,\omega)$ that admit a finite number of lifts $(\cL,[h])\in\overline{\mathfrak{LB}(T^*M, \omega)}$. Indeed, if $\cL_1=\cL$, then for some integer $n$ we would have $(\cL_n,[h_n])=(\cL,[h])$, and Theorem \ref{MainTheorem} would apply.
    However, so far we have not found any $\cL$ admitting finitely many lifts but that is not in $\mathscr{L}^!(T^*M)$.
\end{rem}

We give an application of Theorem \ref{thm:Invariance} to the study of invariant sets of geodesic flows. It was suggested to us by Leonid Polterovitch, and is a generalization of \cite[Corollary 2]{MR1043418} (with an analogous proof).

\begin{coro}\label{cor:geodesic}
    Let $M$ be a closed manifold, equipped with a Riemannian metric $g$ with a ``bump'', i.e. such that there exists a topological ball $B\subset M$ and a point $q_0\in B$ satisfying
    \[2d(q_0,\partial B)>\diam(B).\]
    Let $H$ be the Hamiltonian generating the geodesic flow, that is $H(q,p)=\frac 1 2\Vert p\Vert_g^2$. Let $\alpha$ be a smooth closed $1$-form on $M$, and denote by $\tau_\alpha$ the shift $(q,p)\mapsto (q,p+\alpha(q))$. Then, the flow of $H$ has no invariant set of the form $\tau_\alpha(\cL)$ with $\cL\in\mathscr{L}^!$, except for the zero section.
\end{coro}

In the context of KAM theory, this corollary implies that given a metric $g$ on a torus $\T^d$, all invariant tori of the geodesic flow in $T^*\T^d$ are destroyed when the metric is sufficiently perturbed to introduce a ``bump''.\\

In order to understand the gain in regularity in Theorem \ref{thm:Invariance}, we are led to study the set $\mathscr{L}^!(T^*M,\omega)$ of uniquely liftable Lagrangian subsets in order to understand the extent to which this theorem generalizes the results of \cite{MR2860422,MR3860396}.

\subsubsection{The Set $\mathscr{L}^!$ of Uniquely Liftable Lagrangian Subsets.}
For this definition to be meaningful, the set $\mathscr{L}^!(T^*M,\omega)$ of uniquely liftable Lagrangian subsets should at least contain all smooth closed exact Lagrangian submanifolds. We will see later that this is the case, and that it also contains the exact Lipschitz Lagrangian submanifolds introduced by Amorim–Oh–dos Santos in \cite{MR3860396}. Indeed, it contains a larger class of substantially less regular objects, but to describe them, we first need to introduce a new set of notations and definitions.

\begin{defi}\label{def:ExactSympeo}
    \begin{enumerate}
        \item A homeomorphism $\varphi\colon (U,\lambda_U)\to (V,\lambda_V)$ between exact symplectic manifolds is called a \textit{weakly exact symplectic homeomorphism} if there exists a sequence of $C^1$ exact symplectic diffeomorphisms $\varphi_n$ converging uniformly to $\varphi$ on compact sets.
        \item A homeomorphism $\varphi\colon (U,\lambda_U)\to (V,\lambda_V)$ between exact symplectic manifolds is called a \textit{strongly exact symplectic homeomorphism} if there exist a continuous function $S\colon U\to \R$, a sequence of $C^1$ diffeomorphisms $(\varphi_n)_{n\geq 0}$ and a sequence of $C^1$ functions $(S_n\colon U\to \R)_{n\geq 0}$ with $\varphi_n^*\lambda_V = \lambda_U +dS_n$ for all $n$, such that the sequences $(\varphi_n)$ and $(S_n)$ converge, respectively, to $\varphi$ and $S$ uniformly on compact sets. We will refer to a pair $(\varphi,S)$ satisfying this definition as a strongly exact symplectic homeomorphism brane.
        \item We denote by $\mathscr{L}_\text{Sympeo}$ (resp. $\mathscr{L}_\text{Sympeo}^\text{str}$) the set of images of closed exact smooth Lagrangians of $T^*M$ by weakly exact symplectic homeomorphisms (resp. strongly exact symplectic homeomorphisms) of $T^*M$.
        \item We denote by $\mathscr{L}_{\text{gr}}$ the set of elements of the form $\mathcal{L} = \varphi (\mathcal{G}(dv))$, where $v : M \to \R$ is a $C^1$ map and $\varphi$ is an exact $C^1$ symplectomorphism of $T^*M$.
    \end{enumerate}
\end{defi}

Weakly exact symplectic homeomorphisms of $T^*M$ form a group\footnote{Note that because the spaces we consider are locally compact and locally connected, uniform convergence of diffeomorphisms $\varphi_n$ on compact sets implies the same convergence for the sequence of inverses.}, and we will see in Proposition \ref{prop:strongExactSubgroup} that strongly exact homeomorphisms of $T^*M$ are a subgroup of it.

\begin{theor} \label{thm:UniquenessPrimitiveC0} We have the following results:
	\begin{enumerate}[label=\roman*.]
		\item \label{prop:UniquenessPrimitive3} There is a strict inclusion $\mathscr{L}_{\text{gr}}\subsetneq \mathscr{L}_\text{Sympeo}^\text{str}$.
        \item \label{prop:UniquenessPrimitive4} Strongly exact symplectic homeomorphisms preserve $\mathscr{L}^\exists(T^*M)$. Therefore, $\mathscr{L}_\text{Sympeo}^\text{str}$ is contained in $\mathscr{L}_\text{Sympeo}\cap\mathscr{L}^\exists$.
        \item \label{prop:UniquenessPrimitive6} The set $\mathscr{L}_\text{Sympeo}\cap\mathscr{L}^\exists$ is contained in $\mathscr{L}^!(T^*M)$.
   		\item \label{prop:UniquenessPrimitive5} For $M = \T^1$, let $\mathcal{L}$ be an exact Lagrangian submanifold of $T^*\mathbb{T}^1$, and let $\gamma : [0,1] \to T^*\mathbb{T}^1$ be a smooth embedded curve such that $\mathcal{L} \cap \im(\gamma) = \{ \gamma(0) \}$. The set $\mathcal{K} := \mathcal{L} \cup \im(\gamma)$ belongs to $\mathscr{L}^! \setminus (\mathscr{L}_\text{Sympeo}\cap\mathscr{L}^\exists)$.
	\end{enumerate}
    In summary, 
    \begin{equation}\label{eq:inclusionChain}
	   \mathscr{L}_{\text{gr}} \subsetneq  \mathscr{L}_{\text{Sympeo}}^{\text{str}} \subseteq \mathscr{L}_{\text{Sympeo}} \cap \mathscr{L}^\exists  \subseteq \mathscr{L}^!
    \end{equation} 
    where the last inclusion is strict in the one-dimensional case. In particular, the Birkhoff Theorem \ref{thm:Invariance} applies to these objects.
\end{theor}

Figure \ref{fig:ULLS} below illustrates various examples of uniquely liftable Lagrangian subsets.

\begin{figure}[h]
	\centering
	\begin{subfigure}{0.48\textwidth}
		\centering
		\includegraphics[width=\textwidth]{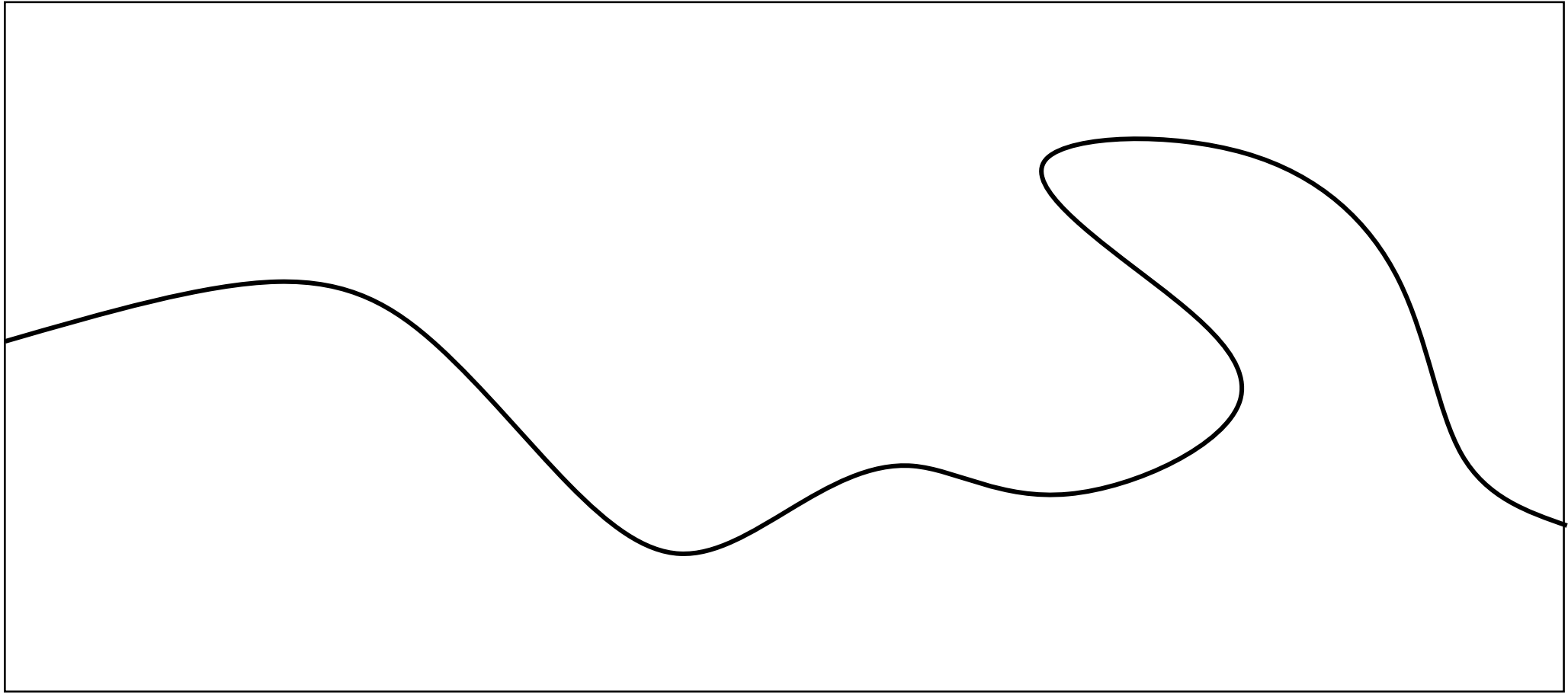}
		\caption{A regular Lagrangian submanifold.}
	\end{subfigure}
	\begin{subfigure}{0.50\textwidth}
		\centering
		\includegraphics[width=\textwidth]{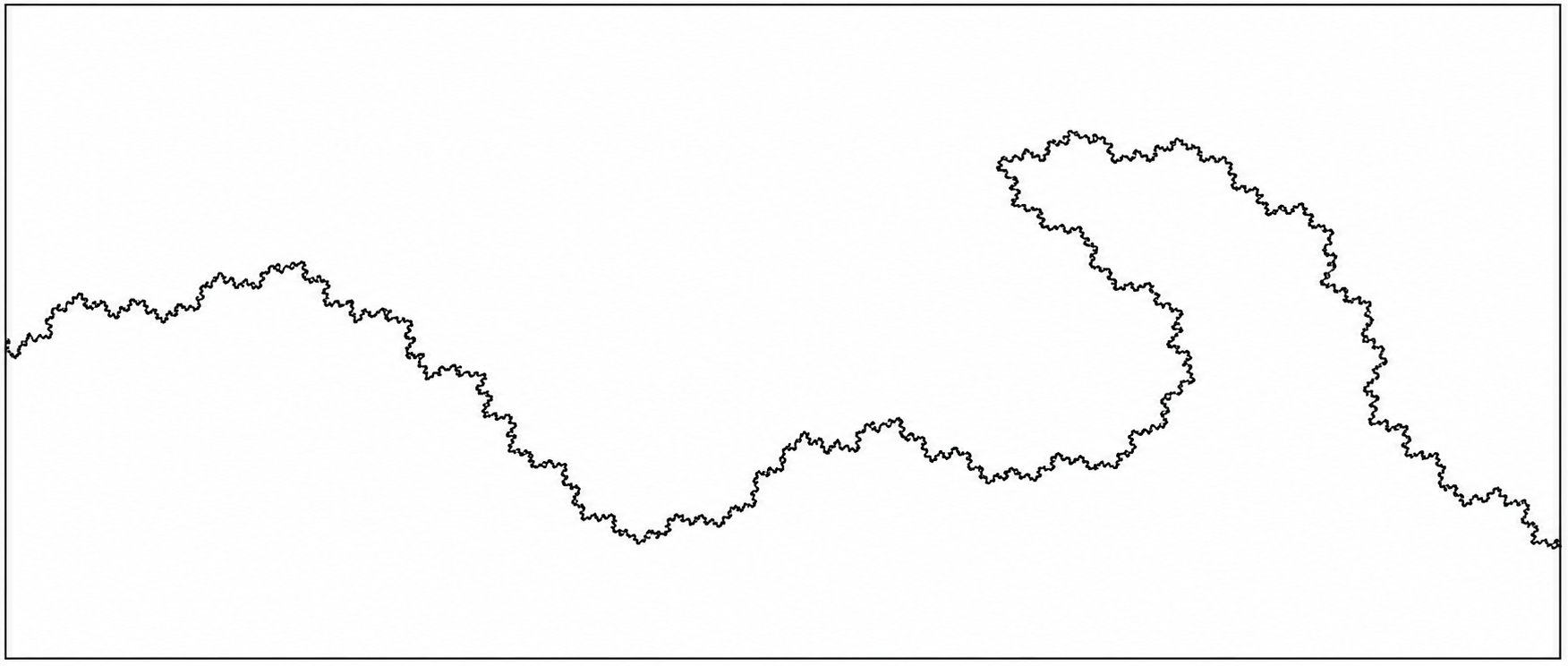}
		\caption{A $C^0$ Lagrangian submanifold $\mathcal{L} = \varphi (\mathcal{G}(dv))$.}
	\end{subfigure}
    \begin{subfigure}{0.48\textwidth}
		\centering
		\includegraphics[width=\textwidth]{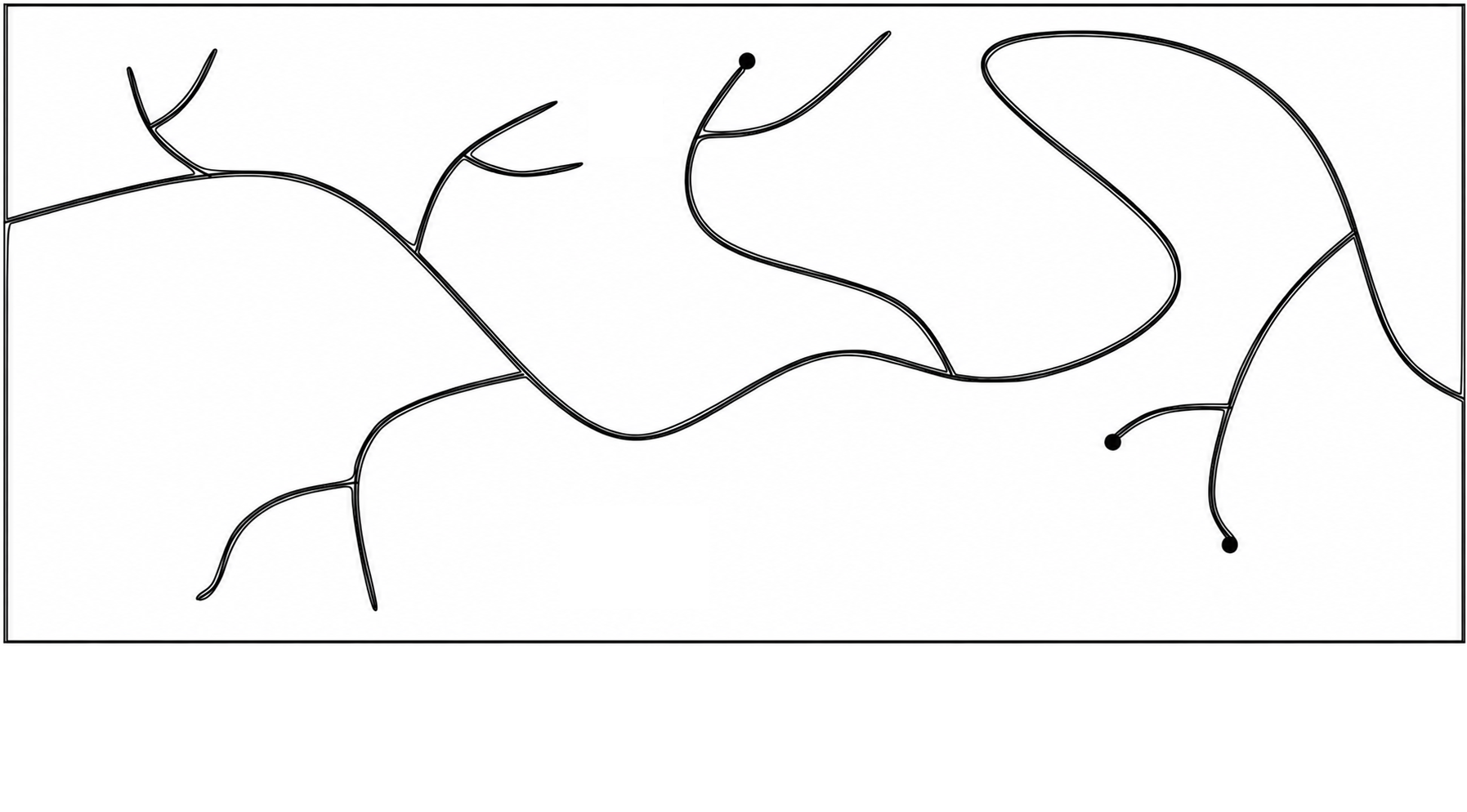}
		\caption{A $1$-dimensional `haired' Lagrangian subset.}
        \label{fig:ULLS:Fils}
	\end{subfigure}
	\begin{subfigure}{0.47\textwidth}
		\centering
		\includegraphics[width=\textwidth]{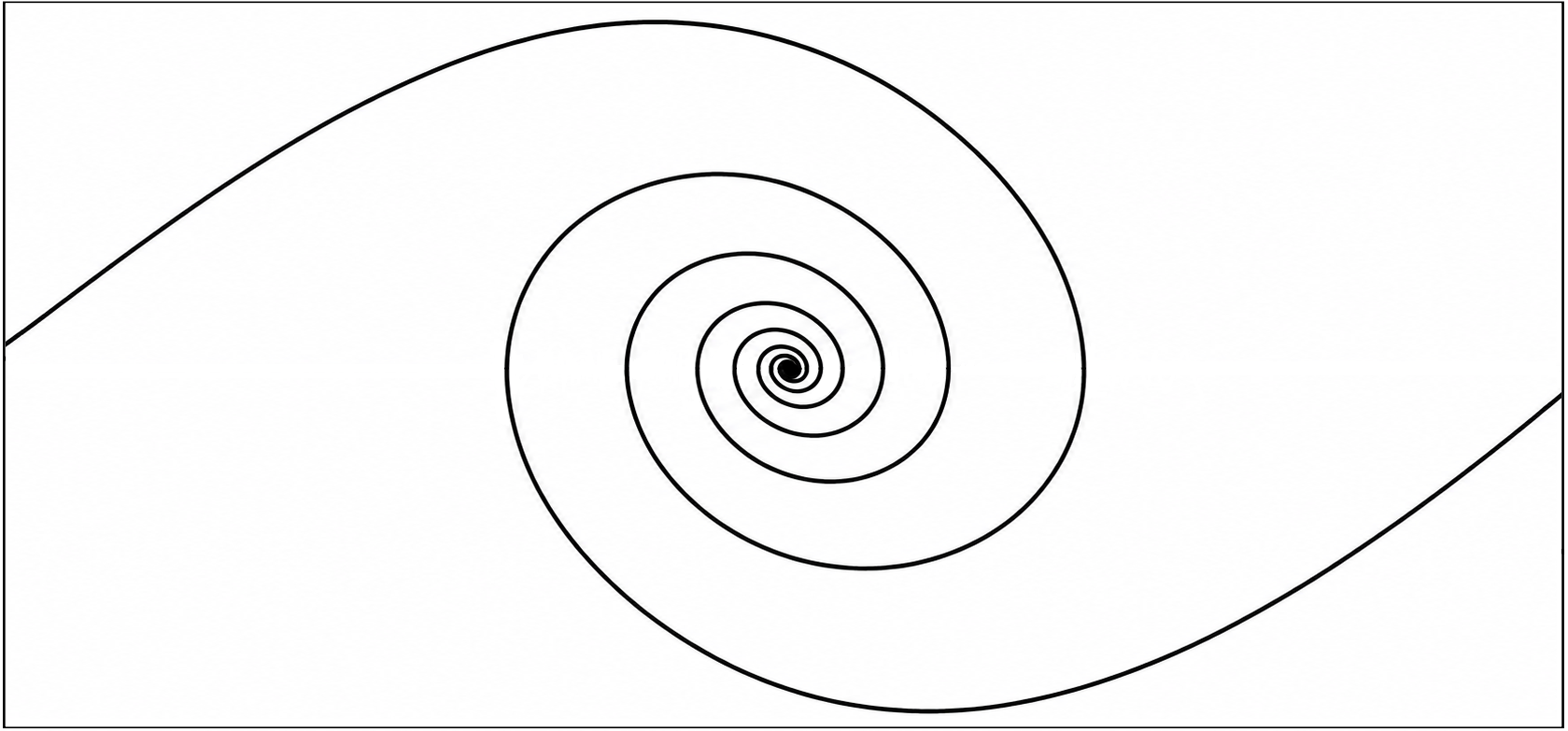}
		\caption{An element of $\mathscr{L}^!(T^*M,\omega) \setminus \mathscr{L}_{\text{gr}}$. If the winding around the origin is controlled, then $\cL$ is liftable to $(\cL, h)$ with a unique finite primitive $h$.}
        \label{fig:ULLS:Uzumaki}
	\end{subfigure}
	\caption{Various uniquely liftable Lagrangian subsets.}
	\label{fig:ULLS}
\end{figure}

\begin{rem}
    \begin{enumerate}
        \item The set $\mathscr{L}_\text{Sympeo}^\text{str}$ clearly contains all closed exact Lagrangians, which are therefore uniquely liftable as we claimed earlier.
    
        \item Although we do not prove this explicitly in the present article, Item \ref{prop:UniquenessPrimitive5} of Theorem \ref{thm:UniquenessPrimitiveC0} can be generalized, with the same proof, to any simple essential curve $\cL$ to which a finite or countable family of pairwise disjoint filaments $\gamma_i$ is attached, possibly with further filaments attached in turn to these filaments. This yields the same conclusion for sets of the type illustrated in Figure \ref{fig:ULLS:Fils}.
        
        \item To show that the inclusion of Item \ref{prop:UniquenessPrimitive3} is strict, we construct an element of $\mathscr{L}^{\text{str}}_{\text{Sympeo}} \setminus \mathscr{L}_{\text{gr}}$ obtained by twisting the zero section faster and faster near a point to obtain a singularity that cannot be turned into a graph via a $C^1$ symplectomorphism. The rate of divergence of the rotation speed needs to be controlled to ensure the Liouville primitive does not blow up at the singularity (see Figure \ref{fig:ULLS:Uzumaki}). We observe from this construction, in the case of a Lagrangian topological manifold, that the obstruction to being uniquely liftable does not arise from the regularity of the curve itself, but rather from the divergence of the associated primitive $h$. This suggests that Theorem \ref{thm:UniquenessPrimitiveC0} could potentially extend to any topological submanifold of dimension $d = \dim M$ that is liftable to a brane, i.e. belonging to the set $\mathscr{L}^\exists(T^*M,\omega)$. We introduce the following definition, inspired by Viterbo's definition of topological Lagrangian submanifolds, and needed to formulate the remaining questions precisely.
		\begin{defi}\label{def:C^0Lag} Recall that $d$ is the dimension of $M$.
        		\begin{enumerate}
           		\item (\cite{viterbo2026supportshumilierecompletiongammacoisotropic}) A topological Lagrangian submanifold is a $\gamma$-support that is a $d$-dimensional topological manifold.
                \item  A topological Lagrangian brane is a bare Lagrangian brane $(\cL,h)$ in $\overline{\mathcal{LB}(T^*M,\omega)}$, where $\cL$ is a $d$-dimensional topological manifold. 
       		\end{enumerate}
  		\end{defi}  
        We will show in Proposition \ref{prop:topologicalLagrangian} that topological Lagrangian branes are topological Lagrangian submanifolds.
  		Birkhoff's Theorem for topological Lagrangian branes can be translated from a dynamical question to a purely topological one, namely the following.
  		\begin{quest} \label{quest:regularity1}
            Are topological Lagrangian branes $(\cL,h)$ uniquely liftable, i.e. does $\cL$ belong to $\mathscr{L}^!(T^*M,\omega)$?
		\end{quest}

        \item In Birkhoff's original theorem in the cylinder $T^*\T^1$, there are continuous curves for which no primitive can be defined, and for which invariance still implies the Lipschitz graph property. Hence, our version of the theorem still lacks topological Lagrangian submanifolds for which the result remains valid. For instance, this includes submanifolds of the form $\mathcal{L} = \varphi(0_M) \in \mathscr{L}_{\text{Sympeo}} \setminus \mathscr{L}^\exists$. However, our proof no longer works for such examples, and one needs new ideas in order to attack this harder question. 
  		
        Observe that one cannot relax the regularity condition too much. In this regard, note that in Example \ref{exp:Pendulum}, the elliptic island can represent the $\gamma$-support of a Lagrangian in Humilière's completion. Since the elliptic island is not a graph, we can conclude that Birkhoff-type theorems do not apply to $\gamma$-supports in full generality.
        \begin{quest} \label{quest:regularity2}
			Does the multidimensional Birkhoff theorem hold for
		      \begin{enumerate}
			     \item topological submanifolds $\mathcal{L}\in\mathscr{L}_\text{Sympeo}$, without assuming they belong to $\mathscr{L}^\exists(T^*M)$?
			     \item topological Lagrangian branes, and topological Lagrangian submanifolds in general?
		      \end{enumerate}
		\end{quest}
    \end{enumerate}
    
\end{rem}

\subsubsection{A Birkhoff Theorem for Stratified Lagrangian Subsets} \label{section:ResultsStratifiedLag}

The examples in Theorem \ref{thm:UniquenessPrimitiveC0} and Figure \ref{fig:ULLS} illustrate some of the possible behaviours of uniquely liftable Lagrangian subsets. In particular, Birkhoff-type theorems may apply to sets that are not even topological manifolds, as shown by the example of the one-dimensional Lagrangian submanifold with an attached filament. However, in higher dimensions, we will see in Theorem \ref{thm:NonUniqueness2D} that such examples are no longer uniquely liftable. Nevertheless, one may still expect a Birkhoff theorem to hold for these stratified Lagrangian subsets.

\begin{theor} \label{thm:BirkhoffStratifiedLag}
	Let $M$ be a closed manifold, and let $H : \T^1 \times T^*M \to \R$ be a Tonelli Hamiltonian. Let $\cL \subset T^*M$ be a set satisfying the following properties:
	\begin{enumerate}[label=\roman*.]
		\item There exists a relatively open subset $U \subset \cL$ such that $\cL \setminus U$ is countable and, for every $x \in U$, there exists a relatively open neighbourhood $\cL_x$ of $x$ in $\cL$ which is a smooth isotropic submanifold of $T^*M$.
		
		\item There exists a continuous map $h : \cL \to \R$ such that $(\cL,h)$ belongs to $\overline{\mathcal{LB}(T^*M,\omega)}$ and, for every $x \in U$,
		\begin{equation}
			d(h|_{\cL_x}) = \lambda|_{\cL_x}
		\end{equation}
		\item $\phi_H^1(\cL) = \cL$.
	\end{enumerate}
	
	Then, $\cL$ is a Lipschitz graph over the base $M$.
\end{theor}

A less regular version is given by the following theorem.

\begin{theor} \label{thm:BirkhoffStratifiedLagC0}
	Let $M$ be a closed $d$-manifold, and let $H : \T^1 \times T^*M \to \R$ be a Tonelli Hamiltonian. Let $\cL \subset T^*M$ be a set satisfying the following properties:
	\begin{enumerate}[label=\roman*.]
		\item \label{item1BirkhoffStratifiedLagC0} There exists a relatively open subset $U \subset \cL$ such that $\cL \setminus U$ is countable and, for every $x \in U$, there exists a relatively open neighbourhood $\cL_x$ of $x$ in $\cL$ of the form
		\begin{equation*}
			\cL_x = \Phi_x(U_x \cap 0_{N_x})
		\end{equation*}
		where $N_x$ is a $d$-manifold and $(\Phi_x : U_x \to V_x,S_x)$ is a strongly exact symplectic homeomorphism brane from an open set $U_x \subset T^*N_x$ onto a neighbourhood $V_x$ of $\cL_x$.
		
		\item There exists a continuous map $h : \cL \to \R$ such that $(\cL,h)$ belongs to $\overline{\mathcal{LB}(T^*M,\omega)}$ and, for every $(q,0) \in U_x \cap 0_{N_x}$,
		\begin{equation} \label{thm:BirkhoffStratifiedLagC0:eq:h}
			h(\Phi_x(q,0)) = S_x(q,0)
		\end{equation}
		
		\item $\phi_H^1(\cL) = \cL$.
	\end{enumerate}
	
	Then, $\cL$ is a Lipschitz graph over the base $M$.
\end{theor}

\begin{rem} \label{rem:BirkhoffStrata}
	\begin{enumerate}
		\item In Theorems \ref{thm:BirkhoffStratifiedLag} and \ref{thm:BirkhoffStratifiedLagC0}, one obstruction to the existence of the map $h$ is the presence of Lagrangian loops enclosing non-zero symplectic area, since such loops prevent local Liouville primitives from extending to a global one. In the $C^0$ setting of Theorem \ref{thm:BirkhoffStratifiedLagC0}, one additionally needs the local strata to admit strongly exact symplectic models of the form appearing in Item \ref{item1BirkhoffStratifiedLagC0} so that a local primitive can be defined.
		
		For instance, if $\cL$ deformation retracts onto an exact Lagrangian submanifold $\cL_0 \subset \cL$, and if the geometry of the strata is sufficiently controlled so that the local Liouville primitives extend to a continuous global primitive $h$, then the theorem applies. Moreover, we will see in the Flexibility Theorem \ref{thm:NonUniqueness2D} that, in dimension at least $2$, such pairs $(\cL,h)$ belong to $\overline{\mathcal{LB}(T^*M,\omega)}$ whenever $\cL$ is obtained as a compact connected enlargement of an exact Lagrangian brane.
        
		\item As in Item \ref{prop:UniquenessPrimitive5} of Theorem \ref{thm:UniquenessPrimitiveC0}, in the one-dimensional case, one can show that the Lagrangian subsets considered in Theorem \ref{thm:BirkhoffStratifiedLagC0} belong to $\mathscr{L}^!$. Hence, Theorem \ref{thm:Invariance} applies directly. However, this one-dimensional result was already known through Fathi's version of Birkhoff's theorem in Herman's Astérisque \cite{MR728564}, which states that the boundary components of any essential annulus of the cylinder invariant under a symplectic twist map are Lipschitz graphs over the zero section.
		
		\item The higher-dimensional case of our theorem is new, and the general Theorem \ref{MainThmInvariant} should apply to a broader class of sets than those described in Theorems \ref{thm:BirkhoffStratifiedLag} and \ref{thm:BirkhoffStratifiedLagC0}. For instance, the assumption that $\cL \setminus U$ is countable is not optimal. For instance, it can be replaced by the weaker condition that $\cL \setminus U$ has countably many connected components, that they are all smooth isotropic submanifolds, and that the restriction of $h$ to each of them is a Liouville primitive.
		
		\item The difference between Theorems \ref{thm:BirkhoffStratifiedLag} and \ref{thm:BirkhoffStratifiedLagC0} lies in the dimension of the strata. In the smooth version, the strata are only assumed to be isotropic and may therefore have dimension strictly smaller than $d$. In the $C^0$ version, by contrast, the proof requires the strata to have the full Lagrangian dimension $d$.

        \item It is worth emphasizing that the absence of invariant filamented exact Lagrangian submanifolds is purely symplectic. Indeed, invariant sets with hairs may appear in the conformally symplectic setting, as is the case for global attractors containing multidimensional Birkhoff attractors; see \cite{BSMF_1932__60__1_0,MR1505024,MR951271,MR5052827}. Note also that our definition naturally produces objects with hairs, since we consider the limit sets of sequences of Lagrangian submanifolds themselves, rather than only the $\gamma$-support of their limits in Humilière's completion. This contrasts with the definition of multidimensional Birkhoff attractors given by Arnaud, Humilière and Viterbo in \cite{MR5052827}.
	\end{enumerate}
\end{rem}

\subsection{Further Results: Rigidity and Flexibility Properties of Bare Lagrangian Branes} 
In this section, we describe some rigidity and flexibility results on bare Lagrangian branes.

\paragraph{Rigidity of bare Lagrangian branes.}
Although we already saw that smooth closed exact Lagrangians are uniquely liftable, we can also view it as a consequence of a more general rigidity statement.

\begin{theor}[Rigidity of bare Lagrangian branes]\label{thm:rigidity}
    Let $(\cL,h)\in \overline{\mathcal{LB}(T^*M)}$, with $\cL$ a topological manifold of dimension $d$. Then, on any open subset of $\cL$ where $\cL$ is $C^1$, $\cL$ is Lagrangian and $h$ is a Liouville primitive:
    \[\lambda|_U=dh|_U.\]
    In particular, if $\cL\in\mathscr{L}^\exists$ is a $C^1$ manifold, then it is an exact Lagrangian submanifold and $h$ is a Liouville primitive, and therefore $\cL\in\mathscr{L}^!$.
\end{theor}

\begin{rem}
    \begin{enumerate}
        \item This theorem can be reformulated as a Legendrian rigidity statement: if the Legendrian lifts of closed exact Lagrangian submanifolds converge in the Hausdorff topology within $J^1M$ to the graph $\Gamma_{(\cL,h)}$ of a function $h$ over a topological manifold $\cL$ of dimension $d$, then $\Gamma_{(\cL,h)}$ is Legendrian over any smooth part of $\cL$. Note that the rigidity of Legendrian submanifolds is false in general. For instance, in dimension 3, one can show, using Gromov's $h$-principle for Legendrian immersions \cite{MR864505}, that any smooth knot can be approached in the Hausdorff topology by Legendrian knots (see \cite{MR2179261,MR2397738,MR3619673}), and therefore a smooth Hausdorff limit of Legendrian knots can fail to be Legendrian. However, some rigidity results for Legendrians have been established under additional assumptions, as in \cite{nakamura2020c0limitslegendriansubmanifolds,MR4719181,MR4876451,asano2025c0rigiditylegendrianscoisotropicssheaf}. Even though our setting is more restrictive, our result is stronger than what already exists in the literature in two aspects. Firstly, our result is semi-local: we only need to assume smoothness locally to get the conclusion locally, but under the global assumption that the projection $\cL$ of the limit $\Gamma_{(\cL,h)}$ into $T^*M$ is a topological manifold of dimension $d$. Secondly, we only need to assume smoothness of the projection $\cL$, and not of $\Gamma_{(\cL,h)}$ itself.
        \item If we wanted to obtain the rigidity in the case when $\cL$ is assumed to be globally a $C^1$ submanifold from existing results in the literature on Lagrangian rigidity, then from the convergence of smooth (exact) Lagrangians $\cL_n$ to $\cL$ in the Hausdorff topology, one could apply \cite{MR1266111} to obtain that $\cL$ is Lagrangian. To get exactness of $\cL$, one could think of applying \cite{MR4334195}. However, their result is stated for convergence under symplectic homeomorphisms, and it seems that their proof extends to the Hausdorff convergence context only if we additionally assume that the projections $\pi_n\colon \cL_n\to\cL$ induce surjections on $H_1(\cL)$ for large $n$. This additional assumption does not seem to immediately follow from our hypotheses alone, even though the result of our theorem implies that it is true: if $\cL$ is assumed to be exact, then applying \cite{MR3849284} in an exact Weinstein neighbourhood of $\cL$ gives that for sufficiently large $n$, the maps $\pi_n$ are simple homotopy equivalences.
        \item This theorem cannot be improved to a completely local one in dimension $d\geq2$. Indeed, if $\cL$ is the union of the zero section and a smooth $d$-dimensional closed disc $D$ intersecting transversely, then we will see in Theorem \ref{thm:NonUniqueness2D} that there exists $h\colon \cL\to\R$ such that $(\cL,h)\in\overline{\mathcal{LB}(T^*M)}$ without $h$ being a Liouville primitive on any smooth open set of $D$. Even worse, $D$ does not have to be Lagrangian.
    \end{enumerate}
\end{rem}

\paragraph{Flexibility phenomena in higher dimensions.}

After presenting several results concerning uniquely liftable branes, we now turn to a flexibility phenomenon that destroys uniqueness in dimensions greater than or equal to $2$. We have already mentioned that the higher-dimensional analogues of the stratified Lagrangian sets illustrated in Figure \ref{fig:ULLS:Fils} are not uniquely liftable. This lack of uniqueness is explained by the flexibility phenomenon expressed in the following theorem.

\begin{theor} \label{thm:NonUniqueness2D}
	Assume that $d = \dim M \geq 2$. Let $(\mathcal{C},c)$ be in $\overline{\mathcal{LB}(T^*M,\omega)}$, and let $\mathcal{K}$ be a compact connected set containing $\mathcal{C}$. Then, for any continuous map $k : \mathcal{K} \to \R$ whose restriction to $\mathcal{C}$ is $c$, the bare brane $(\mathcal{K},k)$ belongs to $\overline{\mathcal{LB}(T^*M,\omega)}$. 
    
    In particular, if $\mathcal{C}$ belongs to $\mathscr{L}^\exists$, then every compact connected set $\mathcal{K}$ containing $\mathcal{C}$ also belongs to $\mathscr{L}^\exists$. Moreover, if $\mathcal{K}$ strictly contains $\mathcal{C}$, then $\mathcal{K}\notin\mathscr{L}^!$.
\end{theor}

\begin{rem}
    \begin{enumerate}
        \item In the one-dimensional case, we have seen in Item \ref{prop:UniquenessPrimitive5} of Theorem \ref{thm:UniquenessPrimitiveC0} that the second statement is false: both $\cL$ and $\cL \cup \im(\gamma)$ belong to $\mathscr{L}^!$.
	
	    The first statement is also false in dimension one. Indeed, one can construct compact sets containing an exact Lagrangian submanifold that do not belong to $\mathscr{L}^\exists$.
	    \begin{prop} \label{prop:NonExistence1D}
		    Let $C \subset T^*\mathbb{T}^1$ be a smooth embedded contractible circle tangent to the zero section at $(0,0)$. Then, $\mathcal{K}:=0_{\mathbb{T}^1}\cup C$ does not belong to $\mathscr{L}^\exists$.
	    \end{prop}

        \item Theorem \ref{thm:NonUniqueness2D} asserts that every element of $\mathscr{L}^!$ is minimal, with respect to inclusion, among subsets of $\mathscr{L}^\exists$. However, the converse is false: minimality with respect to inclusion does not imply uniqueness. Indeed, in Example \ref{exp:MinimalNonUnique}, we construct a compact set $\mathcal{K}$ which is minimal with respect to inclusion in $\mathscr{L}^\exists$, but does not belong to $\mathscr{L}^!$. Although the construction is given in dimension one, it can be generalized to higher dimensions.

        \begin{prop} \label{prop:NonUniqueness&Minimality}
            There exists an element $\mathcal{K}$ in $\mathscr{L}^\exists$ that is minimal with respect to inclusion in $\mathscr{L}^\exists$ and does not belong to $\mathscr{L}^!$.
        \end{prop}

        More strongly, from the viewpoint of the spectral geometry of the Humilière completion, the same example is constructed so that $\mathcal{K}$ is the $\gamma$-support of an element $\underline{\cL}$ of the Humilière completion, is minimal with respect to inclusion among $\gamma$-supports, and nevertheless does not determine $\underline{\cL}$ uniquely. Thus, minimality of the $\gamma$-support does not imply uniqueness of the corresponding element of the completion. This gives a negative answer to one of the questions listed in \cite[Question 7.7]{viterbo2026supportshumilierecompletiongammacoisotropic}.

		\item  From the viewpoint of bare Lagrangian branes, as soon as such a compact set is strictly enlarged, one can construct two distinct classes $[h]$ on the larger set. The situation is more subtle for $\gamma$-supports: given a compact set $\mathcal{K}$ which is a $\gamma$-support, one may ask whether there exists a unique element of the Humilière completion whose $\gamma$-support is $\mathcal{K}$. This is not directly controlled by the uniqueness of the lift $[h]$. Indeed, a sequence of Lagrangian branes may brane-converge to a compact set $\mathcal{K}$, while the $\gamma$-support of the corresponding limit in the Humilière completion is strictly contained in $\mathcal{K}$. Thus, the bare-brane limit retains more geometric information than the $\gamma$-support alone.

        Hence, one may ask the minimality question for $\gamma$-supports. Assume that a $\gamma$-support $\mathcal{K}$ determines a unique element of the Humilière completion. Is this uniqueness necessarily lost as soon as one passes to a strictly larger $\gamma$-support? Contrary to the case of lifts $[h]$, there is no immediate way to construct two distinct elements of the completion after enlarging the support.

		\begin{quest}
			Let $\mathcal{K}$ be a $\gamma$-support. If $\mathcal{K}$ determines a unique element of the Humilière completion, is $\mathcal{K}$ minimal, with respect to inclusion, among $\gamma$-supports? In particular, does every strictly larger $\gamma$-support fail to determine a unique element?
		\end{quest}
    \end{enumerate}
\end{rem}

\subsection{Proof Sketch and Structure of the Article}

The article is divided into two main parts. The dynamical Part \ref{part:Dynamics} is devoted to the proof of the main Theorem \ref{MainTheorem}, while the more geometric Part \ref{part:Geometry} studies bare Lagrangian branes in order to better understand the gain in regularity provided by the main theorem.\\

In Part \ref{part:Dynamics}, the proof of Theorem \ref{MainTheorem} proceeds as follows. We associate to any bare Lagrangian brane $(\cL,h)$ a weak solution of the Hamilton--Jacobi equation \eqref{HJ}, called a variational solution. These solutions are defined through graph selectors of the associated variational branes $(\cL_t,h_t)$. Their construction relies on Floer homology, introduced in Section \ref{section:Floer}, where we first define graph selectors and variational solutions for regular Lagrangian branes. This section also recalls the spectral distance and the Humilière completion, which will later be used in the rigidity results of Section \ref{Section:regularity}. In Section \ref{section:BraneDist}, we introduce the brane topology and extend graph selectors and variational solutions to bare Lagrangian branes.

The main objective is then to prove that, under the assumptions of Theorem \ref{MainTheorem}, these variational solutions are in fact of class $C^{1,1}$. Since they are defined as graph selectors, this regularity yields (with some additional work) the equality $\cL=\mathcal{G}(du)$ and hence shows that $\cL$ is a Lipschitz graph over the zero section. To obtain this regularity, we use weak KAM theory, recalled in Section \ref{SectionCalib}. Although variational solutions do not necessarily coincide with viscosity solutions, Proposition \ref{Domination} shows that they are viscosity subsolutions, which is sufficient for our purposes. In particular, weak KAM theory provides calibrated curves along which such subsolutions enjoy $C^{1,1}$ regularity. The proof therefore reduces to finding suitable calibrated curves through every point. This is carried out in Section \ref{section:Proof}, which contains Corollary \ref{cor:geodesic}.\\

Part \ref{part:Geometry} is devoted to the geometry of bare Lagrangian branes. We first study uniquely liftable Lagrangian subsets, for which the Birkhoff theorem applies directly, and show that this class contains new low-regularity examples. 

Section \ref{Section:regularity} is dedicated to the study of uniquely liftable Lagrangian subsets, namely to the proofs of Theorems \ref{thm:rigidity} and \ref{thm:UniquenessPrimitiveC0} (with the proof of Item \ref{prop:UniquenessPrimitive5} of Theorem \ref{thm:UniquenessPrimitiveC0} delayed to Section \ref{sec:BirkhoffStratified}). It begins with some preliminaries on microlocal sheaf theory, necessary for the proofs of Theorem \ref{thm:rigidity} and Item \ref{prop:UniquenessPrimitive6} of Theorem \ref{thm:UniquenessPrimitiveC0}.

Finally, Section \ref{sec:BirkhoffStratified} studies liftable Lagrangian subsets without assuming uniqueness of the lift. This is necessary because Theorem \ref{MainThmInvariant} also applies to sets that are not uniquely liftable. We prove Theorem \ref{thm:NonUniqueness2D}, which clarifies the structure of the class of liftable Lagrangian subsets and whose proof uses ideas inspired by convex integration. In particular, this result shows that stratified Lagrangian sets provide a large family of liftable examples. We then establish Theorems \ref{thm:BirkhoffStratifiedLag} and \ref{thm:BirkhoffStratifiedLagC0}.

\paragraph{Acknowledgements:}

Both authors thank Noah Porcelli for answering some of our questions about the works \cite{MR2596633,MR2565051,MR2947545} on exact Lagrangians in cotangent bundles, as well as Claude Viterbo for discussions about the Humilière completion and microlocal sheaf theory. We also thank Marie-Claude Arnaud and Patrick Bernard for their questions, which redirected our attention toward the uniqueness of Liouville primitives and helped shape the subsequent development of the project. We thank the FIM, which funded the visit of S.C. at ETH Zürich, where this collaboration started. I.T. thanks his advisor Sobhan Seyfaddini for his continuous feedback on this project, as well as Jean-Philippe Chassé for his suggestion to look into the rigidity properties of Legendrian submanifolds, and Leonid Polterovitch for his comments on the first version of this article. I.T. was partially funded by the ERC Starting Grant number 851701. We used LLMs for proofreading, grammar checking, and literature search. All the results were written and proved by the authors.

\ppart{Proof of the General Birkhoff Theorem \ref{MainTheorem}} \label{part:Dynamics}

\section{Floer-Theoretic Variational Solutions} \label{section:Floer}

\subsection{Lagrangian Floer Homology and Spectral Invariants}

In this section, we recall the definition of the Floer homology of a pair of exact Lagrangians, and of the associated spectral invariant.
To avoid orientation issues, we work with $\Z/2\Z$ coefficients. Our setting is similar to that of \cite{MR3860396}; however, since we will always consider at least one compact Lagrangian, we do not need to work with wrapped Floer homology. We will only give a quick overview of the theory; a more detailed exposition of Lagrangian Floer homology and spectral invariants inside a cotangent bundle can be found in \cite{MR3524783}.

\subsubsection{Lagrangian Floer Homology}

Let $(\cL_0,h_0)$, $(\cL_1,h_1)$ be Lagrangian branes in $T^*M$. We assume that at least one of $\cL_0$ and $\cL_1$ is closed (say for instance $\cL_0$), and that they intersect transversely. If $\cL_1$ is non-compact, we also assume that it is conical at infinity and has vanishing Maslov class (we will mainly consider the case of a cotangent fibre, for which the assumptions hold). We define the Floer chain complex as
\[CF_*(\cL_0, \cL_1)=\bigoplus\limits_{x\in\cL_0\cap \cL_1}\Z/2\Z\cdot x.\]
By a theorem of Abouzaid and Kragh \cite{MR3070514}, closed exact Lagrangians in $T^*M$ also have vanishing Maslov class, so the complex carries a relative $\Z$-valued Maslov grading. We denote by $\mu(x,y)$ the relative grading between two intersection points $x$ and $y$.

This vector space can be equipped with a filtration as follows. We define the action of a generator $x\in\cL_0\cap\cL_1$ as $\mathcal A_{\cL_0,\cL_1}(x) = h_0(x)-h_1(x)$, and denote by $CF^{\leq a}_*(\cL_0, \cL_1)$ the subvector space generated by intersection points of action less than or equal to $a\in\R$. Note that the action $\mathcal A_{\cL_0,\cL_1}$ depends on the choice of Liouville primitive on each Lagrangian submanifold, but we decide to omit it from the notation.

Fix a smooth family of $\omega$-compatible almost complex structures $\{J_t\}_{t\in[0,1]}$ on $T^*M$. For $x,y\in \cL_0\cap \cL_1$, we denote by $\widehat {\mathcal M}(x,y)$ the set of $J_t$-holomorphic strips $u:\R\times [0,1]\to T^*M$, with boundary conditions $u(s,0)\in\cL_0$, $u(s,1)\in \cL_1$, and asymptotic to $x$ at $s \to -\infty$ and to $y$ at $s \to +\infty$. The $J_t$-holomorphic condition means that $u$ satisfies the equation
\[d_{(s,t)}u\circ j=J_t\circ d_{(s,t)}u,\]
where $j$ denotes the standard complex structure on the strip $\R_s\times [0,1]_t$ mapping $\partial_s$ to $\partial_t$. This equation is equivalent to $\partial_su+J_t\partial_tu=0$, and therefore such a curve has a non-negative symplectic area
\[\int_{\R \times [0,1]} u^*\omega =\int_{\R \times [0,1]}\vert\partial_tu\vert^2\geq0\]
where $\vert\cdot\vert$ is the norm associated to the Riemannian metric $\omega(\cdot,J_t\cdot)$, and with equality if and only if $u$ is constant. On the other hand, since $\omega=-d\lambda$ and $h_0$, $h_1$ are Liouville primitives on $\cL_0$ and $\cL_1$, we also have by a straightforward computation that the symplectic area of $u$ coincides with the difference of action between its endpoints:

\begin{equation}
\label{eq:energy}
    \int_{\R \times [0,1]} u^*\omega = \mathcal A_{\cL_0,\cL_1}(x)-\mathcal A_{\cL_0,\cL_1}(y).
\end{equation}

Let $\mathcal M(x,y)$ denote the quotient of $\widehat {\mathcal M}(x,y)$ by the $\R$-action by translation $\tau\cdot u = u(\cdot+\tau,\cdot)$.
Then, for a generic $\{J_t\}$, $\mathcal M(x,y)$ is a smooth manifold of dimension $\mu(x,y)-1$. In particular, when $\mu(x,y)=1$, $\mathcal M(x,y)$ is a zero-dimensional manifold. 
Since at least one of the Lagrangians is compact and the intersection is transverse, $\cL_0\cap\cL_1$ is finite. The energy identity therefore gives a uniform energy bound for all Floer strips. By choosing $J_t$ of contact type outside a compact set, and since $\cL_1$ is either compact or conical at infinity, a maximum principle prevents Floer strips from escaping to infinity. Gromov compactness can then be applied. Exactness excludes nonconstant sphere and disc bubbling. Therefore, the zero-dimensional manifold $\mathcal M(x,y)$ is also compact, and thus consists of a finite number of points. Hence we can define the differential $\partial : CF_*(\cL_0,\cL_1)\rightarrow CF_{*-1}(\cL_0,\cL_1)$ on generators by
\begin{equation} \label{eq:bounrdary map}
    \partial x = \sum\limits_{\substack{y\in\cL_0\cap\cL_1\\ \mu(x,y)=1}}\#_2\mathcal M(x,y)\cdot y
\end{equation}
where $\#_2$ denotes the cardinality modulo 2.

This differential satisfies $\partial^2=0$, and therefore gives rise to a well-defined homology denoted $HF_*(\cL_0,\cL_1)$, whose isomorphism class does not depend on the choice of almost complex structure (even though $\partial$ does depend on $\{J_t\}$).

By \eqref{eq:energy}, the differential is action-decreasing, and therefore restricts to a well-defined differential on $CF^{\leq a}_*(\cL_0, \cL_1)$. We denote by $HF^{\leq a}_*(\cL_0,\cL_1)$ its homology.

\subsubsection{Lagrangian Spectral Invariants}
As in the previous paragraph, we consider two Lagrangian branes $(\cL_0,h_0)$ and $(\cL_1,h_1)$ in $T^*M$, such that $\cL_0$ is closed, and $\cL_0$ and $\cL_1$ intersect transversely. Now, we also assume that the homology $HF_*(\cL_0,\cL_1)$ is nontrivial, and we pick a non-zero homology class $\alpha\in HF_*(\cL_0,\cL_1)$. 
For $a\in\R$, denote by $\iota^a$ the inclusion $CF^{\leq a}_*(\cL_0, \cL_1)\hookrightarrow CF_*(\cL_0,\cL_1)$. It induces a map $\iota^a_*$ in homology, so that we can define the spectral invariant
\begin{equation}
    \rho(\alpha,\cL_0,\cL_1)=\inf\{a\in\R, \alpha\in\operatorname{Im}(\iota^a_*)\}.
\end{equation}

We now list some properties of the spectral invariants. Again, more details can be found in \cite{MR3524783}.

\begin{prop}[Spectrality]
    \label{prop:spectrality}
    If $\cL_0$ and $\cL_1$ intersect transversely, then $\rho(\alpha,\cL_0,\cL_1)\in\Spec(\cL_0,\cL_1)\coloneq\{\mathcal A_{\cL_0,\cL_1}(x),x\in \cL_0\cap\cL_1\}$.
\end{prop}

\begin{prop}
\label{prop:homology_invariance}
    Let $\varphi$ be a Hamiltonian diffeomorphism of $T^*M$, generated by some Hamiltonian $H:[0,1]\times T^*M\to\R$. For $i\in\{0,1\}$, we equip $\varphi(\cL_i)$ with the Liouville primitive given by
    \[\widetilde h_i(\varphi(x))=h_i(x)+\int_0^1\lambda(X^t_H(\phi^t_H(x)))-H_t(\phi^t_H(x))\;dt.\]
    Then,
    \begin{enumerate}
        \item if $\varphi(\cL_0)$ and $\cL_1$ intersect transversely, there exists a \textbf{continuation map} $\Phi$, which is a chain map $CF_*(\cL_0,\cL_1)\to CF_*(\varphi(\cL_0),\cL_1)$ that induces a natural isomorphism in homology. Moreover, if $H$ is bounded, we have
        \[|\rho(\Phi_*\alpha,\varphi(\cL_0),\cL_1)-\rho(\alpha,\cL_0,\cL_1)|\leq \int_0^1\Vert H_t\Vert_\infty\;dt;\]
        \item there exists a chain isomorphism $\Psi:CF_*(\cL_0,\cL_1)\to CF_*(\varphi(\cL_0),\varphi(\cL_1))$\footnote{Note that if $\cL_1$ is non-compact, $\varphi(\cL_1)$ could fail to be conical at infinity. But transporting everything by $\varphi$ still gives a well-defined Floer complex $CF_*(\varphi(\cL_0),\varphi(\cL_1))$} that preserves the action filtration, and therefore
        \[\rho(\alpha,\cL_0,\cL_1)=\rho(\Psi_*\alpha,\varphi(\cL_0),\varphi(\cL_1)).\]
    \end{enumerate}
\end{prop}

This enables us to extend the definition of the homology $HF_*(\cL_0,\cL_1)$ and of the spectral invariant $\rho(\alpha,\cL_0,\cL_1)$ to the case when $\cL_0$ and $\cL_1$ do not intersect transversely: consider a sequence $(H^n)_{n\in\N}$ of compactly supported Hamiltonians such that for all $n$, $\phi^1_{H^n}(\cL_0)$ and $\cL_1$ intersect transversely, and such that $\lim\limits_{n\to\infty}\int_0^1\Vert H^n_t\Vert_\infty\;dt=0$. 
Hamiltonian continuation provides natural isomorphisms between the corresponding Floer homology groups $HF_*(\phi^1_{H^n}(\cL_0),\cL_1)$, satisfying the usual composition properties. Thus one may identify these groups with the Hamiltonian-invariant Floer homology $HF_*(\cL_0,\cL_1)$.
For a non-zero homology class $\alpha$ in $HF_*(\cL_0,\cL_1)$, we still denote by $\alpha$ the corresponding class in all the $HF_*(\phi^1_{H^n}(\cL_0),\cL_1)$. Then, by the above Proposition, and since $\lim\limits_{n\to\infty}\int_0^1\Vert H^n_t\Vert_\infty\;dt=0$, we also have that the sequence $\big(\rho(\alpha,\phi^1_{H^n}(\cL_0),\cL_1)\big)_{n\in\N}$ is Cauchy, and therefore we can define $\rho(\alpha,\cL_0,\cL_1)$ as its limit.

Then, the properties of the spectral invariants listed in Proposition \ref{prop:homology_invariance} still hold in this setting:

\begin{prop}
\label{prop:spectral_invariants}
    Let $(\cL_0,h_0)$ and $(\cL_1,h_1)$ be two Lagrangian branes, with $\cL_0$ closed. Assume that $HF_*(\cL_0,\cL_1)$ is nontrivial, and let $\alpha$ be a non-zero homology class.
    Let $\varphi$ be a Hamiltonian diffeomorphism of $T^*M$, generated by some Hamiltonian $H:[0,1]\times T^*M\to\R$. For $i\in\{0,1\}$, we equip $\varphi(\cL_i)$ with the Liouville primitive given by
    \[\widetilde h_i(\varphi(x))=h_i(x)+\int_0^1\lambda(X^t_H(\phi^t_H(x)))-H_t(\phi^t_H(x))\;dt.\]
    Then,
    \begin{enumerate}
        \item ($L^{1,\infty}$-Lipschitz) if $H$ is bounded, we have
        \[|\rho(\alpha,\varphi(\cL_0),\cL_1)-\rho(\alpha,\cL_0,\cL_1)|\leq \int_0^1\Vert H_t\Vert_\infty\;dt;\]
        \item (Invariance) $\rho(\alpha,\cL_0,\cL_1)=\rho(\alpha,\varphi(\cL_0),\varphi(\cL_1))$,
    \end{enumerate}
    where we still denote by $\alpha$ its own image under continuation isomorphisms.
\end{prop}

\subsubsection{Product Structure and Triangle Inequality}

Consider three Lagrangian branes $(\cL_0,h_0)$, $(\cL_1,h_1)$, and $(\cL_2,h_2)$, with at least two of them closed, and the third one conical at infinity if it is not compact. Assume the pairwise intersections between the Lagrangian submanifolds are transverse, so that the Floer complexes $CF_*(\cL_0,\cL_1)$, $CF_*(\cL_1,\cL_2)$, and $CF_*(\cL_0,\cL_2)$ are well-defined (for some generic choice of a smooth family of almost complex structures).
Then, there exists a chain map $CF_*(\cL_0,\cL_1)\otimes CF_*(\cL_1,\cL_2)\to CF_*(\cL_0,\cL_2)$ defined by counting ``rigid'' pseudo-holomorphic triangles between generators of the three complexes, with Lagrangian boundary conditions. This chain map may depend on the choice of almost complex structures, but different choices give chain-homotopic maps, so that they descend to a well-defined associative product map $*:HF_*(\cL_0,\cL_1)\otimes HF_*(\cL_1,\cL_2)\to HF_*(\cL_0,\cL_2)$. Using the Hamiltonian invariance of Floer homology, this product is still well-defined without the transversality assumption.

The spectral invariants behave well with respect to this product structure:

\begin{prop}[Triangle inequality]\label{prop:TriangleIneq}
    Let $(\cL_0,h_0)$, $(\cL_1,h_1)$, and $(\cL_2,h_2)$ be three Lagrangian branes in $T^*M$, with at least two of them closed.
    Assume that there exist non-zero homology classes $\alpha\in HF_*(\cL_0,\cL_1)$ and $\beta\in HF_*(\cL_1,\cL_2)$, whose product $\alpha*\beta\in HF_*(\cL_0,\cL_2)$ is also non-zero. Then,
    \[\rho(\alpha*\beta,\cL_0,\cL_2)\leq\rho(\alpha,\cL_0,\cL_1)+\rho(\beta,\cL_1,\cL_2).\]
\end{prop}

We shall use the following consequence of deep results on exact Lagrangians in cotangent bundles by \cite{MR2596633,MR2565051,MR2947545}:

\begin{theo}\label{Thm:equivZeroSec}
    Let $\cL_0$ and $\cL_1$ be closed exact Lagrangian submanifolds in $T^*M$. Then,
    \begin{enumerate}
        \item \label{item:fiber} the Floer homology of $\cL_0$ and a cotangent fibre is 1-dimensional: \[\forall q\in M, HF_*(\cL_0,T^*_qM)\cong\Z/2\Z;\]
        \item \label{item:zeroSec} the Floer homology of $\cL_0$ and $\cL_1$ is isomorphic to the Morse (or singular) homology of $M$ with coefficients in $\Z/2\Z$: \[HF_*(\cL_0,\cL_1)\cong H_*(M;\Z/2\Z).\]
    \end{enumerate}
\end{theo} 

In particular, since $M$ is $\Z/2\Z$-orientable, it admits a $\Z/2\Z$-fundamental class $\mu\in H_*(M;\Z/2\Z)$. We denote by $\mu_{\cL_0,\cL_1}$ the corresponding class in $HF_*(\cL_0,\cL_1)$. When the Lagrangians $\cL_0$ and $\cL_1$ are clear from the context, we will simply denote it $\mu$. It has the following property:

\begin{prop}
    \label{prop:unit}
    Let $\cL_0$, $\cL_1$, $\cL_2$, $\cL_3$ be exact Lagrangian submanifolds in $T^*M$, with $\cL_0$, $\cL_1$ and $\cL_2$ closed. Then, 
        \begin{enumerate}
            \item the maps $HF_*(\cL_1,\cL_3)\to HF_*(\cL_0,\cL_3)$ and $HF_*(\cL_3,\cL_0)\to HF_*(\cL_3,\cL_1)$ given by $\alpha\to\mu_{\cL_0,\cL_1}*\alpha$ and $\alpha\to\alpha*\mu_{\cL_0,\cL_1}$ respectively are isomorphisms;
            \item the isomorphisms above are coherent:
            \[\mu_{\cL_0,\cL_1}*\mu_{\cL_1,\cL_2}=\mu_{\cL_0,\cL_2};\]
            \item when $\cL_0=\cL_1$, the maps above coincide with the identity maps.
        \end{enumerate}
\end{prop}

\subsubsection{Spectral Distances, Humili\`ere Completion and $\gamma$-Supports} \label{section:SpectralHumiliere}

We may now introduce the spectral distance, initially defined by Viterbo in \cite{MR1157321} using generating functions, then by Schwarz and Oh \cite{MR1755825,MR2103018} using Floer homology in the Hamiltonian case, and by Leclercq--Zapolsky \cite{MR2383268,MR3850107} in the Lagrangian case.
\begin{defi}[Spectral distance for closed exact Lagrangians] \label{def:SpectralDist}
    Let $\cL_0$, $\cL_1$ be two closed exact Lagrangian submanifolds. We define the spectral distance (or $\gamma$-distance) between $\cL_0$ and $\cL_1$ as
    \[\gamma(\cL_0,\cL_1)=\rho(\mu_{\cL_0,\cL_1},\cL_0,\cL_1)+\rho(\mu_{\cL_1,\cL_0},\cL_1,\cL_0).\]
\end{defi}

\begin{rem}
    \begin{enumerate}
        \item In principle, we need to specify a choice of Liouville primitive for the Lagrangian submanifolds to compute their spectral invariant. However, we claim that $\gamma(\cL_0,\cL_1)$ does not depend on the choice of primitives. Indeed, all primitives differ by a constant. Shifting the primitive of $\cL_0$ by a constant will shift $\rho(\mu,\cL_0,\cL_1)$ by the same constant, and $\rho(\mu,\cL_1,\cL_0)$ by its opposite, therefore it will leave $\gamma(\cL_0,\cL_1)$ unchanged.
        \item In the literature, the spectral norm is often defined as a difference between the spectral invariants associated to the fundamental class and to the point class. Using the duality property of spectral invariants, one can show the two definitions coincide.
    \end{enumerate}
\end{rem}

There also exists a version of this distance for Lagrangian branes, also due to Viterbo:

\begin{defi}[Spectral distance for closed Lagrangian branes] \label{def:SpectralDistBrane}
    Let $(\cL_0,h_0)$, $(\cL_1,h_1)$ be two closed Lagrangian branes. We define the spectral distance between $(\cL_0,h_0)$ and $(\cL_1,h_1)$ as
    \[c((\cL_0,h_0),(\cL_1,h_1))=\max(0,\rho(\mu_{\cL_0,\cL_1},\cL_0,\cL_1))+\max(0,\rho(\mu_{\cL_1,\cL_0},\cL_1,\cL_0)).\]
\end{defi}

\begin{prop} \label{prop:GammaDistance}
    The spectral distances $\gamma$ and $c$ are distances on the set $\mathscr{L}(T^*M)$ of closed exact Lagrangian submanifolds of $T^*M$, and on the set of Lagrangian branes $\mathcal{LB}(T^*M,\omega)$ respectively.
\end{prop}

In the remaining part of this section, we define the notions of Humili\`ere completion, $\gamma$-supports, $\gamma$-coisotropic sets, following \cite{viterbo2026supportshumilierecompletiongammacoisotropic}, and present a quick overview of some of their basic properties. These notions and results will only be used in the proofs of Item \ref{prop:UniquenessPrimitive6} of Theorem \ref{thm:UniquenessPrimitiveC0} (in Section \ref{sec:proof for sympeo}), and of Theorem \ref{thm:rigidity} (in Section \ref{sec:rigidity}).

We denote by $\widehat{\mathscr{L}}(T^*M)$ and $\widehat{\mathcal{LB}}(T^*M)$ the completions of the spaces $(\mathscr{L}(T^*M),\gamma)$ and $(\mathcal{LB}(T^*M,\omega),c)$ respectively. The space $\widehat{\mathscr{L}}(T^*M)$ was originally introduced by Vincent Humili\`ere in his PhD thesis \cite{humil2008} and in \cite{MR2415347}, and is therefore called ``Humili\`ere completion''.

Note that for smooth branes, we have the identity
\[\gamma(\cL_0,\cL_1)=\inf_{a\in\R}c((\cL_0,h_0),(\cL_1,h_1+a)).\]
It implies that the quotient map $\mathcal{LB}(T^*M,\omega)\to\mathfrak{LB}(T^*M,\omega)\cong\mathscr{L}(T^*M)$ induces a natural surjective map $\widehat{\mathcal{LB}}(T^*M)\to \widehat{\mathscr{L}}(T^*M)$, satisfying that two preimages of an element of $\widehat{\mathscr{L}}(T^*M)$ only differ by a constant action shift.

We now define the $\gamma$-support of an element of the completion, following \cite{viterbo2026supportshumilierecompletiongammacoisotropic}.

\begin{defi}[$\gamma$-support, \cite{viterbo2026supportshumilierecompletiongammacoisotropic}]
    Let $\underline \cL$ be an element of the Humili\`ere completion $\widehat{\mathscr{L}}(T^*M)$. Its $\gamma$-support $\operatorname{\gamma-supp}(\underline \cL)$ consists of all points $x\in T^*M$ such that for any neighbourhood $U$ of $x$, there exists a Hamiltonian diffeomorphism $\varphi$ compactly supported in $U$ such that $\varphi(\underline \cL)\neq \underline \cL$.
    
    The $\gamma$-support of an element $\underline \cL\in\widehat{\mathcal{LB}}(T^*M)$ is the $\gamma$-support of its image under the natural map $\widehat{\mathcal{LB}}(T^*M)\to \widehat{\mathscr{L}}(T^*M)$.
\end{defi}

\begin{rem}
    Note that exact symplectomorphisms preserve the $\gamma$-distance. Therefore, they map Cauchy sequences to Cauchy sequences and act naturally on the Humili\`ere completion.
\end{rem}

$\gamma$-supports have the following properties, established in \cite{viterbo2026supportshumilierecompletiongammacoisotropic}:

\begin{prop}[\cite{viterbo2026supportshumilierecompletiongammacoisotropic}]\label{prop:gammaSupports}
    Let $(\cL_n)$ be a sequence of closed exact Lagrangian submanifolds converging in the spectral topology to an element $\underline \cL$ of the Humili\`ere completion $\widehat{\mathscr{L}}(T^*M)$. Then,
    \begin{enumerate}
        \item $\operatorname{\gamma-supp}(\underline \cL)$ is closed;
        \item ($\gamma$-coisotropicity) $\operatorname{\gamma-supp}(\underline \cL)$ is $\gamma$-coisotropic (see Definition \ref{def:gammaCoisotropic} below);
        \item (semi-continuity) $\operatorname{\gamma-supp}(\underline \cL)\subset \liminf_n \cL_n=\{x\in T^*M\mid\exists(x_n)_{n\geq1},x_n\in\cL_n,\lim_n x_n=x\}$.
    \end{enumerate}
\end{prop}

We will not make use of the definition of $\gamma$-coisotropic sets in this article, but we give it for the sake of completeness:

\begin{defi}[$\gamma$-coisotropicity, \cite{viterbo2026supportshumilierecompletiongammacoisotropic}]\label{def:gammaCoisotropic}
    A subset $V\subset T^*M$ is $\gamma$-coisotropic at $x\in V$ if there exists $\varepsilon>0$ such that for any $0<\eta<\varepsilon$, there exists $\delta(\eta)>0$ such that for every Hamiltonian diffeomorphism $\varphi$ supported in $B(x,\varepsilon)$, if $\varphi(V)\cap B(x,\eta)=\emptyset$ then $\gamma(\varphi)>\delta(\eta)$ (where $\gamma$ denotes the spectral norm of Hamiltonian diffeomorphisms).

    $V$ is $\gamma$-coisotropic if it is non-empty and $\gamma$-coisotropic at each $x\in V$.
\end{defi}

The only property of $\gamma$-coisotropic sets we will use is the following:

\begin{prop}[\cite{viterbo2026supportshumilierecompletiongammacoisotropic}]\label{prop:smoothGammaCoisotropic}
    A $C^1$ submanifold $V\subset T^*M$ is $\gamma$-coisotropic if and only if it is coisotropic in the usual sense.
\end{prop}

\subsection{Graph Selectors and Variational Solutions}  \label{section:GS&VarSol}

Let $(\cL,h)$ be a compact Lagrangian brane in $T^*M$. For $q$ in $M$, the Liouville form $\lambda$ vanishes on the cotangent fibre $T^*_qM$, so that the zero function is a Liouville primitive on $T^*_qM$.

By Theorem \ref{Thm:equivZeroSec}, the Floer homology of the pair $(\cL,h)$, $(T^*_qM,0)$ is one-dimensional:
\[HF_*(\cL,T^*_qM)\cong \Z/2\Z.\]
We denote by $\alpha$ its single generator.

 Let $(\cL_t,h_t)$ be the variational brane associated to $(\cL,h)$. By the invariance of Floer homology under Hamiltonian diffeomorphisms, we also have $HF_*(\cL_t,T^*_qM)\cong \alpha\cdot \Z/2\Z$. Therefore, we can define:

 \begin{defi}[Variational solution]
 	\begin{enumerate}
 		\item The \textit{graph selector} associated to a compact Lagrangian brane $(\cL,h)$ is $u: M\to\R$ given by
     \[u(q)=\rho(\alpha,\cL,T^*_qM).\] 
 		\item  The \textit{variational solution} associated to a compact Lagrangian brane $(\cL,h)$ is $u:\R\times M\to\R$ given by
     \[u(t,q)=\rho(\alpha,\cL_t,T^*_qM).\]
 	\end{enumerate}
 \end{defi}
 
We would like to emphasize that this notion of variational solution differs from the classical notion of variational solutions that has already been studied \cite{MR1604386,unknown,MR3957150}, and which solve a Cauchy problem for the Hamilton--Jacobi equation with a fixed initial scalar condition. The difference between the classical notion and the one we introduce is explained in Remark \ref{rem:VariationalSolutions}.   

\begin{prop} \label{GSTranslation}
	Let $(\mathcal{L},h)$ be a Lagrangian brane with associated variational solution $u$. Then, for any real constant $c \in \mathbb{R}$, the variational solution $u_c$ associated to the brane $(\mathcal{L},h+c)$ is $u_c = u+c$.
\end{prop}
\begin{proof}
    This is a straightforward consequence of the fact that shifting the Liouville primitive by a constant only shifts the filtration on the Floer complex by the same constant.
\end{proof}

The convergence of graph selectors and variational solutions is controlled by the spectral distance.
\begin{prop} \label{Prop:GS<Gamma}
	Let $(\mathcal{L}_0,h_0)$ and $(\mathcal{L}_1,h_1)$ be two compact Lagrangian branes with associated variational solutions $u_0$ and $u_1 : \mathbb{R} \times M \to \mathbb{R}$. Then, we have
	\begin{equation}
		\osc ( u_1 - u_0 ) \leq \gamma(\cL_0,\cL_1).
	\end{equation}
\end{prop}
\begin{proof}
	Let $(t,q)$ be a point in $\mathbb{R} \times M$. Using successively the isomorphism property of $\mu$ in Proposition \ref{prop:unit}, the triangle inequality in Proposition \ref{prop:TriangleIneq}, and the invariance property of Proposition \ref{prop:spectral_invariants}, we obtain
	\begin{align*}
		u_0(t,q) = \rho(\alpha, \mathcal{L}_{0,t},T^*_qM) &= \rho(\mu * \alpha, \mathcal{L}_{0,t},T^*_qM) \\
		&\leq \rho(\mu, \mathcal{L}_{0,t}, \mathcal{L}_{1,t}) + \rho(\alpha, \mathcal{L}_{1,t},T^*_qM) \\
		&= \rho(\mu, \mathcal{L}_0, \mathcal{L}_1) + u_1(t,q)
	\end{align*}
	Symmetrically, we have 
	\begin{align*}
		u_1(t,q) \leq \rho(\mu, \mathcal{L}_1, \mathcal{L}_0) + u_0(t,q)
	\end{align*}
	which yields
	\begin{equation} \label{formula:GSSpectralBound}
		-\rho(\mu, \mathcal{L}_0, \mathcal{L}_1) \leq u_1(t,q) - u_0(t,q) \leq \rho(\mu, \mathcal{L}_1, \mathcal{L}_0)
	\end{equation}
	and
	\begin{align*}
		\osc (u_1 - u_0 ) \leq \rho(\mu, \mathcal{L}_0, \mathcal{L}_1) + \rho(\mu, \mathcal{L}_1, \mathcal{L}_0) = \gamma( \cL_0,\cL_1)
	\end{align*}
\end{proof}     

The next statement was already proven in \cite{Birkhoff} for variational solutions defined with generating functions, following Sikorav and Chaperon \cite{MR1094198}. In the Floer case, Amorim--Oh--dos Santos prove it only for graph selectors when there is no time-dependence (\cite[Theorem 4.6]{MR3860396}). We adapt their proof to the time-dependent case.
\begin{prop} \label{Graphext}
	Let $u: \mathbb{R} \times M \to \mathbb{R}$ be a variational solution associated to a compact Lagrangian brane $(\mathcal{L},h)$. Then
	\begin{enumerate}
		\item \label{GraphextItem1} $u$ is locally Lipschitz on $\mathbb{R} \times M$.
		\item \label{GraphextItem2} There exists an open subset $\mathcal{U} \subset \mathbb{R} \times M$ of full Lebesgue measure such that $u$ is as regular as $h_t$ on $\mathcal{U}$ and for all $(t,q) \in \mathcal{U}$,
		\begin{equation} \label{eq:Graphext}
			\big(q, d_qu(t,q)\big) \in \mathcal{L}_t \quad , \quad u(t,q) = h_t(q, d_qu(t,q)) \quad \text{and} \quad \partial_tu + H(t,q,d_qu) = 0
		\end{equation}
	\end{enumerate}
\end{prop}

This proposition justifies the terminology of graph selectors, as we can see that on an open dense subset of the base, a graph selector $u$ of a Lagrangian submanifold $\mathcal{L}$ satisfies $du \in \mathcal{L}$. When generating functions can be used, the variational solution can also be viewed as the graph selector of an extended submanifold $\mathscr{L}$ of $\mathcal{L}$ in $T^*(\mathbb{R} \times M)$, as in \cite{MR3674224,Birkhoff}. It also justifies the terminology of variational solutions, as these are solutions of the Hamilton--Jacobi equation on the same dense open subset $\mathcal{U}$.

\begin{proof}

    \ref{GraphextItem1}. Let $I\subset \R$ be a compact interval. We show that $u$ is Lipschitz on $I\times M$. Let $(t_1,q_1),(t_2,q_2)\in I\times M$. We have
    \begin{equation}
    \label{eq:Lipschitz1}
         |u(t_1,q_1)-u(t_2,q_2)| \leq |u(t_1,q_1)-u(t_2,q_1)|+|u(t_2,q_1)-u(t_2,q_2)|.
    \end{equation}
    The idea of the proof will be to write $\cL_{t_2}$ and $T^*_{q_2}M$ as the image of $\cL_{t_1}$ and $T^*_{q_1}M$ under some Hamiltonian diffeomorphisms, and to use the $L^{1,\infty}$-Lipschitz property of Proposition \ref{prop:spectral_invariants} to control each term of \eqref{eq:Lipschitz1}.

    For the first term, we observe that $\cL_{t_2}= \phi_H^{t_1,t_2} (\cL_{t_1})= \phi^{t_2}_H\circ (\phi^{t_1}_H)^{-1}(\cL_{t_1})$.
    The Hamiltonian diffeomorphism $\phi_H^{t_1,t_2}$ is generated by the Hamiltonian $\widehat H:[0,1]\times T^*M\to \R$ given by
    \[\widehat H(t,x)=(t_2-t_1)H(t_1+(t_2-t_1)t,x).\]
    Consider a cut-off $\widetilde H$ that coincides with $\widehat H$ on $[0,1]\times D^*_RM$ where $D^*_RM$ contains $\bigcup_{t\in I}\cL_t$, and that vanishes away from $[0,1] \times D^*_{R+1}M$. Then, $\phi_H^{t_1,t_2} (\cL_{t_1}) = \phi^1_{\widetilde H}(\cL_{t_1})$, and therefore, by the $L^{1,\infty}$-Lipschitz property of spectral invariants (Proposition \ref{prop:spectral_invariants}), we get
    \begin{align}
        |u(t_1,q_1)-u(t_2,q_1)|&=|\rho(\alpha,\cL_{t_1},T^*_{q_1}M)-\rho(\alpha,\phi^1_{\widetilde H}(\cL_{t_1}),T^*_{q_1}M)|\\
        &\leq \int_0^1\Vert\widetilde H_t\Vert_\infty \;dt\\
        \label{eq:Lipschitz2}&\leq |t_1-t_2|\cdot \max\limits_{I\times D^*_{R+1}M}|H|
    \end{align}

    As for the second term of (\ref{eq:Lipschitz1}), we consider a minimizing geodesic $c:[0,1]\to M$ from $q_1$ to $q_2$. Then, we claim that we can build a time-dependent vector field $X_t$ on $M$ satisfying that for all times $t\in[0,1]$,  $X_t(c(t))=c'(t)$, $\max\limits_{q\in M}|X_t(q)|\leq \ell\coloneq d(q_1,q_2)$, and $\max\limits_{q\in M}|d_qX_t|\leq C$, where the constant $C>0$ depends only on the geometry of $M$. Indeed, fix some $r>0$ smaller than the injectivity radius of $M$. Then, it is enough to construct a smooth family of vector fields $X_t$ supported in the ball of radius $r$ around $c(t)$, satisfying $X_t(c(t))=c'(t)$, $\vert X_t\vert\leq \ell$, and $\vert dX_t\vert\leq C_{M,r}\ell\leq C_{M,r} \operatorname{diam}(M)=C$, for some constant $C_{M,r}$ depending only on $M$ and $r$. Such a family is obtained in exponential coordinates centred at $c(t)$ by extending $c'(t)$ by radial parallel transport and multiplying by a fixed radial cutoff supported in $B(c(t),r)$. Since $M$ is compact, the exponential maps, their inverses, and radial parallel transport have uniformly bounded first derivatives on vectors of norm at most $r$, and the bound on $\vert dX\vert$ follows.

    Let $G_\tau$ be the time-dependent Hamiltonian on $[0,1]\times T^*M$ given by $G_\tau(q,p)=p.X_\tau(q)$. Then, $\partial_pG_\tau(q,p)=X_\tau(q)$ and therefore $\phi^\tau_G(T^*_{q_1}M)=T^*_{c(\tau)}M$. In particular, $\phi^1_G(T^*_{q_1}M)=T^*_{q_2}M$. For all $\tau\in[0,1]$, we have $(\phi^\tau_G)^{-1}=\phi^\tau_{K}$ where $K$ is given by $K_\tau(q,p)=-G_\tau(\phi^\tau_G(q,p))$. Moreover, the bound on $|dX|$ implies, using Gr\"onwall's inequality, that for all $\tau\in[0,1]$, $D^*_RM\subset \phi^\tau_{G}(D^*_{R'}M)$, where $R'=e^CR$, and thus $\phi^\tau_{K}(D^*_RM)\subset D^*_{R'}M$.
    Therefore, if we define a cut-off $\widetilde K_\tau(q,p)\coloneq \chi\left(|p|-R'\right)K_\tau(q,p)$, where $\chi\colon \R\to[0,1]$ satisfies $\chi(s) = 1$ for $s\leq 0$ and $\chi(s)=0$ for $s\geq 1$, then we have that $\phi^1_{\widetilde K}$ is compactly supported inside $D^*_{R'+1}M$, and coincides with $(\phi^1_G)^{-1}$ inside $D^*_RM$. In particular, for all $t\in I$, we have $(\phi^1_G)^{-1}(\cL_t)=\phi^1_{\widetilde K}(\cL_t)$. Moreover, for all $\tau \in[0,1]$, we have 
    \begin{equation*}
        \Vert \widetilde K_\tau\Vert_\infty\leq \Vert G_\tau\Vert_\infty^{\phi^\tau_G(D^*_{R'+1}M)}\leq e^C(R'+1)\max\vert X_\tau\vert\leq e^C(R'+1)\cdot d(q_1,q_2)
    \end{equation*}
    where we used again Gr\"onwall's inequality and the properties of $X_\tau$. Now, using the invariance and the $L^{1,\infty}$-Lipschitz properties of Proposition \ref{prop:spectral_invariants}, we have
    \begin{align*}
        |u(t_2,q_1)-u(t_2,q_2)| &= |\rho(\alpha,\cL_{t_2},T^*_{q_1}M)-\rho(\alpha,\cL_{t_2},\phi^1_G(T^*_{q_1}M))|\\
        &= |\rho(\alpha,\cL_{t_2},T^*_{q_1}M)-\rho(\alpha,(\phi^1_G)^{-1}(\cL_{t_2}),T^*_{q_1}M)|\\
        &= |\rho(\alpha,\cL_{t_2},T^*_{q_1}M)-\rho(\alpha,\phi^1_{\widetilde K}(\cL_{t_2}),T^*_{q_1}M)|\\
        &\leq \int_0^1\Vert\widetilde K_t\Vert_\infty\;dt\\
        &\leq e^C(R'+1)\cdot d(q_1,q_2).
    \end{align*}

This implies, together with (\ref{eq:Lipschitz1}) and (\ref{eq:Lipschitz2}), that $u$ is Lipschitz on $I\times M$.\\

\ref{GraphextItem2}. Consider the Lagrangian suspension
\[\widetilde \cL\coloneq\big\{(t,-H(t,x),x)|t\in\R, x\in\cL_t\big\}\subset T^*(\R\times M)\cong T^*\R\times T^*M.\]
We have that $\widetilde \cL=\{(t,-H(t,\phi^t_H(x)),\phi^t_H(x))|t\in\R,x\in \cL\}$ is the image of the exact Lagrangian submanifold $0_{T^*\R}\times \cL$ under the Hamiltonian diffeomorphism generated by the Hamiltonian $\widetilde H:[0,1]\times T^*\R\times T^*M\to \R$ given by $\widetilde H(\tau,t,E,x)=tH_{t\tau}(x)$. Therefore, $\widetilde\cL$ is itself an exact Lagrangian submanifold, and can be equipped with the Liouville primitive $\widetilde h:\widetilde{\cL}\to \R$ given by $\widetilde h(t,-H(t,x),x)=h_t(x)$.

Fix $(t,q)\in\R\times M$. We would like to compare $u(t,q)$ with the spectral invariant of the pair $(\widetilde \cL,T^*_{(t,q)})$. However, the manifold $\R \times M$ is not compact, so it is not clear whether such a spectral invariant is well-defined. Instead, we will only look at the spectrum of the pair
\[\Spec(\widetilde \cL,T^*_{(t,q)}(\R\times M))=\big\{\mathcal A_{\widetilde \cL,T^*_{(t,q)}(\R\times M)}(y),y\in \widetilde \cL\cap T^*_{(t,q)}(\R\times M)\big\}=\big\{\widetilde h(y),y\in \widetilde \cL\cap T^*_{(t,q)}(\R\times M)\big\}.\]
We observe that for all $(t,q)\in\R\times M$, the map $x\mapsto (t,-H(t,x),x)$ induces a one-to-one correspondence between $\cL_t\cap T^*_qM$ and $\widetilde \cL\cap T^*_{(t,q)}(\R\times M)$, and therefore
\begin{align*}
\Spec(\widetilde \cL,T^*_{(t,q)}(\R\times M))&=\big\{\widetilde h(t,-H(t,x),x),x\in\cL_t\cap T^*_qM\big\}\\
&=\big\{h_t(x),x\in\cL_t\cap T^*_qM\big\}\\
&= \Spec(\cL_t,T^*_qM)
\end{align*}
and this set is finite whenever the intersection is transverse.

Having noticed this, the rest of the proof is identical to that of \cite[Theorem 4.6]{MR3860396}, replacing $\cL\subset T^*M$ by $\widetilde \cL\subset T^*(\R\times M)$ (all the arguments are local and do not use the compactness of $M$). Therefore, we will skip some details and only sketch the main ideas.

Consider the set
\[\mathcal U\coloneq \big\{(t,q)\in \R\times M \; \big| \;\text{$T^*_{(t,q)}(\R\times M)$ intersects $\widetilde \cL$ transversely and $\widetilde h|_{\widetilde \cL\cap T^*_{(t,q)}(\R\times M)}$ is injective}\big\}.\]
Denote by $\pi:T^*(\R\times M)\to \R\times M$ the projection to the base. Using Sard's theorem and the fact that $\widetilde h$ is a Liouville primitive, one can show as in \cite[Proposition 4.2]{MR3860396} that $\mathcal U$ is an open subset of full Lebesgue measure, and that for all $(t,q)\in \mathcal U$, there exists an open connected neighbourhood $U_{(t,q)}\subset \mathcal U$ of $(t,q)$ such that $\pi^{-1}(U_{(t,q)})\cap \widetilde \cL$ decomposes as a finite disjoint union $\pi^{-1}(U_{(t,q)})\cap \widetilde \cL=\coprod\limits_{i=1}^{k_{(t,q)}} V_i$ where the $V_i$ are open subsets of $\widetilde \cL$, diffeomorphic to $U_{(t,q)}$ via $\pi|_{V_i}$.

Fix some $(t,q)$ in $\mathcal U$. Since $T^*_{(t,q)}(\R\times M)$ intersects $\widetilde \cL$ transversely, we also have that $T^*_qM$ intersects $\cL_t$ transversely, and by spectrality (Proposition \ref{prop:spectrality}), $u(t,q)\in\Spec(\cL_t,T^*_qM)=\Spec(\widetilde \cL,T^*_{(t,q)}(\R\times M))$. Therefore, there exists $y_{(t,q)}\in \widetilde \cL\cap T^*_{(t,q)}(\R\times M)$ such that $u(t,q)=\widetilde h(y_{(t,q)})$. Now consider a neighbourhood $U_{(t,q)}$ as above. Since $y_{(t,q)}\in \pi^{-1}(U_{(t,q)})\cap \widetilde \cL$, there exists an index $i$ such that $y_{(t,q)}\in V_i$. For $(t',q')\in U_{(t,q)}$, we define $y(t',q')=\pi|_{V_i}^{-1}(t',q')$, and consider the subset $U^=_{(t,q)}=\{(t',q')\in U_{(t,q)}|u(t',q')=\widetilde h(y(t',q'))\}$.
This set is non-empty since it contains $(t,q)$. It is also closed in $U_{(t,q)}$ since $u$ and $\widetilde h\circ y$  are continuous ($u$ is locally Lipschitz by \ref{GraphextItem1}.). Using the injectivity of $\widetilde h$ on $\widetilde \cL\cap T^*_{(t',q')}(\R\times M)$ and the finiteness of the spectrum for $(t',q')$ in $\mathcal U$, we also have that $U^=_{(t,q)}$ is open. Thus, $U^=_{(t,q)}=U_{(t,q)}$ since $U_{(t,q)}$ is connected, and therefore $u(t',q')=\widetilde h(y(t',q'))$ for all $(t',q')\in U_{(t,q)}$.
Therefore, $u$ has the same regularity as $\widetilde h$ (since $y$ is smooth) on $U_{(t,q)}$, and therefore the same regularity as $h_t$ on $\mathcal U$. Using the fact that $\widetilde h$ is a Liouville primitive, we get that
\begin{equation}\label{eq:graph}
    ((t,q),d_{(t,q)}u)=y_{(t,q)}\in \widetilde\cL\cap T^*_{(t,q)}(\R\times M).
\end{equation}
Writing $y_{(t,q)}$ as $(t,-H(t,x_{(t,q)}),x_{(t,q)})\in T^*\R\times T^*M$, with $x_{(t,q)}\in\cL_t$, (\ref{eq:graph}) implies that
\begin{equation}
	\big(q, d_qu(t,q)\big)=x_{(t,q)} \in \mathcal{L}_t \quad \text{and} \quad \partial_tu + H(t,q,d_qu) = 0.
\end{equation}
Finally, $u(t,q)=\widetilde h(y_{(t,q)})=h_t(x_{(t,q)})=h_t(q, d_qu(t,q))$.
\end{proof}

\begin{rem} \label{rem:VariationalSolutions}
	Let us explain the difference between the notion of variational solution introduced in this article and the classical notion of variational solution present in the literature. 
	
	In \cite{MR1604386,unknown}, variational solutions are first defined for the Cauchy problem
	\begin{equation}
		\begin{cases}
			\partial_tu + H(t,q,d_qu) = 0 \\
			u(0,\cdot) = u_0
		\end{cases}
	\end{equation}
	with Lipschitz initial data $u_0$. In this setting, the classical variational solution $u : [0,+\infty) \times M \to \mathbb{R}$ is first defined for $C^1$ initial data using graph selectors, in the same way as we have done. It is then observed that the map $R : u_0 \mapsto u(t,\cdot) = R^t u_0$ extends to Lipschitz initial data in the uniform topology.\\
	
	Joukovskaia's theorem \cite{jouko1994} states that, in the case of Tonelli Hamiltonians, the classical variational solutions $u : [0,+\infty) \times M \to \mathbb{R}$ generated by the operator $R^t$ coincide with the viscosity solutions of the Hamilton--Jacobi equation, introduced by Crandall--Lions \cite{MR0690039} and later developed by A. Fathi \cite{MR1451248} in the Tonelli setting.\\
	
	In this framework, the initial condition is a scalar function, whereas in our definition of variational solution, the solution depends on the entire Lagrangian brane $(\mathcal{L},h)$ at time $0$, and not only on its graph selector $u_0$. The solutions constructed in this way are no longer guaranteed to be viscosity solutions, which is one of the main difficulties of this article.\\
	
	If one wishes to relate the classical notion to the one introduced here, the correspondence is as follows: given a Lipschitz initial condition $u_0$, the classical variational solution $u : [0,+\infty) \times M \to \mathbb{R}$ generated by the operator $R^t$ coincides with our variational solution associated to the initial Lagrangian brane $(\mathcal{L},h)$, where $\mathcal{L} = \mathcal{PG}(du_0)$ is the pseudograph associated to $du_0$, i.e. the fibrewise convex hull of the closure of the graph $\mathcal{G}(du_0)$. Note that even though $(\mathcal{PG}(du_0), u_0 \circ \pi_M)$, where $\pi_M : T^*M \to M$ is the canonical projection, is not a smooth Lagrangian brane, it is nevertheless a bare Lagrangian brane, since it can be approximated by graphs of differentials of smooth maps \cite{MR2231385}. We associate to it a variational solution in Subsection \ref{SectionGSBC}.
	
	Below, we illustrate an example where viscosity solutions and variational solutions do not coincide.
	\begin{figure}[h]
		\centering
		\begin{subfigure}{0.35\textwidth}
			\centering
			\includegraphics[width=\textwidth]{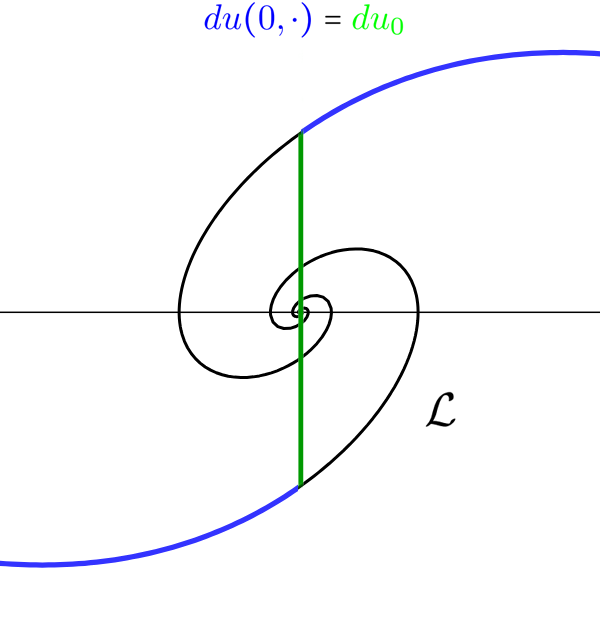}
			\caption{Initial Data}
		\end{subfigure}
		\begin{subfigure}{0.375\textwidth}
			\centering
			\includegraphics[width=\textwidth]{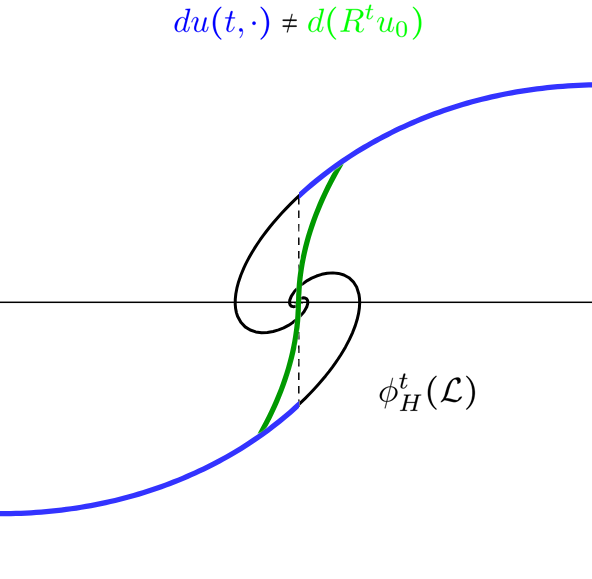}
			\caption{Solutions at time $t$}
		\end{subfigure}
		\caption{A case where variational and viscosity solutions are distinct}
	\end{figure}
	
	Another subtlety is that the variational solutions $u : \mathbb{R} \times M \to \mathbb{R}$ that we define correspond, when the initial condition is a pseudograph, to viscosity solutions only for positive times $t \geq 0$, whereas the evolution for negative times is different, and intuitively corresponds to the symmetry between what Albert Fathi calls positive and negative weak KAM solutions.
\end{rem}

\section{Calibration of Variational Solutions} \label{SectionCalib}

In this section, we introduce ideas from weak KAM theory that allow us to detect the regularity of viscosity solutions. Even though we have seen in Remark \ref{rem:VariationalSolutions} that viscosity and variational solutions do differ, we will show in Proposition \ref{Domination} that the latter are dominated (or, equivalently in the viscosity terminology, viscosity subsolutions). Hence, some regularity properties still apply to variational solutions, which we will use in the proof of the main theorem.\\

Tonelli Hamiltonians are the right setting to have a correspondence between Hamiltonian and Lagrangian dynamics. This is precisely the framework of Fathi's weak KAM theory.

\begin{defi} \label{Tonelli}
	 A $C^2$ time-periodic Hamiltonian $H(t,q,p) : \mathbb{T}^1 \times T^*M \to \mathbb{R}$ is called \textit{Tonelli} if it satisfies the following classical hypotheses:
	\begin{itemize}[label=--]
		\item (\textit{Strict convexity}) $\partial_{pp}H(t,q,p)>0$ for all $(t,q,p) \in \mathbb{T}^1 \times T^*M$.
		\item (\textit{Superlinearity}) $\frac{|H(t,q,p)|}{\Vert p\Vert } \to \infty$ as $\Vert p\Vert  \to \infty$ for each $(t,q) \in \mathbb{T}^1 \times M$.
		\item (\textit{Completeness}) The Hamiltonian vector field $X_H$ and hence its flow $\phi_H^{s,t}$ are complete in the sense that the flow curves are defined for all times $t \in \mathbb{R}$.
	\end{itemize}
\end{defi}

In this case, we can associate to the Tonelli Hamiltonian its convex conjugate known as the \textit{Lagrangian} $L : \T^1 \times TM \to \R$ defined as
\begin{equation} \label{HamLag}
	L(t,q,v) = \sup_{p \in T^*_qM} \{ p.v - H(t,q,p) \}\quad \text{and} \quad H(t,q,p) = \sup_{v \in T_qM} \{ p.v - L(t,q,v) \} 
\end{equation}

From this convexity relation, we deduce Fenchel's inequality.
\begin{equation} \label{Fenchel}
	p.v \leq L(t,q,v) + H(t,q,p) \quad \text{with equality if and only if } (t,q,p) = \leg(t,q,v)
\end{equation}
where the Legendre transform $\leg$ is defined in (\ref{Legendrian}).

The Lagrangian dynamics is obtained from a variational principle, that is, by minimizing the action $A_L(\gamma) = \int L(\tau, \gamma(\tau), \dot{\gamma}(\tau)) \; d\tau$ among curves with fixed endpoints. This gives rise to the Lagrangian flow $ \phi_L^{s,t} $, which is conjugate to the Hamiltonian flow $\phi_H^{s,t}$ via the Legendre map $ \leg : \mathbb{T}^1 \times TM \to \mathbb{T}^1 \times T^*M $, a diffeomorphism defined by
\begin{equation} \label{Legendrian}
	\leg (t,q,v) = (t,q,\partial_v L(t,q,v)) \quad \text{with inverse } \leg^{-1}(t,q,p) = (t,q,\partial_pH(t,q,p))
\end{equation}

\begin{rem} \label{LagsImagesLag}
	In the case of a Tonelli Hamiltonian, the Hamiltonian action of an orbit $x_\tau=\phi_H^\tau(x)$ is equal to its Lagrangian action. Hence, the definition of the variational bare brane $(\cL_t,h_t)$ in \eqref{LagsImagesFormula} becomes
	\begin{equation}
		h_t(x_t) = h (x) + \int_0^t L(\tau,q_\tau,\dot{q}_\tau) \;d\tau \quad \text{where} \quad x_\tau=(q_\tau,p_\tau) = \phi_H^\tau (x).
	\end{equation}
\end{rem}

\subsection{Domination}

\begin{defi}
	Let $u : \mathbb{R} \times M \to \mathbb{R}$ be a variational solution associated to a Lagrangian brane and let $\gamma : [a,b] \to M$ be a $C^1$ curve on $M$. The \textit{defect of calibration} of $\gamma$ by $u$ is defined as
	\begin{equation}
		\delta(u,\gamma,[a,b]) = \int_a^b L(s,\gamma(s),\dot{\gamma}(s)) \; ds - [u(b,\gamma(b)) - u(a,\gamma(a))]
	\end{equation}
\end{defi}

\begin{prop} \label{Domination}
	For a variational solution $u$ associated to a Lagrangian brane, the defect of calibration $\delta$ is non-negative. We say that the map $u : \mathbb{R} \times M \to \mathbb{R}$ is \textit{dominated} by $L$.
\end{prop}

\begin{proof}
	Let $\gamma : [a,b] \to M$ be a $C^1$ curve on $M$ such that for almost every time $t \in [a,b]$, $(t,\gamma(t))$ belongs to the dense open set $\mathcal{U}$ defined in Proposition \ref{Graphext}.
	\begin{align*}
		u(b,\gamma(b)) - u(a,\gamma(a)) & = \int_a^b du(s,\gamma(s)).(1,\dot{\gamma}(s)) \;ds = \int_a^b \Big( \partial_t u(s,\gamma(s)) + d_q u (s,\gamma(s)).\dot{\gamma}(s) \Big) \;ds \\
		& \leq \int_a^b \Big( \partial_t u(s,\gamma(s)) + H\big(s,\gamma(s),d_q u (s,\gamma(s))\big) + L\big(s,\gamma(s),\dot{\gamma}(s)\big) \Big) \;ds \\
		& = \int_a^b 0 + L(s,\gamma(s),\dot{\gamma}(s)) \; ds 
	\end{align*}
	where in the second line we used the Fenchel inequality \eqref{Fenchel}. Then, we get
	\begin{equation*}
		\delta(u,\gamma,[a,b]) = \int_a^b L(s,\gamma(s),\dot{\gamma}(s)) \; ds - [u(b,\gamma(b)) - u(a,\gamma(a))] \geq 0.
	\end{equation*}
	The general configuration case is given by the perturbation process described in Lemma \ref{Pertrub}.
\end{proof}

\begin{lem}[\cite{MR3674224}] \label{Pertrub}
	Let $\gamma : [a,b] \to M$ be a $C^1$ curve and let $\mathcal{U}$ be a full Lebesgue measure set of $\mathcal{M} = \mathbb{R} \times M$. Then there exists a sequence of $C^1$ curves $\gamma_k : [a,b] \to M$ such that
	\begin{enumerate}[label=\roman*.]
		\item For almost all $t \in [a,b]$, $(t, \gamma_k(t)) \in \mathcal{U}$.
		\item The sequence $(\gamma_k)_k$ converges to $\gamma$ in the $C^1$-topology.
	\end{enumerate}
\end{lem}

\begin{prop} \label{DominatedSubsolution}
	Let $u : \mathbb{R} \times M \to \mathbb{R}$ be a locally Lipschitz function that is dominated by $L$. Then, $u$ is a very weak subsolution of the Hamilton--Jacobi equation, i.e., at every point $(t,q)$ where $u$ is differentiable, we have
	\begin{equation}
		\partial_t u + H(t,q,d_qu) \leq 0
	\end{equation}
\end{prop}

\begin{rem}
	In fact, Proposition \ref{DominatedSubsolution} is an equivalence: a locally Lipschitz function $u$ is dominated by $L$ if and only if it is a very weak subsolution of the Hamilton--Jacobi equation.
\end{rem}

\begin{proof}
	Let $(t,q)$ be a point of differentiability of $u$. Fix $v \in T_qM$, and let $\gamma : [t,t+1] \to M$ be a smooth curve such that $(\gamma(t),\dot{\gamma}(t))=(q,v)$. By the domination inequality, for every $s \in (t,t+1]$, we have
	\begin{align*}
		\frac{u(s,\gamma(s))-u(t,\gamma(t))}{s-t} \leq \frac{1}{s-t} \int_t^s L(\tau,\gamma(\tau),\dot{\gamma}(\tau)) \; d\tau
	\end{align*}
	Letting $s \to t$ yields
	\begin{equation*}
		\frac{d}{dt}\big( u(t,\gamma(t)) \big) = \partial_t u(t,q) + d_qu(t,q).v \leq L(t,q,v)
	\end{equation*}
	Taking the supremum over $v \in T_qM$ and using the conjugation formula \eqref{HamLag}, we obtain the desired inequality 
	\begin{align*}
		\partial_t u(t,q) + H(t,q,d_qu(t,q)) \leq 0
	\end{align*}
\end{proof}

\subsection{Calibration}

We can pay closer attention to curves that achieve equality in this domination, or in other words curves having vanishing calibration defect.

\begin{defi}
	Let $u : \mathbb{R} \times M \to \mathbb{R}$ be a dominated map. A $C^1$ curve $\gamma : I \to M$ on $M$ is \textit{calibrated} by $u$ if for all $[a,b] \subset I$, $\delta(u,\gamma,[a,b]) = 0$.
\end{defi}

\begin{rem}
	\begin{enumerate}
		\item We know from Proposition \ref{Domination} that for all times $t<t'$ and points $q$ and $q'$ of $M$,
		\begin{equation}
			u(t',q') - u(t,q) \leq \inf \left\{ \int_t^{t'} L(\tau,\gamma(\tau),\dot{\gamma}(\tau)) \; d\tau \; \left| \;
			\begin{matrix}
				\gamma : & [t,t'] \to M \\
				& t \mapsto q \\
				& t' \mapsto q'
				\end{matrix} \right.	\right\}
		\end{equation}
		Hence, if $\gamma : [a,b] \to M$ is $u$-calibrated, then it is minimizing for the Lagrangian $L$.
		\item For dominated maps, the non-negativity of the defect of calibration implies that this defect is non-decreasing with respect to the size of the interval. As a consequence, if a curve $\gamma$ is calibrated on an interval $[a,b]$, then it is calibrated on every subinterval $[s,t] \subset [a,b]$.
	\end{enumerate}
	
\end{rem}

We now state the theorem from weak KAM theory that detects the regularity of dominated functions along their calibrated curves.

\begin{theo}[\cite{fathi2008weak}] \label{Fathi}
	\begin{enumerate}
		\item If $\gamma : I \to M$ is a $u$-calibrated curve, then for every time $t$ in the interior of $I$, $u$ is differentiable at $(t, \gamma(t))$ and
		\begin{equation}
			d_qu(t,\gamma(t)) = \partial_vL \big(t,\gamma(t), \dot{\gamma}(t) \big) \quad \text{and} \quad \partial_t u(t,\gamma(t)) + H(t,\gamma(t),d_qu(t,\gamma(t))) =0
		\end{equation}

		\item Let $A_{\varepsilon,u}$ be the set of points $(t,q) \in \mathbb{R} \times M$ such that there exists a curve $\gamma_{t,q} : (t-\varepsilon, t+\varepsilon) \to M$ with $\gamma_{t,q}(t) = q$ which is $u$-calibrated. Then, the map 
		\begin{align*}
			A_{\varepsilon,u} &\longrightarrow T^*(\mathbb{R} \times M) \\
			(t,q) &\longmapsto du (t,q) = \big(t,\partial_tu(t,q), q, d_qu(t,q)\big)
		\end{align*}
		is locally Lipschitz. 
	\end{enumerate}		
\end{theo}

\section{The Brane Distance} \label{section:BraneDist}

In this section, we study the brane distance introduced in Definition \ref{BraneDef}. We prove several convergence properties, as well as its behaviour under symplectomorphisms and variational bare branes. We conclude by extending the notion of variational solutions to bare Lagrangian branes.

\subsection{Basic Properties}

The formula of $d_B : \mathcal{B}_c(T^*M) \times \mathcal{B}_c(T^*M) \to \mathbb{R}$ on the set of bare branes is written as follows. For $(\mathcal{K},h),(\mathcal{K}',h')\in \mathcal{B}_c(T^*M)$ and $x\in \mathcal{K}$, we first set
\begin{equation} \label{eq:BraneDistanceFormula0}
    \delta_{(\mathcal{K},h)}(x,(\mathcal{K}',h')) \coloneq \inf_{x' \in \mathcal{K}'} \Big\{ | h(x) - h'(x') | + d(x,x') \Big\}
\end{equation}
and
\begin{equation}
    \delta_{(\mathcal{K},h)}(\mathcal{K}',h') \coloneq \sup_{x \in \mathcal{K}} \delta_{(\mathcal{K},h)}(x,(\mathcal{K}',h')) = \sup_{x \in \mathcal{K}} \inf_{x' \in \mathcal{K}'} \Big\{ | h(x) - h'(x') | + d(x,x') \Big\}.
\end{equation}
Finally, we have
\begin{equation} \label{BraneDistanceFormula}
	\begin{split}
		d_B \big((\mathcal{K},h),(\mathcal{K}',h')\big)  &= \max  \Big\{ \delta_{(\mathcal{K},h)}(\mathcal{K}',h'), \delta_{(\mathcal{K}',h')}(\mathcal{K},h) \Big\}
	\end{split}
\end{equation}

When the maps $h$ and $h'$ are clear from the context, we may, for simplicity, write $d_B(\mathcal{K},\mathcal{K}')$, $\delta_\mathcal{K}(\mathcal{K}')$ and $\delta_\mathcal{K}(x,\mathcal{K}')$ instead of $d_B\big((\mathcal{K},h),(\mathcal{K}',h')\big)$, $\delta_{(\mathcal{K},h)}(\mathcal{K}',h')$ and $\delta_{(\mathcal{K},h)}(x,(\mathcal{K}',h'))$.

We recall that $\mathfrak{B}_c(T^*M)$ denotes the quotient of $\mathcal{B}_c(T^*M)$ by the action of $\R$ by addition of a constant, and that $d_B$ induces a map $d_{\mathfrak{B}}:\mathfrak{B}_c(T^*M)\times \mathfrak{B}_c(T^*M)\to \R$ defined by (\ref{distance on quotient}).

\begin{prop}
	The brane distance $d_{\mathfrak{B}}$ defines a metric on $\mathfrak{B}_c(T^*M)$.
\end{prop}
\begin{proof}
	Non-negativity and symmetry are immediate from the fact that $d_B$ is a metric. Finiteness comes from the finiteness of $d_B$. In order to show definiteness and the triangle inequality, we need the following lemma.
	\begin{lem} \label{BraneModuloLem}
		Let $(\mathcal{K},[h])$ and $(\mathcal{K}',[h'])$ be two compact bare branes in $\mathfrak{B}_c(T^*M)$. Then, the infimum in $d_{\mathfrak{B}} \big((\mathcal{K},[h]),(\mathcal{K}',[h'])\big)$ is realized i.e. there exists a real constant $c \in \mathbb{R}$ such that 
		\begin{equation}
			d_{\mathfrak{B}} \big((\mathcal{K},[h]),(\mathcal{K}',[h'])\big) = d_B \big((\mathcal{K},h+c),(\mathcal{K}',h')\big)
		\end{equation}
	\end{lem}
	\begin{proof}
		 We prove that 
		\begin{align*}
			\lim_{|c| \to +\infty} d_B \big((\mathcal{K},h+c),(\mathcal{K}',h')\big) = + \infty
		\end{align*}		
		Indeed, for every pair of points $(x,x')$ in $\mathcal{K} \times \mathcal{K}'$, we have
		\begin{align*}
			| h(x)+c - h'(x') | + d(x,x') &\geq |c| - | h(x) - h'(x')| + d(x,x') \geq |c| - |h(x)| - |h'(x')| + 0 \\
			&\geq |c| - \Vert h \Vert_\infty - \Vert h' \Vert_\infty  \longrightarrow	+ \infty \quad \text{as } |c| \to +\infty
		\end{align*}
		where the bound is uniform on the pair $(x,x')$. This yields the claimed limit. Hence, by continuity of the map $c \mapsto d_B \big((\mathcal{K},h+c),(\mathcal{K}',h')\big)$, its minimum is realized at some finite real constant $c \in  \mathbb{R}$.
	\end{proof}
	
	\textit{Definiteness.} Assume that $d_{\mathfrak{B}} \big((\mathcal{K},[h]),(\mathcal{K}',[h'])\big)=0$. Then, there exists a constant $c \in \mathbb{R}$ such that $d_B \big((\mathcal{K},h+c),(\mathcal{K}',h')\big)=0$. Thus, by definiteness of the distance $d_B$, we obtain the equalities $\mathcal{K} = \mathcal{K}'$ and $h+c = h'$ so that $[h] = [h']$.
	
	\textit{Triangle Inequality.} Consider three compact bare branes $(\mathcal{K},[h])$, $(\mathcal{K}',[h'])$ and $(\mathcal{K}'',[h''])$ in $\mathfrak{B}_c(T^*M)$. Let $c$ and $c''$ be two real constants realizing the infima in $d_{\mathfrak{B}} \big((\mathcal{K},[h]),(\mathcal{K}',[h'])\big)$ and $d_{\mathfrak{B}} \big((\mathcal{K}',[h']),(\mathcal{K}'',[h''])\big)$ as follows
	\begin{align*}
		d_{\mathfrak{B}} \big((\mathcal{K},[h]),(\mathcal{K}',[h'])\big) &= 	d_B \big((\mathcal{K},h+c),(\mathcal{K}',h')\big)	\\
		d_{\mathfrak{B}} \big((\mathcal{K}',[h']),(\mathcal{K}'',[h''])\big) &= d_B \big((\mathcal{K}',h'),(\mathcal{K}'',h''+c'')\big)
	\end{align*}
	Then, using the triangle inequality on $d_B$, we obtain
	\begin{align*}
		d_{\mathfrak{B}} \big((\mathcal{K},[h]),(\mathcal{K}'',[h''])\big) &\leq d_B \big((\mathcal{K},h+c),(\mathcal{K}'',h''+c'')\big) \\
		&\leq d_B \big((\mathcal{K},h+c),(\mathcal{K}',h')\big) + d_B \big((\mathcal{K}',h'),(\mathcal{K}'',h''+c'') \big) \\
		&= d_{\mathfrak{B}} \big((\mathcal{K},[h]),(\mathcal{K}',[h'])\big) + d_{\mathfrak{B}} \big((\mathcal{K}',[h']),(\mathcal{K}'',[h''])\big)
	\end{align*}
\end{proof}

We list some properties of the brane distance.
\begin{prop} \label{BraneHausdorffProp}
	Let $(\mathcal{K}^n,h^n)$ and $(\mathcal{K},h)$ be compact bare branes of $\mathcal{B}_c(T^*M)$.
    If $(\mathcal{K}^n,h^n)$ brane-converges to $(\mathcal{K},h)$, then $\mathcal{K}^n$ converges to $\mathcal{K}$ in the Hausdorff topology. More precisely, if we denote by $d_H$ the Hausdorff distance on the compact sets of $T^*M$, then we have the inequality $d_H \leq d_B$.
\end{prop}

\begin{proof}
This is due to the inequality
	\begin{align*}
		\sup_{x \in \mathcal{K}} \inf_{x' \in \mathcal{K}'} d(x,x') \leq \sup_{x \in \mathcal{K}} \inf_{x' \in \mathcal{K}'} \Big\{ |h(x) - h'(x')| + d(x,x') \Big\} = \delta_{(\mathcal{K},h)}(\mathcal{K}',h')
	\end{align*}
	and to the Definition \eqref{Hausdorff} of the Hausdorff distance.	
\end{proof}

A useful pointwise version of the previous proposition is the following
\begin{prop} \label{BranePointwiseConv}
	Let $(\mathcal{K}^n,h^n)$ be a sequence of bare branes that brane-converges to $(\mathcal{K},h)$ in $B_c(T^*M)$. If $x^n \in \mathcal{K}^n$ is a sequence of points converging to a point $x$ of $\mathcal{K}$, then $\lim_n h^n(x^n) = h(x)$.
\end{prop}
\begin{proof}
	Let $y_n$ be a point in $\mathcal{K}$ realizing the infimum in $\delta_{(\mathcal{K}^n,h^n)}(x^n,(\mathcal{K},h))$. Then
	\begin{align*}
		|h^n(x^n) - h(x) | \leq |h^n(x^n) - h(y_n)| + |h(y_n) - h(x)| \leq \delta_{(\mathcal{K}^n,h^n)}(\mathcal{K},h) + |h(y_n) - h(x)|
	\end{align*}
	Moreover, we have
	\begin{align*}
		d(y_n,x) \leq d(y_n,x^n) + d(x^n,x) \leq \delta_{(\mathcal{K}^n,h^n)}(\mathcal{K},h) + d(x^n,x) \longrightarrow 0 \quad \text{as } n \to +\infty
	\end{align*}
	Hence, by continuity of the map $h$, we get that $\lim_n h(y_n) = h(x)$ so that $\lim_n h^n(x^n) = h(x)$.
\end{proof}

\begin{prop} \label{prop:Connectedness}
	Every bare Lagrangian brane $(\cL,h)$ in $\overline{\mathcal{LB}}(T^*M,\omega)$ is connected.
\end{prop}

\begin{proof}
	Let $(\cL_n,h_n)$ be a sequence of compact Lagrangian branes that brane-converges to $(\cL,h)$. Each $\cL_n$ is connected. Indeed, otherwise two of its connected components would be disjoint compact exact Lagrangian submanifolds of $T^*M$, contradicting the exact Lagrangian intersection theorem \cite{MR2565051} that asserts that two distinct exact Lagrangian submanifolds do intersect. Since we know from Proposition \ref{BraneHausdorffProp} that brane-convergence implies Hausdorff convergence of the supports, we deduce that $\cL$ is connected.
\end{proof}

\subsection{Operations on Branes}

\subsubsection{Image under an Exact Symplectomorphism}

\begin{prop} \label{BraneInvariance}
	Let $(\mathcal{K}^n,h^n)$ and $(\mathcal{K},h)$ be compact bare branes in $\mathcal{B}_c(T^*M)$, and let $\varphi \in \sympex (T^*M,\lambda)$ be an exact symplectic diffeomorphism. Let $f : T^*M \to \R$ be such that $\varphi^* \lambda - \lambda = df$. Let $l^n : \varphi(\mathcal{K}^n) \to \mathbb{R}$ and $l : \varphi(\mathcal{K}) \to \mathbb{R}$ be the maps defined by
	\begin{equation}
		l^n = (h^n + f) \circ \varphi^{-1} \quad \text{and} \quad l = (h + f) \circ \varphi^{-1} 
	\end{equation}
    
    \begin{enumerate}
        \item \label{item:BraneInvariance1} If the sequence $(\mathcal{K}^n,h^n)$ brane-converges to $(\mathcal{K},h)$ in $\mathcal{B}_c(T^*M)$, then the sequence $(\varphi(\mathcal{K}^n),l^n)$ brane-converges to $(\varphi(\mathcal{K}),l)$. The same also holds in $\mathfrak{B}_c(T^*M)$.
        \item \label{item:BraneInvariance2} If $(\mathcal{K},h)$ is a Lagrangian brane, then $(\varphi(\mathcal{K}),l)$ is a Lagrangian brane.
        \item \label{item:BraneInvariance3} If $(\mathcal{K},h)$ is a bare Lagrangian brane, then $(\varphi(\mathcal{K}),l)$ is a bare Lagrangian brane.
        
    \end{enumerate}          
\end{prop}

\begin{proof}
    \ref{item:BraneInvariance1}. Consider the lift $\Phi\colon J^1M\to J^1M$ of $\varphi$ to $J^1M=T^*M\times\R_z$ given by
    \[\Phi(x,z)=(\varphi(x),z+f(x)).\]
    It is a diffeomorphism that maps $\Gamma_{(\mathcal{K}^n,h^n)}$ to $\Gamma_{(\varphi(\mathcal{K}^n),l^n)}$. By brane-convergence, $\Gamma_{(\mathcal{K}^n,h^n)}$ converges in the Hausdorff topology to $\Gamma_{(\mathcal{K},h)}$, and therefore $\Phi(\Gamma_{(\mathcal{K}^n,h^n)})=\Gamma_{(\varphi(\mathcal{K}^n),l^n)}$ converges to $\Phi(\Gamma_{(\mathcal{K},h)})=\Gamma_{(\varphi(\mathcal{K}),l)}$. Hence, the sequence $(\varphi(\mathcal{K}^n),l^n)$ brane-converges to $(\varphi(\mathcal{K}),l)$.

    For convergence in $\mathfrak{B}_c(T^*M)$, one needs to see that if $(\mathcal{K}^n,[h^n])$ brane-converges to $(\mathcal{K},[h])$, then there exists a sequence of constants $c_n \in \mathbb{R}$ such that $(\mathcal{K}^n,h^n+c_n)$ brane-converges to $(\mathcal{K},h)$. Hence, by the preceding case, we get that $(\varphi(\mathcal{K}^n),l^n+c_n)$ brane-converges to $(\varphi(\mathcal{K}),l)$, and consequently $(\varphi(\mathcal{K}^n),[l^n])$ brane-converges to $(\varphi(\mathcal{K}),[l])$ in $\mathfrak{B}_c(T^*M)$.
    
    \ref{item:BraneInvariance2}. If the $(\mathcal{K}^n,h^n)$ are Lagrangian branes, then since $\varphi$ is a symplectomorphism, the $\varphi(\mathcal{K}^n)$ are Lagrangian submanifolds. We also have that the graphs $\Gamma_{(\mathcal{K}^n,h^n)}$ are Legendrians for the standard contact structure on $J^1M$ given by the contact form $\alpha = \lambda - dz$. But since $\varphi$ is exact, its lift $\Phi$ is a contactomorphism:
    \begin{align*}
        \Phi^*\alpha&=\varphi^*\lambda - d(z+f)\\
        &= \lambda + df - dz -df\\
        &= \alpha.
    \end{align*}
    Therefore, it maps Legendrians to Legendrians. The Legendrian condition for the graph of $l^n$ over the Lagrangian $\varphi(\mathcal{K}^n)$ exactly tells us that $l^n$ is a Liouville primitive. Therefore, $(\varphi(\mathcal{K}^n),l^n)$ is a Lagrangian brane.

    \ref{item:BraneInvariance3}. This is a straightforward consequence of \ref{item:BraneInvariance1} and \ref{item:BraneInvariance2}.   
\end{proof}

\subsubsection{Variational Branes}

Recall from Definition \ref{LagsImages} the notion of variational bare branes.

\begin{prop} \label{LagsImagesPrimitive}
	If $(\mathcal{L},h)\in \mathcal{LB}(T^*M,\omega)$ is a Lagrangian brane, then for all times $t \in \mathbb{R}$, $(\mathcal{L}_t, h_t) \in \mathcal{LB}(T^*M,\omega)$ is also a Lagrangian brane. 
\end{prop}
\begin{proof}
	Fix a time $t \in \mathbb{R}$. Let us define explicitly a map $f_t : T^*M \to \R$ such that $(\phi_H^t)^*\lambda - \lambda = df_t$. If we denote by $L$ the Lie derivative, then we have $\frac{d}{dt}(\phi_H^t)^*\lambda = (\phi_H^t)^*(L_{X^t_H}\lambda)$. The Cartan formula yields
	\begin{align*}
		L_{X^t_H}\lambda = \iota_{X_H^t}d\lambda + d(\iota_{X^t_H}\lambda) = -\iota_{X_H^t} \omega+ d(\lambda(X^t_H))=-dH_t + d(\lambda(X^t_H)) = d \big( \lambda(X^t_H) - H_t \big)
	\end{align*}
	where $X_H^t$ denotes the Hamiltonian vector field, and $H_t = H(t,\cdot,\cdot)$. Hence, by integrating $\frac{d}{dt}(\phi_H^t)^*\lambda$ between $0$ and $t$, we get
	\begin{align} \label{eq:HamiltonianExactForm}
		(\phi_H^t)^*\lambda - \lambda = \int_0^t \frac{d}{d\tau}(\phi_H^\tau)^*\lambda \; d\tau = d \left( \int_0^t (\phi_H^\tau)^* \big(\lambda(X^\tau_H) - H_\tau\big) \; d\tau \right)
	\end{align}
	At a point $x=(q,p)$ in $T^*M$ of orbit $x_\tau =(q_\tau,p_\tau) = \phi_H^\tau(x)$, we have
	\begin{align*}
		\lambda(X^t_H)(x) = p.\partial_pH(t,q,p) \quad \text{and} \quad (\phi_H^t)^* \big(\lambda(X^t_H)\big)(x) = p_\tau.\partial_pH(t,q_t,p_t) = p_t.\dot{q}_t = x_t^*\lambda
	\end{align*}
	so that if $f_t : T^*M \to \R$ denotes the map
	\begin{equation}
		f_t(x) = \int_0^t \big( x_\tau^* \lambda - H(\tau, x_\tau) \big) \;d\tau
	\end{equation}
	Then, we have $(\phi_H^t)^*\lambda - \lambda = df_t$. Furthermore, the identity \eqref{LagsImagesFormula} can be written as $h_t = (h +f_t) \circ (\phi_H^t)^{-1}$. Therefore, since $\phi_H^t \in \ham (T^*M,\omega)$, Proposition \ref{BraneInvariance} yields the result.
\end{proof}

\begin{prop} \label{ConvTimeProp}
	Let $(\mathcal{L}^n,h^n)$ be a sequence of bare branes that brane-converges to a bare brane $(\mathcal{L},h)$. Then for all times $t \in \mathbb{R}$, the sequence of bare branes $(\mathcal{L}_t^n,h_t^n)$ brane-converges to the bare brane $(\mathcal{L}_t,h_t)$. In particular, if $(\mathcal{L},h)$ belongs to $\overline{\mathcal{LB}(T^*M, \omega)}$, then for all times $t \in \R$, $(\mathcal{L}_t,h_t)$ belongs to $\overline{\mathcal{LB}(T^*M, \omega)}$.
\end{prop}

\begin{proof}
	This is a direct application of Propositions \ref{BraneInvariance} and \ref{LagsImagesPrimitive}.
\end{proof}

\subsection{Variational Solutions for Bare Lagrangian Branes} \label{SectionGSBC}

Recall from Subsection \ref{section:GS&VarSol} the definitions of graph selectors and variational solutions. In this section, we extend these notions from Lagrangian branes to bare Lagrangian branes.

\begin{prop} \label{BraneGammaProp}
	Let $(\mathcal{L}^n,h^n)$ be a sequence of Lagrangian branes that brane-converges to a bare Lagrangian brane $(\mathcal{L},h)$ in $\overline{\mathcal{LB}(T^*M,\omega)}$. Then the sequences of spectral invariants $\big(\rho(\mu, \mathcal{L}^{n+p},\mathcal{L}^n)\big)_n$ and $\big(\rho(\mu, \mathcal{L}^n,\mathcal{L}^{n+p})\big)_n$ converge to $0$ uniformly on $p\geq 0$. In particular, the sequence $\mathcal{L}^n$ (resp. $(\cL^n,h^n)$) is Cauchy for the spectral topology induced by the $\gamma$-distance (Definition \ref{def:SpectralDist}) (resp. $c$-distance, see Definition \ref{def:SpectralDistBrane}).
\end{prop}

\begin{proof}
	Since $\cL^{n+p}$ and $\cL^n$ are closed exact Lagrangian submanifolds, by Theorem \ref{Thm:equivZeroSec} their Floer homology is isomorphic to the homology of $M$ with coefficients in $\Z/2\Z$. Let $\mu$ denote its fundamental class. We know from the spectrality property (Proposition \ref{prop:spectrality}), that if $\cL^{n+p}$ and $\cL^n$ intersect transversely, then
	\begin{equation} \label{DemBraneGamma1}
		| \rho(\mu, \mathcal{L}^{n+p},\mathcal{L}^n) | \leq \sup_{x^n \in \mathcal{L}^{n+p} \cap \mathcal{L}^n} \Big\{ |h^{n+p}(x^n) - h^n(x^n)| \Big\}.
	\end{equation}
    By the $L^{1,\infty}$-Lipschitz property (Proposition \ref{prop:spectral_invariants}), this still holds if we do not assume transversality (take a $C^2$-small Hamiltonian perturbation to make them transverse; both the spectral invariants and the primitives vary continuously with respect to $C^2$-small perturbations).
    
     Let $\omega_h : \mathbb{R}_{\geq 0} \to \mathbb{R}_{\geq 0}$ be the continuity modulus of the uniformly continuous map $h : \mathcal{L} \to \mathbb{R}$. This means that $\omega_h$ is non-decreasing, that $\lim_{t \to 0} \omega_h(t) = 0$ and for every $x$ and $y$ in $\mathcal{L}$, $|h(x) - h(y)| \leq \omega_h(d(x,y))$. Recall the definition \eqref{eq:BraneDistanceFormula0} of $\delta_{\mathcal{L}^n}(x,\mathcal{L})$. Hence, we can take a point $x$ in $\mathcal{L}$ such that
    \begin{equation*}
		\delta_{\mathcal{L}^n}(x^n,\mathcal{L}) = |h^n(x^n) - h(x)| + d(x^n,x) 
	\end{equation*}
	Similarly, we take another point $x'$ in $\mathcal{L}$ such that
	\begin{equation*}
		\delta_{\mathcal{L}^{n+p}}(x^n,\mathcal{L}) = |h^{n+p}(x^n) - h(x')| + d(x^n,x')
	\end{equation*}
	Gathering the inequalities, we obtain
	\begin{equation} \label{DemBraneGamma2}
		\begin{split}
			|h^{n+p}(x^n) - h^n(x^n)| &\leq |h^{n+p}(x^n) - h(x')| + |h(x') - h(x)| + |h(x) - h^n(x^n)| \\
			&\leq \delta_{\mathcal{L}^n}(x^n,\mathcal{L}) + \delta_{\mathcal{L}^{n+p}}(x^n,\mathcal{L}) + |h(x') - h(x)| \\
			&\leq d_B(\mathcal{L}^n, \mathcal{L}) + d_B(\mathcal{L}^{n+p}, \mathcal{L}) + \omega_h( d(x',x))\\
		\end{split}
	\end{equation}
	Moreover, 
	\begin{align*}
		d(x,x') \leq d(x,x^n) + d(x^n,x') \leq \delta_{\mathcal{L}^n}(x^n,\mathcal{L}) + \delta_{\mathcal{L}^{n+p}}(x^n,\mathcal{L}) \leq d_B(\mathcal{L}^n, \mathcal{L}) + d_B(\mathcal{L}^{n+p}, \mathcal{L})
	\end{align*}
	Thus, we deduce from \eqref{DemBraneGamma1} and \eqref{DemBraneGamma2} the inequalities
	\begin{align} \label{inequality:SpectralDomination}
		| \rho(\mu, \mathcal{L}^{n+p},\mathcal{L}^n) | \leq d_B(\mathcal{L}^n, \mathcal{L}) + d_B(\mathcal{L}^{n+p}, \mathcal{L}) + \omega_h \big( d_B(\mathcal{L}^n, \mathcal{L}) + d_B(\mathcal{L}^{n+p}, \mathcal{L}) \big)
	\end{align}
	and the same inequality for $|\rho(\mu, \mathcal{L}^{n},\mathcal{L}^{n+p})|$, yielding that these sequences are Cauchy.
	
	For the spectral distance, we have
	\begin{equation}\label{DemBraneGamma3}
		\begin{split}
			\gamma( \mathcal{L}^{n+p}, \mathcal{L}^n) &= \rho(\mu, \mathcal{L}^{n+p},\mathcal{L}^n) + \rho(\mu, \mathcal{L}^{n},\mathcal{L}^{n+p}) \\
			&\leq 2d_B(\mathcal{L}^n, \mathcal{L}) + 2d_B(\mathcal{L}^{n+p}, \mathcal{L}) + 2\omega_h \big( d_B(\mathcal{L}^n, \mathcal{L}) + d_B(\mathcal{L}^{n+p}, \mathcal{L}) \big)
		\end{split}
	\end{equation}
	and the sequence $\mathcal{L}^n$ is Cauchy for the $\gamma$-distance. Similarly, $(\cL^n,h^n)$ is Cauchy for the $c$-distance.
\end{proof}

This allows us to define the graph selectors and the variational solutions associated to bare Lagrangian branes in $\overline{\mathfrak{LB}(T^*M,\omega)}$, as intended in the corollary below.

\begin{cor} \label{GSConv}
	Let $H : \mathbb{T}^1 \times T^*M \to \mathbb{R}$ be a Hamiltonian, and let $(\mathcal{L}^n,h^n) \in \mathcal{LB}(T^*M,\omega)$ be a sequence of Lagrangian branes that brane-converges to a bare Lagrangian brane $(\mathcal{L},h) \in \overline{\mathfrak{LB}(T^*M,\omega)}$. Let $u^n : M \to \mathbb{R}$ be graph selectors associated to $(\mathcal{L}^n,h^n)$. Then the sequence $u^n$ converges uniformly to a continuous function $u : M \to \mathbb{R}$. More generally, the variational solutions $u^n : \R \times M \to \R$ converge uniformly to a continuous function $u : \R \times M \to \R$.
	
	Moreover, these limit graph selectors and variational solutions $u$ are independent of the converging sequence of Lagrangian branes $(\mathcal{L}^n,h^n)$.
\end{cor}

\begin{proof}
	From Proposition \ref{Prop:GS<Gamma}, we already know that $\osc(u^{n+p}-u^n) \leq \gamma(\mathcal{L}^{n+p},\mathcal{L}^n)$. However, in our case, since the Liouville primitives $h^n$ converge (and not merely up to an additive constant), we can obtain a bound on $\Vert u^{n+p}-u^n \Vert_\infty^{\mathbb{R} \times M}$. Indeed, we have seen in \eqref{formula:GSSpectralBound} that for all $(t,q)\in\R\times M$,
    \begin{align*}
		-\rho(\mu,\cL^n,\cL^{n+p})\leq u^{n+p}(t,q) - u^n(t,q) \leq \rho(\mu,\cL^{n+p},\cL^n).
	\end{align*}
	so that
	\begin{align*}
		|u^{n+p}(t,q) - u^n(t,q)| \leq |\rho(\mu,\cL^{n+p},\cL^n)| + |\rho(\mu,\cL^n,\cL^{n+p})|
	\end{align*}
	Thus, we infer from Proposition \ref{BraneGammaProp} that the sequence $\Vert u^{n+p}- u^n \Vert_\infty^{\mathbb{R} \times M} \to 0$ uniformly on $p \geq 0$ and the sequence $u^n$ is Cauchy. By completeness of $C^0([-T,T] \times M, \R)$ for all times $T>0$, we conclude that the sequence $u^n : \mathbb{R} \times M \to \mathbb{R}$ does converge to a continuous limit $u : \mathbb{R} \times M \to \mathbb{R}$.
	
	For the independence of the converging sequence, consider a sequence of Lagrangian branes $(\cL_n,h_n)$ that brane-converges to $(\cL,h)$ and denote the corresponding variational solutions by $u_n$. Then, the inequality \eqref{inequality:SpectralDomination} applied to $\cL^n$ and $\cL_n$ yields
    \begin{equation*}
    		\Vert u^n - u_n \Vert_\infty^{\mathbb{R} \times M} \leq 2d_B(\mathcal{L}^n, \mathcal{L}) + 2d_B(\mathcal{L}_n, \mathcal{L}) + 2\omega_h \big( d_B(\mathcal{L}^n, \mathcal{L}) + d_B(\mathcal{L}_n, \mathcal{L}) \big) \longrightarrow 0 \quad \text{as } n\to +\infty,
    \end{equation*}
	and we conclude that their limits are equal.
\end{proof}

This corollary allows the following definition.
\begin{defi}
	Let $(\cL,h)$ be a bare Lagrangian brane in $\overline{\mathfrak{LB}(T^*M,\omega)}$. 
	\begin{enumerate}
		 \item The limit $u : M \to \mathbb{R}$ of any sequence of graph selectors of Lagrangian branes that brane-converge to $(\mathcal{L},h)$ is the \textit{graph selector} of $(\mathcal{L},h)$.
		 \item The limit $u : \R \times M \to \mathbb{R}$ of any sequence of variational solutions associated to Lagrangian branes that brane-converge to $(\mathcal{L},h)$ is the \textit{variational (sub)solution} associated to $(\mathcal{L},h)$.
	\end{enumerate}
\end{defi}

At the end of this subsection, in Remark \ref{rem:GSGamma-supp}, we describe a connection between this definition and the Humilière completion.\\

For now, let us present some properties of these variational subsolutions. As in the case of variational solutions associated to Lagrangian branes, these limiting variational subsolutions are also viscosity subsolutions, and in particular they are locally Lipschitz.

\begin{prop} \label{GSSubsolution}
    Let $(\cL,h) \in \overline{\mathcal{LB}(T^*M,\omega)}$ be a bare Lagrangian brane. For every sequence of exact Lagrangian branes $(\cL^n,h^n)$ that brane-converges to $(\cL,h)$, the associated variational solutions $u^n : \mathbb{R} \times M \to \mathbb{R}$ are equi-Lipschitz on $I \times M$ for every compact interval $I \subset \mathbb{R}$.

    In particular, the variational solution $u : \mathbb{R} \times M \to \mathbb{R}$ associated to $(\cL,h)$ is a locally Lipschitz very weak subsolution of the Hamilton--Jacobi equation \eqref{HJ}.
\end{prop}

\begin{rem}
	One can show that the variational solution $u$ is in fact a viscosity subsolution of the Hamilton--Jacobi equation \eqref{HJ}.
\end{rem}

\begin{proof}
	Let $(\cL^n,h^n)$ be a sequence of exact Lagrangian branes that brane-converges to $(\cL,h)$, and let $u^n : \R \times M \to \R$ be the associated variational solutions. By Corollary \ref{GSConv}, the sequence $u^n$ converges uniformly to the variational solution $u$ associated to $(\cL,h)$.

	We first show that the sequence $u^n$ is equi-Lipschitz on every compact time interval. Fix $T>0$. Since brane-convergence implies Hausdorff convergence of the supports, there exists a compact subset $K \subset T^*M$ containing all the Lagrangian submanifolds $\cL^n$. We set
	\begin{equation*}
		K_T := \left\{ \phi_H^t(x) \; \middle| \; x\in K,\ |t|\leq T \right\}
	\end{equation*}
	The set $K_T$ is compact and, for every $n$ and $|t|\leq T$,
	\begin{equation*}
		\cL_t^n = \phi_H^t(\cL^n) \subset K_T
	\end{equation*}
	By Proposition \ref{Graphext}, for almost every $(t,q)\in[-T,T]\times M$, we have
	\begin{equation*}
		(q,d_qu^n(t,q))\in\cL_t^n \quad \text{and} \quad \partial_tu^n(t,q) +H(t,q,d_qu^n(t,q)) =0
	\end{equation*}
	Hence, for almost every $(t,q)\in[-T,T]\times M$,
	\begin{equation*}
		\Vert d_qu^n(t,q)\Vert\leq P_T:= \max_{(q,p)\in K_T}\Vert p\Vert \quad \text{and} \quad |\partial_tu^n(t,q)|\leq C_T := \max_{(t,q,p)\in [-T,T] \times K_T} |H(t,q,p)|
	\end{equation*}
	uniformly in $n$. Since the maps $u^n$ are locally Lipschitz, these bounds on their differentials imply that the sequence $u^n$ is equi-Lipschitz on $[-T,T]\times M$. By uniform convergence of $u^n$ to $u$, we deduce that $u$ is locally Lipschitz on $\R\times M$. 
	
	We now prove the subsolution property. By Proposition \ref{Domination}, every $u^n$ is dominated by $L$. Thus, for every $a<b$ and every $C^1$ curve $\gamma:[a,b]\to M$, we have
	\begin{equation*}
		u^n(b,\gamma(b))-u^n(a,\gamma(a)) \leq \int_a^b L(\tau,\gamma(\tau),\dot{\gamma}(\tau)) \; d\tau
	\end{equation*}
	Passing to the limit $n\to+\infty$ yields
	\begin{equation*}
		u(b,\gamma(b))-u(a,\gamma(a)) \leq \int_a^b L(\tau,\gamma(\tau),\dot{\gamma}(\tau)) \; d\tau
	\end{equation*}
	Hence, $u$ is also dominated by $L$. By Proposition \ref{DominatedSubsolution}, at every point of differentiability $(t,q)$ of $u$, we have
	\begin{equation*}
		\partial_tu(t,q) + H(t,q,d_qu(t,q)) \leq 0
	\end{equation*}
	Therefore, $u$ is a locally Lipschitz very weak subsolution of the Hamilton--Jacobi equation \eqref{HJ}.
\end{proof}

We now state the analogue of Proposition \ref{Graphext} for these variational subsolutions. This proposition will not be used, but we include it for the sake of completeness.

\begin{prop}
	Let $(\cL,h)$ be a bare Lagrangian brane in $\overline{\mathcal{LB}(T^*M,\omega)}$ with associated variational subsolution $u$. Then there exists a dense set of full measure $\mathcal{V}$ of $\mathbb{R} \times M$ such that $u$ is differentiable in $\mathcal{V}$ and for all $(t,q) \in \mathcal{V}$, 
	\begin{equation} \label{eq:Graphext2}
		\big(q, d_qu(t,q)\big) \in \conv (\mathcal{L}_t\cap T^*_qM) \quad \text{and} \quad \partial_tu + H(t,q,d_qu) \leq 0
	\end{equation}
    where $\conv$ denotes the convex hull.
\end{prop}
\begin{proof}
	Let $(\mathcal{L}^n,h^n)$ be a sequence of Lagrangian branes that brane-converges to $(\mathcal{L},h)$. Let $u^n$ be the associated variational solutions, and set $\mathcal{U}^n$ to be the open dense sets obtained by Proposition \ref{Graphext}. We know from Proposition \ref{GSSubsolution} that $u$ is Lipschitz, so that there exists a dense full measure set $\mathcal{V}'$ where $u$ is differentiable. Then, we set $\mathcal{V} =\mathcal{V}' \cap ( \bigcap_{n \geq 0} \mathcal{U}^n )$. The identities in \eqref{eq:Graphext2} are direct consequences of Lemma \ref{lem:BdS} and the identity \eqref{eq:SpectralityLemmaDem5}.
\end{proof}

This proposition explains why the term \textit{variational subsolution} may be more appropriate than that of \textit{variational solution} for bare Lagrangian branes. Moreover, the spectrality relation $h_t(x,d_q u) = u(t,q)$ stated in \eqref{eq:Graphext} no longer holds, which creates an additional difficulty in the proof of the main theorem. However, Proposition \ref{SpectralityLemma} allows one to recover this spectrality along calibrated curves.

\begin{rem} \label{rem:GSGamma-supp}
	The notion of graph selector extends naturally to the Humilière completion (see \ref{section:SpectralHumiliere}). Indeed, the estimate $\osc( u_{\cL_1}-u_{\cL_2})\leq \gamma(\cL_1,\cL_2)$ shows that the graph-selector map is continuous with respect to the spectral distance. Hence, if a sequence $\cL_n$ converges to an element $\underline{\cL}$ of the Humilière completion, the corresponding graph selectors $u_{\cL_n}$ converge uniformly, up to additive constants, to a function depending only on $\underline{\cL}$, which defines the graph selector of $\underline{\cL}$. The same interpretation applies to the variational (sub)solutions considered here, and the variational (sub)solution associated to a bare Lagrangian brane coincides with the variational solution associated to the corresponding element of the Humilière completion.
\end{rem}

\section{Proof of the Main Theorem} \label{section:Proof}

In this section, we prove the main Theorem \ref{MainTheorem}. We adopt the notations of its statement.\\

Let $(\mathcal{L},[h])$ be a bare Lagrangian brane in $\overline{\mathfrak{LB}(T^*M,\omega)}$ with associated variational subsolution $u : \R \times M \to \R$. By Lemma \ref{BraneModuloLem}, there exists a sequence $(\mathcal{L}^k,h^k)$ of Lagrangian branes in $\mathcal{LB}(T^*M,\omega)$ that brane-converges to $(\mathcal{L},h)$. We denote the associated variational solutions by $u^k: \mathbb{R} \times M \to \mathbb{R}$. We have shown in Proposition \ref{GSSubsolution} that $u$ and $u^k$ are locally equi-Lipschitz subsolutions of the Hamilton--Jacobi equation \eqref{HJ}.\\

We will see in Lemma \ref{lem:RealtoInteger} that the sequences of times $t_k$ and $s_k$ can be taken to be integers. Hence, we may assume that there exist two increasing integer sequences $n_k$ and $m_k \to +\infty$ such that the sequences $(\mathcal{L}_{n_k},[h_{n_k}])$ and $(\mathcal{L}_{-m_k},[h_{-m_k}])$ brane-converge, respectively, to two bare branes $(\mathcal{L}_\omega,[h_\omega])$ and $(\mathcal{L}_\alpha,[h_\alpha])$ in $\mathfrak{B}_c(T^*M)$. Hence, by the same Lemma \ref{BraneModuloLem}, there exist two sequences $c_{n_k}$ and $c_{-m_k}$ of real constants such that $(\mathcal{L}_{n_k},h_{n_k}+c_{n_k})$ and $(\mathcal{L}_{-m_k},h_{-m_k}+c_{-m_k})$ brane-converge, respectively, to $(\mathcal{L}_\omega,h_\omega)$ and $(\mathcal{L}_\alpha,h_\alpha)$.\\

Our goal is to prove that the variational subsolution $u$ is $C^{1,1}$ so that $\cL_t \subset \mathcal{G}(d_q u(t,\cdot))$. This inclusion then implies the equality $\cL_t = \mathcal{G}(d_q u(t,\cdot))$, showing that $\cL_t$ is a Lipschitz graph over the base $M$. Our main tool will be Theorem \ref{Fathi} on the regularity of dominated maps along calibrated curves, which, as we have seen, applies to $u$. We will look for tame calibrated curves for $u$.

\subsection{Reduction to Integer Times}

\begin{lem} \label{lem:RealtoInteger}
	Under the hypotheses of Theorem \ref{MainTheorem}, there exist two increasing sequences $m_k$ and $n_k$ of positive integers and two bare branes $(\mathcal{L}^0_\omega,[h^0_\omega])$ and $(\mathcal{L}^0_\alpha,[h^0_\alpha])$ in $\mathcal{B}_c(T^*M)$ such that $(\mathcal{L}_{n_k},[h_{n_k}])$ and $(\mathcal{L}_{-m_k},[h_{-m_k}])$ brane-converge, respectively, to $(\mathcal{L}^0_\omega,[h^0_\omega])$ and $(\mathcal{L}^0_\alpha,[h^0_\alpha])$.
\end{lem}

\begin{proof}
	We only prove the result for $\cL_\omega$. For all real $t \in \R$, $\lfloor t \rfloor$ denotes the integer part of $t$ and $\{t\} \coloneq t - \lfloor t \rfloor \in [0,1)$ denotes its fractional part. Up to extraction, we can assume that $\lim_k \{t_k\} = T \in [0,1]$, and we set $\cL^0_\omega = (\phi_H^T)^{-1}(\cL_\omega)$ and $h^0_\omega : \cL^0_\omega \to \R$ defined by
	\begin{align*}
		h^0_\omega(x_0) = h_\omega(x_T) - \int_0^T \big( x_\tau^* \lambda - H(\tau, x_\tau) \big) \;d\tau
	\end{align*}
	where $x_0$ belongs to $\cL^0_\omega$ and $x_\tau = \phi_H^\tau(x_0)$. We will show that $(\cL_{\lfloor t_k \rfloor},[h_{\lfloor t_k \rfloor}])$ brane-converges to $(\mathcal{L}^0_\omega,[h^0_\omega])$. By Proposition \ref{BraneInvariance} applied to $\varphi = (\phi_H^T)^{\pm1}$, this is equivalent to $(\cL_{\lfloor t_k \rfloor+T},[h_{\lfloor t_k \rfloor+T}])$ being brane-convergent to $(\mathcal{L}_\omega,[h_\omega])$. We have
	\begin{align*}
		d_B\big((\cL_{\lfloor t_k \rfloor+T},[h_{\lfloor t_k \rfloor+T}]),(\mathcal{L}_\omega,[h_\omega])\big) \leq  d_B\big((\cL_{\lfloor t_k \rfloor+T},[h_{\lfloor t_k \rfloor+T}]),(\cL_{ t_k },[h_{t_k}])\big) + d_B\big((\cL_{ t_k },[h_{t_k}]),(\mathcal{L}_\omega,[h_\omega])\big)
	\end{align*}
	Moreover, since $\lim_k d_B\big((\cL_{ t_k },[h_{t_k}]),(\mathcal{L}_\omega,[h_\omega])\big)=0$, it suffices to show that 
	\begin{align*}
		\lim_k d_B\big((\cL_{\lfloor t_k \rfloor+T},[h_{\lfloor t_k \rfloor+T}]),(\cL_{ t_k },[h_{t_k}])\big) = 0
	\end{align*}
	Let $K$ be a compact neighbourhood of $\cL_\omega$ for the Hausdorff topology, so that for large $k$, $\cL_{t_k}$ is included in $K$. Thus, since $\lim_k \{t_k\} = T$, we deduce that the restriction of $\phi_H^{\{t_k\},T}$ to $K$ uniformly converges to the identity. We obtain 
	\begin{align*}
		\sup_{x_{t_k} \in \cL_{t_k}} d \big( \phi_H^{\{t_k\},T}(x_{t_k}), x_{t_k}  \big) \leq d\left(\phi_{H}^{\{t_k\},T}\big|_K, \mathrm{Id}_K \right) \longrightarrow 0 \quad \text{as } k \to +\infty
	\end{align*}
	Moreover, if $x_{t_k}$ is a point of $\cL_{t_k}$ and $x_\tau = (q_\tau,p_\tau) = \phi_H^{t_k,\tau}(x_{t_k})$, then we have $x_{\lfloor t_k \rfloor+T}$ belongs to $\cL_{\lfloor t_k \rfloor+T}$ and
	\begin{align*}
		h_{\lfloor t_k \rfloor+T}(x_{\lfloor t_k \rfloor+T}) - h_{t_k}(x_{t_k}) &= \int_{t_k}^{\lfloor t_k \rfloor+T} \big( x_\tau^*\lambda - H(\tau,x_\tau) \big) \; d\tau =  \int_{t_k}^{\lfloor t_k \rfloor+T} L(\tau, q_\tau, \dot{q}_\tau) \; d\tau  \\
		&= \int_{t_k}^{\lfloor t_k \rfloor+T} L(\tau, q_\tau, \partial_pH(\tau,x_\tau) ) \; d\tau
	\end{align*}
	where we have used Remark \ref{LagsImagesLag} and the Legendre map $\leg$ defined in \eqref{Legendrian}. We consider the compact set
	\begin{align*}
		\phi_H^{[0,1],[-1,1]}(K) \coloneq \left\{ \phi_H^{s,t}(x) \; | \; x \in K, \; s \in [0,1], \; t \in [-1,1]  \right\}
	\end{align*}
	and we set $C>0$ to be larger than the maximum of the map $(\tau,q,p) \mapsto |L(\tau, q, \partial_pH(\tau,q,p))|$ on $\T^1 \times \phi_H^{[0,1],[-1,1]}(K)$. We have for all time $\tau$ between $t_k$ and $\lfloor t_k \rfloor+T$
	\begin{align*}
		x_\tau \in \cL_\tau = \phi_H^{t_k,\tau}(\cL_{t_k}) = \phi_H^{\{t_k\},\tau-\lfloor t_k \rfloor}(\cL_{t_k}) \subset \phi_H^{\{t_k\},\tau-\lfloor t_k \rfloor}(K) \subset \phi_H^{[0,1],[-1,1]}(K)
	\end{align*}
	for large $t_k$. Hence, we deduce that 
	\begin{align*}
		\left|h_{\lfloor t_k \rfloor+T}(x_{\lfloor t_k \rfloor+T}) - h_{t_k}(x_{t_k})\right| \leq C|\lfloor t_k \rfloor+T - t_k| = C|T - \{ t_k \}| \longrightarrow 0 \quad \text{as } k \to \infty
	\end{align*}
	By definition of the brane distance $d_B$, we conclude that 
	\begin{align*}
		d_B\big((\cL_{\lfloor t_k \rfloor+T},[h_{\lfloor t_k \rfloor+T}]),(\cL_{ t_k },[h_{t_k}])\big) &\leq \sup_{x_{t_k} \in \cL_{t_k}} d \big( \phi_H^{\{t_k\},T}(x_{t_k}), x_{t_k}  \big)  + \sup_{x_{t_k} \in \cL_{t_k}} \left\{ \left|h_{\lfloor t_k \rfloor+T}(x_{\lfloor t_k \rfloor+T}) - h_{t_k}(x_{t_k})\right| \right\} \\
		&\leq d\left(\phi_{H}^{\{t_k\},T}\big|_K, \mathrm{Id}_K \right) + C|T - \{ t_k \}| \longrightarrow 0 \quad \text{as } k \to +\infty
	\end{align*}
\end{proof}

\subsection{Calibration on all Orbits}
By the previous section, we may assume that there exist increasing sequences of integer times $(n_k)$ and $(m_k)$ such that the sequences $(\cL_{n_k},[h_{n_k}])$ and $(\cL_{-m_k},[h_{-m_k}])$ converge to some $(\cL_\omega,[h_\omega])$ and $(\cL_\alpha,[h_\alpha])$ respectively.

Our goal will be to show that the variational graph selector $u\colon\R\times M\to \R$ associated to $\cL$ admits calibrated curves through every point, in order to apply Theorem \ref{Fathi}.

Let $(t,q)\in \R\times M$. Since the projection $\cL_t\to M$ is surjective, there exists a point $x=(q,p)\in\cL_t$. We denote by $x_\tau=(q_\tau,p_\tau)=\phi^{t,\tau}_H(x)$ its orbit under the Hamiltonian flow. This section is dedicated to the proof of

\begin{prop} \label{CalibTout}
	The curve $q_\tau$ is calibrated by the variational solution $u$ of $(\mathcal{L},h)$.
\end{prop}

We divide the proof in three steps: we first show a calibration estimate involving a defect of spectrality at limit points, then prove calibration at limit points, and finally prove a spectrality identity above calibrated curves for bare Lagrangian branes.

\subsubsection{A Calibration Estimate}

Let $[a,b]$ be a time interval containing $t$ in its interior. We want to show that the calibration defect for the curve $\tau\in[a,b]\mapsto q_\tau\in M$ vanishes. For $k$ large enough, we have $[a,b]\subset[-m_k,n_k]$, and since $u$ is dominated, positivity of the defect implies that
\begin{equation}
\begin{split}
    0\leq \delta(u,q_\tau,[a,b])&\leq \delta(u,q_\tau,[-m_k,n_k])\\
    &\leq \int_{-m_k}^{n_k} L(\tau, q_\tau, \dot{q}_\tau) \; d\tau - \big[ u(n_k,q_{n_k}) -  u(-m_k,q_{-m_k}) \big].
\end{split}
\end{equation}
We know from the Definition \ref{LagsImages} of variational bare branes and from Remark \ref{LagsImagesLag} that
	\begin{align*}
		\int_{-m_k}^{n_k} L(\tau,q_\tau,\dot{q}_\tau) \; d\tau &= h_{n_k}(x_{n_k}) - h_{-m_k}(x_{-m_k})
	\end{align*}
so that we have
\begin{equation}\label{eq:calibration_ineq}
\begin{split}
    0\leq \delta(u,q_\tau,[a,b])&\leq \big[ h_{n_k}(x_{n_k}) - h_{-m_k}(x_{-m_k}) \big] - \big[ u_{n_k}(0,q_{n_k}) -  u_{-m_k}(0,q_{-m_k}) \big],
\end{split}
\end{equation}
where we have used the notation $u_\tau=u(\tau+\cdot,\cdot)$.

Since $(\cL_{n_k},[h_{n_k}])$ and $(\cL_{-m_k},[h_{-m_k}])$ brane converge to compact branes, we may assume, up to extraction, that the sequences $(x_{n_k})$ and $(x_{-m_k})$ converge respectively to points $x_\omega=(q_\omega,p_\omega)\in\cL_\omega$ and $x_\alpha=(q_\alpha,p_\alpha)\in\cL_\alpha$. Then, brane convergence implies the existence of sequences of real numbers $(c_{n_k})$ and $(c_{-m_k})$ such that
\[\lim_{k\to+\infty}h_{n_k}(x_{n_k})+c_{n_k}=h_\omega(x_\omega) \quad, \quad \lim_{k\to+\infty}h_{-m_k}(x_{-m_k})+c_{-m_k}=h_\alpha(x_\alpha)\]
and
\[\lim_{k\to+\infty}u_{n_k}(0,q_{n_k})+c_{n_k}=u_\omega(0,q_\omega) \quad, \quad \lim_{k\to+\infty}u_{-m_k}(0,q_{-m_k})+c_{-m_k}=u_\alpha(0,q_\alpha),\]
where $u_\omega$ and $u_\alpha$ are the variational solutions of $(\cL_\omega,h_\omega)$ and $(\cL_\omega,h_\omega)$ respectively.

One can artificially introduce cancelling constants $\pm c_{n_k}$ and $\pm c_{-m_k}$ in \eqref{eq:calibration_ineq} and take the limit to get
\begin{align}\label{eq:calibration_ineq2}
    0\leq \delta(u,q_\tau,[a,b])&\leq \big[h_\omega(x_\omega)-u_\omega(q_\omega)\big]-\big[h_\alpha(x_\alpha)-u_\alpha(q_\alpha)\big].
\end{align}

\subsubsection{Calibration at Limit Points}
We keep the notations of the previous section. For $i\in\{\alpha,\omega\}$, let $x_{i,\tau}=(q_{i,\tau},p_{i,\tau})=\phi^\tau_H(x_i)$ be the orbit of $x_i$.

\begin{prop} \label{CalibLimit}
	The curves $q_{\alpha,\tau}$ and $q_{\omega,\tau}$ are respectively calibrated by $u_\alpha$ and $u_\omega$.
\end{prop}

\begin{proof}
	We only prove the calibration for $q_{\omega,\tau}$. The case of $q_{\alpha,\tau}$ is done analogously. Let $a<b$ be two times. We need to prove that the defect of calibration $\delta(u_\omega, q_\omega,[a,b])$ vanishes.
	
	The following lemma links the calibration defect of $q_{\omega,\tau}$ to that of $q_{n_k+\tau}$.
	\begin{lem}\label{LemDeltaConv}
		We have
		\begin{equation*}
			\delta(u_\omega, q_{\omega,\cdot},[a,b]) = \lim\limits_k  \delta(u_{n_k}, q_{n_k+\cdot},[a,b])
		\end{equation*}
	\end{lem}
	\begin{proof}
		We are studying the limit of 
		\begin{equation*}
			\delta(u_{n_k}, q_{n_k+\tau},[a,b]) = \int_a^b L(\tau, q_{n_k+\tau}, \dot{q}_{n_k+\tau}) \; d\tau - \big[ u_{n_k}(b, q_{n_k+b}) - u_{n_k}(a, q_{n_k+a}) \big].
		\end{equation*}
        Note that we have used here the periodicity of $L$ ($H$ is a 1-periodic Hamiltonian, and the $n_k$ are integers):
        \[L(n_k+\tau, q_{n_k+\tau}, \dot{q}_{n_k+\tau})=L(\tau, q_{n_k+\tau}, \dot{q}_{n_k+\tau}).\]
		We look at the limit of the curve $q_{n_k+\cdot} : [a,b] \to M$. Recall that the sequence $x_{n_k}$ converges to $x_\omega$. This implies that the curve $x_{n_k+\cdot} \colon [a,b] \to T^*M$ uniformly converges to the curve $x_{\omega,\cdot} \colon [a,b] \to T^*M$. Since $\leg^{-1}(x_{n_k+\tau}) = (q_{n_k+\tau},\dot{q}_{n_k+\tau})$ and $\leg^{-1}(x_{\omega,\tau}) = (q_{\omega,\tau},\dot{q}_{\omega,\tau})$, we obtain the convergence of the Lagrangian actions
		\begin{equation} \label{LemDeltaConvDem1}
			\lim_{k \to + \infty} \int_a^b L(\tau, q_{n_k+\tau}, \dot{q}_{n_k+\tau}) \; d\tau  = \int_a^b L(\tau, q_{\omega,\tau}, \dot{q}_{\omega,\tau}) \; d\tau. 
		\end{equation}
		
		Additionally, we know from Corollary \ref{GSConv} that the variational solutions $u_{n_k}+c_{n_k}$ converge uniformly on $\mathbb{R} \times M$ to the variational solution $u_\omega$. Hence, the pointwise convergence $\lim_k q_{n_k+a} = q_{\omega,a}$ and $\lim_k q_{n_k+b} = q_{\omega,b}$ yield the limit
		\begin{equation} \label{LemDeltaConvDem2}
			\begin{split}
				\lim_{k\to +\infty} \big[ u_{n_k}(b, q_{n_k+b}) - u_{n_k}(a, q_{n_k+a}) \big] &= \lim_{k\to +\infty} \big[ (u_{n_k}(b, q_{n_k+b}) + c_{n_k}) - (u_{n_k}(a, q_{n_k+a}) + c_{n_k}) \big] \\
				&= u_\omega(b, q_{\omega,b}) - u_\omega(a, q_{\omega,a}).
			\end{split}
		\end{equation}
		
		Gathering \eqref{LemDeltaConvDem1} and \eqref{LemDeltaConvDem2} gives the desired limit
		\begin{align*}
			\lim_{k \to +\infty} \delta(u_{n_k}, q_{n_k+\cdot},[a,b]) &= \lim_{k \to +\infty} \left( \int_a^b L(\tau, q_{n_k+\tau}, \dot{q}_{n_k+\tau}) \; d\tau - \big[ u_{n_k}(b, q_{n_k+b}) - u_{n_k}(a, q_{n_k+a}) \big] \right) \\
			&= \int_a^b L(\tau, q_{\omega,\tau}, \dot{q}_{\omega,\tau}) \; d\tau - \big[ u_\omega(b, q_{\omega,b}) - u_\omega(a, q_{\omega,a}) \big]\\
			&= \delta(u_\omega, q_{\omega,\cdot},[a,b]).
		\end{align*}
	\end{proof}
	
	Now set $b_k = a + n_{k+1}-n_k$. We may assume, up to extraction, that $n_{k+1}-n_k\longrightarrow+\infty$ as $k$ goes to infinity, so that we have the inclusion $[a,b] \subset [a,b_k]$ for large $k$. Thus, the lemma and the positivity of the defect of calibration (Proposition \ref{Domination}) lead to
	\begin{equation} \label{CalibLimDem0}
		0 \leq \delta(u_\omega, q_{\omega,\cdot},[a,b]) = \lim\limits_{k \to +\infty} \delta(u_{n_k}, q_{n_k+\cdot},[a,b]) \leq \liminf_{k \to +\infty} \delta(u_{n_k}, q_{n_k+\cdot},[a,b_k]).
	\end{equation}
	Let us evaluate
	\begin{align*}
		\delta(u_{n_k}, q_{n_k+\cdot},[a,b_k]) &= \int_a^{b_k} L(\tau,q_{n_k+\tau},\dot{q}_{n_k+\tau}) \; d\tau - \big[u_{n_k}(b_k,q_{n_k+b_k}) - u_{n_k}(a,q_{n_k+a}) \big].
	\end{align*}
	Once again, from the Definition \ref{LagsImages} of variational bare branes and from Remark \ref{LagsImagesLag}, we get that
	\begin{equation} \label{CalibLimDem1}
		\begin{split}
			\delta(u_{n_k}, q_{n_k+\cdot},[a,b_k]) &= \big[ h_{n_{k}+b_k}(x_{n_{k}+b_k}) - h_{n_k+a}(x_{n_k+a}) \big] - \big[u_{n_k}(b_k,q_{n_k+b_k}) - u_{n_k}(a,q_{n_k+a}) \big] \\
			&=\big[ h_{n_{k+1}+a}(x_{n_{k+1}+a}) - h_{n_k+a}(x_{n_k+a}) \big] - \big[u_{n_{k+1}}(a,q_{n_{k+1}+a}) - u_{n_k}(a,q_{n_k+a}) \big] \\
			&= \Big[ \big( h_{n_{k+1}+a}(x_{n_{k+1}+a}) + c_{n_{k+1}} \big) - \big( h_{n_k+a}(x_{n_k+a}) + c_{n_k} \big) \Big] \\
			&- \Big[ \big( u_{n_{k+1}}(a,q_{n_{k+1}+a}) + c_{n_{k+1}} \big) - \big( u_{n_k}(a,q_{n_k+a}) + c_{n_k} \big) \Big].
		\end{split}	
	\end{equation}
    
	Recall that we have brane convergence of the branes $(\mathcal{L}_{n_{k+1}},h_{n_{k+1}} + c_{n_{k+1}})$ and $(\mathcal{L}_{n_k},h_{n_k}+ c_{n_k})$ towards $(\mathcal{L}_\omega,h_\omega)$. Hence, using Proposition \ref{ConvTimeProp}, we also have brane convergence of the branes $(\mathcal{L}_{n_{k+1}+a},h_{n_{k+1}+a}+ c_{n_{k+1}})$ and $(\mathcal{L}_{n_k+a},h_{n_k+a}+c_{n_k})$ towards $(\mathcal{L}_{\omega,a},h_{\omega,a})$. Moreover, the points $x_{n_k}$ and $x_{n_{k+1}}$ converge to $x_\omega$, which implies the convergence of the points $x_{n_k+a}$ and $x_{n_{k+1}+a}$ towards $x_{\omega,a}$. Thus, we infer from Proposition \ref{BranePointwiseConv} the limits
	\begin{equation}\label{CalibLimDem2}
		\lim_{k\to +\infty} h_{n_{k+1}+a}(x_{n_{k+1}+a}) + c_{n_{k+1}} = h_{\omega,a}(x_{\omega,a}) \quad \text{and} \quad \lim_{k\to +\infty} h_{n_k+a}(x_{n_k+a}) + c_{n_k} = h_{\omega,a}(x_{\omega,a}).
	\end{equation}	
	Additionally, the brane convergences mentioned above, together with Proposition \ref{GSTranslation} and Corollary \ref{GSConv}, imply that the variational solutions $u_{n_{k+1}}+c_{n_{k+1}}$ and $u_{n_k}+ c_{n_k}$ both converge uniformly on $\mathbb{R} \times M$ to the variational solution $u_\omega$. This gives the limits
	\begin{equation}\label{CalibLimDem3}
		\lim_{k \to +\infty } u_{n_{k+1}}(a,q_{n_{k+1}+a}) + c_{n_{k+1}} = u_\omega(a,q_{\omega,a}) \quad \text{and} \quad \lim_{k \to +\infty } u_{n_k}(a,q_{n_k+a}) + c_{n_k} = u_\omega(a,q_{\omega,a}).
	\end{equation}	
	Gathering \eqref{CalibLimDem1}, \eqref{CalibLimDem2} and \eqref{CalibLimDem3}, we obtain the desired limit
	\begin{equation}
		\lim_{k \to +\infty} \delta(u_{n_k}, q_{n_k+\cdot},[a,b_k]) = \big[ h_{\omega,a}(x_{\omega,a}) - h_{\omega,a}(x_{\omega,a}) \big] - \big[ u_\omega(a,q_{\omega,a}) - u_\omega(a,q_{\omega,a}) \big] = 0.
	\end{equation}
	Therefore, the inequalities in \eqref{CalibLimDem0} become
	\begin{align*}
		0 \leq \delta(u_\omega, q_{\omega,\cdot},[a,b]) \leq \liminf_{k \to +\infty} \delta(u_{n_k}, q_{n_k+\tau},[a,b_k]) = 0
	\end{align*}
	and the calibration defect $\delta(u_\omega, q_{\omega,\cdot},[a,b])=0$ vanishes.
\end{proof}		

\subsubsection{Spectrality at Calibrated Points}

We prove the spectrality identity above calibrated curves for bare Lagrangian branes.

	\begin{prop} \label{SpectralityLemma}
		Let $(\cL,h)\in\overline{\mathcal{LB}(T^*M,\omega)}$. Let $x_0=(q_0,p_0)$ be a point of $\mathcal{L}$ with orbit $x_\tau = (q_\tau,p_\tau) = \phi_H^\tau(x)$ under the Hamiltonian flow. Assume that the curve $q_\tau\colon\R\to M$ is calibrated by the variational solution $u$ associated to $(\cL,h)$. Then, for all $\tau\in \R$, we have
		\begin{equation}
			u(\tau,q_\tau) = h_\tau(x_\tau).
		\end{equation}
	\end{prop}

\begin{proof}
	By calibration, we can apply Theorem \ref{Fathi} at time $\tau$ to get that $u$ is differentiable at $(\tau,q_\tau)$ and 
	\begin{equation} \label{SpectralityLemmaDem1}
		x_\tau=(q_\tau,p_\tau) = (q_\tau, d_{q_\tau} u(\tau,q_\tau)) \in \mathcal{L}_\tau \quad \text{and} \quad  \partial_t u(\tau,q_\tau) + H(\tau,q_\tau,d_{q_\tau} u(\tau,q_\tau)) =0 
	\end{equation}
	The spectral equality $u(\tau,q_\tau) = h_\tau(x_\tau)$ that we want to prove is satisfied by graph selectors of smooth Lagrangian branes according to Proposition \ref{Graphext}. To obtain this equality for bare Lagrangian branes, we use an approximating sequence of smooth Lagrangian branes $(\mathcal{L}^k,h^k)$ which brane-converges to $(\mathcal{L}, h)$. We denote their associated variational solutions by $u^k$. By Proposition \ref{Graphext}, there exists an open dense subset $\mathcal{U}^k \subset \mathbb{R} \times M$ of full Lebesgue measure where $u^k$ is $C^1$ and  for all $(t,q) \in \mathcal{U}^k$, we have
	\begin{equation} \label{SpectralityLemmaDem2}
        \begin{split}
		      (q, d_q u^k(t,q)) \in \mathcal{L}^k_{t}, \quad & \partial_t u^k(t,q) + H(t,q,d_q u^k) =0 ,\\ 
            \text{and} & \quad h^k_{t}\big(q, d_q u^k(t,q)\big) = u^k(t,q).
        \end{split}
	\end{equation}
	Set $\mathcal{U} := \bigcap_{k \geq 0} \mathcal{U}^k$ which is still a dense set of full Lebesgue measure in $\mathbb{R} \times M$. We claim that there exists a sequence $(t^k,q^k) \in \mathcal{U}$ converging to $(\tau,q_\tau)$ such that $d_q u^k(t^k,q^k)$ converges to $d_qu(\tau,q_\tau)$. This is true due to the fibrewise strict convexity of the Tonelli Hamiltonian $H$ and to the following lemma by Bernard and dos Santos \cite{MR2860422}, which extends an earlier result of Clarke \cite{MR2346451,MR0709590}.
	
	\begin{lem} \label{lem:BdS}
		Let $U \subset \mathbb{R}^d$ be an open domain and let $f_k : U \to \mathbb{R}$ be an equi-Lipschitz sequence of functions converging uniformly to a Lipschitz function $f : U \to \mathbb{R}$. Let $U_0 \subset U$ be a set of full Lebesgue measure where the maps $f_k$ are differentiable. For each $q \in U$, consider the following sets 
		\begin{equation}
			\begin{split}
				K^{U_0}_{f_k}(q) &:= \big\{ \text{limit points of } \big(df_k(q_k)\big)_{k \geq 0} \;  \big| \; q_k \in U_0, \; \lim\limits_k q_k = q  \big\} \subset T^*_qU \\
				C^{U_0}_{f_k}(q) &:= \conv \left(K^{U_0}_{f_k}(q)\right)
			\end{split}
		\end{equation}
		where $\conv$ stands for the convex hull. Then, whenever $f$ is differentiable at a point $q \in U$, we have $df(q) \in C^{U_0}_{f_k}(q)$.
	\end{lem}
	
	Recall from Proposition \ref{GSSubsolution} that all the variational solutions $u^k$ are locally equi-Lipschitz. Hence, we can apply the lemma to these functions, on an open chart $U$ around $(\tau,q_\tau)$ in $\mathbb{R} \times M$, and for the full Lebesgue measure set $\mathcal{U}_0 = \mathcal{U} \cap U$. We obtain that $du(\tau,q_\tau) = \big(\partial_t u(\tau,q_\tau), d_qu(\tau,q_\tau) \big)$ belongs to the set $C^{\mathcal{U}_0}_{u^k}(\tau,q_\tau)$.
	In order to prove that $du(\tau,q_\tau) \in K^{\mathcal{U}_0}_{u^k}(\tau,q_\tau)$, we introduce the following Hamiltonian $\mathscr{H} : T^*\mathbb{R} \times T^*M \to \mathbb{R}$ defined by
	\begin{equation}
		\mathscr{H}(t,E,q,p) = E + H(t,q,p)
	\end{equation}
	so that the Hamilton--Jacobi equations in \eqref{SpectralityLemmaDem1} and \eqref{SpectralityLemmaDem2} translate into
	\begin{equation} \label{SpectralityLemmaDem4}
		\big( \tau, \partial_t u(\tau,q_\tau), q_\tau, d_q u(\tau,q_\tau) \big) \in \{ \mathscr{H}_{(\tau,q_\tau)} = 0 \} :=  \{\mathscr{H}=0\} \cap T^*_{(\tau,q_\tau)}( \mathbb{R} \times M )
	\end{equation}
	and for all $(t,q) \in \mathcal{U}$, 
	\begin{equation}
		\big( t, \partial_t u^k(t,q), q, d_q u^k(t,q) \big) \in \{ \mathscr{H}_{(t,q)} = 0 \} :=  \{\mathscr{H}=0\} \cap T^*_{(t,q)}( \mathbb{R} \times M ).
	\end{equation}
	Hence, we deduce by continuity of $\mathscr{H}$ that 
	\begin{equation} \label{SpectralityLemmaDem3}
		K^{\mathcal{U}_0}_{u^k}(\tau,q_\tau) \subset \{\mathscr{H}_{(\tau,q_\tau)}=0\} \quad \text{and} \quad C^{\mathcal{U}_0}_{u^k}(\tau,q_\tau) \subset \conv \{\mathscr{H}_{(\tau,q_\tau)}=0\}.
	\end{equation}
	Moreover, we have
	\begin{align*}
		\{\mathscr{H}_{(\tau,q_\tau)} \leq 0\} = \left\{ (E,p) \in T^*_{(\tau,q_\tau)}( \mathbb{R} \times M ) \; | \;  H(\tau,q_\tau,p ) \leq -E \right\}
	\end{align*}
	which corresponds, up to the symmetry $E \mapsto -E$, to the epigraph of the strictly convex Hamiltonian $H$ restricted to the fibre $T^*_{(\tau,q_\tau)}( \mathbb{R} \times M )$. Thus, this set is strictly convex with boundary (or equivalently, extremal points) equal to
	\begin{align*}
		\partial \{\mathscr{H}_{(\tau,q_\tau)} \leq 0\} = \left\{ (E,p) \in T^*_{(\tau,q_\tau)}( \mathbb{R} \times M ) \; | \;  H(\tau,q_\tau,p ) = -E \right\} =  \{\mathscr{H}_{(\tau,q_\tau)} = 0\}.
	\end{align*}
	Thus, we deduce that 
	\begin{equation} \label{eq:SpectralityLemmaDem5}
		\conv \{\mathscr{H}_{(\tau,q_\tau)}=0\} = \{\mathscr{H}_{(\tau,q_\tau)} \leq 0\}.
	\end{equation}
	The second inclusion in \eqref{SpectralityLemmaDem3} becomes
	\begin{equation*}
		du(\tau,q_\tau) = \big(\partial_t u(\tau,q_\tau), d_qu(\tau,q_\tau) \big) \in C^{\mathcal{U}_0}_{u^k}(\tau,q_\tau) \subset \{\mathscr{H}_{(\tau,q_\tau)}\leq 0\}.
	\end{equation*}
	However, we have seen in \eqref{SpectralityLemmaDem4} that $du(\tau,q_\tau)$ belongs to the boundary $\partial \{\mathscr{H}_{(\tau,q_\tau)} \leq 0\}=\{\mathscr{H}_{(\tau,q_\tau)} = 0\}$ of the strictly convex set $\{\mathscr{H}_{(\tau,q_\tau)} \leq 0\}$. Hence, it is an extremal point of $\{\mathscr{H}_{(\tau,q_\tau)} \leq 0\}$, and by the inclusion \eqref{SpectralityLemmaDem3}, we infer that $du(\tau,q_\tau)$ is an extremal point of $C^{\mathcal{U}_0}_{u^k}(\tau,q_\tau)$. Therefore, we get that $du(\tau,q_\tau)$ belongs to $K^{\mathcal{U}_0}_{u^k}(\tau,q_\tau)$ and by definition of this set, there exists, up to extraction, a sequence $(t^k,q^k) \in \mathcal{U}_0$ converging to $(\tau,q_\tau)$ such that $du^k(t^k,q^k)$ converges to $du(\tau,q_\tau)$. \\
	
	We now deal with the proof of the desired spectral equality $u(\tau,q_\tau) = h_\tau(x_\tau)$. Set $x^k_{t^k} = (q^k, d_qu^k(t^k,q^k))$ with associated orbit $x^k_{t}=\phi^{t^k,t}_H(x^k_{t^k})$. We have
	\begin{align*}
		d(x^k_{\tau},x_\tau) &\leq d(x^k_{\tau},x^k_{t^k}) + d(x^k_{t^k},x_\tau) \\
		&= d( \phi_H^{t^k,\tau}(x^k_{t^k}),x^k_{t^k}) + d(x^k_{t^k},x_\tau).
	\end{align*}
	We know by definition of $x^k_{t^k}$ and \eqref{SpectralityLemmaDem1} that $\lim_k x^k_{t^k} = x_\tau$ and $\lim_k t^k =0$. Moreover, the Hamiltonian flow $\phi_H$ is uniformly continuous on a compact neighbourhood of $(\tau,x_\tau) \in \mathbb{R} \times T^*M$, hence $\lim_k \phi_H^{t^k,\tau}(x^k_{t^k}) = x_\tau$. Then, the inequality above shows that the sequence of points $x^k_{\tau}$ converges to $x_\tau$. Furthermore, we have
	\begin{align*}
		\big| h^k_{t^k}(x^k_{t^k}) - h_\tau(x_\tau) \big| &\leq  \big| h^k_{t^k}(x^k_{t^k}) - h^k_{\tau}(x^k_{\tau}) \big| + \big| h^k_{\tau}(x^k_{\tau}) - h_\tau(x_\tau) \big| \\
		&\leq \left| \int_\tau^{t^k} L(t,q^k_{t}, \dot{q}^k_{t} ) \; dt \right| + \big| h^k_{\tau}(x^k_{\tau}) - h_\tau(x_\tau) \big|
	\end{align*}
	By the convergence of $x^k_{\tau}$ towards $x_\tau$, the curve $q^k_{t}$ converges in the $C^1$ topology on compact sets towards the curve $q_t$. Thus, the maps $t \mapsto L(t,q^k_{t}, \dot{q}^k_{t} )$ converge uniformly on compact sets to $t \mapsto L(t,q_t, \dot{q}_{t})$, and we obtain the limit 
	\begin{align*}
		\lim_{k\to +\infty} \int_\tau^{t^k} L(t,q^k_{t}, \dot{q}^k_{t} ) \; dt =0.
	\end{align*}
	Moreover, by brane convergence of $(\mathcal{L}^k_\tau, h^k_\tau)$ towards $(\mathcal{L}_\tau,h_\tau)$, and using Proposition \ref{BranePointwiseConv}, we get $\lim_k h^k_{\tau}(x^k_{\tau}) = h_\tau(x_\tau)$. Consequently, we deduce the limit
	\begin{equation} \label{SpectralityLemmaDem5}
		\lim_{k \to +\infty} h^k_{t^k}(x^k_{t^k}) = h_\tau(x_\tau). 
	\end{equation}
	
	Furthermore, Proposition \ref{GSConv} yields the convergence of $u^k$ towards $u$ uniformly on $\mathbb{R} \times M$. Hence, we get the limit
	\begin{equation} \label{SpectralityLemmaDem6}
		\lim_{k \to +\infty} u^k(t^k,q^k) = u(\tau,q_\tau).
	\end{equation}
	Therefore, by gathering \eqref{SpectralityLemmaDem2}, \eqref{SpectralityLemmaDem5} and \eqref{SpectralityLemmaDem6}, we conclude the desired equality $h_\tau(x_\tau) = u(\tau,q_\tau)$.
\end{proof}

\begin{proof}[Proof of Proposition \ref{CalibTout}]
    Since we know from Proposition \ref{CalibLimit} that the curves $q_\alpha$ and $q_\omega$ are respectively calibrated by $u_\alpha$ and $u_\omega$, we deduce from Proposition \ref{SpectralityLemma} that
    \begin{align}
        h_\alpha(x_\alpha) = u_\alpha(q_\alpha) \quad \text{and} \quad h_\omega(x_\omega)  u_\omega(q_\omega)
    \end{align}
    Therefore, we conclude from the calibration estimate \eqref{eq:calibration_ineq2} the vanishing of the defect of calibration and hence the $u$-calibration of the curve $q_\tau$.
\end{proof}

\subsection{Proofs of the Main Result}

\subsubsection{Proof of Theorem \ref{MainTheorem}}

\begin{proof}[Proof of Theorem \ref{MainTheorem}]
	Let $t \in \mathbb{R}$ be a fixed time. Let $x_t = (q_t,p_t)$ be a point of $\mathcal{L}_t$ with associated Hamiltonian orbit $x_\tau = (q_\tau,p_\tau)$. We know from Proposition \ref{CalibTout} that $q_\tau$ is calibrated by $u$. Hence, we infer from Theorem \ref{Fathi} that $u$ is differentiable at $(t,q_t)$ and $x_t = (q_t,p_t) = (q_t, d_qu(t,q_t))$. Thus, if we denote by $\mathcal{G}(du_t)$ the graph of $du_t=d_qu(t,\cdot)$ in $T^*M$, we get the inclusion $\mathcal{L}_t \subset \mathcal{G}(du_t)$. Since the projection $\mathcal{L}_t \to M$ is onto, we conclude that $\mathcal{L}_t = \mathcal{G}(d_qu_t)$ is a graph over $0_M$.
	
	In order to obtain the Lipschitz regularity, and using the notation of Theorem \ref{Fathi}, we proved that for all $\varepsilon >0$, we have  $A_{\varepsilon,u} = \mathbb{R} \times M$. Hence, $u$ is locally $C^{1,1}$ and $du$ is locally Lipschitz on $\mathbb{R} \times M$. The fact that $\mathcal{L}_t = \mathcal{G}(du_t)$ yields that it is a Lipschitz graph over $0_M$.
\end{proof}

\subsubsection{Proof of Corollary \ref{OptimalityCor}}

\begin{proof}[Proof of Corollary \ref{OptimalityCor}]
	The only non-trivial part is to prove the brane convergence in the general case. The rest follows directly from Theorem \ref{TheoremViscosityAsymptotic}.
	
	We keep the same notation as in the preceding proofs and subsections. We have shown that $\mathcal{L}_t = \mathcal{G}(d_qu_t)$, where $u : \mathbb{R} \times M \to \mathbb{R}$ is a locally $C^{1,1}$ solution of the Hamilton--Jacobi equation \eqref{HJ}. Let $\alpha_0$ be the \Mane critical value defined in \eqref{HJAlpha} and set $\tilde{u}(t,q) = u(t,q) + \alpha_0.t$ which is a locally $C^{1,1}$ solution of \eqref{HJAlpha}. Hence, Theorem \ref{TheoremViscosityAsymptotic} yields two increasing sequences $n_k$ and $m_k$ of positive integers such that both $\tilde{u}(n_k,\cdot)$ and $\tilde{u}(-m_k,\cdot)$ converge uniformly to $\tilde{u}_0 = \tilde{u}(0,\cdot)$ and both $\mathcal{L}_{n_k} = \mathcal{G}(d_qu_{n_k})$ and $\mathcal{L}_{-m_k} = \mathcal{G}(d_qu_{-m_k})$ converge to $\mathcal{L} = \mathcal{G}(d_qu_0)$ in the Hausdorff topology. In particular, both $\tilde{u}(n_k,\cdot)$ and $\tilde{u}(-m_k,\cdot)$ converge to $\tilde{u}_0 = \tilde{u}(0,\cdot)$ in the $C^1$ topology.

	Set $\tilde{h}_t = h_t + \alpha_0.t$, so that the variational solution $\tilde{u}$ of $\eqref{HJAlpha}$ is associated to the variational bare brane $(\mathcal{L}_t, \tilde{h}_t)$. We need to prove that the sequences $(\mathcal{L}_{n_k}, \tilde{h}_{n_k})$ and $(\mathcal{L}_{-m_k}, \tilde{h}_{-m_k})$ brane-converge to $(\mathcal{L}, \tilde{h}_0)$. We only prove the first convergence.
	
	Let $x_t \in \mathcal{L}_t$ with associated curve $x_\tau = (q_\tau,p_\tau)$ under the Hamiltonian flow. Proposition \ref{CalibTout} shows that the curve $q_\tau$ is calibrated by $u$, and Proposition \ref{SpectralityLemma} yields the spectrality equality $h_t(x_t) = u_t(q_t)$ so that $\tilde{h}_t(x_t) = \tilde{u}_t(q_t)$.
    
    Therefore, we have for any $t$ that the graph of $\tilde{h}_{t}$ coincides with the $1$-jets of $\tilde{u}_t$:
    \begin{equation}\label{eq:graph1jet}
        \Gamma_{(\cL_{t},\tilde{h}_{t})}=\{(q,d_q\tilde u_{t},\tilde{u}_{t}(q)),q\in M\}.
    \end{equation}
    Since $\tilde{u}_{n_k}$ converges to $\tilde{u}_0$ in the $C^1$-topology, the $1$-jet of the $\tilde{u}_{n_k}$ converge to that of $\tilde{u}_{0}$ in the Hausdorff topology. Then, by (\ref{eq:graph1jet}), $(\mathcal{L}_{n_k}, \tilde{h}_{n_k})$ brane-converges towards $(\mathcal{L}, \tilde{h}_0)$. Since $[\tilde{h}_{n_k}] = [h_{n_k}]$, we conclude that $(\mathcal{L}_{n_k}, [h_{n_k}])$ brane-converges to $(\mathcal{L}, [h])$ in $\mathfrak{B}_c(T^*M)$.

\end{proof}

\subsubsection{Proof of Corollary \ref{cor:geodesic}}

    We argue by contradiction. Assume that such an invariant set $\tau_\alpha(\cL)$ exists. Then, $\cL$ is invariant under $\tau_\alpha^{-1}\phi^1_H\tau_\alpha=\phi^1_{H\circ\tau_\alpha}$. But $H_\alpha := H\circ\tau_\alpha$ is a Tonelli Hamiltonian (it is given by $(q,p)\mapsto \frac 1 2\Vert p+\alpha_q\Vert_g^2$). Therefore, Theorem \ref{thm:Invariance} implies that $\cL$ is a Lipschitz graph over $M$ of the differential of a $C^{1,1}$ regular solution $u : M \to \mathbb{R}$ of the Hamilton-Jacobi equation 
    \begin{equation}\label{eq:HJ stationary}
        H_\alpha(q,d_qu) = \alpha_0.
    \end{equation}
    Here, we used Theorem \ref{TheoremViscosityAsymptotic} and the fact that the Hamiltonian is autonomous.

    Let $x_0=(q_0,p_0)$ be the point of $\cL$ that projects to $q_0\in B$, and denote, for $t\in\R$, $x_t=\phi^t_{H\circ\tau_\alpha}(x_0)=(q_t,p_t)$ its trajectory under the conjugated geodesic flow. Note that since the projection is invariant under $\tau_\alpha$, $q_t$ is also the projection of the trajectory of $\tau_\alpha(x_0)$ under the geodesic flow, and therefore $q_t$ is a geodesic. We know from Proposition \ref{CalibTout} that $q_t$ is a $u$-calibrated curve, which therefore minimizes the Lagrangian action in its homotopy class with fixed endpoints. In our case, the Lagrangian action dual to $H_\alpha$ of a curve $\gamma\colon[a,b]\to M$ is
    \begin{align*}
        \int_a^b L_\alpha (t,\gamma(t),\dot\gamma(t))\;dt&=\int_a^b\sup_p\big(p.\dot\gamma(t)-\frac 1 2\Vert p+\alpha_{\gamma(t)}\Vert_g^2\big)\;dt\\
        &=\int_a^b\frac 1 2\Vert\dot\gamma(t)\Vert^2_g\;dt-\int_\gamma\alpha.
    \end{align*}
    Therefore, since $\int_\gamma\alpha$ only depends on the homotopy class of $\gamma$ with fixed endpoints, the curve $q_t$ is a geodesic that minimizes the energy in its homotopy class with fixed endpoints on any time interval, hence it is a minimising geodesic. 
    
    By \eqref{eq:HJ stationary}, $\cL$ is contained in the level set $\{H_\alpha=\alpha_0\}$. If $\alpha_0=0$, then $\cL$ would be the graph of $-\alpha$ and $\tau_\alpha(\cL)$ would be the zero section, which contradicts our hypothesis. Therefore, we may assume that $\cL\subset\{H_\alpha=\alpha_0\}$ with $\alpha_0\neq 0$ and consequently that the geodesic $q_t$ has positive constant speed. Since the diameter of $B$ is finite, there must exist two times $a<0<b$ such that $q_a, q_b\in \partial B$. Then, the curve $q_t$ must minimize the distance between $q_a$ and $q_b$ within $B$. But since it passes through $q_0$, this curve has length at least $2d(q_0,\partial B)$, and therefore cannot be minimizing since by assumption $2d(q_0,\partial B)>\diam(B)$, giving us a contradiction. 

\qed

\ppart{Geometrical Study of Bare Lagrangian Branes} \label{part:Geometry}

\section{Uniquely Liftable Lagrangian Subsets}\label{Section:regularity}

\subsection{Preliminaries on Microlocal Sheaf Theory}

Our proofs of Theorem \ref{thm:rigidity} and of Item \ref{prop:UniquenessPrimitive6} of Theorem \ref{thm:UniquenessPrimitiveC0} rely on several results from microlocal sheaf theory in symplectic geometry. In this section, we introduce all the notations and known results from this theory that we will use. We will avoid entering into the details of each definition or statement, but we encourage the interested reader to read the books \cite{MR1299726} and \cite{MR4211770} to find a more extensive exposition. 

Let $N$ be a manifold. We introduce the following notations and definitions:
\begin{enumerate}[label=-]
    \item We denote by $D(N)$ the (unbounded) derived category of sheaves of $\Z/2\Z$-vector spaces on $N$. As usual, we refer to objects of $D(N)$ simply as sheaves.
    \item A sheaf $F\in D(N)$ admits a \textbf{singular support} (or \textbf{micro-support}), denoted $SS(F)$, which is the subset of $T^*N$ defined as the closure of the set of cotangent directions in which ``\textit{$F$ does not propagate}'' (which means broadly that the sheaf cohomology changes locally in this cotangent direction). The singular support is \textbf{conical}, i.e. stable under multiplication by positive constants in the fibres.
    \item For $N=M\times\R$ We denote by $D_{\{\tau>0\}}(N)$ the subcategory of $D(N)$ whose objects are sheaves $F$ satisfying $SS(F)\setminus 0_{M\times\R}\subset \{\tau>0\}$, where $\tau$ denotes the fibre coordinate in $T^*\R$.
    \item By \cite[Theorem 6.5.4]{MR1299726} singular supports are also \textbf{involutive}. The definition of involutivity in this context involves contingent and paratingent cones, hence involutivity is also sometimes referred to as \textbf{cone-coisotropicity} (for instance in \cite{MR4768582}) to avoid confusion with other notions of involutivity. We will adopt the same terminology here. We will not need the explicit definition of cone-coisotropic subsets, but only the following property that follows from \cite[Proposition 6.5.2]{MR1299726}:
\end{enumerate}

\begin{prop}\label{prop:cone-coiso}
    Let $S$ be cone-coisotropic and closed in an open subset $U\subset T^*N$, and let $H\colon [0,1]\times U\to\R$ be a Hamiltonian vanishing on $S$ at every time $t\in[0,1]$ and whose flow is $C^1$. Then, for every point $x$ in $S$, the trajectory $\phi^t_H(x)$ remains in $S$ as long as it is well-defined in $U$.
\end{prop}

Suppose now that $N$ is of the form $M\times \R$, where $M$ is a closed manifold.
Then, $D_{\{\tau>0\}}(M\times \R)$ can be equipped with an \textbf{interleaving-type pseudo-distance $d$}. By \cite[Proposition 6.26]{MR4768582}, singular supports are lower semi-continuous with respect to this topology:
\begin{prop}[Lower semi-continuity of singular supports \cite{MR4768582}]\label{prop:singular_support_is_lsc}
    Let $(F_n)_{n\geq0}$ be a sequence of sheaves in $D_{\{\tau>0\}}(M\times\R)$ converging in the interleaving topology to a sheaf $F$. Then, we have
    \[SS(F)\subset \liminf_{n\to\infty}SS(F_n).\]
\end{prop}

We can also define the \textbf{reduced singular support} $RS(F)$ of a sheaf $F\in D(M\times \R)$ as the closure of $\rho(SS(F)\cap\{\tau>0\})$, and $\rho\colon T^*M\times T^*\R\setminus\{\tau=0\}\to T^*M$ is given by $\rho(q,p,t,\tau)=(q,\frac{p}{\tau})$.

Putting together results due to Viterbo and Guillermou \cite{viterbo2019sheafquantizationlagrangiansfloer,guillermou2015,MR4768582} and to Asano--Guillermou--Humilière--Ike--Viterbo \cite{MR4683312}, we have the following Lagrangian quantization theorem:

\begin{theo}[Lagrangian quantization \cite{viterbo2019sheafquantizationlagrangiansfloer,guillermou2015,MR4768582,MR4683312}]\label{thm:lag_quantization}
    There exists a map $Q\colon \mathcal{LB}(T^*M,\omega)\to D_{\{\tau>0\}}(M\times\R)$ satisfying:
    \begin{enumerate}
        \item \cite{MR4768582} $Q$ is an isometric embedding for the $c$-distance on $\mathcal{LB}(T^*M,\omega)$ and the interleaving pseudo-distance $d$ on $D_{\{\tau>0\}}(M\times \R)$;
        \item \cite{guillermou2015,viterbo2019sheafquantizationlagrangiansfloer} for $(\cL,h)\in \mathcal{LB}(T^*M,\omega)$, $RS(Q(\cL,h))=\cL$ and $SS(Q(\cL,h))\cap\{\tau>0\}$ is the conification of the graph of $-h$, that is
        \[SS(Q(\cL,h))\cap\{\tau>0\}=\{(q,\tau p,-h(q,p),\tau)\mid (q,p)\in \cL,\tau >0\}=\nu^{-1}(\Gamma_{(\cL,-h)}),\]
        where $\nu\colon (T^*M\times T^*\R)\cap\{\tau>0\} \to J^1M$ is defined by $\nu(q,p,t,\tau)=(\rho(q,p,t,\tau),t)$;
        \item \cite{MR4768582} $Q$ extends to an isometric embedding $\widehat{Q}\colon \widehat{\mathcal{LB}}(T^*M,\omega)\to D_{\{\tau>0\}}(M\times\R)$ from the Humilière completion to the derived category of sheaves;
        \item \cite{MR4683312} for $\underline \cL\in \widehat{\mathcal{LB}}(T^*M,\omega)$, $RS(\widehat{Q}(\underline{\cL}))=\operatorname{\gamma-supp}(\underline{\cL})$.
    \end{enumerate}
\end{theo}

\subsection{Proof of Item \ref{prop:UniquenessPrimitive6} of Theorem \ref{thm:UniquenessPrimitiveC0}}
\label{sec:proof for sympeo}

Recall that if $(\cL_n,h_n)$ are Lagrangian branes converging to $(\cL,h)\in \overline{\mathcal{LB}(T^*M)}$, then by Proposition \ref{BraneGammaProp} the sequence $(\cL_n,h_n)$ is Cauchy for the $c$-distance. Therefore, it defines an element $\underline{\cL}_h$ of the Humilière completion $\widehat{\mathcal{LB}}(T^*M,\omega)$. This element depends only on $(\cL,h)$ (not on the approximating sequence $(\cL_n,h_n)$), which justifies the notation.

We also recall the notions of topological Lagrangian brane and topological Lagrangian submanifold introduced in Definition \ref{def:C^0Lag}. We first show the following result, as a consequence of \cite[Theorem 1.4]{MR5052827}.

\begin{prop}\label{prop:topologicalLagrangian}
    Let $(\cL,h)$ be a topological Lagrangian brane. Then, $\cL$ is a topological Lagrangian submanifold. More precisely, $\cL=\operatorname{\gamma-supp}(\underline{\cL}_h)$.
\end{prop}

\begin{proof}
    We want to show that $\cL$ is a $\gamma$-support. Let $(\cL_n,h_n)$ be a sequence of smooth Lagrangian branes converging to $(\cL,h)$ in the brane topology. By the semi-continuity of $\gamma$-supports (Proposition \ref{prop:gammaSupports}), $\operatorname{\gamma-supp}(\underline \cL_h)\subset \liminf_n \cL_n=\cL$. Assume that the inclusion is strict. Let $x\in\cL\setminus\operatorname{\gamma-supp}(\underline \cL_h)$. Since $\operatorname{\gamma-supp}(\underline \cL_h)$ is closed and contained in the compact set $\cL$, it is compact and we can apply \cite[Theorem 1.4]{MR5052827}: the natural map $H^d(M)\to \check H^d(\operatorname{\gamma-supp}(\underline \cL_h))$, where $\check H$ denotes the Čech cohomology, is injective. It factors through
    \begin{equation}\label{eq:cohomology}
        H^d(M)\longrightarrow H^d(\cL\setminus\{x\})\longrightarrow \check H^d(\operatorname{\gamma-supp}(\underline \cL_h)).
    \end{equation}
    But $\cL\setminus\{x\}$ is a topological manifold of dimension $d$, which is not compact (since $\cL$ is connected by Proposition \ref{prop:Connectedness} and $\cL\setminus\{x\}$ is open). Therefore, $H^d(\cL\setminus\{x\})=0$. This implies that the map $H^d(M)\to \check H^d(\operatorname{\gamma-supp}(\underline \cL_h))$ is trivial. However, since $M$ is a closed manifold of dimension $d$, $H^d(M)$ is non-trivial, which contradicts the injectivity of the map. By contradiction, we thus have $\operatorname{\gamma-supp}(\underline \cL_h)=\cL$.
\end{proof}

In order to prove Item \ref{prop:UniquenessPrimitive6} of Theorem \ref{thm:UniquenessPrimitiveC0}, we will also need the following proposition.

\begin{prop}\label{prop:SStopologicalBrane}
    Let $(\cL,h)$ be a topological Lagrangian brane. Then,
    \[SS(\widehat Q(\underline{\cL}_h))\cap\{\tau=1\}=\Gamma_{(\cL,-h)}\times\{1\}.\]
\end{prop}

\begin{proof}
    Let $(\cL_n,h_n)$ be a sequence of Lagrangian branes converging to $(\cL,h)$. Denote by $F_n=Q(\cL_n,h_n)\in D_{\{\tau>0\}}(M\times\R)$ the sheaf associated to $(\cL_n,h_n)$ by the quantization Theorem \ref{thm:lag_quantization}, and by $F=\widehat{Q}(\underline \cL_h)$ the sheaf associated to $\underline{\cL}_h$.\\
    Recall that the reduced singular support of $F$ is $RS(F)=\overline{\rho(SS(F)\cap\{\tau>0\})}$, where $\rho\colon T^*M\times T^*\R\setminus\{\tau=0\}\to T^*M$ is given by $\rho(q,p,t,\tau)=(q,\frac{p}{\tau})$. Then, by the fourth item of Theorem \ref{thm:lag_quantization} and Proposition \ref{prop:topologicalLagrangian}, we have $RS(F)=\operatorname{\gamma-supp}(\underline{\cL}_h)=\cL$.
    Moreover, for any $n$, $(\cL_n,h_n)$ is a smooth Lagrangian brane, so by the second item of Theorem \ref{thm:lag_quantization} we have
    \[SS(F_n)\cap\{\tau>0\}=\{(q,\tau p,-h_n(q,p),\tau)\mid (q,p)\in \cL_n,\tau >0\}=\nu^{-1}(\Gamma_{(\cL_n,-h_n)}),\]
    where $\nu\colon (T^*M\times T^*\R)\cap\{\tau>0\}\to J^1M$ is defined by $\nu(q,p,t,\tau)=(\rho(q,p,t,\tau),t)$. By Proposition \ref{prop:singular_support_is_lsc}, since the sheaves $F_n$ converge to $F$ in the interleaving topology, we have $SS(F)\subset \liminf SS(F_n)$. Therefore, by conicity of singular supports,
    \begin{align*}
        SS(F)\cap\{\tau=1\}&\subset\liminf (SS(F_n)\cap\{\tau=1\})\\
        &\subset \liminf \Gamma_{(\cL_n,-h_n)}\times\{1\}\\
        &\subset \Gamma_{(\cL,-h)}\times\{1\}
    \end{align*}
    since $\Gamma_{(\cL_n,-h_n)}$ converges to $\Gamma_{(\cL,-h)}$ in the Hausdorff topology by definition of the brane distance $d_B$. Therefore, $\nu(SS(F)\cap\{\tau>0\})\subset \Gamma_{(\cL,-h)}$. But since we know that $RS(F)=\cL$, the image $K$ of $\nu(SS(F)\cap\{\tau>0\})$ by the projection $J^1M\to T^*M$ has to be dense in $\cL$. Since $SS(F)$ is closed, so is $\nu(SS(F)\cap\{\tau>0\})\times\{1\}=SS(F)\cap\{\tau=1\}$. Therefore, $\nu(SS(F)\cap\{\tau>0\})$ is closed, and it contains the set $\Gamma_{(K,-h|_K)}$, whose closure is $\Gamma_{(\cL,-h)}$. This implies that $\nu(SS(F)\cap\{\tau>0\})=\Gamma_{(\cL,-h)}$, and therefore $SS(F)\cap\{\tau=1\}=\Gamma_{(\cL,-h)}\times\{1\}$.
\end{proof}

We can now prove Item \ref{prop:UniquenessPrimitive6} of Theorem \ref{thm:UniquenessPrimitiveC0}, which is a direct consequence of the following proposition.

\begin{prop}\label{prop:uniqueGammaSupportImpliesL^!}
    Let $(\cL,h)$ be a topological Lagrangian brane. If $\cL$ has the property that, for all $K,K'\in \widehat{\mathcal{L}}(T^*M)$, $\operatorname{\gamma-supp}(K)=\operatorname{\gamma-supp}(K')=\cL\implies K=K'$, then $\cL\in \mathscr{L}^!(T^*M)$. In particular, this is the case if $\cL\in\mathscr{L}_\text{Sympeo}\cap\mathscr{L}^\exists$.
\end{prop}

\begin{proof}
    Let $h'$ be such that $(\cL,h')\in \overline{\mathcal{LB}(T^*M)}$. We want to prove that $[h]=[h']$. By Proposition \ref{prop:topologicalLagrangian}, we have $\operatorname{\gamma-supp}(\underline{\cL}_h)=\cL=\operatorname{\gamma-supp}(\underline{\cL}_{h'})$, and therefore by hypothesis the images of $\underline{\cL}_h$ and $\underline{\cL}_{h'}$ coincide inside $\widehat{\cL}(T^*M)$. This implies that they differ by an action shift in $\widehat{\mathcal{LB}}(T^*M)$, which in turn implies that their respective quantizations $F=\widehat Q(\underline{\cL}_h)$ and $F'=\widehat Q(\underline{\cL}_{h'})$, as well as their singular supports, differ by a translation in the $\R$ direction. But by Proposition \ref{prop:SStopologicalBrane}, we have $SS(F)\cap\{\tau=1\}=\Gamma_{(\cL,-h)}\times\{1\}$ and $SS(F')\cap\{\tau=1\}=\Gamma_{(\cL,-h')}\times\{1\}$. Therefore, $h$ and $h'$ differ by a constant, and $\cL\in \mathscr{L}^!$.

    We now show that if $\cL\in\mathscr{L}_\text{Sympeo}\cap\mathscr{L}^\exists$, then it satisfies the hypothesis of the proposition. Indeed, let $K\in\widehat{\mathcal{L}}(T^*M)$ be such that $\operatorname{\gamma-supp}(K)=\cL$. Let $\varphi$ be a homeomorphism which is a uniform limit of $C^1$ exact symplectomorphisms on compact sets, and such that $\cL = \varphi(\cL_0)$, with $\cL_0$ a smooth closed exact Lagrangian submanifold. As explained in Corollary 5.25 of \cite{viterbo2026supportshumilierecompletiongammacoisotropic} and the paragraph before it, since the spectral distance is $C^0$-continuous by \cite{MR4263685}, and since exact symplectomorphisms preserve the spectral distance between Lagrangians, $\varphi$ defines an isometry of $\widehat{\mathcal{L}}(T^*M)$, which commutes with $\gamma$-supports. Therefore, we have $\operatorname{\gamma-supp}(\varphi^{-1}(K))=\varphi^{-1}(\operatorname{\gamma-supp}(K))=\varphi^{-1}(\cL)=\cL_0$. Since $\cL_0$ is a smooth closed exact Lagrangian, by \cite{MR5021042} we have $\varphi^{-1}(K)=\cL_0$, and therefore $K=\varphi(\cL_0)$ is uniquely determined by $\cL$.
\end{proof}

\begin{rem}
    In \cite[Question 7.7]{viterbo2026supportshumilierecompletiongammacoisotropic}, Viterbo also asks whether topological Lagrangians correspond to a unique element of the Humilière completion. A positive answer would imply that the hypotheses of Proposition \ref{prop:uniqueGammaSupportImpliesL^!} are always satisfied, and therefore that topological Lagrangian branes are in $\mathscr{L}^!(T^*M)$, answering Question \ref{quest:regularity1} positively.
\end{rem}

\subsection{Proof of Theorem \ref{thm:rigidity}}
\label{sec:rigidity}

    We prove Theorem \ref{thm:rigidity}.
    By Proposition \ref{prop:SStopologicalBrane}, the sheaf $F=\widehat{Q}(\underline \cL_h)$ satisfies 
    \begin{equation*}
        SS(F)\cap\{\tau>0\}=\{(q,\tau p,-h(q,p),\tau)\mid (q,p)\in \cL,\tau >0\}=\nu^{-1}(\Gamma_{(\cL,-h)}).
    \end{equation*}
    Let $U$ be an open subset of $\cL$ where $\cL$ is $C^1$. Since $\cL$ is a topological manifold, we have $\operatorname{\gamma-supp}(\underline{\cL}_h)= \cL$, and therefore $\cL$ is $\gamma$-coisotropic. Since $U$ is $C^1$, it is then coisotropic in the usual sense, and therefore Lagrangian. Let $H\colon[0,1]\times T^*M\to\R$ be a Hamiltonian function vanishing on $U$. Its homogeneous lift $\widetilde{H}\colon [0,1]\times T^*M\times T^*\R \setminus \{ \tau = 0 \} \to\R$ is given by $(s,q,p,t,\tau)\mapsto \tau H(s,q,\frac{p}{\tau})$. Since it does not depend on the $t$ coordinate, its flow preserves the levels of $\tau$, and its restriction to $\{\tau=1\}$ is given by
    \[\phi^s_{\widetilde{H}}(x,t,1)=\bigg(\phi^s_H(x),t+\int_0^s\big(H_{s'}(\phi^{s'}_H(x))-\lambda_{\phi^{s'}_H(x)}(X_{H_{s'}}(\phi^{s'}_H(x)))\big)\;ds',1\bigg).\]
    Since $U$ is open in $\cL$, there exists an open subset $W$ of $T^*M$ whose intersection with $\cL$ is $U$. Because $H$ vanishes on $U$, $\widetilde{H}$ vanishes on $\nu^{-1}(\Gamma_{(U,-h|_U)})=SS(F)\cap(\rho^{-1}(W)\cap\{\tau>0\})$. Since singular supports are closed and cone-coisotropic, $\nu^{-1}(\Gamma_{(U,-h|_U)})$ is closed and cone-coisotropic in $\rho^{-1}(W)\cap\{\tau>0\}$, and by Proposition \ref{prop:cone-coiso} this implies that for any point $y$ in $\nu^{-1}(\Gamma_{(U,-h|_U)})$, the trajectory $\phi^s_{\widetilde{H}}(y)$ remains in $\nu^{-1}(\Gamma_{(U,-h|_U)})$ for small positive times $s$. In particular, in the level $\{\tau=1\}$, we get for $x$ in $U$ and small times $s\geq 0$, that
    \[\phi^s_{\widetilde{H}}(x,-h(x),1)=\bigg(\phi^s_H(x),-h(x)+\int_0^s\big(H_{s'}(\phi^{s'}_H(x))-\lambda_{\phi^{s'}_H(x)}(X_H^{s'}(\phi^{s'}_H(x)))\big)\;ds',1\bigg)\in \Gamma_{(U,-h|_U)} \times \{1\} \]
    i.e.
    \begin{align*}h(\phi^s_H(x))&=h(x)-\int_0^s\big(H_{s'}(\phi^{s'}_H(x))-\lambda_{\phi^{s'}_H(x)}(X_{H_{s'}}(\phi^{s'}_H(x)))\big)\;ds'\\
    &=h(x)+\int_0^s\lambda_{\phi^{s'}_H(x)}(X_{H_{s'}}(\phi^{s'}_H(x)))\;ds'
    \end{align*}
    since $H$ vanishes on $U$. This formula implies that $h$ is a primitive of $\lambda$ along the trajectories of $H$. To conclude the proof, it only remains to show that any $C^1$ injective path $\gamma\colon[0,1]\to U$ can be realized, up to reparametrization, by a $C^1$ flow of a Hamiltonian vanishing on $U$. Indeed, this would imply that $\lambda|_U=dh|_U$. If $U$ is at least $C^2$, then this follows easily from Weinstein's neighbourhood theorem. We present a proof for the $C^1$ case.
    
    Since $U$ is a $C^1$ Lagrangian and $\gamma([0,1])$ is compact, for a sufficiently fine finite subdivision $0=t_0<t_1<\ldots<t_n=1$ we can find $C^1$ symplectic charts $(\varphi_j\colon V_j\to V'_j)_{0\leq j\leq n}$ from neighbourhoods $V_j$ of $\gamma(t_j)$ in $T^*M$ to neighbourhoods $V'_j$ of the origin in $\C^d$ that map $U\cap V_j$ to $\R^d\cap V'_j$, and such that the $V_j$ cover $\gamma([0,1])$. For $0\leq j \leq n$, let $I_j=\gamma^{-1}(V_j)$ and $\gamma_j=\varphi_j\circ \gamma|_{I_j}$. Since $\gamma$ is injective, up to taking a finer subdivision and shrinking the $V_j$, we may assume that the $I_j$ are open intervals. We also pick points $0=s_0<s_1<\ldots<s_{n+1}=1$ satisfying $s_j\in I_{j-1}\cap I_j$ for $1\leq j\leq n$, so that each $I_j$ contains $[s_j,s_{j+1}]$. Note that for each $j$, $\gamma_j$ is a $C^1$ path contained in $\R^d$, and therefore the Hamiltonian $G^j\colon I_j\times\C^d\to \R$ given by $G^j_t(q+ip)=p\cdot \dot{\gamma}_j(t)$ is well-defined. For a fixed time $t$, this Hamiltonian is smooth, and its associated vector field is given by $X_{G^j}^t=\dot{\gamma_j}(t)$. Its flow is given by $\phi_{G^j}^{s,t}(q+ip)=q+\gamma_j(t)-\gamma_j(s)+ip$. It is $C^1$, and satisfies $\phi_{G^j}^{s,t}(\gamma_j(s))=\gamma_j(t)$. By multiplying $G^j$ by a smooth cut-off function supported in $V'_j$, and equal to $1$ in a neighbourhood of $\gamma_j([s_j,s_{j+1}])$, we obtain a Hamiltonian $\widetilde{G}^j$, compactly supported in $V'_j$, vanishing on $\R^d\cap V'_j$, and whose flow is $C^1$ and realizes the path $\gamma_j|_{[s_j,s_{j+1}]}$. Also, since we only care about realizing the image of the path $\gamma$, we may reparametrize it so that it stays constant in small neighbourhoods of the $s_j$, so that $\widetilde{G}^j_t$ is identically zero for $t$ close to $s_j$ or $s_{j+1}$. Then, we define $H^j\colon[s_j,s_{j+1}]\times V_j\to\R$ by the formula $H^j_t=\widetilde{G}^j_t\circ\varphi_j$. It is compactly supported, and therefore extends to $[s_j,s_{j+1}]\times T^*M$. It also vanishes on $U\cap V_j$, is $C^1$ for fixed times, and its flow is given by $\phi_{H^j}^{s,t}=\varphi_j^{-1} \circ \phi_{\widetilde{G}^j}^{s,t}\circ\varphi_j$ which is $C^1$ and realizes $\gamma|_{[s_j,s_{j+1}]}$. Moreover, each $H^j$ vanishes identically near $s_j$ and $s_{j+1}$, so we can concatenate them to obtain a Hamiltonian $H\colon[0,1]\times T^*M\to\R$ satisfying all the desired properties.
    \qed

\subsection{Proof of Item \ref{prop:UniquenessPrimitive4} of Theorem \ref{thm:UniquenessPrimitiveC0}}

We start by proving that strongly exact symplectic homeomorphisms preserve $\mathscr{L}^\exists(T^*M)$. Let $\cL\in\mathscr{L}^\exists(T^*M)$ and $\varphi$ be a strongly exact symplectic homeomorphism. Let $h$ be a continuous function on $\cL$ such that $(\cL,h)\in \overline{\mathcal{LB}(T^*M,\omega)}$, and $(\cL_n,h_n)$ be a sequence of smooth Lagrangian branes converging to $(\cL,h)$ in the brane topology. Let $(\varphi_n)$ be a sequence of $C^1$ exact symplectic diffeomorphisms converging to $\varphi$ and satisfying $\varphi_n^*\lambda-\lambda=dS_n$, with $(S_n)$ converging uniformly on compact sets to a continuous function $S$.

Then, since the $\varphi_n$ are exact, we have by Proposition \ref{BraneInvariance} that $(\varphi_n(\cL),g_n)\in \overline{\mathcal{LB}(T^*M,\omega)}$, with $g_n\coloneq h\circ\varphi_n^{-1}+(S_n\circ\varphi_n^{-1})|_{\varphi_n(\cL)}$. We claim that the bare Lagrangian branes $(\varphi_n(\cL),g_n)$ converge in the brane topology to $(\varphi(\cL),g)$, where $g=h\circ\varphi^{-1}+(S\circ\varphi^{-1})|_{\varphi(\cL)}$. This would imply $(\varphi(\cL),g)\in\overline{\mathcal{LB}(T^*M,\omega)}$ since $\overline{\mathcal{LB}(T^*M,\omega)}$ is closed. To prove the claim, consider the lifts $\Phi_n\colon T^*M\times \R\to T^*M\times \R$ given by $\Phi_n(x,t)=(\varphi_n(x),t+S_n(x))$. By the defining properties of the sequences $(\varphi_n)$ and $(S_n)$, we have that the sequence of lifts $(\Phi_n)$ converges uniformly on compact sets to the lift $\Phi$ of $\varphi$ given by $\Phi(x,t)=(\varphi(x),t+S(x))$. Therefore, the sets $\Phi_n(\Gamma_{(\cL,h)})$ converge in the Hausdorff topology to the set $\Phi(\Gamma_{(\cL,h)})$. But $\Phi_n(\Gamma_{(\cL,h)})=\{(\varphi_n(x),h(x)+S_n(x))\mid x\in\cL\}=\Gamma_{(\varphi_n(\cL),g_n)}$, and similarly $\Phi(\Gamma_{(\cL,h)})=\Gamma_{(\varphi(\cL),g)}$. Hence, the Hausdorff convergence above implies the brane convergence of $(\varphi_n(\cL),g_n)$ to $(\varphi(\cL),g)$, which completes the proof.
\qed

\subsection{Proof of Item \ref{prop:UniquenessPrimitive3} of Theorem \ref{thm:UniquenessPrimitiveC0}} \label{Subsection:regularity}

The goal of this section is to prove the strict inclusion $\mathscr{L}_{\text{gr}}\subsetneq \mathscr{L}_\text{Sympeo}^\text{str}$. We start by showing the inclusion in the lemma below, and then produce an example of an element in $\mathscr{L}_\text{Sympeo}^\text{str}\setminus \mathscr{L}_{\text{gr}}$.

\begin{lem}\label{lem:L_gr in L_sympeo}
    We have the inclusion $\mathscr{L}_{\text{gr}}\subset \mathscr{L}_\text{Sympeo}^\text{str}$.
\end{lem}

\begin{proof}
    Let $\cL$ be an element of $\mathscr{L}_{\text{gr}}$. Then, there exists a $C^1$ map $v\colon M\to \R$ and a $C^1$ exact symplectomorphism $\varphi$ such that $\cL=\varphi(\G(dv))$. Let $(v_n)_{n\in\N}$ be a sequence of smooth maps approximating $v$ in the $C^1$ topology. For $n\in\N$, let $H_n\colon T^*M\to \R$ be the autonomous Hamiltonian given by $H_n(q,p)=-v_n(q)$. Let $\psi_n$ denote the time-1 flow of $H_n$. Then, for all $n\in\N$ and $(q,p)\in T^*M$, we have $\psi_n(q,p)=(q,p+d_qv_n)$, and $\psi_n^*\lambda = \lambda +dv_n$. Therefore, since $(v_n)$ converges to $v$ in the $C^1$ topology, we have that $(\psi_n)$ converges in the $C^0$ topology to the strongly exact symplectic homeomorphism $\psi$ given by $\psi(q,p)=(q,p+d_qv)$, and hence $\G(dv)=\psi(0_M)$ with $\psi$ a strongly exact symplectic homeomorphism. Since $\varphi$ is a $C^1$ exact symplectic diffeomorphism, $\varphi\circ\psi$ is still a strongly exact symplectic homeomorphism, and therefore $\cL\in \mathscr{L}_{\text{Sympeo}}^\text{str}$.
\end{proof}

We now show that this inclusion is strict by using ``infinite twist maps''. Let $R>0$, and $f\colon(0,R)\to[0,+\infty)$ be a smooth non-increasing function supported in $(0,R/2)$ and such that $f(r)\xrightarrow[r\to 0]{}+\infty$. Let $B(R)$ denote the open ball of radius $\sqrt{R}$ centred at the origin in $\C^d$ equipped with the symplectic form $\omega_0=-d\lambda_0$, where $\lambda_0$ is the standard Liouville form.  We define the infinite twist $\varphi_0\colon B(R)\setminus\{0\}\to B(R)\setminus\{0\}$ by $\varphi_0(z)=e^{if(\Vert z\Vert^2)}\cdot z$. It is a Hamiltonian diffeomorphism generated by the autonomous Hamiltonian $H(z)=-\frac{1}{2}F(\Vert z\Vert^2)$, where $F$ is a primitive of $f$. Therefore, it is exact, and satisfies $\varphi_0^*\lambda-\lambda = dS$, where 

\begin{align*}
S(z)&=\int_0^1{\lambda_0}_{e^{itf(\Vert z\Vert^2)}z}(if(\Vert z\Vert^2)e^{itf(\Vert z\Vert^2)}z)+\frac{1}{2}F(\Vert z\Vert^2)\;dt\\
&=\frac 1 2F(\Vert z\Vert^2)-f(\Vert z\Vert^2)\int_0^1\sum_{i=1}^dr_i^2\sin^2(\theta_i+tf(\Vert z\Vert^2))\;dt\\
&=\frac 1 2(F(\Vert z\Vert^2)-\Vert z\Vert^2f(\Vert z\Vert^2))+\frac 1 4 \sum_{i=1}^dr_i^2(\sin(2\theta_i+2f(\Vert z\Vert^2))-\sin(2\theta_i)),
\end{align*}
where $z=(r_1e^{i\theta_1},\ldots,r_de^{i\theta_d})\in B(R)$.   

$\varphi_0$ extends continuously to a compactly supported homeomorphism $B(R)\to B(R)$ mapping the origin to itself, which we still denote $\varphi_0$. One can show it is a $C^0$-limit of compactly supported smooth Hamiltonian diffeomorphisms of $B(R)$. Now, pick a point $q\in M$. By Weinstein's neighbourhood theorem, there exists a neighbourhood $U$ of $(q,0)$ in $T^*M$, $R>0$, and an exact smooth symplectomorphism $\psi\colon (U,\lambda)\to (B(R),\lambda_0)$ mapping $(q,0)$ to the origin and $U\cap 0_M$ to $B(R)\cap \R^d$. Let $S'$ be a smooth function on $U$ satisfying $\psi^*\lambda_0=\lambda+dS'$, and let $\varphi\coloneq\psi^{-1}\varphi_0\psi$. It is a Hamiltonian homeomorphism compactly supported in $U$, and therefore extends to a compactly supported Hamiltonian homeomorphism of $T^*M$. It is a smooth exact symplectomorphism away from $(q,0)$, with $\varphi^*\lambda-\lambda = d(S\circ\psi+S'-S'\circ\varphi)$. Let $\cL\coloneq \varphi(0_M)$.

By construction, $\cL$ is an element of $\mathscr{L}_{\text{Sympeo}}$. It is a smooth Lagrangian away from $(q,0)$, and it admits a Liouville primitive $h$ away from $(q,0)$ given by the formula
\begin{align*}h(\varphi(x))&=S(\psi(x)) +S'(x)-S'\circ\varphi(x)\\
&=\frac 1 2(F(\Vert x\Vert^2)-\Vert x\Vert^2f(\Vert x\Vert^2))+\frac 1 4 \Vert x\Vert^2\sin(2f(\Vert x\Vert^2))+S'(x)-S'\circ\varphi(x),
\end{align*}
for $x$ in $U\cap 0_M\setminus\{(q,0)\}$, where we used the notation $\Vert x\Vert$ for $\Vert\psi(x)\Vert$. By applying Theorem \ref{thm:rigidity} to the smooth part of the Lagrangian, we have that if there exists a continuous function $h'\colon\cL\to\R$ such that $(\cL,h')\in\overline{\mathcal{LB}(T^*M,\omega)}$, then $h'$ coincides with $h$ on $\cL\setminus\{(q,0)\}$ up to a locally constant function. By continuity of $h'$, this implies that if $\cL\in\mathscr{L}^\exists$, then $\cL\in\mathscr{L}^!$ and $h$ extends continuously at $0$.
    
Because $S'-S'\circ\varphi$ is continuous at $(q,0)$, if $f$ is chosen so that $rf(r)-F(r)$ diverges at zero (this is for instance the case if $f$ coincides with $r\mapsto\frac{1}{r}$ near $0$), then the primitive blows up at $(q,0)$. In this case, we have $\cL\in \mathscr{L}_{\text{Sympeo}}\setminus\mathscr{L}^\exists$. Note that by the inclusion $\mathscr{L}_\text{gr}\subset\mathscr{L}_\text{Sympeo}^\text{str}$ of Lemma \ref{lem:L_gr in L_sympeo}, together with the inclusions of Items \ref{prop:UniquenessPrimitive4} and \ref{prop:UniquenessPrimitive6} of Theorem \ref{thm:UniquenessPrimitiveC0}, and $\mathscr{L}^!\subset\mathscr{L}^\exists$, we have for such a choice of rotation speed that $\cL\notin\mathscr{L}_\text{gr}$.

On the other hand, if $f$ is chosen so that $rf(r)$ and $F(r)$ converge at $0$ (for instance if $f$ coincides with $r\mapsto -\frac 1 2\log r$ near $0$), then the global primitive $S$ extends continuously at the origin of $B(R)$, and one can easily show, using a sequence of smooth $f_n\colon [0,R)\to[0,+\infty)$ coinciding with $f$ away from a ball whose radius decreases to zero, that $\varphi_0$ is a strongly exact symplectic homeomorphism, and therefore $\cL\in \mathscr{L}_\text{Sympeo}^\text{str}$. From now on, we assume that $f$ coincides with $r\mapsto -\frac 1 2\log r$ for $0 < r <\varepsilon$, for a sufficiently small $\varepsilon>0$. We will show that in this case we also have $\cL\notin \mathscr{L}_{\text{gr}}$, which will conclude the proof of Item \ref{prop:UniquenessPrimitive3} of Theorem \ref{thm:UniquenessPrimitiveC0}.

Arguing by contradiction, suppose that $\cL\in \mathscr{L}_{\text{gr}}$, and let $v\colon M\to \R$ be a $C^1$ function and $\varphi'$ an exact $C^1$ symplectomorphism such that $\cL=\varphi'(\mathcal{G}(dv))$. Let $q'\in M$ be the unique point such that $\varphi'(q',d_{q'}v)=\varphi(q,0)=(q,0)$. Then, since $T^*_{q'}M$ intersects $\mathcal{G}(dv)$ at the single point $(q',d_{q'}v)$, we also have that the $C^1$ exact Lagrangian submanifold $\varphi'(T^*_{q'}M)$ intersects $\cL=\varphi(0_M)$ at the single point $(q,0)$. To obtain a contradiction, we will show that in the chart $U$, $\varphi_0(B(R)\cap\R^d)$ intersects any $C^1$ exact Lagrangian submanifold passing through the origin infinitely many times arbitrarily close to the origin.

We begin with the following lemma.

\begin{lem}\label{lem:intersection Lag plane}
    $\cL_0\coloneq \varphi_0(B(R)\cap\R^d)$ intersects cleanly any Lagrangian plane $P$ passing through the origin infinitely many times arbitrarily close to the origin.
\end{lem}

\begin{proof}
    Let $x\in B(R)\cap\R^d$. If $P$ intersects $\cL_0$ at $\varphi_0(x)=e^{if(\Vert x\Vert^2)}\cdot x$, then we have $P_{f(\Vert x\Vert^2)}\cap P\neq \{0\}$, where $P_\theta\coloneq e^{i\theta}\cdot \R^d$ for $\theta\in \R$. Conversely, if $P_{\theta}\cap P\neq \{0\}$ for some $\theta\in \R$, then there exists $x\in \R^d\setminus\{0\}$ such that $e^{i\theta}\cdot x\in P$. Since $f$ goes to infinity at $0$, there are arbitrarily small $r>0$ such that $f(r^2)=\theta \text{ mod } \pi$. Then, $\varphi_0\Big(\frac{r}{\Vert x\Vert}\cdot x\Big)=e^{if(r^2)}\cdot \frac{r}{\Vert x\Vert} x=\pm\frac{r}{\Vert x\Vert}e^{i\theta}\cdot x\in P$.
    
    We will show that the loop of Lagrangian planes $\theta\in \R/2\pi\Z\mapsto P_\theta$ intersects any Lagrangian plane $P$ non-trivially at a finite number of angle values.
    Let us first assume that $P\cap i\R^d=\{0\}$. Then, the projection by taking the real part is a bijection between $P$ and $\R^d$, and $P$ can be written as a graph over $\R^d$:
    \[P=\{x+iAx\mid x\in \R^d\}\]
    where $A$ is a real matrix. Moreover, since $P$ is Lagrangian, we have for $x,y\in \R^d$:
    \begin{align*}
        0&=\omega(x+iAx,y+iAy)=\text{Im}\langle x+iAx,y+iAy\rangle\\
        &=\langle Ax,y\rangle - \langle x,Ay\rangle
    \end{align*}
    and therefore $A$ is a symmetric matrix. Hence, there are real numbers $\lambda_1,\ldots,\lambda_d$ and a real orthogonal matrix $O$ such that $A=O\text{diag}[\lambda_1,\ldots,\lambda_d]O^{-1}$. Since $O$ is real orthogonal, it maps $\R^d$ to itself, and $P_\theta=e^{i\theta}\cdot \R^d$ to itself for any $\theta\in \R$. Therefore, we can assume that the matrix $A$ is diagonal, and $P$ is of the form
    \[P=\big\{((1+i\lambda_1)x_1,\ldots,(1+i\lambda_d)x_d) \big\mid (x_1,\ldots,x_d)\in \R^d\big\}.\]
    For $1\leq j\leq d$, we write $1+i\lambda_j=r_je^{i\theta_j}$ for some $r_j>0$ and $\theta_j\in[0,2\pi)$. Then, we have that $P_\theta\cap P\neq \{0\}$ whenever $\theta\in\{\theta_1,\ldots,\theta_d\}$ mod $\pi$.
    It remains to consider the case when $P\cap i\R^d\neq\{0\}$. Then, $P$ intersects $P_{\frac{\pi}{2}}=i\R^d$, and after applying a unitary change of coordinates, one can still show that there are finitely many angles $\{\theta_1,\ldots,\theta_d\}$ such that $P_{\theta_j}\cap P\neq \{0\}$.

    We denote by $S^{2d-1}(r)$ the sphere of radius $r$. Note that for all small $r>0$, $\cL_0\cap P\cap S^{2d-1}(r)=P_\theta\cap P\cap S^{2d-1}(r)$, where $f(r^2)=\theta \text{ mod } \pi$. Now, fix $\theta \in\{\theta_1,\ldots,\theta_d\}$, and let $r>0$ be small enough and such that $f(r^2)=\theta \text{ mod } \pi$. Since there are finitely many intersection angles, and since $f$ is strictly monotone near 0, $C_r\coloneq \cL_0\cap P\cap S^{2d-1}(r)=P_\theta\cap P\cap S^{2d-1}(r)$ is an isolated component of the intersection $\cL_0\cap P$. We want to check that this intersection is clean, that is $T\cL_0|_{C_r}\cap TP|_{C_r}=TC_r$.
    We start by computing the tangent space of $\cL_0$ at a point $e^{i\theta} x\in C_r$, where $x\in \R^d\cap S^{2d-1}(r)$. Since $P_\theta\cap S^{2d-1}(r)\subset \cL_0$, $T_{e^{i\theta}x}\cL_0$ contains the $(d-1)$-dimensional subspace $T_{e^{i\theta} x}(P_\theta\cap S^{2d-1}(r))$. Moreover, by differentiating $x\mapsto e^{if(\Vert x\Vert^2)}x$ in the radial direction, we get that $T_{e^{i\theta} x}\cL_0$ also contains the line generated by $(1+2ir^2f'(r^2))e^{i\theta}x$. On the other hand, $T_{e^{i\theta}x}P=T_{e^{i\theta}x}(P\cap S^{2d-1}(r))\oplus \operatorname{Span}(e^{i\theta}x)$.
    Therefore, since $f'(r^2)\neq 0$, we have $T_{e^{i\theta}x}P\cap \operatorname{Span}((1+2ir^2f'(r^2))e^{i\theta}x)=\{0\}$, and
    \begin{align*}
        T_{e^{i\theta} x}\cL_0\cap T_{e^{i\theta}x}P&=T_{e^{i\theta} x}(P_\theta\cap S^{2d-1}(r))\cap T_{e^{i\theta}x}(P\cap S^{2d-1}(r))\\
        &=T_{e^{i\theta} x}P_\theta\cap T_{e^{i\theta}x}P\cap T_{e^{i\theta} x}S^{2d-1}(r)\\
        &=T_{e^{i\theta} x}(P_\theta\cap P\cap S^{2d-1}(r))\\
        &=T_{e^{i\theta} x}C_r.
    \end{align*}
\end{proof}

Note that in this proof, we have not used the explicit form of the function $f$, but only that it diverges at $0$ and has non-vanishing derivative near $0$. However, the assumption that $f(r)=-\frac{1}{2}\log(r)$ for $0<r<\varepsilon$ will become crucial in the proof of the next lemma.

\begin{lem}
    $\cL_0$ intersects any $C^1$ exact Lagrangian submanifold $N$ passing through the origin infinitely many times arbitrarily close to the origin.
\end{lem}

\begin{proof}
    Let $P$ be the tangent space to $N$ at the origin. It is a Lagrangian plane, and following the proof of Lemma \ref{lem:intersection Lag plane}, there exists an angle $\theta$ such that for all $r>0$ sufficiently small and such that $f(r^2)=\theta \text{ mod } \pi$, $\cL_0$ intersects $P$ cleanly along $C_r=\cL_0\cap P\cap S^{2d-1}(r)=P_\theta\cap P\cap S^{2d-1}(r)$. For $k\in \N$, let $r(k)\coloneq e^{-\theta-k\pi}$. Then, for $k$ sufficiently large, $r(k)$ is close to $0$ and $f(r(k)^2)=-\frac 1 2\log(r(k)^2)=\theta + k\pi$, and therefore $\cL_0$ intersects $P$ cleanly along $C_{r(k)}$.
    
    The goal will be to show that for $k$ large enough, $\cL_0$ also intersects $N$ in a small neighbourhood of $C_{r(k)}$. We will prove that this happens if we can write $N$ as an exact graph over $P$ in a Weinstein neighbourhood of $P$ in which $\cL_0$ is identified with the conormal bundle over $C_{r(k)}$. The normal form theorem for clean Lagrangian intersections \cite[Proposition 3.4.1]{MR1736217} provides us with the desired Weinstein neighbourhood, but to make sure that $N$ becomes a graph in this neighbourhood when $k$ is large enough, we need the Weinstein chart to be uniform in some sense. In our case, this uniformity comes from the following rescaling property that we get from our choice of function $f$: for $x\in \R^d$ and $k$ such that $r(k)x\in B(\varepsilon)$, we have 
    \begin{align*}
        \varphi_0(r(k)x)&=r(k)e^{if(r(k)^2\Vert x\Vert^2)}x=r(k)e^{-\frac i 2\log(r(k)^2\Vert x\Vert^2)}x\\
        &=r(k)e^{i(\theta+k\pi)}e^{-\frac i 2\log(\Vert x\Vert^2)}x=(-1)^kr(k)e^{i\theta}e^{-\frac i 2\log(\Vert x\Vert^2)}x.
    \end{align*}
    Since $\cL_0=-\cL_0$, this implies that $r(k)^{-1}(\cL_0\cap B(\varepsilon))=e^{i\theta}\widetilde{\cL}_0\cap B(r(k)^{-2}\varepsilon)$, where 
    \begin{equation*}
        \widetilde{\cL}_0=\{e^{-\frac{i}{2}\log(\Vert x\Vert^2)}x,x\in \R^d\} 
    \end{equation*}
    Therefore, for $k$ large enough,
    \begin{align*}
        r(k)^{-1}C_{r(k)}&=r(k)^{-1}(\cL_0\cap P\cap S^{2d-1}(r(k)))\\
        &=(r(k)^{-1}\cL_0)\cap(r(k)^{-1} P)\cap(r(k)^{-1} S^{2d-1}(r(k)))\\
        &=(e^{i\theta}\widetilde{\cL}_0)\cap P\cap S^{2d-1}(1)
    \end{align*}
    which does not depend on $k$. Let $C\coloneq (e^{i\theta}\widetilde{\cL}_0)\cap P\cap S^{2d-1}(1)$. To show that $\cL_0$ intersects $N$ infinitely many times arbitrarily close to the origin, it is now enough to show that for $k$ large enough, $N_k\coloneq r(k)^{-1}N$ intersects $e^{i\theta}\widetilde{\cL}_0$ in a neighbourhood of $C$ that does not contain $0$.
    Since $P$ intersects $e^{i\theta}\widetilde{\cL}_0$ cleanly along $C$ (by Lemma \ref{lem:intersection Lag plane}), by the normal form theorem for clean Lagrangian intersections \cite[Proposition 3.4.1]{MR1736217}, we can find a small neighbourhood $V$ of $C$ not containing the origin, an open set $W\subset T^*P$ and an exact symplectomorphism $\psi\colon V\to W$ mapping $P\cap V$ to $0_P\cap W$ and $e^{i\theta}\widetilde{\cL}_0\cap V$ to the intersection of the conormal bundle $\nu^*\psi(C)$ with $W$.
    Since $N_k$ converges to $P$ uniformly in the $C^1$ topology inside $V$, we have that for $k$ large enough, $\psi(N_k\cap V)$ is a graph over $0_P\cap W$. Moreover, since $N$ and $\psi$ are exact, $\psi(N_k\cap V)$ is also exact and therefore it is of the form $\G(dh_k)$, where $h_k\colon 0_P\cap W\to\R$ is a $C^1$ function.
    Then, 
    \begin{equation*}
        \psi(N_k\cap e^{i\theta} \widetilde{\cL}_0\cap V)=\G(dh_k)\cap \nu^*\psi(C)=\{(x,d_xh_k)\mid x\in\psi(C),d_xh_k|_{T_x\psi(C)}=0\}
    \end{equation*}
    Since $h_k$ is $C^1$ and $\psi(C)$ is compact, $h_k|_{\psi(C)}$ admits a critical point, and therefore $N_k\cap e^{i\theta}\widetilde{\cL}_0\cap V$ is non-empty, which concludes the proof.    
\end{proof}

\section{Birkhoff Theorems for Stratified Lagrangian Sets}
\label{sec:BirkhoffStratified}

\subsection{A Weinstein neighbourhood for a Lagrangian with an Attached Filament}

We begin with a criterion for constructing Weinstein neighbourhoods with a prescribed differential.
    
	\begin{lem} \label{lem:RelativeSymplecticExtension}
		Let $(N_1,\omega_1)$ and $(N_2,\omega_2)$ be two symplectic manifolds, and let $S_1$ and $S_2$ be two diffeomorphic submanifolds of $N_1$ and $N_2$, respectively. Let $f : S_1 \to S_2$ be a diffeomorphism, and fix a compact subset $A$ of $S_1$ together with a neighbourhood $V$ of $A$ in $N_1$. We also fix a symplectomorphism $\Psi_V : V \to N_2$. Assume that
		\begin{enumerate}[label=\roman*.]
			\item The map $\Psi_V$ restricts to $f$ on $S_1 \cap V$.
			\item There exists a smooth fibrewise symplectic bundle isomorphism $\mathcal{F} : T_{S_1}N_1 \to f^*T_{S_2}N_2$ such that $\mathcal{F}|_{TS_1}=df$ and $\mathcal{F}=d\Psi_V$ over $S_1 \cap V$.
		\end{enumerate}
		Then, there exist neighbourhoods $U_1$ and $U_2$ of $S_1$ and $S_2$ in $N_1$ and $N_2$, respectively, and a symplectomorphism $\Psi : U_1 \to U_2$ such that
		\begin{equation}
			\Psi|_{S_1} = f, \quad d\Psi|_{T_{S_1}N_1} = \mathcal{F}, \quad \text{and} \quad  \Psi = \Psi_V \text{ on a (smaller) neighbourhood of } A.
		\end{equation}
	\end{lem}		

	\begin{proof}
		Choose Riemannian metrics $g_1$ and $g_2$ on $N_1$ and $N_2$, respectively, such that $g_1 = \Psi_V^*g_2$ on a neighbourhood of $A$, after shrinking $V$ if necessary. Let $E_1$ be the normal bundle of $S_1$ in $N_1$, and set $E_2 = \mathcal{F}(E_1)$. Denote by $\exp_i$ the exponential map associated to $g_i$. After restricting to sufficiently small neighbourhoods of the zero sections, the maps $\exp_i$ identify neighbourhoods of $S_i$ in $E_i$ with tubular neighbourhoods $U_i$ of $S_i$ in $N_i$.

		After shrinking $U_1$ and $U_2$ if necessary, we define the diffeomorphism $F : U_1 \to U_2$ by
		\begin{equation}
			F = \exp_2 \circ \mathcal{F}|_{E_1} \circ \exp_1^{-1}
		\end{equation}
		By construction,
		\begin{equation}
			F|_{S_1} = f, \quad dF|_{T_{S_1}N_1} = \mathcal{F}, \quad \text{and } F = \Psi_V \text{ near } A
		\end{equation}
		The last identity follows from the compatibility of the metrics $g_1 = \Psi_V^*g_2$ near $A$, which implies $\exp_2 \circ d\Psi_V = \Psi_V \circ \exp_1$.
		
		We now correct $F$ by a Moser argument. Consider the closed $2$-form $\eta = F^*\omega_2 - \omega_1$. Since $\mathcal{F}$ is fibrewise symplectic and $dF|_{T_{S_1}N_1} = \mathcal{F}$, the form $\eta$ vanishes along $S_1$. Moreover, since $F=\Psi_V$ near $A$, it also vanishes on a neighbourhood of $A$. We identify the tubular neighbourhood $U_1$ of $S_1$ with its preimage under $\exp_1$ in the normal bundle $E_1$. Let $R : [0,1] \times U_1 \to U_1$ be the deformation retraction defined by $R(s,(x,v)) = (x,sv)$, and denote $R_s = R(s,\cdot)$. By Cartan's formula,
		\begin{align*}
			R_1^*\eta - R_0^*\eta &= \int_0^1 \frac{d}{ds}(R_s^*\eta) \; ds = \int_0^1 L_{\partial_s}(R^*\eta) \; ds \\
			&= d \left( \int_0^1 \iota_{\partial_s}(R^*\eta) \; ds \right) + \int_0^1 \iota_{\partial_s}R^*(d\eta) \; ds \\
			&= d \left( \int_0^1 \iota_{\partial_s}(R^*\eta) \; ds \right)
		\end{align*}
		where we have used $d\eta=0$. Since $R_1$ is the identity and $R_0$ takes values in $S_1$, where $\eta$ vanishes, we obtain
		\begin{equation}
			\eta = d\mu
			\quad \text{with} \quad
			\mu := \int_0^1 \iota_{\partial_s}(R^*\eta) \; ds
		\end{equation}
		Moreover, $\mu$ vanishes near $A$, and its $1$-jet vanishes along $S_1$. Set $\Omega_s = \omega_1 + s d\mu$. Since $\Omega_s=\omega_1$ along $S_1$, after shrinking $U_1$ if necessary, we may assume that $\Omega_s$ is symplectic on $U_1$ for every $s \in [0,1]$. We now apply Moser's argument. Let $X_s$ be the vector field defined by $\iota_{X_s}\Omega_s = -\mu$
		and let $\varphi_s$ denote its flow. Then
		\begin{align*}
			\frac{d}{ds}(\varphi_s^*\Omega_s) &= \varphi_s^* \left( \partial_s \Omega_s + L_{X_s}\Omega_s \right) = \varphi_s^* \left( d\mu + d(\iota_{X_s}\Omega_s) \right) = \varphi_s^* \left( d\mu - d\mu \right) = 0
		\end{align*}
		Hence, $\varphi_1^*\Omega_1 = \Omega_0 = \omega_1$. Finally, set $\Psi = F \circ \varphi_1$. We obtain
		\begin{align*}
			\Psi^*\omega_2
			= \varphi_1^*F^*\omega_2
			= \varphi_1^*\Omega_1
			= \omega_1
		\end{align*}
		so that $\Psi$ is a symplectomorphism.

		It remains to verify the prescribed behaviour along $S_1$ and near $A$. Since $\mu$ has vanishing $1$-jet along $S_1$, the same holds for $X_s$. Hence, $\varphi_s$ restricts to the identity on $S_1$ and its differential is the identity along $S_1$. Moreover, since $\mu$ vanishes near $A$, so does $X_s$, and therefore $\varphi_s$ is the identity near $A$. Consequently,
		\begin{equation}
			\Psi|_{S_1} = f, \quad d\Psi|_{T_{S_1}N_1} = \mathcal{F}, \quad \text{and } \Psi = \Psi_V \text{ near } A
		\end{equation}
		which concludes the proof.
	\end{proof}

	The main result of this section is given by the following lemma.
    
	\begin{lem} \label{lem:WeinsteinLag+Curve}
		Let $\mathcal{K} = \cL \cup \im(\gamma)$, where $\cL$ is an exact Lagrangian submanifold, and let $\gamma : [0,1] \to T^*M$ be a smooth embedded curve satisfying
        \begin{equation}
	       	\gamma([0,1]) \cap \cL = \{ \gamma(0) \}, \quad \dot{\gamma}(0) \notin T_{\gamma(0)} \cL
	    \end{equation}
        Then, there exists a neighbourhood $U$ of $\mathcal{K}$ and an exact symplectomorphism $\Psi : U \to \Psi(U) \subset  T^*\cL$ such that the image of $\cL$ is the zero section $0_\cL$ and the image of $\gamma$ is the curve $\gamma_0 : [0,1] \to T^*\cL$ defined by $\gamma_0(t) = (\gamma(0), p_1 = t, 0, \ldots, 0)$ written in some coordinate chart of $T^*\cL$.
	\end{lem}

	\begin{proof}
		This lemma is proved in four steps.
        
		\textit{1. Defining $\Psi$ around $\im(\gamma)$.} Set $x = \gamma(0)$, let $e$ be a non-zero element of $T^*_{\gamma(0)}\cL$, and set $\tilde{\gamma}_0(t) = (x,te)$. Since $\dot{\gamma}(0) \notin T_x\cL$, we can choose a symplectic linear isomorphism $A : T_x(T^*M) \to T_x(T^*\cL)$ such that
		\begin{equation}
			A(T_x\cL) = T_x0_{\cL} \quad \text{and} \quad A(\dot{\gamma}(0)) = e
		\end{equation}
	
		Set $v(t) = \dot{\gamma}(t)$ and $v_0(t) = \dot{\tilde{\gamma}}_0(t) = e$. Denoting by $\omega_0$ the standard symplectic form on $T^*\cL$, we choose smooth vector fields $w(t)$ and $w_0(t)$ such that
		\begin{align*}
			\omega(v(t),w(t)) = 1, \quad \omega_0(v_0(t),w_0(t)) = 1, \quad \text{and} \quad w_0(0) = Aw(0)
		\end{align*}
		Set
		\begin{equation}
			E(t) = \Span(v(t),w(t))^\omega \quad \text{and} \quad E_0(t) = \Span(v_0(t),w_0(t))^{\omega_0}
		\end{equation}
		We choose smooth symplectic frames of $E(t)$ and $E_0(t)$ such that, at $t=0$, the second one is the image under $A$ of the first one. By sending $v$ to $v_0$, $w$ to $w_0$, and these symplectic frames to one another, we obtain a smooth fibrewise symplectic bundle isomorphism
		\begin{equation}
			\mathcal{F}_\gamma : T_{\im(\gamma)}(T^*M) \to T_{\im(\tilde{\gamma}_0)}(T^*\cL)
		\end{equation}
		which restricts to $A$ at $\gamma(0)$ and to the differential of the map $f_\gamma : \gamma(t) \mapsto \tilde{\gamma}_0(t)$ on $T\im(\gamma)$.

		Applying Lemma \ref{lem:RelativeSymplecticExtension} with an empty relative set yields neighbourhoods $U_\gamma$ and $U_{\tilde{\gamma}_0}$ of $\im(\gamma)$ and $\im(\tilde{\gamma}_0)$, respectively, and a symplectomorphism $\Psi_\gamma : U_\gamma \to U_{\tilde{\gamma}_0}$ such that
		\begin{equation}
			\Psi_\gamma(\gamma(t)) = \tilde{\gamma}_0(t) \quad \text{and} \quad d\Psi_\gamma(T_{\gamma(0)}\cL) = T_{\tilde{\gamma}_0(0)}0_{\cL} 
		\end{equation}\\

		\textit{2. Adjustment around $\cL$.} We have defined a map such that $d\Psi_\gamma(T_{\gamma(0)}\cL) = T_{\tilde{\gamma}_0(0)}0_{\cL}$. However, we would like it to send a neighbourhood of the Lagrangian $\cL$ into the zero section $0_\cL$. This adjustment is done as follows.
		
		By smoothness of $\cL$ and $\Psi_\gamma$, we deduce that, in a sufficiently small neighbourhood of $x=\tilde{\gamma}_0(0)$, $\Psi_\gamma(\cL)$ is the graph of a $1$-form $\alpha$ over the zero section $0_\cL$, with $\alpha(x)=0$. Since $\Psi_\gamma(\cL)$ is Lagrangian, we have $d\alpha=0$. Thus, the translation $\tau_{\alpha} : (q,p) \mapsto (q,p-\alpha_q)$ is symplectic, sends $\Psi_\gamma(\cL)$ to $0_\cL$, and fixes every point of the fibre $T^*_x\cL$, which contains $\tilde{\gamma}_0$. Hence, the restriction of $\tau_{\alpha} \circ \Psi_\gamma$ to a neighbourhood of $x$ in $\cL$ is of the form $(q,p) \mapsto (r(q,p),0)$ for some local diffeomorphism $r$ satisfying $r(x)=x$. 
		
		We shrink $U_\gamma$ and $U_{\tilde{\gamma}_0}$ so that $r : \cL \cap U_\gamma \to 0_{\cL} \cap U_{\tilde{\gamma}_0}$ is well defined. Define the symplectic cotangent lift of its inverse by
		\begin{equation}
			\widehat{r^{-1}}(q,p) = \left( r^{-1}(q), (d_{r^{-1}(q)}r)^*p \right)
		\end{equation}
		
		It sends $(x,te)$ to $(x,te')$ for some $e' \neq 0$ in $T^*_{\gamma(0)}\cL$. Choose coordinates $(q_1,\ldots,q_d)$ centered at $x$ such that $dq_1(x)=e'$. In these coordinates, set $\gamma_0(t)= (\gamma(0),p_1=t,0,\ldots,0)$. We then define
		\begin{equation}
			\Psi_0 = \widehat{r^{-1}} \circ \tau_{\alpha} \circ \Psi_{\gamma}
		\end{equation}
		which satisfies
		\begin{equation}
			\Psi_0(x) = (x,0) \quad \text{for all } x \in U_\gamma \cap \cL, \quad \text{and} \quad \Psi_0(\gamma(t)) = \gamma_0(t) \quad \text{for all } t \in [0,1] 
		\end{equation}\\

		\textit{3. Extension to a neighbourhood of $\cL$.} We apply Lemma \ref{lem:RelativeSymplecticExtension} with $S_1 = \cL$ and $S_2 = 0_\cL$, and with the diffeomorphism $f : \cL \to 0_\cL$ given by $f(y) = (y,0)$. We take as relative set a compact neighbourhood $A$ of $x = \gamma(0)$ in $\cL$, and prescribe the map $\Psi_0$ on a neighbourhood of $A$.

		Choose $A$ small enough so that $A \subset U_\gamma$. Let $\mathcal{C}$ be a Lagrangian complement of $T\cL$ over a neighbourhood of $A$, and set $\mathcal{C}_0 = d\Psi_0(\mathcal{C})$. We first extend $\mathcal{C}$ and $\mathcal{C}_0$ to global Lagrangian complements of $T\cL$ and $T0_\cL$, respectively. To do so, choose compatible almost-complex structures $J$ and $J_0$ and consider the global Lagrangian complements $J(T\cL)$ and $J_0(T0_\cL)$. The local complements $\mathcal{C}$ and $\mathcal{C}_0$ can be glued to these global complements outside $A$. Indeed, relative to a fixed Lagrangian complement, any other Lagrangian complement is the graph of a bundle map whose associated bilinear form is symmetric. One can therefore interpolate between the corresponding bundle maps using a smooth cutoff function while preserving the Lagrangian condition.

		We now define a smooth fibrewise symplectic bundle isomorphism $\mathcal{F} : T_\cL(T^*M) \to f^*T_{0_\cL}(T^*\cL)$. On $T\cL$, set $\mathcal{F}|_{T\cL} = df$. For every $y \in \cL$, consider the linear isomorphisms
		\begin{equation}
			\begin{matrix}
				\beta_y : & \mathcal{C}_y & \longrightarrow & T_y^*\cL &,& \beta_{0,y} : & (\mathcal{C}_0)_{f(y)} & \longrightarrow & T_y^*\cL, \\
				& c & \longmapsto & \big(v \mapsto \omega(v,c)\big) &, && c_0 & \longmapsto & \big(v \mapsto \omega_0(df(v),c_0)\big)
			\end{matrix}
		\end{equation}
		For every $c \in \mathcal{C}_y$, define
		\begin{equation}
			\mathcal{F}(c) = \beta_{0,y}^{-1}\big(\beta_y(c)\big)
		\end{equation}
		Equivalently, $\mathcal{F}(c)$ is the unique element of $(\mathcal{C}_0)_{f(y)}$ such that
		\begin{equation}
			\omega(v,c) = \omega_0(df(v),\mathcal{F}(c))
			\quad \text{for all } v \in T_y\cL
		\end{equation}
	
		The map $\mathcal{F}$ is a smooth bundle isomorphism. We claim that it is fibrewise symplectic. Indeed, for all $u,u' \in T_y\cL$ and $c,c' \in \mathcal{C}_y$, we have
		\begin{align*}
			\omega_0\big(\mathcal{F}(u+c),\mathcal{F}(u'+c')\big) &= \omega_0\big(\mathcal{F}(u),\mathcal{F}(u')\big) + \omega_0\big(\mathcal{F}(u),\mathcal{F}(c')\big) \\
			&\quad + \omega_0\big(\mathcal{F}(c),\mathcal{F}(u')\big) + \omega_0\big(\mathcal{F}(c),\mathcal{F}(c')\big) \\
			&= \omega_0\big(df(u),\mathcal{F}(c')\big) + \omega_0\big(\mathcal{F}(c),df(u')\big) \\
			&= \omega(u,c') + \omega(c,u') \\
			&= \omega(u+c,u'+c')
		\end{align*}
		where we have used that $T_y\cL$, $\mathcal{C}_y$, $T_{f(y)}0_\cL$, and $(\mathcal{C}_0)_{f(y)}$ are Lagrangian.

		Moreover, by construction, $\mathcal{F}=d\Psi_0$ over $A$. We may therefore apply Lemma \ref{lem:RelativeSymplecticExtension}. We obtain neighbourhoods $U_\cL$ and $U_{0_\cL}$ of $\cL$ and $0_\cL$ in $T^*M$ and $T^*\cL$, respectively, together with a symplectomorphism $\Psi_\cL : U_\cL \to U_{0_\cL}$ such that
		\begin{equation}
			\Psi_\cL(y) = (y,0) \quad \text{for all } y \in \cL, \quad \text{and} \quad  \Psi_\cL = \Psi_0 \quad \text{on a neighbourhood of } A
		\end{equation}

		By shrinking $U_\gamma$ and $U_\cL$ if necessary, we may assume that $U_\gamma \cap U_\cL$ is contained in the neighbourhood of $A$ on which $\Psi_\cL = \Psi_0$. The two maps therefore glue to a symplectomorphism
		\begin{equation}
			\Psi : U = U_\gamma \cup U_\cL \longrightarrow U_{\gamma_0} \cup U_{0_\cL}
		\end{equation}
		satisfying the required geometric properties. \\

		\textit{4. Exactness.} It remains to prove that the symplectomorphism $\Psi$ is exact. After shrinking the neighbourhoods if necessary, we may assume that $U_\cL$ is a fibrewise star-shaped neighbourhood of $\cL$, while $U_\gamma$ is a tubular neighbourhood of $\gamma$, chosen so that both $U_\gamma$ and $U_\gamma \cap U_\cL$ are contractible.

		Since $U_\cL$ deformation retracts onto $\cL$, we have $H^1(U_\cL,\R) \simeq H^1(\cL,\R)$. Moreover,
		\begin{equation}
			H^1(U_\gamma,\R)=0 \quad \text{and} \quad H^1(U_\cL \cap U_\gamma,\R)=0
		\end{equation}
		By the Mayer--Vietoris sequence in de Rham cohomology, and since the sets involved are connected, we obtain
		\begin{align*}
			0
			\longrightarrow H^1(U,\R)
			\longrightarrow H^1(U_\cL,\R) \oplus H^1(U_\gamma,\R)
			\longrightarrow H^1(U_\cL \cap U_\gamma,\R),
		\end{align*}
		and therefore
		\begin{equation}
			H^1(U,\R) \simeq H^1(\cL,\R)
		\end{equation}

		Let $\lambda_0$ denote the Liouville form on $T^*\cL$, and set $\eta = \Psi^*\lambda_0 - \lambda$. Since $\Psi$ is symplectic, $\eta$ is closed. On $\cL$, the map $\Psi$ takes values in the zero section, and hence $\Psi^*\lambda_0|_\cL = 0$. Additionally, since $(\cL,h)$ is exact, we have $\lambda|_\cL = dh$ and we infer $\eta|_\cL = -dh$. Thus, the restriction of the cohomology class $[\eta]$ to $\cL$ vanishes. Under the isomorphism $H^1(U,\R) \simeq H^1(\cL,\R)$, this implies $[\eta]=0$. Hence, $\eta$ is exact, and consequently $\Psi$ is an exact symplectomorphism.
	\end{proof}

\subsection{One Dimensional Case: Unique Liftability}

We show the following local rigidity result in the one-dimensional case.
\begin{prop} \label{prop:1DLocalRigidity}
	Let $(\mathcal{K},k)$ be in $\overline{\mathcal{LB}}(T^*\T^1,\omega)$ and assume that there exists an open set $U$ of $T^*\T^1$ such that $\mathcal{K} \cap U$ is an embedded smooth curve $\gamma:(0,1) \to T^*\T^1$. Then, $k$ restricted to $\gamma$ is a Liouville primitive, i.e. $d(k|_{\gamma}) = \lambda|_{\gamma}$.
\end{prop}

\begin{proof}
	Let $(\cL_n,h_n)$ be a sequence of smooth exact Lagrangian branes that brane-converges to $(\mathcal{K},k)$. Fix $0<a<b<1$. We choose an exact Weinstein chart $\Psi : V \subset U \to V_0 \subset T^*\R$ from a neighbourhood $V$ of $\gamma([a,b])$, such that
	\begin{equation*}
		\Psi(\gamma(t))=(t,0), \quad \Psi^*\lambda_0-\lambda=dS, \quad \lambda_0=p\;dq
	\end{equation*}
	After shrinking $V$ if necessary, we may assume that $V_0$ contains a box $[a-\delta_0,b+\delta_0]\times[-\delta_0,\delta_0]$ for some $\delta_0>0$. We set
	\begin{equation*}
		\Lambda_n=\Psi(\cL_n\cap V), \quad g_n=(h_n+S)\circ\Psi^{-1}, \quad k_0=(k+S)\circ\Psi^{-1}
	\end{equation*}
	By Proposition \ref{BraneInvariance}, $g_n$ is a Liouville primitive on $\Lambda_n$ and $(\Lambda_n,g_n)$ brane converges to $(\Psi(\mathcal{K}\cap V),k_0)$.

	We prove that $k_0$ is constant on $[a,b]\times\{0\}$. Fix $\delta>0$ sufficiently small and set
	\begin{equation*}
		R=[a,b]\times[-\delta,\delta]\subset V_0, \quad \text{and} \quad \Psi(\mathcal{K}\cap V)\cap R=[a,b]\times\{0\}
	\end{equation*}
	By Hausdorff convergence, there exists a sequence of positive real numbers $\delta_n\to0$ such that
	\begin{equation} \label{eq:FilBox}
		\Lambda_n\cap R
		\subset
		[a,b]\times(-\delta_n,\delta_n).
	\end{equation}

	Fix $t\in(a,b)$. By Hausdorff convergence, we can choose $z_n=(t_n,p_n)\in\Lambda_n\cap\operatorname{int}(R)$ such that $z_n\to(t,0)$. Let $\cL_n(s):\T^1\to\cL_n$ be a parametrization of $\cL_n$, and choose $s_n$ such that $\Psi(\cL_n(s_n))=z_n$. Lifting the parametrization to $\R$, we consider it on $[s_n,s_n+1]$. Since $\cL_n$ is not contained in $\Psi^{-1}(R)$, there exists a first exit time
	\begin{equation*}
		\tau_n := \inf \left\{ \tau\in[s_n,s_n+1] \; \middle| \; \Psi(\cL_n(\tau))\in\partial R \right\}
	\end{equation*}
	By \eqref{eq:FilBox}, for $n$ sufficiently large, $\Lambda_n$ does not intersect the horizontal sides of $R$. Hence,
	\begin{equation*}
		y_n:=\Psi(\cL_n(\tau_n)) \in \{a,b\}\times(-\delta_n,\delta_n).
	\end{equation*}
	We write $y_n=(e_n,r_n)$ with $e_n\in\{a,b\}$. Now set
	\begin{equation*}
		\tau_n^0 := \sup \left\{ \tau\in[s_n,\tau_n] \; \middle| \; \Psi(\cL_n(\tau))\in \{t_n\}\times(-\delta_n,\delta_n) \right\}
	\end{equation*}
	and define $x_n:=\Psi(\cL_n(\tau_n^0))$. We obtain embedded arcs $\gamma_n := \Psi\circ\cL_n|_{[\tau_n^0,\tau_n]}$ joining $x_n$ to $y_n$. By definition of $\tau_n$ and $\tau_n^0$, the interior of $\gamma_n$ lies strictly between the vertical lines $q=t_n$ and $q=e_n$.

	Let $\sigma_n$ be the polygonal path joining
	\begin{equation*}
		y_n \longrightarrow (e_n,\delta_n) \longrightarrow (t_n,\delta_n) \longrightarrow x_n
	\end{equation*}
	The simple closed curve $\gamma_n\cup\sigma_n$ bounds a domain contained in the rectangle of area at most $2\delta_n|e_n-t_n|$. Hence, by Stokes' theorem,
	\begin{equation*}
		\left| \int_{\gamma_n}\lambda_0 + \int_{\sigma_n}\lambda_0 \right| \leq 2\delta_n|e_n-t_n|
	\end{equation*}
	Moreover, we have
	\begin{equation*}
		\left| \int_{\sigma_n}\lambda_0 \right| \leq \delta_n(e_n+t_n)
	\end{equation*}
	Since $g_n$ is a Liouville primitive on $\Lambda_n$, we obtain
	\begin{align*}
		|g_n(y_n)-g_n(x_n)| &= \left| \int_{\gamma_n}\lambda_0 \right| \leq \delta_n(e_n+t_n) + 2\delta_n|e_n-t_n| \leq 3\delta_n(a+b) \longrightarrow 0 \quad \text{as } n\to+\infty
	\end{align*} 
	
	Up to extracting a subsequence, we may assume that $e_n=e$ for some $e\in\{a,b\}$. Since $\lim_n x_n = (t,0)$ and $\lim_n y_n = (e,0)$, Proposition \ref{BranePointwiseConv} yields $k_0(t,0)=k_0(e,0)$. We have therefore shown that, for every $t\in(a,b)$,
	\begin{equation*}
		k_0(t,0) \in \{k_0(a,0),k_0(b,0)\}
	\end{equation*}
	Since $k_0$ is continuous, its image on the connected set $[a,b]\times\{0\}$ is connected and contained in a two-point set. Hence, $k_0$ is constant on $[a,b]\times\{0\}$. Finally, We infer from Proposition \ref{BraneInvariance} that $d(k|_{\gamma}) = \lambda|_{\gamma}$.
\end{proof}          

\subsubsection{Proof of Item \ref{prop:UniquenessPrimitive5} of Theorem \ref{thm:UniquenessPrimitiveC0}}

    Following Proposition \ref{prop:1DLocalRigidity}, it suffices to prove that $\mathcal{K}$ belongs to $\mathscr{L}^\exists$. We first assume that $\gamma$ is transverse to $\cL$. We apply Lemma \ref{lem:WeinsteinLag+Curve} to $\mathcal{K}$ and obtain an exact symplectomorphism $\Psi : U \to U_0$ from a neighbourhood $U$ of $\mathcal{K}$ onto a neighbourhood $U_0$ of $\mathcal{K}_0 = 0_{\mathbb{T}^1} \cup \im(\gamma_0)$, where $\gamma_0 : [0,1] \to T^*\mathbb{T}^1$ is given by $\gamma_0(t) = (0,t)$, and such that $\Psi(\mathcal{K}) = \mathcal{K}_0$. By Proposition \ref{BraneInvariance}, it is enough to prove that $\mathcal{K}_0$ belongs to $\mathscr{L}^\exists$.

	Let $\eta : [-1,1] \to \R$ be a bump function such that $0 \leq \eta \leq 1$, $\eta(0) = 1$, and $\int_{-1}^1 \eta = 1$. Let $\delta_n$ be a sequence of positive real numbers converging to $0$. Let $\delta>0$ be a small real number such that $[-\delta,\delta] \times [-\delta,1+\delta] \subset U_0$. We define the smooth map $v_n : \R \to \R$, supported in $(-\delta,\delta)$, by
	\begin{equation} \label{eq:BumpApprox}
		v_n(t) = \eta \left( \frac{t}{\delta_n} \right) - \frac{\delta_n}{\delta} \eta \left( \frac{t}{\delta} \right)
	\end{equation}
	We have
	\begin{align*}
		\int_{-\delta}^\delta v_n &= \int_{-\delta}^\delta \eta \left( \frac{t}{\delta_n} \right) \; dt - \int_{-\delta}^\delta \frac{\delta_n}{\delta} \eta \left( \frac{t}{\delta} \right) \; dt \\
		&= \delta_n \int_{-\delta/\delta_n}^{\delta/\delta_n} \eta - \delta_n \int_{-1}^1 \eta = \delta_n - \delta_n = 0
	\end{align*}

	Moreover, $\lim_n v_n(0) = \lim_n 1-\frac{\delta_n}{\delta} = 1$ while, for every non-zero $t$, we have $v_n(t)\to0$. In addition, $v_n \leq 1$ and $\inf v_n \geq -\delta_n/\delta$. It follows that the graphs of $v_n$ converge in the Hausdorff topology to $(\R \times \{0\}) \cup (\{0\} \times [0,1])$. Set
	\begin{equation} \label{eq:BumpApprox2}
		w_n(t) = \int_{-\infty}^t v_n(\tau) \; d\tau
	\end{equation}
	Since $\int_{-1/2}^{1/2} v_n = 0$, we have $w_n(-1/2)=w_n(1/2)$, and therefore $w_n|_{[-\frac{1}{2},\frac{1}{2}]}$ descends to a smooth map on $\T^1$. Moreover,
	\begin{equation*}
		\Vert w_n \Vert_\infty \leq \int_{-\delta}^{\delta}|v_n(\tau)| \; d\tau \longrightarrow 0 \quad \text{as } n \to +\infty
	\end{equation*}
	We set $h_n = w_n \circ \pi_{\T^1}$ and let $\mathcal{L}_n = \mathcal{G}(dw_n)$ be the graph of $v_n$. Then, the sequence of exact Lagrangian branes $(\mathcal{L}_n,h_n)$ brane-converges to $(\mathcal{K}_0,0)$. In particular, $\mathcal{K}_0$ belongs to $\mathscr{L}^\exists$.\\

\textit{Non-transverse case.} Set $x=\gamma(0)$ and let $h$ be a Liouville primitive on $\cL$. We approximate $\gamma$ in the $C^1$ topology by a sequence of smooth embedded curves $\gamma_j : [0,1] \to T^*\T^1$ such that
\begin{equation*}
	\gamma_j(0)=x, \quad \im(\gamma_j)\cap\cL=\{x\}, \quad \text{and} \quad \dot{\gamma}_j(0)\notin T_x\cL
\end{equation*}
Such an approximation can be constructed locally by taking coordinates in which $\cL=\{p=0\}$ and writing the initial arc as $p=f(q)$ for $q\geq0$. We then replace $f(q)$ by $f(q)+\text{sign}(f)\varepsilon_j\chi(q)q$, where $\varepsilon_j \to 0$ and $\chi$ is a non-negative smooth bump function supported near $0$ and equal to $1$ in a neighbourhood of $0$.

Set $\mathcal{K}_j=\cL\cup\im(\gamma_j)$. By the transverse case, $\mathcal{K}_j$ belongs to $\mathscr{L}^\exists$. By the local primitive property on the smooth arcs and continuity at $x$, we can normalize the associated primitive $k_j$ so that $k_j|_{\cL}=h$ and
\begin{equation*}
	k_j(\gamma_j(t)) = h(x) + \int_0^t \gamma_j^*\lambda
\end{equation*}
We define $k:\mathcal{K}\to\R$ by
\begin{equation*}
	k|_{\cL}=h \quad \text{and} \quad k(\gamma(t)) = h(x) + \int_0^t \gamma^*\lambda
\end{equation*}
Since $\gamma_j\to\gamma$ in the $C^1$ topology, we have
\begin{align*}
	d_B\bigl((\mathcal{K}_j,k_j),(\mathcal{K},k)\bigr) \leq \sup_{t\in[0,1]} \left( d(\gamma_j(t),\gamma(t)) + \left|\int_0^t(\gamma_j^*\lambda-\gamma^*\lambda)\right|\right) \longrightarrow 0 \quad \text{as } j\to+\infty.
\end{align*}
Since $\overline{\mathcal{LB}}(T^*\T^1,\omega)$ is closed for the brane topology, we conclude that $(\mathcal{K},k)\in\overline{\mathcal{LB}}(T^*\T^1,\omega)$. Hence, $\mathcal{K}$ belongs to $\mathscr{L}^\exists$.

    \qed

\subsubsection{Proof of Proposition \ref{prop:NonExistence1D}}
    
    Let $C \subset T^*\mathbb{T}^1$ be a smooth embedded contractible circle tangent to the zero section at $(0,0)$. Assume, by contradiction, $\mathcal{K} = 0_{\T^1} \cup C$ belongs to $\mathscr{L}^\exists$ and let $k : \mathcal{K} \to \R$ be a continuous map such that $(\mathcal{K},k) \in \overline{\mathcal{LB}}(T^*\T^1,\omega)$.

	Let $c : [0,1] \to C$ be a smooth parametrization, injective on $[0,1)$, such that $c(0)=c(1)=(0,0)$. Fix $0<a<b<1$. Then, $c([a,b])$ is a portion of $C$ that does not contain $(0,0)$. By Proposition \ref{prop:1DLocalRigidity}, $k$ restricts to a local Liouville primitive on $c([a,b])$. Since $0<a<b<1$ are arbitrary and $(0,1)$ is connected, there exists a constant $\varsigma$ such that, for every $t\in(0,1)$,
	\begin{align*}
		k(c(t)) = \varsigma + h(t) \quad \text{where } h(t) = \int_0^t \lambda_{c(\tau)}.\dot{c}(\tau)\;d\tau
	\end{align*}

	By continuity of $h$ and $k$, and since $c(0)=c(1)$, we obtain
	\begin{align*}
		0= h(0) = k(c(0))-\varsigma = k(c(1))-\varsigma = h(1) =\int_C\lambda 
	\end{align*}
	However, if $D$ denotes the disk bounded by $C$, Stokes' theorem gives
	\begin{equation*}
		\left|\int_C \lambda\right| = \leb(D) > 0
	\end{equation*}
	which yields a contradiction.

    \qed

\subsection{Higher-Dimensional Case}

\subsubsection{Non-Unique Liftability: Proof of Theorem \ref{thm:NonUniqueness2D}}

\begin{prop} \label{prop:fil}
	Assume that $\dim M \geq 2$. Let $(\mathcal{L},h)$ be a compact Lagrangian brane in $\mathcal{LB}(T^*M,\omega)$. Let $\gamma : [0,1] \to T^*M$ be an embedded continuous curve that intersects $\cL$ only at $\gamma(0)$, and let $f : \im(\gamma) \to \R$ be a continuous map such that $f(\gamma(0)) = h(\gamma(0))$. Let $(\mathcal{K},k)$ be the bare brane defined by $\mathcal{K} = \cL \cup \im(\gamma)$, where $k : \mathcal{K} \to \R$ restricts to $h$ on $\cL$ and to $f$ on $\im(\gamma)$. Then, $(\mathcal{K},k)$ is a bare Lagrangian brane in $\overline{\mathcal{LB}}(T^*M,\omega)$. 
	
	More precisely, for every small neighbourhood $U$ of $\gamma$ in $T^*M$, there exists a sequence of Lagrangian branes $(\cL_n,h_n)$ that brane-converges to $(\mathcal{K},k)$ and such that the $\cL_n$ coincide with $\cL$ outside of $U$.
\end{prop}
	
\begin{proof}
    We first outline the proof, which is divided into six steps.
    \begin{itemize}[label=-]
        \item \textit{Step 1. (Smooth approximation).} We approximate the filament and its prescribed primitive by smooth ones.
        \item \textit{Step 2. (Reduction to $T^*\cL$).} Using a relative exact Weinstein neighbourhood theorem, we reduce to a local model where $\cL$ is the zero section and the filament is vertical.
        \item \textit{Step 3. (Reparametrization).} We modify the primitive so that it is constant near the attaching point.
        \item \textit{Step 4. (Approximation by exact curves).} This is the heart of the proof. Using an explicit oscillatory construction in the spirit of convex integration, we approximate the vertical filament by brane-converging exact curves $\sigma_l$. See \cite{MR864505,MR3024860,MR4381219} for the general philosophy of convex integration, and \cite{MR3619673} for a related application to Legendrian approximation in $3$-dimensional contact manifolds, where convex integration is used to give a constructive approximation of arbitrary continuous curves by Legendrian curves.
        \item \textit{Step 5. (Second reduction).} A second exact Weinstein reduction brings the exact oscillating curve back to the standard vertical filament with constant primitive.
        \item \textit{Step 6. (Lagrangian approximation).} We approximate the standard model by graphs of exact one-forms.
    \end{itemize}

	\textit{Step 1. (Smooth approximation).} We begin by approximating $(\gamma,f)$ by a sequence of smooth bare branes $(\gamma^\infty_n,f^\infty_n)$. By relative smooth approximation of arcs, one can uniformly approximate $\gamma$ by a smooth embedding $\gamma^\infty_n : [0,1] \to T^*M$ such that
	\begin{equation} \label{prop:fil:dem:filtransv}
		\gamma^\infty_n(0) = \gamma(0), \quad \gamma^\infty_n([0,1]) \cap \cL = \{ \gamma(0) \}, \quad \dot{\gamma}^\infty_n(0) \notin T_{\gamma(0)} \cL
	\end{equation}
	and
	\begin{equation}
		\sup_{t \in [0,1]} d(\gamma^\infty_n(t),\gamma(t)) < \frac{1}{2n}
	\end{equation}

	We use the notation $f(t) = f(\gamma(t))$ and approximate $f$ by a smooth map $f^\infty_n : [0,1] \to \R$ such that $f^\infty_n(0) = f(0)$ and $\Vert f^\infty_n - f \Vert_\infty < 1/2n$. We identify $f^\infty_n$ with the map defined on $\im(\gamma^\infty_n)$ by $f^\infty_n(\gamma^\infty_n(t)) = f^\infty_n(t)$. Then, for all $t \in [0,1]$, we have
	\begin{align*}
		d(\gamma^\infty_n(t),\gamma(t)) + |f^\infty_n(\gamma^\infty_n(t)) - f(\gamma(t))| < \frac{1}{n}
	\end{align*}
	and we obtain $d_B((\gamma^\infty_n,f^\infty_n),(\gamma,f)) < 1/n$.
	We finally set $\mathcal{K}^\infty_n := \cL \cup \im(\gamma^\infty_n)$ and we set $k^\infty_n : \mathcal{K}^\infty_n \to \R$ to be the map that restricts to $h$ on $\cL$ and to $f^\infty_n$ on $\im(\gamma^\infty_n)$. The sequence $(\mathcal{K}^\infty_n,k^\infty_n)$ brane-converges to $(\mathcal{K},k)$.\\

	\textit{Step 2. (Reduction to $T^*\cL$).} We now reduce the problem to the cotangent bundle $T^*\cL$, in such a way that $\gamma^\infty_n$ becomes a straight segment contained in the fibre over $x = \gamma^\infty_n(0)$. 
    
    This is done by applying Lemma \ref{lem:WeinsteinLag+Curve} to $\mathcal{K}^\infty_n = \cL \cup \im(\gamma^\infty_n)$. We obtain an exact symplectomorphism $\Psi^0_n  : U^\infty_n \to U^0_n$ from a neighbourhood $U^\infty_n$ of $\mathcal{K}^\infty_n$ onto a neighbourhood $U^0_n$ of the set $\mathcal{K}^0 = 0_{\cL} \cup \im(\gamma^0)$. Moreover, after shrinking $U^0_n$ around $\im(\gamma^0)$ if necessary, we may choose local coordinates on $\cL$ centered at $x=\gamma(0)$ and identify $U^0_n$ with an open subset of $T^*\R^d$ in such a way that $0_\cL \cap U^0_n$ is identified with the zero section, $x=0$ and $\gamma^0(t) = (0,\ldots,0,p_1=t,0,\ldots,0)$. Under this identification, we regard $\mathcal{K}^0 \cap U^0_n$ as a subset of $T^*\R^d$.	

	By exactness of $\Psi^0_n $, there exists a map $S^0_n : U^\infty_n \to \R$ such that $(\Psi^0_n) ^*\lambda_0 - \lambda = dS^0_n$, where $\lambda_0$ denotes the Liouville form on $T^*\cL$. We define $\tilde{k}^0_n : \mathcal{K}^0 \to \R$ by
	\begin{equation}
		\tilde{k}^0_n = (k^\infty_n + S^0_n) \circ (\Psi^0_n) ^{-1}
	\end{equation}	
	
	We denote by $\tilde{f}^0_n$ the restriction of $\tilde{k}^0_n$ to $\im(\gamma^0)$, and by $h^0$ its restriction to $0_\cL$. Although it will not be used, we note from Proposition \ref{BraneInvariance} that $h^0$ is constant, since it is a Liouville primitive on the zero section. We now work with the bare brane $(\mathcal{K}^0,\tilde{k}^0_n)$, using the above local identification with $T^*\R^d$ near $\im(\gamma^0)$.\\

	\textit{Step 3. Reparametrization of $\tilde{f}^0_n$.} We approximate $\tilde{f}^0_n$ by an $(\varepsilon = \frac{1}{m})$-close map $f^0_{n,m} : \im(\gamma^0) \to \R$ which is constant equal to $h^0(\gamma^0(0))$ on a neighbourhood $\gamma^0([0,a])$ of $\gamma^0(0)$. For simplicity, we will omit the $n$ in the notation.
	
	The map $\tilde{f}^0_n$ is uniformly continuous on $\im(\gamma^0)$, hence, there exists a real constant $\delta_1 >0$ such that if $d(\gamma^0(s), \gamma^0(t)) < \delta_1$, we get $| \tilde{f}^0_n ( \gamma^0(s) ) - \tilde{f}^0_n (\gamma^0(t)) | < 1/m$. Similarly, the curve $\gamma^0$ is uniformly continuous on $[0,1]$, hence, there exists a small real constant $\delta_2 >0$ such that if $|s-t| < \delta_2$, then $d(\gamma^0(s), \gamma^0(t)) < \delta_1$. Now set $r : [0,1] \to [0,1]$ to be a smooth map such that $r(t) = 0$ on $[0,\delta_2/3]$, $r(t) = t$ on $[2\delta_2/3,1]$ and $|r(t) - t| < \delta_2$ for all $t \in [0,1]$. We set $f^0_m : \im ( \gamma^0 ) \to \R$ defined by
	\begin{equation}
		f^0_m(\gamma^0(t)) = \tilde{f}^0_n \circ \gamma^0 \circ r(t)
	\end{equation}
	If we set $a = \delta_2/3$, we have for all $t \in [0,a]$, 
	\begin{align*}
		f^0_m(t) = \tilde{f}^0_n \circ \gamma^0(r(t)) = \tilde{f}^0_n \circ \gamma^0(0) = h^0(\gamma^0(0))
	\end{align*}		
	Moreover, for all $t \in [0,1]$, $|r(t)-t| < \delta_2$, which implies that $d(\gamma^0(r(t)), \gamma^0(t)) < \delta_1$, which yields 
	\begin{align*}
		| f^0_m ( \gamma^0(t) ) - \tilde{f}^0_n (\gamma^0(t)) | = | \tilde{f}^0_n  ( \gamma^0(r(t)) ) - \tilde{f}^0_n (\gamma^0(t)) | < \frac{1}{m}
	\end{align*}
	Hence, $\Vert f^0_m - \tilde{f}^0_n \Vert_\infty < 1/m$ and, if we set $k^0_m : \mathcal{K}^0 \to \R$ to be the map that restricts to $h^0$ on $0_\cL$ and to $f^0_m$ on $\gamma^0$, we get 
	\begin{equation}
		d_B \big( (\mathcal{K}^0,k^0_m) , ( \mathcal{K}^0, \tilde{k}^0_n) \big) < \frac{1}{m}.
	\end{equation}
	
	We work with the bare brane $(\mathcal{K}^0,k^0_m)$, using the above local identification with $T^*\R^d$ near $\im(\gamma^0)$. Then, the sequence $(\mathcal{K}^0,k^0_m)_m$ brane-converges to $(\mathcal{K}^0,\tilde{k}^0_n)$.\\

	\textit{Step 4. Approximation of $(\gamma^0,f^0_m)$ by a sequence of exact curves $(\sigma_{m,l},g_{m,l})$.} In this step, we work in a neighbourhood of $\im(\gamma^0)$ identified with an open subset of $T^*\R^d$, where $0_\cL$ is identified with the zero section $0_{\R^d}$. We have $\gamma^0(t) = (0, \ldots ,0, p_1 = t, 0 ,\ldots,0)$, we use the notation $f^0_m(t) := f^0_m(\gamma^0(t))$, and we recall that $f^0_m|_{[0,a]}$ is constant for some sufficiently small $0<a<1$.
	
	 Our goal is to approximate the pair $(\gamma^0,f^0_m)$, with respect to the brane distance, by an exact curve $(\sigma_{m,l},g_{m,l})$, i.e. a pair satisfying $\sigma_{m,l}^*\lambda = dg_{m,l}$. For simplicity, we will omit $m$ in the notation. For each $l \geq 1$, we consider a curve $\sigma_l : [0,1] \to T^*\R^d$ of the form
	\begin{equation}
		\sigma_l(t) = (0,\mu_l(t),0, \ldots, 0, p_1 = t, \nu_l(t),0, \ldots , 0) \quad \text{and} \quad g'_l(t) = \nu_l(t).\mu'_l(t).
	\end{equation}
	It then follows that $\sigma_l^*\lambda = \nu_l\mu'_l dt = g'_l dt = dg_l$. We further require that $\lim_l \mu_l = \lim_l \nu_l = 0$, so that $\sigma_l$ converges uniformly to $\gamma^0$, while at the same time imposing $\lim_l g_l = f^0_m$. Notice that $g'_l$ should not converge to zero. Thus, although both $\mu_l$ and $\nu_l$ converge to zero, the product $\mu'_l.\nu_l$ must remain nontrivial. To achieve this, we introduce increasingly rapid oscillations in $\mu_l$, so that its derivatives become large while $\mu_l$ itself remains uniformly small. This mechanism is reminiscent of ideas arising in convex integration. 
	
	We set $\mu_l(t) = a_l \sin \theta_l(t)$ with $a_l = 1/\sqrt{l}$. The angle map $\theta_l$ is defined as follows. Let $\rho: [0,1] \to [0,1]$ be a smooth non-decreasing map such that $\rho (t) = 0$ on $[0,a/3]$ and $\rho(t) = 1$ on $[a,1]$, and set $\theta_l(t) := 2\pi l t\rho(t)$.

	In order to construct $g_l$, fix $\delta_l := l^{-\frac{1}{4}}$ and consider a smooth $\pi$-periodic map $\psi_l : \R \to [0,1]$ that satisfies
	\begin{equation}
		\psi_l(u) =
		\begin{cases}
			0 & \text{if } | \cos u | \leq  \delta_l \\
			1 & \text{if } | \cos u | \geq 2 \delta_l
		\end{cases}
	\end{equation}
	and set $\chi_l(t) := \rho(t) \psi_l(\theta_l(t))$. We then define $g_l$ and $\nu_l : [0,1] \to \R$ by
	\begin{equation}
		g_l(t) = f^0_m(0) + \int_0^t \chi_l(\tau) (f^0_m)' (\tau) \; d\tau \quad \text{and} \quad \nu_l(t) = 
		\begin{cases} 
			\frac{g'_l(t)}{\mu'_l(t)} & \text{if } t \in [a,1] \text{ and } \mu'_l(t) \neq 0 \\
			0 & \text{elsewhere }
		\end{cases}
	\end{equation}
	
	 We first show that $g_l$ converges uniformly to $f^0_m$ on $[0,1]$. Indeed, we have
	\begin{align*}
		g_l(t) - f^0_m(t) = \int_0^t (\chi_l(\tau) - 1) (f^0_m)'(\tau) \; d\tau
	\end{align*}
	On $[0,a]$, we have $(f^0_m)' = 0$, and hence $g_l = f^0_m$. On $[a,1]$, whenever $|\cos(\theta_l(\tau))| \geq 2\delta_l$, we have $\psi_l(\theta_l(\tau)) = 1$ and therefore $1-\chi_l(\tau) = 0$. Thus, the integrand can be non-zero only on the set
	\begin{equation*}
		A = \{\tau \in [a,1] \;|\; |\cos(\theta_l(\tau))| \leq 2\delta_l \}
	\end{equation*}
	We bound the measure of $A$. If $\tau \in A$, then, since $\rho(\tau)=1$ on $[a,1]$,
	\begin{equation*}
		|\cos(\theta_l(\tau))| = |\cos(2 \pi l \tau)| \leq 2\delta_l
	\end{equation*}
	Thus, $\tau$ belongs to
	\begin{equation*}
		\bigcup_k \left[ \frac{\arccos(2\delta_l)}{2 \pi l} + \frac{k}{2l}, \frac{1}{2l} - \frac{\arccos(2 \delta_l)}{2 \pi l} + \frac{k}{2l} \right] \cap [a,1]
	\end{equation*}
	where $k$ ranges from $0$ to $2l-1$. Hence,
	\begin{align*}
		\leb(A) &\leq 2l \left( \frac{1}{2l} - 2\frac{\arccos(2 \delta_l)}{2 \pi l} \right) = 1-2\frac{\arccos(2 \delta_l)}{\pi} \underset{l \to \infty}{\sim} \frac{4}{\pi}\delta_l
	\end{align*}
	Consequently, there exists a constant $C>0$, independent of $l$, such that
	\begin{align*}
		\Vert g_l - f^0_m \Vert_\infty \leq C \Vert (f^0_m)' \Vert_\infty \delta_l \longrightarrow 0 \quad \text{as } l \to \infty
	\end{align*}
	
	We now show that the curve $\sigma_l$ is smooth and embedded, and that it converges uniformly to $\gamma^0$. Its injectivity follows immediately from the coordinate $p_1 = t$ appearing in the definition of $\sigma_l$. Let us first examine the map $\nu_l$. For $t \in [a,1]$ such that $\mu'_l(t) \neq 0$, we have
	\begin{align*}
		\nu_l(t) &= \frac{g'_l(t)}{\mu'_l(t)} = \frac{\chi_l(t)(f^0_m)'(t)}{\mu'_l(t)} = \frac{(f^0_m)'(t)}{2 \pi l a_l} \frac{\rho(t)}{(t \rho(t))'} \frac{\psi_l(\theta_l(t))}{\cos(\theta_l(t))} = \frac{(f^0_m)'(t)}{2 \pi l a_l} \frac{\psi_l(\theta_l(t))}{\cos(\theta_l(t))}
	\end{align*}
	where, in the last equality, we used that $\rho(t)=1$ on $[a,1]$. Since $f^0_m$ is constant on $[0,a]$, there is no regularity issue at $t=a$. Moreover, whenever $|\cos(\theta_l(t))| \leq \delta_l$, we have $\psi_l(\theta_l(t))=0$. This allows $\nu_l$ to extend smoothly across the points where $\mu'_l(t)=0$. Hence, $\nu_l$ is smooth, and therefore so is the curve $\sigma_l$.
	
	We now prove that $\nu_l$ converges uniformly to $0$. Whenever $|\cos(\theta_l(t))| \geq \delta_l$, we have
	\begin{align*}
		|\nu_l(t)| &= \left| \frac{(f^0_m)'(t)}{2 \pi l a_l} \frac{\psi_l(\theta_l(t))}{\cos(\theta_l(t))} \right| \leq \frac{\Vert (f^0_m)' \Vert_\infty}{2 \pi l a_l \delta_l} = \frac{\Vert (f^0_m)' \Vert_\infty}{2 \pi l^{\frac{1}{4}}} \longrightarrow 0 \quad \text{as } l \to \infty
	\end{align*}
	On the other hand, if $|\cos(\theta_l(t))| \leq \delta_l$, then $\nu_l(t)=0$. Thus, $\nu_l$ converges uniformly to $0$. 
	
	Finally, since $\mu_l = a_l \sin \theta_l$ and $a_l \to 0$, the sequence $\mu_l$ also converges uniformly to $0$. We conclude that $\sigma_l$ converges uniformly to $\gamma^0$.	 In particular, $(\sigma_l,g_l)$ brane-converges to $(\gamma^0,f^0_m)$. Let $\mathcal{K}'_l = 0_\cL \cup \im(\sigma_l)$, and let $k'_l : \mathcal{K}'_l \to \R$ be the map that restricts to $h^0$ on $0_\cL$ and to $g_l$ on $\im(\sigma_l)$. Then, $(\mathcal{K}'_l,k'_l)_l$ brane-converges to $(\mathcal{K}^0,k^0_m)$.\\

	\textit{Step 5. Second reduction to $T^*\cL$.} We apply Lemma \ref{lem:WeinsteinLag+Curve} to $(\mathcal{K}'_l,k'_l)$ and obtain an exact symplectomorphism $\Psi^1_l : U'_l \to U^1_l$ from a neighbourhood $U'_l$ of $\mathcal{K}'_l$ onto a neighbourhood $U^1_l$ of the set $\mathcal{K}^1 = 0_{\cL} \cup \im(\sigma^0)$, where $\sigma^0(t) = (\gamma(0),p_1=t,0,\ldots,0)$. Moreover, after shrinking $U^1_l$ around $\im(\sigma^0)$ if necessary, we may choose new local coordinates on $\cL$ centered at $x=\gamma(0)$ and identify $U^1_l$ with an open subset of $T^*\R^d$ in such a way that $0_\cL \cap U^1_l$ is identified with the zero section, $x=0$, and $\sigma^0(t) = (0,\ldots,0,p_1=t,0,\ldots,0)$. Under this identification, we regard $\mathcal{K}^1 \cap U^1_l$ as a subset of $T^*\R^d$.

	By exactness of $\Psi^1_l$, there exists a map $S^1_l : U'_l \to \R$ such that $(\Psi^1_l)^*\lambda_0 - \lambda_0 = dS^1_l$, where $\lambda_0$ denotes the Liouville form on $T^*\cL$. We define $k^1_l : \mathcal{K}^1 \to \R$ by
	\begin{equation}
		k^1_l = (k'_l + S^1_l) \circ (\Psi^1_l)^{-1}
	\end{equation}

	We denote by $g^1_l$ the restriction of $k^1_l$ to $\im(\sigma^0)$, and by $h^1_l$ its restriction to $0_\cL$. By construction of $\sigma_l$, we have $\lambda_0|_{\sigma_l} = dg_l$. Hence, by Proposition \ref{BraneInvariance},
	\begin{equation}
		\lambda_0|_{\sigma^0} = dg^1_l \quad \text{and} \quad \lambda_0|_{0_\cL} = dh^1_l
	\end{equation}
	Since $\sigma^0$ is contained in the fibre over $x$, we have $\lambda_0|_{\sigma^0}=0$, while $\lambda_0$ also vanishes on the zero section. Therefore, both $g^1_l$ and $h^1_l$ are constant. Since $\sigma^0$ intersects $0_\cL$ at $x$ and $k^1_l$ is well defined on their union, these constants coincide. Hence, $k^1_l$ is constant on $\mathcal{K}^1$.\\

	\textit{Step 6. Construction of a Lagrangian brane.} We work in a neighbourhood of $\im(\sigma^0)$ identified with $[-\delta,\delta]^d \times [-\delta,1+\delta] \times [-\delta,\delta]^{d-1}$ of $T^*\R^d$, where $0_\cL$ is identified with the zero section $0_{\R^d}$. We have $\sigma^0(t) = (0, \ldots ,0, p_1 = t, 0 ,\ldots,0)$.

	Let $\eta : [-1,1] \to \R$ be a bump function such that $0 \leq \eta \leq 1$, $\eta(0) = 1$, and $\int_{-1}^1 \eta = 1$. Let $\delta_i$ be a sequence of positive real numbers converging to $0$. For $i$ sufficiently large so that $\delta_i<\delta$, define as in \eqref{eq:BumpApprox} the smooth map $v_i : \R \to \R$, supported in $(-\delta,\delta)$, by $v_i(t) = \eta \left( \frac{t}{\delta_i} \right) - \frac{\delta_i}{\delta} \eta \left( \frac{t}{\delta} \right)$. Then, we have $\int_{-\delta}^\delta v_i =0$ and the graphs of $v_i$ converge in the Hausdorff topology to $(\R \times \{0\}) \cup (\{0\} \times [0,1])$. Moreover, as in \eqref{eq:BumpApprox2}, set $w_i(t) = \int_{-\infty}^t v_i(s) \; ds$ which descends to $\T^1$ and satisfies $\lim_i \Vert w_i \Vert_\infty =0$. We then set $r_i = \sqrt{\Vert w_i \Vert_\infty}$ so that $r_i \to 0$.
	
	We now extend the construction to $\R^d$. Let $\bar{\eta} : \R^{d-1} \to \R$ be a smooth bump function supported in $(-\delta,\delta)^{d-1}$ such that $0 \leq \bar{\eta} \leq 1$ and $\bar{\eta} \equiv 1$ near $0$. We define $u_i : \R^d \to \R$ by
	\begin{equation}
		u_i(q_1,\bar{q}_1) = w_i(q_1)\bar{\eta}\left(\frac{\bar{q}_1}{r_i}\right)
	\end{equation}
	where $\bar{q}_1=(q_2,\ldots,q_d)$. For $i$ sufficiently large, $u_i$ is supported in $(-\delta,\delta)^d$, and $u_i$ converges uniformly to $0$, since
	\begin{equation}
		\Vert u_i \Vert_\infty \leq \Vert w_i \Vert_\infty \longrightarrow 0
	\end{equation}

	We claim that the graph of its differential $\mathcal{G}(du_i)$ converges in the Hausdorff topology to $\mathcal{K}_{\R^d} := 0_{\R^d} \cup \im(\sigma^0)$. Indeed, using $w'_i=v_i$, we have
	\begin{align*}
		du_i(q_1,\bar{q}_1) &= (\partial_{q_1}u_i,\partial_{\bar{q}_1}u_i) = \left( v_i(q_1)\bar{\eta}\left(\frac{\bar{q}_1}{r_i}\right), \frac{w_i(q_1)}{r_i} d\bar{\eta}\left(\frac{\bar{q}_1}{r_i}\right) \right)
	\end{align*}

	Moreover, the parts of the graphs $\mathcal{G}(du_i)$ that differ from the zero section are contained in a fixed compact subset of $T^*\R^d$. Indeed, the maps $u_i$ are supported in the fixed compact set $[-\delta,\delta]^d$, the first component of $du_i$ is uniformly bounded, and
	\begin{equation*}
		\Vert \partial_{\bar{q}_1}u_i \Vert_\infty \leq \Vert d\bar{\eta} \Vert_\infty \sqrt{\Vert w_i \Vert_\infty} \longrightarrow 0 \quad \text{as } i \to +\infty
	\end{equation*}
	Hence, every subsequence of $\mathcal{G}(du_i)$ admits a further subsequence converging in the Hausdorff topology. Let $\mathcal{G}_\infty$ be the Hausdorff limit of such a subsequence, which we still index by $i$.

	We first prove the inclusion $\mathcal{G}_\infty \subset \mathcal{K}_{\R^d}$. Consider a sequence $(q^i,p^i=du_i(q^i)) \in \mathcal{G}(du_i)$ converging to $(q,p) \in \mathcal{G}_\infty$. Writing $q^i=(q^i_1,\bar{q}^i_1) \quad \text{and} \quad p^i=(p^i_1,\bar{p}^i_1)$, we have seen that $\lim_i \bar{p}^i_1 = \lim_i \partial_{\bar{q}_1}u_i(q^i) =0$. Hence, $\bar{p}_1=(p_2,\ldots,p_d)=0$. We assume first that $p_1 \neq 0$. We have
	\begin{equation}
		p^i_1 = v_i(q^i_1) \bar{\eta}\left(\frac{\bar{q}^i_1}{r_i}\right) \longrightarrow p_1 \neq 0
	\end{equation}
	Since the graphs of $v_i$ converge to $(\R \times \{0\}) \cup (\{0\} \times [0,1])$, we deduce that $q_1=\lim_i q^i_1=0$. Moreover, for $i$ sufficiently large, $\bar{q}^i_1$ must belong to the support of $\bar{\eta}\left(\frac{\bar{q}_1}{r_i}\right)$, which is contained in $(-\delta r_i,\delta r_i)^{d-1}$. Since $\lim_i r_i = 0$, it follows that $\bar{q}_1=\lim_i \bar{q}^i_1=0$. Finally, since $0 \leq \bar{\eta} \leq 1$, $v_i \leq 1$, and the negative part of $v_i$ converges uniformly to $0$, we obtain $0 \leq p_1 \leq 1$. Thus, $(q,p) \in \im(\sigma^0) \subset \mathcal{K}_{\R^d}$. If $p_1=0$, then $\bar{p}_1=0$ also, and hence $p=0$. Thus, $(q,p) \in 0_{\R^d} \subset \mathcal{K}_{\R^d}$. This proves $\mathcal{G}_\infty \subset \mathcal{K}_{\R^d}$.
	
	We now prove the reverse inclusion. First, let $q \neq 0$. We claim that $\lim_i du_i(q) = 0$. Indeed, if $\bar{q}_1 \neq 0$, then
	\begin{equation}
		\bar{\eta}\left(\frac{\bar{q}_1}{r_i}\right) 	= d\bar{\eta}\left(\frac{\bar{q}_1}{r_i}\right) = 0
	\end{equation}
	for $i$ sufficiently large. If $\bar{q}_1=0$ and $q_1\neq0$, then $\lim_i v_i(q_1) = 0$, while $d\bar{\eta}(0)=0$ since $\bar{\eta}$ is constant near $0$. Hence, in both cases, $\lim_i du_i(q)=0$ and therefore $(q,0)\in\mathcal{G}_\infty$ for every $q \neq 0$. By closedness of $\mathcal{G}_\infty$, we also deduce that $(0,0)$ belongs to $\mathcal{G}_\infty$. Moreover, at $q=0$, we have
	\begin{align*}
		du_i(0) &= \left( v_i(0)\bar{\eta}(0), \frac{w_i(0)}{r_i}d\bar{\eta}(0) \right) = \left( v_i(0),0,\ldots,0 \right) \longrightarrow (1,0,\ldots,0) \quad \text{as } i\to+\infty
	\end{align*}
	where we have used that $\bar{\eta}(0)=1$ and $d\bar{\eta}(0)=0$. Hence, the endpoint $(0,\ldots,0,p_1=1,0,\ldots,0)$ of $\im(\sigma^0)$ belongs to $\mathcal{G}_\infty$. Since $\mathcal{G}_\infty$ is connected, is contained in $\mathcal{K}_{\R^d}$, and contains both the zero section $0_{\R^d}$ and the endpoint of $\im(\sigma^0)$, we deduce that $\im(\sigma^0) \subset \mathcal{G}_\infty$. Therefore, $\mathcal{G}_\infty = \mathcal{K}_{\R^d}$. Since the initial subsequence was arbitrary, every subsequence of $\mathcal{G}(du_i)$ admits a further subsequence converging in the Hausdorff topology to the same set $\mathcal{K}_{\R^d}$. We conclude that the whole sequence $\mathcal{G}(du_i)$ converges to $\mathcal{K}_{\R^d}$ in the Hausdorff topology.

	Since $u_i$ is supported in $(-\delta,\delta)^d$ and $\mathcal{G}(du_i)$ converges to $\mathcal{K}_{\R^d}$ in the Hausdorff topology, for $i$ sufficiently large, the part of $\mathcal{G}(du_i)$ that differs from the zero section is contained in $U^1_l$. We can therefore glue $\mathcal{G}(du_i) \cap U^1_l$ with the zero section $0_\cL$ outside $U^1_l$ to obtain an exact Lagrangian submanifold $\cL^1_i$. A Liouville primitive is given by $h^1_i = u_i \circ \pi_{0_\cL} + k_l^1(\gamma(0))$, where $\pi_{0_\cL} : T^*\cL \to 0_\cL$ denotes the canonical projection. Then, the sequence $(\cL^1_i,h^1_i)$ brane-converges to $(\mathcal{K}^1,k^1_l)$.\\

	\textit{Conclusion.} Let us summarize the constructions above. We have obtained the following convergences:
	\begin{align*}
		\lim_n (\mathcal{K}^\infty_n,k^\infty_n) &= (\mathcal{K},k), \\
		\lim_m (\mathcal{K}^0,k^0_m) &= (\mathcal{K}^0,\tilde{k}^0_n) \quad \text{with } (\mathcal{K}^0,\tilde{k}^0_n) = (\Psi^0_n)_* (\mathcal{K}^\infty_n,k^\infty_n), \\
		\lim_l (\mathcal{K}'_l,k'_l) &= (\mathcal{K}^0,k^0_m), \\
		\lim_i (\cL^1_i,h^1_i) &= (\mathcal{K}^1,k^1_l) \quad \text{with } (\mathcal{K}^1, k^1_l) = (\Psi^1_l)_*(\mathcal{K}'_l,k'_l)
	\end{align*}

	We set
	\begin{equation}
		\cL_i = (\Psi^0_n)^{-1} \circ (\Psi^1_l)^{-1}(\cL_i^1) \quad \text{and} \quad h_i = (h_i^1 \circ \Psi^1_l - S^1_l) \circ \Psi^0_n - S^0_n,
	\end{equation}
	so that $(\cL_i,h_i) = (\Psi^1_l \circ \Psi^0_n)^*(\cL_i^1,h_i^1)$ is an exact Lagrangian brane. By Proposition \ref{BraneInvariance} and a diagonal extraction in the indices $i$, $l$, $m$, and $n$, we obtain a sequence $(\cL_i,h_i)$ that brane-converges to $(\mathcal{K},k)$. Moreover, for every open neighbourhood $U$ of $\im(\gamma)$, the branes $(\cL_i,h_i)$ coincide with $(\mathcal{K},k)$ on the complement of $U$ for all sufficiently large $i$.
\end{proof}

\begin{proof}[Proof of Theorem \ref{thm:NonUniqueness2D}]
	We show that $(\mathcal{K},k)$ belongs to $\overline{\mathcal{LB}}(T^*M,\omega)$. More precisely, we show that, for every $\varepsilon>0$, we can construct $(\mathcal{K}_\varepsilon,k_\varepsilon)$ in $\overline{\mathcal{LB}}(T^*M,\omega)$ such that $d_B \big( (\mathcal{K}_\varepsilon,k_\varepsilon),(\mathcal{K},k) \big) < 3\varepsilon$. Since $\overline{\mathcal{LB}}(T^*M,\omega)$ is closed for the brane topology, we will conclude that $(\mathcal{K},k)$ belongs to $\overline{\mathcal{LB}}(T^*M,\omega)$.

	By the Tietze extension theorem, there exists a continuous map $f : T^*M \to \R$ whose restriction to $\mathcal{K}$ is $k$. We set
	\begin{equation} \label{thm:NonUniqueness2D:dem:1}
		V := \{ x \in T^*M \; | \; \delta_{(T^*M,f)}(x,(\mathcal{K},k)) < \varepsilon \} \subset T^*M
	\end{equation}
	where $\delta_{(T^*M,f)}(x,(\mathcal{K},k))$ is defined as in \eqref{eq:BraneDistanceFormula0}, even though $T^*M$ is non-compact. The set $V$ is open and contains $\mathcal{K}$. Let $U$ denote the connected component of $V$ containing $\mathcal{K}$.

	Since $(\mathcal{C},c)$ belongs to $\overline{\mathcal{LB}}(T^*M,\omega)$, we can choose an exact Lagrangian brane $(\mathcal{L}_\varepsilon,h_\varepsilon)$ sufficiently close to $(\mathcal{C},c)$ so that
	\begin{equation} \label{thm:NonUniqueness2D:dem:2}
		d_B \big( (\mathcal{L}_\varepsilon,h_\varepsilon),(\mathcal{C},c) \big) < \varepsilon \quad \text{and} \quad \mathcal{L}_\varepsilon \subset U
	\end{equation}
	Moreover, by choosing the approximation sufficiently close, we may fix a point $y_\varepsilon \in \mathcal{L}_\varepsilon$ such that
	\begin{equation} \label{thm:NonUniqueness2D:dem:3}
		|h_\varepsilon(y_\varepsilon)-f(y_\varepsilon)|<\varepsilon
	\end{equation}

	Let $\big((x_i,z_i)\big)_i$ be a finite $\varepsilon$-dense subset of the graph $\Gamma_{(\mathcal{K},k)}$ of $k$ in $T^*M \times \R$. For every $i$, define
	\begin{equation}
		O_i = \{x \in U \; | \; d(x,x_i) + |f(x)-z_i| < \varepsilon \}
	\end{equation}
	Since $z_i=k(x_i)=f(x_i)$, each $O_i$ is a non-empty open subset of $U$ containing $x_i$. For every $i$, fix a point $y_i$ in $O_i \setminus \mathcal{L}_\varepsilon$. Since $\codim \mathcal{L}_\varepsilon = d \geq 2$, the set $U \setminus \mathcal{L}_\varepsilon$ is path-connected. We can therefore connect $y_\varepsilon$ to $y_0$ by an arc that intersects $\mathcal{L}_\varepsilon$ only at $y_\varepsilon$, and connect successively $y_i$ to $y_{i+1}$ by paths contained in $U \setminus \mathcal{L}_\varepsilon$. After a small perturbation, we may assume that the resulting curve $\gamma_\varepsilon : [0,1] \to U$ is embedded and satisfies
	\begin{equation}
		\gamma_\varepsilon(0) = y_\varepsilon, \quad \mathcal{L}_\varepsilon \cap \im(\gamma_\varepsilon) = \{y_\varepsilon\}, \quad \text{and} \quad \im(\gamma_\varepsilon) \cap O_i \neq \emptyset \quad \text{for all } i
	\end{equation}

	Set $\varsigma = h_\varepsilon(y_\varepsilon)-f(y_\varepsilon)$. By \eqref{thm:NonUniqueness2D:dem:3}, we have $|\varsigma|<\varepsilon$. We consider the bare brane $(\mathcal{K}_\varepsilon,k_\varepsilon)$ defined by
	\begin{equation}
		\mathcal{K}_\varepsilon
		:=
		\mathcal{L}_\varepsilon \cup \im(\gamma_\varepsilon)
		\quad \text{and} \quad
		k_\varepsilon =
		\begin{cases}
			h_\varepsilon & \text{on } \mathcal{L}_\varepsilon \\
			\varsigma + f & \text{on } \im(\gamma_\varepsilon)
		\end{cases}
	\end{equation}
	The two definitions agree at $y_\varepsilon$, so $k_\varepsilon$ is continuous. Proposition \ref{prop:fil} then implies that $(\mathcal{K}_\varepsilon,k_\varepsilon)$ is a bare Lagrangian brane in $\overline{\mathcal{LB}}(T^*M,\omega)$.

	We now show that $d_B \big( (\mathcal{K}_\varepsilon,k_\varepsilon),(\mathcal{K},k) \big) < 3\varepsilon$. Let $x_\varepsilon$ be in $\mathcal{K}_\varepsilon$. If $x_\varepsilon$ belongs to $\im(\gamma_\varepsilon)$, then, since $x_\varepsilon \in U \subset V$, we have
	\begin{align*}
		\delta_{(\mathcal{K}_\varepsilon,k_\varepsilon)}(x_\varepsilon,(\mathcal{K},k)) &= \inf_{x \in \mathcal{K}} \Big\{ d(x_\varepsilon,x) + |k_\varepsilon(x_\varepsilon)-k(x)| \Big\} \\
		&\leq |\varsigma| + \inf_{x \in \mathcal{K}} \Big\{ d(x_\varepsilon,x) + |f(x_\varepsilon)-k(x)| \Big\} \\
		&= |\varsigma| + \delta_{(T^*M,f)}(x_\varepsilon,(\mathcal{K},k)) < 2\varepsilon
	\end{align*}
	If $x_\varepsilon$ belongs to $\mathcal{L}_\varepsilon$, then, since $\mathcal{C} \subset \mathcal{K}$ and $k|_{\mathcal{C}}=c$, \eqref{thm:NonUniqueness2D:dem:2} gives
	\begin{align*}
		\delta_{(\mathcal{K}_\varepsilon,k_\varepsilon)}(x_\varepsilon,(\mathcal{K},k)) &= \delta_{(\mathcal{L}_\varepsilon,h_\varepsilon)}(x_\varepsilon,(\mathcal{K},k)) \\
		&\leq \delta_{(\mathcal{L}_\varepsilon,h_\varepsilon)}(x_\varepsilon,(\mathcal{C},c)) < \varepsilon
	\end{align*}

	Conversely, let $x$ be in $\mathcal{K}$. By the $\varepsilon$-density of $\big((x_i,z_i)\big)_i$ in $\Gamma_{(\mathcal{K},k)}$, there exists $i$ such that 
	\begin{equation*}
		d(x,x_i)+|k(x)-z_i|<\varepsilon
	\end{equation*}	
	Choose a point $y_i \in \im(\gamma_\varepsilon) \cap O_i$. By definition of $O_i$,
	\begin{equation*}
		d(x_i,y_i)+|z_i-f(y_i)|<\varepsilon
	\end{equation*}
	Since $k(x)=f(x)$ and $z_i=f(x_i)$, we obtain
	\begin{align*}
		d(x,y_i)+|f(x)-f(y_i)| &\leq \big[d(x,x_i)+|f(x)-f(x_i)|\big] + \big[d(x_i,y_i)+|f(x_i)-f(y_i)|\big] < 2\varepsilon
	\end{align*}
	Thus,
	\begin{align*}
		\delta_{(\mathcal{K},k)}(x,(\mathcal{K}_\varepsilon,k_\varepsilon)) &\leq d(x,y_i)+|k(x)-k_\varepsilon(y_i)| \\
		&= d(x,y_i)+|f(x)-f(y_i)-\varsigma| \\
		&\leq d(x,y_i)+|f(x)-f(y_i)|+|\varsigma| < 3\varepsilon
	\end{align*}
	
	We conclude that $d_B \big( (\mathcal{K}_\varepsilon,k_\varepsilon),(\mathcal{K},k) \big) < 3\varepsilon$.\\
	
	Let $\mathcal{C}$ be in $\mathscr{L}^\exists$, let $c : \mathcal{C} \to \R$ be such that $(\mathcal{C},c)$ belongs to $\overline{\mathcal{LB}}(T^*M,\omega)$, and let $\mathcal{K}$ be a compact connected set containing $\mathcal{C}$. By the Tietze extension theorem, there exists a continuous map $k : \mathcal{K} \to \R$ whose restriction to $\mathcal{C}$ is equal to $c$. By the preceding result, $(\mathcal{K},k) \in \overline{\mathcal{LB}}(T^*M,\omega)$ is a bare Lagrangian brane. In particular, $\mathcal{K}$ belongs to $\mathscr{L}^\exists$.\\

    Moreover, if $\mathcal{K}$ contains $\mathcal{C}$ strictly, define $\tilde{k} : \mathcal{K} \to \R$ by $\tilde{k}(x) = k(x) + d(x,\mathcal{C})$. Since $\mathcal{C} \subsetneq \mathcal{K}$, the map $x \mapsto d(x,\mathcal{C})$ is not constant, and therefore $[\tilde{k}] \neq [k]$. Moreover, both $k$ and $\tilde{k}$ restrict to $c$ on $\mathcal{C}$. Hence, by the preceding result, $(\mathcal{K},k)$ and $(\mathcal{K},\tilde{k})$ belong to $\overline{\mathcal{LB}}(T^*M,\omega)$. We conclude that $\mathcal{K}$ does not belong to $\mathscr{L}^!$.
\end{proof}

\subsubsection{A Non-Uniquely Liftable Minimal Set: Proof of Proposition \ref{prop:NonUniqueness&Minimality}}

One may ask whether every element of $\mathscr{L}^\exists$ contains an element of $\mathscr{L}^!$. However, this is false, as shown by the following example:

\begin{exmp}[A minimal element of $\mathscr{L}^\exists$ containing no element of $\mathscr{L}^!$] \label{exp:MinimalNonUnique}
	We construct an element $\mathcal{K}$ of $\mathscr{L}^\exists$ which does not belong to $\mathscr{L}^!$ and which is minimal with respect to inclusion, meaning that if $\mathcal{C} \in \mathscr{L}^\exists$ is a subset of $\mathcal{K}$, then $\mathcal{C}=\mathcal{K}$.
	
	We first work on $[0,1]$ before passing to the circle $\mathbb{T}^1$. Let $C$ be the Smith--Volterra Cantor set constructed inductively as follows: set $C_0=[0,1]$, and if $C_{n-1}$ is a disjoint union of $2^{n-1}$ intervals, remove from the center of each component an open interval of length $4^{-n}$. Then, $C:=\bigcap_n C_n$ and $\leb(C)=1/2>0$.
	
	The complement of $C$ in $[0,1]$ is a disjoint union of open intervals $I_i=(a_i,b_i)$. We define the map $f : [0,1]\setminus C \to \R$ and the map $\sigma : C \to \R$ by
	\begin{align*}
		f(q) &= \sin \left( 2\pi \frac{q-a_i}{b_i-a_i} \right) \quad \text{on } I_i=(a_i,b_i) \\
		\sigma(q) &=
		\begin{cases}
			+1 & \text{if } q \in C \cap [0,\frac{1}{2}) \\
			-1 & \text{if } q \in C \cap [\frac{1}{2},1]
		\end{cases}
	\end{align*}
	We then define $f_\pm : [0,1] \to \R$ by requiring that $f_\pm$ restricts to $f$ on $[0,1]\setminus C$ and to $\pm \sigma$ on $C$. We consider the maps $u_\pm : [0,1] \to \R$ defined by
	\begin{equation}
		u_\pm(q) = \int_0^q f_\pm(\tau) \; d\tau
	\end{equation}
	Note that $u_\pm(0)=u_\pm(1)=0$, so that $u_\pm$ descend to maps on the circle $\mathbb{T}^1$. Moreover, we define $\mathcal{K}$ as the closure of the graph $\mathcal{G}(f)$ of $f$. In particular,
	\begin{equation}
		\mathcal{K} = \mathcal{G}(f) \cup \big(C\times[-1,1]\big) \subset T^*\mathbb{T}^1
	\end{equation}
	
	Set $k_\pm = u_\pm \circ \pi_{\mathbb{T}^1}|_{\mathcal{K}} : \mathcal{K} \to \R$, where $\pi_{\mathbb{T}^1} : T^*\mathbb{T}^1 \to \mathbb{T}^1$ is the canonical projection.	We claim the following.
	\begin{prop}
		The set $\mathcal{K}$ belongs to $\mathscr{L}^\exists$, is minimal with respect to the inclusion in $\mathscr{L}^\exists$, and the pairs $(\mathcal{K},[k_\pm])$ define two distinct elements of $\overline{\mathfrak{LB}}(T^*\T^1,\omega)$. 
	\end{prop}

\begin{proof}
	We construct a sequence of smooth maps $f_{\pm,n} : \mathbb{T}^1 \to \R$ such that $\mathcal{G}(f_{\pm,n})$ converges to $\mathcal{K}$ in the Hausdorff topology and $\lim_n f_{\pm,n} = f_\pm$ in $L^1$. This can be achieved by choosing smooth approximations that converge to $f_\pm$ in $L^1$ on shrinking neighbourhoods of $C$, while converging uniformly to $f$ on their complements. Since $\int_{\mathbb{T}^1} f_\pm =0$, after subtracting a constant converging to $0$, we may further assume that $\int_{\mathbb{T}^1} f_{\pm,n}=0$. Set
	\begin{equation}
		u_{\pm,n}(q)=\int_0^q f_{\pm,n}(\tau)\;d\tau
	\end{equation}
	Then $u_{\pm,n}$ is smooth, $u_{\pm,n}(0) = u_{\pm,n}(1)$, $\mathcal{G}(du_{\pm,n})=\mathcal{G}(f_{\pm,n})$, and
	\begin{equation*}
		\Vert u_{\pm,n} - u_\pm \Vert_\infty \leq \Vert f_{\pm,n} - f_\pm \Vert_{L^1} \longrightarrow 0 \quad \text{as } n \to +\infty
	\end{equation*}
	Hence, $(\mathcal{G}(du_{\pm,n}),u_{\pm,n}\circ\pi_{\mathbb{T}^1})$ brane-converges to $(\mathcal{K},k_\pm)$. Thus, $(\mathcal{K},[k_\pm])$ are in $\overline{\mathfrak{LB}}(T^*\T^1,\omega)$.

	Moreover,
	\begin{equation}
		(u_+-u_-)'=2\sigma
		\quad \text{a.e. on } C
	\end{equation}
	and $\leb(C)>0$. Hence, $u_+-u_-$ is not constant and $[k_+]\neq[k_-]$.

	Finally, let $\mathcal{C}\subset\mathcal{K}$ belong to $\mathscr{L}^\exists$. Its projection onto $\mathbb{T}^1$ is surjective. For every $q$ in $\mathbb{T}^1 \setminus C$, the fibre of $\mathcal{K}$ over $q$ consists of the single point $(q,f(q))$, yielding $(q,f(q))\in\mathcal{C}$. Hence, $\mathcal{G}(f)$ is a subset of $\mathcal{C}$. Since $\mathcal{C}$ is closed and $\mathcal{K}=\overline{\mathcal{G}(f)}$, we obtain $\mathcal{C}=\mathcal{K}$. Thus, $\mathcal{K}$ is minimal with respect to inclusion in $\mathscr{L}^\exists$.
\end{proof}

	\begin{rem}
		\begin{enumerate}
			\item More generally, the elements $\Lambda_\pm$ of the Humilière completion defined in Proposition \ref{BraneGammaProp} and associated to $(\mathcal{K},[k_\pm])$ are distinct. Indeed, any continuous function $u : \mathbb{T}^1 \to \R$ defines an element $\Lambda_u$ of the Humilière completion, and
			\begin{equation}
				\gamma(\Lambda_{u_1},\Lambda_{u_2})=\osc(u_2-u_1)
			\end{equation}
			see \cite[Section 4]{MR4683312}.

			In our case, the approximating exact graphs converge in the Hausdorff topology to $\mathcal{K}$, hence $\gamma\text{-supp}(\Lambda_\pm)\subset \mathcal{K}$ by \cite[Proposition 6.20]{viterbo2026supportshumilierecompletiongammacoisotropic}. Conversely, every compact $\gamma$-support meets every cotangent fibre \cite[Proposition 6.11]{viterbo2026supportshumilierecompletiongammacoisotropic}. Since, for every $q$ in $\mathbb{T}^1 \setminus C$, the fibre of $\mathcal{K}$ over $q$ is the singleton $\{(q,f(q))\}$, we obtain $\mathcal{G}(f)\subseteq \gamma\text{-}\supp(\Lambda_\pm)$. By closedness and $\mathcal{K}=\overline{\mathcal{G}(f)}$, it follows that
			\begin{equation*}
				\gamma\text{-supp}(\Lambda_+) = \gamma\text{-supp}(\Lambda_-) = \mathcal{K}
			\end{equation*}
			Moreover,
			\begin{align*}
				\gamma(\Lambda_+, \Lambda_-) &=\osc(u_+-u_-) = \int_{C \cap \left[ 0,\frac{1}{2} \right] } 2\sigma = 2 \leb \left( C \cap \left[ 0,\frac{1}{2} \right] \right) = \frac{1}{2}
			\end{align*}
			Hence $\Lambda_+\neq\Lambda_-$ although they have the same $\gamma$-support. 
	
			The same singleton-fibre argument shows that $\mathcal{K}$ is minimal for inclusion among $\gamma$-supports. Thus, an inclusion-minimal $\gamma$-support does not necessarily determine a unique element of the Humili\`ere completion, giving a negative answer to the corresponding uniqueness question in \cite[Question 7.7]{viterbo2026supportshumilierecompletiongammacoisotropic}. 

            \item The example extends directly to higher dimensions. Assume $d \geq 2$ and work on $T^*\T^d$. We set
			\begin{equation*}
				\mathcal K_d:=\mathcal K\times 0_{\mathbb T^{d-1}} \quad \text{and} \quad k_{\pm,d}:=k_\pm\circ \text{pr}_1
			\end{equation*}
			where $\text{pr}_1$ denotes the projection onto the first factor $(q_1,p_1) \in T^*\T^1$, and define $v_{\pm,n}(q_1,\ldots,q_d):=u_{\pm,n}(q_1)$. Then, the exact graphs $\mathcal{G}(dv_{\pm,n})$ brane-converge to $(\mathcal K_d,k_{\pm,d})$. Hence, $\mathcal K_d\in\mathscr L^\exists$, $[k_{+,d}]\neq[k_{-,d}]$ and $\mathcal K_d$ is minimal with respect to inclusion in $\mathscr L^\exists$.
		\end{enumerate}
	\end{rem}
\end{exmp}

\subsubsection{Regular Case: Proof of Theorem \ref{thm:BirkhoffStratifiedLag}}

\begin{proof}[Proof of Theorem \ref{thm:BirkhoffStratifiedLag}]
	Let $(\cL_t,h_t)$ be the variational brane associated to $(\cL,h)$, with $h_0=h$. We prove that $[h_1]=[h_0]$, which allows us to apply Theorem \ref{MainThmInvariant} and conclude that $\cL$ is a Lipschitz graph over the base $M$.	

	We set $\phi=\phi_H^1$ and define $\mathcal{A} : T^*M \to \R$ by
	\begin{equation} \label{eq:ActionTimeOne}
		\mathcal{A}(x) = \mathcal{A}_H\left(x_\tau|_{[0,1]}\right) = \int_0^1 \Big( \lambda_{x_\tau}(\dot{x}_\tau) - H(\tau,x_\tau) \Big) \; d\tau, \quad \text{where } x_\tau=\phi_H^\tau(x)
	\end{equation}
    We have seen in \eqref{eq:HamiltonianExactForm} that $\phi^*\lambda-\lambda=d\mathcal{A}$. We define $k : \cL \to \R$ by
	
	\begin{equation} \label{thm:BirkhoffStratifiedLag:dem1}
		k = (h_1-h_0)\circ\phi = h+\mathcal{A}-h\circ\phi
	\end{equation}

	Let us show that $k$ is constant on $\cL$. Replacing $U$ by the union of all the neighbourhoods $\cL_x$, we may set 
	\begin{equation} \label{eq:Stratas}
		U=\bigcup_{x\in U}\cL_x \quad  V:=U\cap\phi^{-1}(U)
	\end{equation}

	Let $y\in V$, and choose $x,x'\in U$ such that $y\in\cL_x$ and $\phi(y)\in\cL_{x'}$. After shrinking the neighbourhoods if necessary, we may assume that $\phi(\cL_x)\subset\cL_{x'}$. On $\cL_x$, we have
	\begin{align*}
		d(k|_{\cL_x}) &= d(h|_{\cL_x}) +d(\mathcal{A}|_{\cL_x}) -d(h\circ\phi|_{\cL_x}) \\
		&= \lambda|_{\cL_x} + (\phi^*\lambda-\lambda)|_{\cL_x} - \phi^*\big(d(h|_{\cL_{x'}})\big) \\
		&= \lambda|_{\cL_x} +(\phi^*\lambda-\lambda)|_{\cL_x} - \phi^*\lambda|_{\cL_x} \\
		&=0
	\end{align*}
	Hence, $k$ is locally constant near $y$, and therefore on $V$.

	Since $\cL\setminus U$ is countable and $\phi(\cL)=\cL$, we have
	\begin{equation*}
		\cL\setminus V = (\cL\setminus U) \cup \phi^{-1}(\cL\setminus U),
	\end{equation*}
	which is countable. Hence, $k(\cL\setminus V)$ is countable. Moreover, $V$ is second countable and $k$ is locally constant on $V$, so $k(V)$ is countable. Therefore,
	\begin{equation*}
		k(\cL) = k(V)\cup k(\cL\setminus V)
	\end{equation*}
	is countable. On the other hand, $k$ is continuous and we know from Proposition \ref{prop:Connectedness} that $\cL$ is connected, so $k(\cL)$ is connected. Since every connected countable subset of $\R$ is a singleton, we conclude that $k$ is constant on $\cL$.

	By \eqref{thm:BirkhoffStratifiedLag:dem1} and the equality $\phi(\cL)=\cL$, it follows that $h_1-h_0$ is constant on $\cL$. Hence, $[h_0]=[h_1]$, and Theorem \ref{MainThmInvariant} concludes the proof.
\end{proof}

\subsubsection{Continuous case: Proof of Theorem \ref{thm:BirkhoffStratifiedLagC0}}

Recall from Definition \ref{def:ExactSympeo} the definition of strongly exact symplectic homeomorphism branes. This definition allows for the transport of primitives $h$ by the formula $(h+S)\circ\Phi^{-1}$. Additive constants can be absorbed in $S$.

\begin{prop}\label{prop:strongExactSubgroup}
    Strongly exact symplectic homeomorphisms of an exact symplectic manifold $(U,\lambda)$ form a subgroup of weakly exact symplectic homeomorphisms. Moreover, at the level of branes, the composition law is given by
    \[(\varphi,S)\circ(\psi,S')=(\varphi\circ\psi,S\circ\psi+S'),\]
    and the inverse by
    \[(\varphi,S)^{-1}=(\varphi^{-1},-S\circ\varphi^{-1}).\]
\end{prop}

\begin{proof}
    The identity is clearly a strongly exact symplectic homeomorphism (take the constant sequence equal to the identity, with primitive $0$). Now, let $\varphi$ and $\psi$ be strongly exact symplectic homeomorphisms of $U$, with approximating sequences $(\varphi_n)$ and $(\psi_n)$ of $C^1$ exact symplectic diffeomorphisms converging to $\varphi$ and $\psi$ respectively, and satisfying $\varphi_n^*\lambda-\lambda=dS_n$, $\psi_n^*\lambda-\lambda=dS'_n$, with $(S_n)$ and $(S'_n)$ converging uniformly on compact sets to continuous functions $S$ and $S'$.

    Then, for $n\geq 0$, the diffeomorphisms $\varphi_n\circ\psi_n$ and $\varphi_n^{-1}$ are exact, and satisfy $(\varphi_n\circ\psi_n)^*\lambda-\lambda = d(S_n\circ\psi_n+S'_n)$ and $(\varphi_n^{-1})^*\lambda-\lambda = d(-S_n\circ\varphi_n^{-1})$. Moreover, $\varphi_n\circ\psi_n$ and $\varphi_n^{-1}$ converge uniformly on compact sets to $\varphi\circ\psi$ and $\varphi^{-1}$ respectively. $S_n\circ\psi_n+S'_n$ and $-S_n\circ\varphi_n^{-1}$ also converge uniformly on compact sets to the continuous functions $S\circ\psi+S'$ and $-S\circ\varphi^{-1}$ respectively. Therefore, $\varphi\circ\psi$ and $\varphi^{-1}$ are strongly exact symplectic homeomorphisms, and the claim about their primitives holds.
\end{proof}

We use the following result by Humilière--Leclercq--Seyfaddini.

\begin{theo}[Humilière--Leclercq--Seyfaddini \cite{HLS}] \label{thm:HLS}
	Assume that a symplectic homeomorphism $\Phi : U \to V$ between open subsets of symplectic manifolds sends a smooth coisotropic submanifold $C$ onto a smooth submanifold $C'$. Then, $C'$ is coisotropic and $\Phi$ sends the characteristic foliation of $C$ to that of $C'$.
\end{theo}

This is a local statement, and neither $C$ nor $C'$ is required to be closed.\\

Recall that the characteristic leaves of a coisotropic manifold $C$ are the integral leaves of $\ker(\omega|_{TC})$.

\begin{lem} \label{lem:SympeoLocalConstance}
	Let $(\Phi : U_1 \to U_2,S)$ be a strongly exact symplectic homeomorphism brane between open subsets $U_i\subset T^*M_i$ for some manifolds $M_i$. If $\Phi(U_1 \cap 0_{M_1}) = U_2 \cap 0_{M_2}$, then $S|_{U_1 \cap 0_{M_1}}$ is locally constant.
\end{lem}

\begin{proof}
	On $\mathscr{U}_i=U_i \times \R_z \times \R_r$, let $\Omega_i = d \big( e^r (dz- \lambda_{M_i}) \big)$ be a symplectic form on $\mathscr{U}_i$ with Liouville form $\Lambda_i = e^r(\lambda_{M_i} - dz)$. We define a homeomorphism $\Psi : \mathscr{U}_1 \to \mathscr{U}_2$ by
	\begin{equation*}
		\Psi (x,z,r) = \big( \Phi(x),z + S(x), r \big)
	\end{equation*}

	Analogously, for the $C^1$ approximations $(\Phi_n,S_n)$ of $(\Phi,S)$, consider the corresponding lifts $\Psi_n$. We have
	\begin{align*}
		\Psi_n^*(\lambda_{M_2}-dz) &= \Phi_n^*\lambda_{M_2} - d(z+S_n) \\
		&= \lambda_{M_1} + dS_n - dz - dS_n\\
		&= \lambda_{M_1}-dz
	\end{align*}
	Since the coordinate $r$ is preserved, we get $\Psi_n^*\Lambda_2 = \Lambda_1$, and the maps $\Psi_n$ are exact symplectic. Moreover, they converge locally uniformly to $\Psi$. Hence, $\Psi$ is a symplectic homeomorphism.

	Now consider
	\begin{equation*}
		C_i=(U_i\cap 0_{M_i})\times\R_z\times\R_r
	\end{equation*}
	We have
	\begin{equation*}
		\Omega_i|_{C_i}= e^r dr \wedge dz, \quad \ker(\Omega_i|_{TC_i}) = TM_i \times \{0\} \times \{0\}
	\end{equation*}
	Hence, $C_i$ is coisotropic, and its local characteristic leaves are of the form
	\begin{equation}\label{eq:CharLeaves}
		\{(q,0,z_0,r_0) \; | \; q\in U_i \cap 0_{M_i}\}
	\end{equation}
	for fixed $z_0$ and $r_0$. By assumption, $\Psi(C_1)=C_2$. Hence, Theorem \ref{thm:HLS} applies. Writing $\Phi(q,0)=(f(q),0)$, we have
	\begin{equation*}
		\Psi(q,0,z_0,r_0) = \big( f(q),0, z_0+S(q,0),r_0 \big)
	\end{equation*}
	Theorem \ref{thm:HLS} asserts that the image of a local characteristic leaf of $C_1$ is contained in a characteristic leaf of $C_2$. Thus, we deduce from \eqref{eq:CharLeaves}, that the $z$-coordinate of $\Psi(q,0,z_0,r_0)$ is constant on such a leaf. Therefore, $S(q,0)$ is locally constant on $U_1 \cap 0_{M_1}$.
\end{proof}

\begin{proof}[Proof of Theorem \ref{thm:BirkhoffStratifiedLagC0}]
	We proceed as in the proof of Theorem \ref{thm:BirkhoffStratifiedLag}. Let $(\cL_t,h_t)$ be the variational brane associated to $(\cL,h)$, with $h_0=h$. We prove that $[h_1]=[h_0]$, which allows us to apply Theorem \ref{MainThmInvariant} and conclude that $\cL$ is a Lipschitz graph over the base $M$.

	We set $\phi=\phi_H^1$ and $\mathcal{A} : T^*M \to \R$ defined in \eqref{eq:ActionTimeOne} so that $\phi^*\lambda-\lambda=d\mathcal{A}$. Hence, $(\phi,\mathcal{A})$ is a strongly exact symplectomorphism. We define $k : \cL \to \R$ by
	\begin{equation} \label{thm:StratifiedLagC0:dem1}
		k = (h_1-h_0) \circ\phi = h+ \mathcal{A}- h \circ\phi
	\end{equation}

	We show that $k$ is constant on $\cL$. Replacing $U$ by the union of all points of $\cL$ admitting a local model as in the statement, we may consider the set $V$ and assume that $U$ verifies \eqref{eq:Stratas}. 
	
	Let $y\in V$, and choose two local models $(\Phi_x,S_x)$ and $(\Phi_{x'},S_{x'})$ such that $y\in\cL_x$ and $\phi(y)\in\cL_{x'}$. After shrinking the domains if necessary, set
	\begin{equation} \label{thm:StratifiedLagC0:dem2}
		\Phi = \Phi_{x'}^{-1}\circ\phi\circ\Phi_x \quad \text{and} \quad S = S_x + \mathcal{A}\circ\Phi_x - S_{x'}\circ\Phi
	\end{equation}

    By Proposition \ref{prop:strongExactSubgroup}, $(\Phi,S)$ is a strongly exact symplectic homeomorphism brane. Since $\phi(\cL)=\cL$ and $\cL_x$ and $\cL_{x'}$ are precisely the intersections of $\cL$ with the corresponding ambient neighbourhoods, $\Phi$ sends the local zero section of $0_{N_x}$ onto the local zero section of $0_{N_{x'}}$. Hence, Lemma \ref{lem:SympeoLocalConstance} applies and $S$ is locally constant on the local zero section.

	Let $z$ be in a sufficiently small neighbourhood of $y$ in $\cL_x$, and write $z_x=\Phi_x^{-1}(z)$. Then,
	\begin{align*}
		k(z) &= h(z) + \mathcal{A}(z) -h\circ\phi(z) \\
		&= h\circ\Phi_x(z_x) +\mathcal{A}\circ\Phi_x(z_x) - h\circ\phi\circ\Phi_x(z_x) \\
		&= S_x(z_x)+ \mathcal{A}\circ\Phi_x(z_x) -h\circ\Phi_{x'}\circ\Phi(z_x) \\
		&= S_x(z_x) + \mathcal{A}\circ\Phi_x(z_x) - S_{x'}\circ\Phi(z_x) \\
		&= S(z_x)
	\end{align*}
	where we used Property \eqref{thm:BirkhoffStratifiedLagC0:eq:h} of the statement. Hence, $k$ is locally constant on $V$. The rest of the proof is identical to that of Theorem \ref{thm:BirkhoffStratifiedLag}.
\end{proof}

\appendix

\section{Equivalence between Brane Convergence and Reduced Complexity Convergence} \label{SectionRC}

We show here that, for Lagrangian branes Hamiltonian isotopic to the zero section, brane convergence is equivalent to the “Reduced Complexity Convergence” introduced in \cite{Birkhoff}.

\begin{defi} \label{ReducedComplexity}
	Let $(\mathcal{L}_n)_{n \geq 0}$ and $\mathcal{L}$ be exact Lagrangian submanifolds of $T^*M$ Hamiltonian-isotopic to the zero section $0_M$. We say that the sequence $(\mathcal{L}_n)_{n \geq 0}$ \textit{converges with reduced complexity} if the two conditions below are satisfied:
	\begin{enumerate}[label = \roman*.]
		\item $\lim\limits_n d_H(\mathcal{L}_n , \mathcal{L})=0$.
		\item For a Hamiltonian diffeomorphism $\varphi$ such that $\mathcal{L} = \varphi(0_M)$,  if $l_n$ is a Liouville primitive on the Lagrangian submanifold $\varphi^{-1}(\mathcal{L}_n)$, then $\lim\limits_n\osc(l_n)=0$.
	\end{enumerate}
\end{defi}

\begin{prop}
	Let $(\mathcal{L}_n = \varphi_n(0_M),h_n)$ and $(\mathcal{L} = \varphi(0_M),h)$ be Lagrangian branes, with $\varphi_n$ and $\varphi$ in $\ham(T^*M,\omega)$. Then, $\mathcal{L}_n$ converges to $\mathcal{L}$ with reduced complexity if and only if $(\mathcal{L}_n = \varphi_n(0_M),[h_n])$ brane-converges to $(\mathcal{L} = \varphi(0_M),[h])$ in $\mathfrak{B}_c(T^*M)$.
\end{prop}

\begin{proof}
	We start by showing the equivalence in the case where $\mathcal{L} = 0_M$. We assume brane convergence. Then we have for any constant $h \in \mathbb{R}$, and any point $x_n \in \mathcal{L}_n$,
	\begin{align*}
		|h_n(x_n) - h| = \inf_{x \in 0_M} |h_n(x_n) - h(x)| \leq \delta_{(\mathcal{L}_n,h_n)}(x_n,(0_M,h)) \leq d_B \big( (\mathcal{L}_n,h_n), (0_M,h) \big)
	\end{align*}
	Taking successively the supremum over $x_n \in \mathcal{L}_n$ and the infimum over $h \in \mathbb{R}$, we obtain
	\begin{align*}
		\inf_{h \in \mathbb{R}} \Vert h_n - h \Vert_\infty \leq d_{\mathfrak{B}} \big( (\mathcal{L}_n,[h_n]), (0_M,[0]) \big)
	\end{align*}
	Moreover, the infimum is realized at $h = \frac{1}{2}(\max h_n + \min h_n)$ which satisfies $\osc(h_n) = 2\Vert h_n - h \Vert_\infty$. Thus, we get
	\begin{align*}
		\osc(h_n) \leq 2d_{\mathfrak{B}} \big( (\mathcal{L}_n,[h_n]), (0_M,[0]) \big) \longrightarrow 0 \quad \text{as } n \to +\infty
	\end{align*}
	
	Now we assume reduced complexity convergence. There exists a sequence of constants $c_n \in \mathbb{R}$ such that
	\begin{align*}
		\Vert h_n - c_n \Vert_\infty = \frac{1}{2} \osc( h_n ) \longrightarrow 0 \quad \text{as } n \to +\infty
	\end{align*}
    
    Together with the Hausdorff convergence of $\cL_n$ to $0_M$, this implies that the topological $\limsup$ and $\liminf$ of the graphs $\Gamma_{(\mathcal{L}_n,h_n-c_n)}$ satisfy
    \begin{align*}
    \limsup_{n\to\infty} \Gamma_{(\mathcal{L}_n,h_n-c_n)}&\coloneq \bigcap_{n\geq 0}\overline{\bigcup_{k\geq n}\Gamma_{(\cL_k,h_k-c_k)}}\\
    &\subset (\limsup_{n\to\infty}\cL_n)\times \{0\}\\
    &\subset 0_M\times\{0\}
    \end{align*}
    and
    \begin{align*}\liminf_{n\to\infty} \Gamma_{(\mathcal{L}_n,h_n-c_n)}&\coloneq \{(x,z)\in J^1M\mid\exists(x_n,z_n)_{n\geq0},(x_n,z_n)\in \Gamma_{(\mathcal{L}_n,h_n-c_n)},\lim_n (x_n,z_n)=(x,z)\}\\
    &\supset (\liminf_{n\to\infty}\cL_n)\times\{0\}\\
    &\supset 0_M\times\{0\}.
    \end{align*}

    Therefore, the topological $\liminf$ and $\limsup$ agree and coincide with $0_M\times\{0\}$. Moreover, all the graphs are eventually contained in a common compact set, and therefore the equality of the $\liminf$ and $\limsup$ implies that $\Gamma_{(\mathcal{L}_n,h_n-c_n)}$ converges to $0_M\times\{0\}$ in the Hausdorff topology. Hence, $(\cL_n,[h_n])$ brane-converges to $(0_M,[0])$.\\

	The general case follows from Proposition \ref{BraneInvariance} which states that $(\mathcal{L}_n,[h_n])$ brane-converges to $(\mathcal{L} = \varphi(0_M),[h])$ if and only if $(\varphi^{-1}(\mathcal{L}_n),[l_n])$ brane-converges to $(0_M,[0])$, where $l_n$ are Liouville primitives on $\varphi^{-1}(\mathcal{L}_n)$.
\end{proof}

\bibliographystyle{alpha}
\addcontentsline{toc}{section}{References}
\bibliography{Biblio}

@article {MR3674224,
    AUTHOR = {Arnaud, Marie-Claude and Venturelli, Andrea},
     TITLE = {A multidimensional {B}irkhoff theorem for time-dependent
              {T}onelli {H}amiltonians},
   JOURNAL = {Calc. Var. Partial Differential Equations},
  FJOURNAL = {Calculus of Variations and Partial Differential Equations},
    VOLUME = {56},
      YEAR = {2017},
    NUMBER = {4},
     PAGES = {Paper No. 122, 27},
      ISSN = {0944-2669,1432-0835},
   MRCLASS = {37J50 (53D12 70H20)},
  MRNUMBER = {3674224},
MRREVIEWER = {Pawe\l \ Urbanski},
       DOI = {10.1007/s00526-017-1210-0},
       URL = {https://doi.org/10.1007/s00526-017-1210-0},
}

@article {MR3860396,
    AUTHOR = {Amorim, Lino and Oh, Yong-Geun and dos Santos, Joana Oliveira},
     TITLE = {Exact {L}agrangian submanifolds, {L}agrangian spectral
              invariants and {A}ubry-{M}ather theory},
   JOURNAL = {Math. Proc. Cambridge Philos. Soc.},
  FJOURNAL = {Mathematical Proceedings of the Cambridge Philosophical
              Society},
    VOLUME = {165},
      YEAR = {2018},
    NUMBER = {3},
     PAGES = {411--434},
      ISSN = {0305-0041,1469-8064},
   MRCLASS = {53D12 (37J05 53D37)},
  MRNUMBER = {3860396},
MRREVIEWER = {M\'{a}rio\ Jorge Dias Carneiro},
       DOI = {10.1017/S0305004117000561},
       URL = {https://doi.org/10.1017/S0305004117000561},
}

@PHDTHESIS{humil2008,
url = "http://www.theses.fr/2008EPXX0005",
title = "Continuité en topologie symplectique",
author = "Humilière, Vincent",
year = "2008",
pages = "1 vol. ( 154 p.)",
note = "Thèse de doctorat dirigée par Viterbo, Claude",
url = "http://www.theses.fr/2008EPXX0005/document",
school = "{\'E}cole polytechnique",
}

@article {MR1157321,
    AUTHOR = {Viterbo, Claude},
     TITLE = {Symplectic topology as the geometry of generating functions},
   JOURNAL = {Math. Ann.},
  FJOURNAL = {Mathematische Annalen},
    VOLUME = {292},
      YEAR = {1992},
    NUMBER = {4},
     PAGES = {685--710},
      ISSN = {0025-5831,1432-1807},
   MRCLASS = {58F05 (53C15 53C23 57R50 58E05)},
  MRNUMBER = {1157321},
MRREVIEWER = {Nikolai\ K.\ Smolentsev},
       DOI = {10.1007/BF01444643},
       URL = {https://doi.org/10.1007/BF01444643},
}

@article {MR2025275,
    AUTHOR = {Paternain, Gabriel P. and Polterovich, Leonid and Siburg, Karl
              Friedrich},
     TITLE = {Boundary rigidity for {L}agrangian submanifolds, non-removable
              intersections, and {A}ubry-{M}ather theory},
      NOTE = {Dedicated to Vladimir I. Arnold on the occasion of his 65th
              birthday},
   JOURNAL = {Mosc. Math. J.},
  FJOURNAL = {Moscow Mathematical Journal},
    VOLUME = {3},
      YEAR = {2003},
    NUMBER = {2},
     PAGES = {593--619, 745},
      ISSN = {1609-3321,1609-4514},
   MRCLASS = {37J50 (37J05 53D12 57R17 57R30)},
  MRNUMBER = {2025275},
MRREVIEWER = {Renato\ Iturriaga},
       DOI = {10.17323/1609-4514-2003-3-2-593-619},
       URL = {https://doi.org/10.17323/1609-4514-2003-3-2-593-619},
}

@incollection {MR1604386,
    AUTHOR = {Viterbo, C.},
     TITLE = {Solutions of {H}amilton-{J}acobi equations and symplectic
              geometry. {A}ddendum to: {\it {S}\'{e}minaire sur les
              \'{E}quations aux {D}\'{e}riv\'{e}es {P}artielles. 1994--1995}
              [\'{E}cole {P}olytech., {P}alaiseau, 1995; {MR}1362548
              (96g:35001)]},
 BOOKTITLE = {S\'{e}minaire sur les \'{E}quations aux {D}\'{e}riv\'{e}es
              {P}artielles, 1995--1996},
    SERIES = {S\'{e}min. \'{E}qu. D\'{e}riv. Partielles},
     PAGES = {8},
 PUBLISHER = {\'{E}cole Polytech., Palaiseau},
      YEAR = {1996},
      ISBN = {2-7302-0366-8},
   MRCLASS = {35F25 (58F05)},
  MRNUMBER = {1604386},
}

@unknown{unknown,
author = {Ottolenghi, Alberto},
year = {2015},
month = {11},
pages = {},
title = {Solutions généralisées pour l'équation de {H}amilton-{J}acobi dans le cas d'évolution - {P}ure {M}athematics {P}ostgraduate {D}iploma {D}issertation 1992},
doi = {10.13140/RG.2.1.1322.4401}
}

@article {MR3957150,
    AUTHOR = {Roos, Valentine},
     TITLE = {Variational and viscosity operators for the evolutionary
              {H}amilton-{J}acobi equation},
   JOURNAL = {Commun. Contemp. Math.},
  FJOURNAL = {Communications in Contemporary Mathematics},
    VOLUME = {21},
      YEAR = {2019},
    NUMBER = {4},
     PAGES = {1850018, 76},
      ISSN = {0219-1997,1793-6683},
   MRCLASS = {49L25 (35D40 35F21 35F25)},
  MRNUMBER = {3957150},
MRREVIEWER = {Shuang\ Liu},
       DOI = {10.1142/S0219199718500189},
       URL = {https://doi.org/10.1142/S0219199718500189},
}

@article {MR3151090,
    AUTHOR = {Wei, Qiaoling},
     TITLE = {Viscosity solution of the {H}amilton-{J}acobi equation by a
              limiting minimax method},
   JOURNAL = {Nonlinearity},
  FJOURNAL = {Nonlinearity},
    VOLUME = {27},
      YEAR = {2014},
    NUMBER = {1},
     PAGES = {17--41},
      ISSN = {0951-7715,1361-6544},
   MRCLASS = {35F21 (35D40 35F25 37J05 49L25)},
  MRNUMBER = {3151090},
MRREVIEWER = {Gawtum\ Namah},
       DOI = {10.1088/0951-7715/27/1/17},
       URL = {https://doi.org/10.1088/0951-7715/27/1/17},
}

@PHDTHESIS{jouko1994,
url = "
		http://www.theses.fr/1994PA077046",
title = "Singularités de minimax et solutions faibles d'équations aux derivées partielles",
author = "Joukovskaïa, Tatiana",
year = "1994",
pages = "1 vol. (50 p.)",
note = "Thèse de doctorat dirigée par Chaperon, Marc",
school = "Paris 7",
}

@article {MR2393423,
    AUTHOR = {Bernard, Patrick},
     TITLE = {The dynamics of pseudographs in convex {H}amiltonian systems},
   JOURNAL = {J. Amer. Math. Soc.},
  FJOURNAL = {Journal of the American Mathematical Society},
    VOLUME = {21},
      YEAR = {2008},
    NUMBER = {3},
     PAGES = {615--669},
      ISSN = {0894-0347,1088-6834},
   MRCLASS = {37J40 (37J50)},
  MRNUMBER = {2393423},
MRREVIEWER = {Karl\ Friedrich\ Siburg},
       DOI = {10.1090/S0894-0347-08-00591-2},
       URL = {https://doi.org/10.1090/S0894-0347-08-00591-2},
}

@book{fathi2008weak,
  title={The Weak KAM Theorem in Lagrangian Dynamics },
  author={Fathi, A.},
  isbn={9780521822282},
  series={(Book in Preparation). Cambridge Studies in Advanced Mathematics},
  year={2008},
  publisher={Cambridge University Press}
}

@article {MR1555175,
    AUTHOR = {Birkhoff, George D.},
     TITLE = {Surface transformations and their dynamical applications},
   JOURNAL = {Acta Math.},
  FJOURNAL = {Acta Mathematica},
    VOLUME = {43},
      YEAR = {1922},
    NUMBER = {1},
     PAGES = {1--119},
      ISSN = {0001-5962,1871-2509},
   MRCLASS = {99-04},
  MRNUMBER = {1555175},
       DOI = {10.1007/BF02401754},
       URL = {https://doi.org/10.1007/BF02401754},
}

@article {MR1157317,
    AUTHOR = {Bialy, Misha and Polterovich, Leonid},
     TITLE = {Hamiltonian systems, {L}agrangian tori and {B}irkhoff's
              theorem},
   JOURNAL = {Math. Ann.},
  FJOURNAL = {Mathematische Annalen},
    VOLUME = {292},
      YEAR = {1992},
    NUMBER = {4},
     PAGES = {619--627},
      ISSN = {0025-5831,1432-1807},
   MRCLASS = {58F05 (58F27)},
  MRNUMBER = {1157317},
MRREVIEWER = {Jaroslav\ Stark},
       DOI = {10.1007/BF01444639},
       URL = {https://doi.org/10.1007/BF01444639},
}

@article {MR1001842,
    AUTHOR = {Bialy, M. L. and Polterovich, L. V.},
     TITLE = {Lagrangian singularities of invariant tori of {H}amiltonian
              systems with two degrees of freedom},
   JOURNAL = {Invent. Math.},
  FJOURNAL = {Inventiones Mathematicae},
    VOLUME = {97},
      YEAR = {1989},
    NUMBER = {2},
     PAGES = {291--303},
      ISSN = {0020-9910,1432-1297},
   MRCLASS = {58F05 (58F27 70H05)},
  MRNUMBER = {1001842},
MRREVIEWER = {Jair\ Koiller},
       DOI = {10.1007/BF01389043},
       URL = {https://doi.org/10.1007/BF01389043},
}

@article {MR1159829,
    AUTHOR = {Bialy, M. and Polterovich, L.},
     TITLE = {Hamiltonian diffeomorphisms and {L}agrangian distributions},
   JOURNAL = {Geom. Funct. Anal.},
  FJOURNAL = {Geometric and Functional Analysis},
    VOLUME = {2},
      YEAR = {1992},
    NUMBER = {2},
     PAGES = {173--210},
      ISSN = {1016-443X,1420-8970},
   MRCLASS = {58F05 (53C23 57R50)},
  MRNUMBER = {1159829},
MRREVIEWER = {Nikolai\ K.\ Smolentsev},
       DOI = {10.1007/BF01896972},
       URL = {https://doi.org/10.1007/BF01896972},
}

@article {MR1027911,
    AUTHOR = {Bialy, M. L.},
     TITLE = {Aubry-{M}ather sets and {B}irkhoff's theorem for geodesic
              flows on the two-dimensional torus},
   JOURNAL = {Comm. Math. Phys.},
  FJOURNAL = {Communications in Mathematical Physics},
    VOLUME = {126},
      YEAR = {1989},
    NUMBER = {1},
     PAGES = {13--24},
      ISSN = {0010-3616,1432-0916},
   MRCLASS = {58F17 (53C22 58F05)},
  MRNUMBER = {1027911},
MRREVIEWER = {Victor\ Bangert},
       URL = {http://projecteuclid.org/euclid.cmp/1104179721},
}

@article {MR1067380,
    AUTHOR = {Herman, Michael-R.},
     TITLE = {In\'{e}galit\'{e}s ``a priori'' pour des tores lagrangiens
              invariants par des diff\'{e}omorphismes symplectiques},
   JOURNAL = {Inst. Hautes \'{E}tudes Sci. Publ. Math.},
  FJOURNAL = {Institut des Hautes \'{E}tudes Scientifiques. Publications
              Math\'{e}matiques},
    VOLUME = {70},
      YEAR = {1989},
     PAGES = {47--101},
      ISSN = {0073-8301,1618-1913},
   MRCLASS = {58F05 (58F22 58F27 58F30)},
  MRNUMBER = {1067380},
       URL = {http://www.numdam.org/item?id=PMIHES_1989__70__47_0},
}

@article {MR2860422,
    AUTHOR = {Bernard, Patrick and dos Santos, Joana Oliveira},
     TITLE = {A geometric definition of the {M}a\~{n}\'{e}-{M}ather set and
              a theorem of {M}arie-{C}laude {A}rnaud},
   JOURNAL = {Math. Proc. Cambridge Philos. Soc.},
  FJOURNAL = {Mathematical Proceedings of the Cambridge Philosophical
              Society},
    VOLUME = {152},
      YEAR = {2012},
    NUMBER = {1},
     PAGES = {167--178},
      ISSN = {0305-0041,1469-8064},
   MRCLASS = {37J50 (37J05 70H20)},
  MRNUMBER = {2860422},
MRREVIEWER = {M\'{a}rio\ Jorge Dias Carneiro},
       DOI = {10.1017/S0305004111000685},
       URL = {https://doi.org/10.1017/S0305004111000685},
}

@article {MR4597700,
    AUTHOR = {Dias Carneiro, M\'{a}rio J. and Ruggiero, Rafael O.},
     TITLE = {On the graph theorem for {L}agrangian invariant tori with
              totally irrational invariant sets},
   JOURNAL = {Manuscripta Math.},
  FJOURNAL = {Manuscripta Mathematica},
    VOLUME = {171},
      YEAR = {2023},
    NUMBER = {3-4},
     PAGES = {423--436},
      ISSN = {0025-2611,1432-1785},
   MRCLASS = {53D25 (37J51 53D12)},
  MRNUMBER = {4597700},
       DOI = {10.1007/s00229-022-01391-1},
       URL = {https://doi.org/10.1007/s00229-022-01391-1},
}

@article {MR2738994,
    AUTHOR = {Arnaud, Marie-Claude},
     TITLE = {On a theorem due to {B}irkhoff},
   JOURNAL = {Geom. Funct. Anal.},
  FJOURNAL = {Geometric and Functional Analysis},
    VOLUME = {20},
      YEAR = {2010},
    NUMBER = {6},
     PAGES = {1307--1316},
      ISSN = {1016-443X,1420-8970},
   MRCLASS = {37J50 (37E40 37J05 37J10 70H20)},
  MRNUMBER = {2738994},
MRREVIEWER = {Darko\ Milinkovi\'{c}},
       DOI = {10.1007/s00039-010-0091-6},
       URL = {https://doi.org/10.1007/s00039-010-0091-6},
}

@article {MR1451248,
    AUTHOR = {Fathi, Albert},
     TITLE = {Th\'{e}or\`eme {KAM} faible et th\'{e}orie de {M}ather sur les
              syst\`emes lagrangiens},
   JOURNAL = {C. R. Acad. Sci. Paris S\'{e}r. I Math.},
  FJOURNAL = {Comptes Rendus de l'Acad\'{e}mie des Sciences. S\'{e}rie I.
              Math\'{e}matique},
    VOLUME = {324},
      YEAR = {1997},
    NUMBER = {9},
     PAGES = {1043--1046},
      ISSN = {0764-4442},
   MRCLASS = {58F27 (58F05 70H35)},
  MRNUMBER = {1451248},
MRREVIEWER = {Ugo\ C.\ Bessi},
       DOI = {10.1016/S0764-4442(97)87883-4},
       URL = {https://doi.org/10.1016/S0764-4442(97)87883-4},
}

@article {MR0690039,
    AUTHOR = {Crandall, Michael G. and Lions, Pierre-Louis},
     TITLE = {Viscosity solutions of {H}amilton-{J}acobi equations},
   JOURNAL = {Trans. Amer. Math. Soc.},
  FJOURNAL = {Transactions of the American Mathematical Society},
    VOLUME = {277},
      YEAR = {1983},
    NUMBER = {1},
     PAGES = {1--42},
      ISSN = {0002-9947,1088-6850},
   MRCLASS = {35F20},
  MRNUMBER = {690039},
MRREVIEWER = {Moshe\ Marcus},
       DOI = {10.2307/1999343},
       URL = {https://doi.org/10.2307/1999343},
}

@article {MR2346451,
    AUTHOR = {Fathi, Albert and Maderna, Ezequiel},
     TITLE = {Weak {KAM} theorem on non compact manifolds},
   JOURNAL = {NoDEA Nonlinear Differential Equations Appl.},
  FJOURNAL = {NoDEA. Nonlinear Differential Equations and Applications},
    VOLUME = {14},
      YEAR = {2007},
    NUMBER = {1-2},
     PAGES = {1--27},
      ISSN = {1021-9722,1420-9004},
   MRCLASS = {49L25 (37J15 37J40)},
  MRNUMBER = {2346451},
       DOI = {10.1007/s00030-007-2047-6},
       URL = {https://doi.org/10.1007/s00030-007-2047-6},
}

@book {MR0709590,
    AUTHOR = {Clarke, Frank H.},
     TITLE = {Optimization and nonsmooth analysis},
    SERIES = {Canadian Mathematical Society Series of Monographs and
              Advanced Texts},
      NOTE = {A Wiley-Interscience Publication},
 PUBLISHER = {John Wiley \& Sons, Inc., New York},
      YEAR = {1983},
     PAGES = {xiii+308},
      ISBN = {0-471-87504-X},
   MRCLASS = {49-02 (49A52 58C20 90C48)},
  MRNUMBER = {709590},
MRREVIEWER = {Jean-Baptiste\ Hiriart-Urruty},
}

@article {MR2150356,
    AUTHOR = {Arnaud, M.-C.},
     TITLE = {Convergence of the semi-group of {L}ax-{O}leinik: a geometric
              point of view},
   JOURNAL = {Nonlinearity},
  FJOURNAL = {Nonlinearity},
    VOLUME = {18},
      YEAR = {2005},
    NUMBER = {4},
     PAGES = {1835--1840},
      ISSN = {0951-7715,1361-6544},
   MRCLASS = {37J50 (70H20)},
  MRNUMBER = {2150356},
MRREVIEWER = {M\'ario\ Jorge Dias Carneiro},
       DOI = {10.1088/0951-7715/18/4/021},
       URL = {https://doi.org/10.1088/0951-7715/18/4/021},
}

@article {MR1650261,
    AUTHOR = {Fathi, Albert},
     TITLE = {Sur la convergence du semi-groupe de {L}ax-{O}leinik},
   JOURNAL = {C. R. Acad. Sci. Paris S\'{e}r. I Math.},
  FJOURNAL = {Comptes Rendus de l'Acad\'{e}mie des Sciences. S\'{e}rie I.
              Math\'{e}matique},
    VOLUME = {327},
      YEAR = {1998},
    NUMBER = {3},
     PAGES = {267--270},
      ISSN = {0764-4442},
   MRCLASS = {37J50 (37J40)},
  MRNUMBER = {1650261},
MRREVIEWER = {Luigi Chierchia},
       DOI = {10.1016/S0764-4442(98)80144-4},
       URL = {https://doi.org/10.1016/S0764-4442(98)80144-4},
}

@article {MR2041603,
    AUTHOR = {Bernard, Patrick and Roquejoffre, Jean-Michel},
     TITLE = {Convergence to time-periodic solutions in time-periodic
              {H}amilton-{J}acobi equations on the circle},
   JOURNAL = {Comm. Partial Differential Equations},
  FJOURNAL = {Communications in Partial Differential Equations},
    VOLUME = {29},
      YEAR = {2004},
    NUMBER = {3-4},
     PAGES = {457--469},
      ISSN = {0360-5302},
   MRCLASS = {35F20 (35B10 37J40 37J50)},
  MRNUMBER = {2041603},
MRREVIEWER = {Paola Loreti},
       DOI = {10.1081/PDE-120030404},
       URL = {https://doi.org/10.1081/PDE-120030404},
}

@misc{Representation,
      title={Representation of Global Viscosity Solutions for {T}onelli {H}amiltonians}, 
      author={Skander Charfi},
      year={2025},
      eprint={2503.16035},
      archivePrefix={arXiv},
      primaryClass={math.DS},
      url={https://arxiv.org/abs/2503.16035}, 
}

@misc{Recurrent,
      title={A Smooth, Recurrent, Non-Periodic Viscosity Solution of the {H}amilton-{J}acobi Equation}, 
      author={Skander Charfi},
      year={2025},
      eprint={2505.08700},
      archivePrefix={arXiv},
      primaryClass={math.DS},
      url={https://arxiv.org/abs/2505.08700}, 
}

@misc{Birkhoff,
      title={A Multidimensional {B}irkhoff Theorem for Recurrent {L}agrangian Submanifolds by a {T}onelli {H}amiltonian}, 
      author={Skander Charfi},
      year={2025},
      eprint={2507.14561},
      archivePrefix={arXiv},
      primaryClass={math.DS},
      url={https://arxiv.org/abs/2507.14561}, 
}

@article{HLS,
author = {Humilière, Vincent and Leclercq, Rémi and Seyfaddini, Sobhan},
year = {2015},
month = {03},
pages = {767-799},
title = {Coisotropic rigidity and ${C}^{0}$ -symplectic geometry},
volume = {164},
journal = {Duke Mathematical Journal},
doi = {10.1215/00127094-2881701}
}

@incollection {MR2596633,
    AUTHOR = {Fukaya, K. and Seidel, P. and Smith, I.},
     TITLE = {The symplectic geometry of cotangent bundles from a
              categorical viewpoint},
 BOOKTITLE = {Homological mirror symmetry},
    SERIES = {Lecture Notes in Phys.},
    VOLUME = {757},
     PAGES = {1--26},
 PUBLISHER = {Springer, Berlin},
      YEAR = {2009},
      ISBN = {978-3-540-86374-8},
   MRCLASS = {53D37 (18G40 53D35 53D40 57R17)},
  MRNUMBER = {2596633},
MRREVIEWER = {Timothy\ Perutz},
       DOI = {10.1007/978-3-540-68030-7\_1},
       URL = {https://doi.org/10.1007/978-3-540-68030-7_1},
}

@article {MR2947545,
    AUTHOR = {Abouzaid, Mohammed},
     TITLE = {Nearby {L}agrangians with vanishing {M}aslov class are
              homotopy equivalent},
   JOURNAL = {Invent. Math.},
  FJOURNAL = {Inventiones Mathematicae},
    VOLUME = {189},
      YEAR = {2012},
    NUMBER = {2},
     PAGES = {251--313},
      ISSN = {0020-9910,1432-1297},
   MRCLASS = {53D12 (53D37 53D40)},
  MRNUMBER = {2947545},
MRREVIEWER = {Michael\ J.\ Usher},
       DOI = {10.1007/s00222-011-0365-0},
       URL = {https://doi.org/10.1007/s00222-011-0365-0},
}

@article {MR2565051,
    AUTHOR = {Nadler, David},
     TITLE = {Microlocal branes are constructible sheaves},
   JOURNAL = {Selecta Math. (N.S.)},
  FJOURNAL = {Selecta Mathematica. New Series},
    VOLUME = {15},
      YEAR = {2009},
    NUMBER = {4},
     PAGES = {563--619},
      ISSN = {1022-1824,1420-9020},
   MRCLASS = {53D37 (32S60 53D40 57R56)},
  MRNUMBER = {2565051},
MRREVIEWER = {Richard\ P.\ Thomas},
       DOI = {10.1007/s00029-009-0008-0},
       URL = {https://doi.org/10.1007/s00029-009-0008-0},
}

@article {MR1094198,
    AUTHOR = {Chaperon, Marc},
     TITLE = {Lois de conservation et g\'eom\'etrie symplectique},
   JOURNAL = {C. R. Acad. Sci. Paris S\'er. I Math.},
  FJOURNAL = {Comptes Rendus de l'Acad\'emie des Sciences. S\'erie I.
              Math\'ematique},
    VOLUME = {312},
      YEAR = {1991},
    NUMBER = {4},
     PAGES = {345--348},
      ISSN = {0764-4442},
   MRCLASS = {58F05 (35D05 35F25 49S05)},
  MRNUMBER = {1094198},
MRREVIEWER = {Jean-Claude\ Sikorav},
}

@misc{guillermou2015,
      title={Quantization of conic {L}agrangian submanifolds of cotangent bundles}, 
      author={Stéphane Guillermou},
      year={2015},
      eprint={1212.5818},
      archivePrefix={arXiv},
      primaryClass={math.SG},
      url={https://arxiv.org/abs/1212.5818}, 
}

@book {MR3524783,
    AUTHOR = {Oh, Yong-Geun},
     TITLE = {Symplectic topology and {F}loer homology. {V}ol. 2},
    SERIES = {New Mathematical Monographs},
    VOLUME = {29},
      NOTE = {Floer homology and its applications},
 PUBLISHER = {Cambridge University Press, Cambridge},
      YEAR = {2015},
     PAGES = {xxiii+446},
      ISBN = {978-1-107-10967-4},
   MRCLASS = {53D40 (53Dxx 57R58)},
  MRNUMBER = {3524783},
MRREVIEWER = {Hansj\"org\ Geiges},
}

@article {MR1755825,
    AUTHOR = {Schwarz, Matthias},
     TITLE = {On the action spectrum for closed symplectically aspherical
              manifolds},
   JOURNAL = {Pacific J. Math.},
  FJOURNAL = {Pacific Journal of Mathematics},
    VOLUME = {193},
      YEAR = {2000},
    NUMBER = {2},
     PAGES = {419--461},
      ISSN = {0030-8730,1945-5844},
   MRCLASS = {53D40 (57R58)},
  MRNUMBER = {1755825},
MRREVIEWER = {David\ E.\ Hurtubise},
       DOI = {10.2140/pjm.2000.193.419},
       URL = {https://doi.org/10.2140/pjm.2000.193.419},
}

@incollection {MR2103018,
    AUTHOR = {Oh, Yong-Geun},
     TITLE = {Construction of spectral invariants of {H}amiltonian paths on
              closed symplectic manifolds},
 BOOKTITLE = {The breadth of symplectic and {P}oisson geometry},
    SERIES = {Progr. Math.},
    VOLUME = {232},
     PAGES = {525--570},
 PUBLISHER = {Birkh\"auser Boston, Boston, MA},
      YEAR = {2005},
      ISBN = {0-8176-3565-3},
   MRCLASS = {53D40 (53D45 58E30)},
  MRNUMBER = {2103018},
MRREVIEWER = {Darko\ Milinkovi\'c},
       DOI = {10.1007/0-8176-4419-9\_18},
       URL = {https://doi.org/10.1007/0-8176-4419-9_18},
}

@article {MR2383268,
    AUTHOR = {Leclercq, R\'emi},
     TITLE = {Spectral invariants in {L}agrangian {F}loer theory},
   JOURNAL = {J. Mod. Dyn.},
  FJOURNAL = {Journal of Modern Dynamics},
    VOLUME = {2},
      YEAR = {2008},
    NUMBER = {2},
     PAGES = {249--286},
      ISSN = {1930-5311,1930-532X},
   MRCLASS = {57R17 (53D40 55T10)},
  MRNUMBER = {2383268},
       DOI = {10.3934/jmd.2008.2.249},
       URL = {https://doi.org/10.3934/jmd.2008.2.249},
}

@article {MR3850107,
    AUTHOR = {Leclercq, R\'emi and Zapolsky, Frol},
     TITLE = {Spectral invariants for monotone {L}agrangians},
   JOURNAL = {J. Topol. Anal.},
  FJOURNAL = {Journal of Topology and Analysis},
    VOLUME = {10},
      YEAR = {2018},
    NUMBER = {3},
     PAGES = {627--700},
      ISSN = {1793-5253,1793-7167},
   MRCLASS = {57R17 (53D12 53D40)},
  MRNUMBER = {3850107},
MRREVIEWER = {Roman\ Golovko},
       DOI = {10.1142/S1793525318500267},
       URL = {https://doi.org/10.1142/S1793525318500267},
}

@article {MR1266111,
    AUTHOR = {Laudenbach, F. and Sikorav, J.-C.},
     TITLE = {Hamiltonian disjunction and limits of {L}agrangian
              submanifolds},
   JOURNAL = {Internat. Math. Res. Notices},
  FJOURNAL = {International Mathematics Research Notices},
      YEAR = {1994},
    NUMBER = {4},
     PAGES = {161 ff., approx. 8 pp.\},
      ISSN = {1073-7928,1687-0247},
   MRCLASS = {58F05 (53C15)},
  MRNUMBER = {1266111},
MRREVIEWER = {Wilhelm\ Klingenberg},
       DOI = {10.1155/S1073792894000176},
       URL = {https://doi.org/10.1155/S1073792894000176},
}

@book {MR728564,
    AUTHOR = {Herman, Michael-R.},
     TITLE = {Sur les courbes invariantes par les diff\'eomorphismes de
              l'anneau. {V}ol. 1},
    SERIES = {Ast\'erisque},
    VOLUME = {103-104},
      NOTE = {With an appendix by Albert Fathi,
              With an English summary},
 PUBLISHER = {Soci\'et\'e{} Math\'ematique de France, Paris},
      YEAR = {1983},
     PAGES = {i+221},
   MRCLASS = {58F05 (54H20 57R50 58-02 58F11)},
  MRNUMBER = {728564},
MRREVIEWER = {Helmut\ R\"ussmann},
}

@article{BSMF_1932__60__1_0,
     author = {Birkhoff, George D.},
     title = {Sur quelques courbes ferm\'ees remarquables},
     journal = {Bulletin de la Soci\'et\'e Math\'ematique de France},
     pages = {1--26},
     year = {1932},
     publisher = {Soci\'et\'e math\'ematique de France},
     volume = {60},
     doi = {10.24033/bsmf.1182},
     zbl = {0005.22002},
     language = {fr},
     url = {https://www.numdam.org/articles/10.24033/bsmf.1182/}
}

@article {MR1505024,
    AUTHOR = {Charpentier, M.},
     TITLE = {Sur quelques propri\'et\'es des courbes de {M}. {B}irkhoff},
   JOURNAL = {Bull. Soc. Math. France},
  FJOURNAL = {Bulletin de la Soci\'et\'e{} Math\'ematique de France},
    VOLUME = {62},
      YEAR = {1934},
     PAGES = {193--224},
      ISSN = {0037-9484},
   MRCLASS = {99-04},
  MRNUMBER = {1505024},
       URL = {http://www.numdam.org/item?id=BSMF_1934__62__193_0},
}

@article {MR951271,
    AUTHOR = {Le Calvez, P.},
     TITLE = {Propri\'et\'es des attracteurs de {B}irkhoff},
   JOURNAL = {Ergodic Theory Dynam. Systems},
  FJOURNAL = {Ergodic Theory and Dynamical Systems},
    VOLUME = {8},
      YEAR = {1988},
    NUMBER = {2},
     PAGES = {241--310},
      ISSN = {0143-3857,1469-4417},
   MRCLASS = {58F12},
  MRNUMBER = {951271},
MRREVIEWER = {Dietrich\ Flockerzi},
       DOI = {10.1017/S0143385700004442},
       URL = {https://doi.org/10.1017/S0143385700004442},
}

@article {MR5052827,
    AUTHOR = {Arnaud, Marie-Claude and Humili\`ere, Vincent and Viterbo,
              Claude},
     TITLE = {Higher dimensional {B}irkhoff attractors (with an appendix by
              {M}axime {Z}avidovique)},
   JOURNAL = {J. \'Ec. polytech. Math.},
  FJOURNAL = {Journal de l'\'Ecole polytechnique. Math\'ematiques},
    VOLUME = {13},
      YEAR = {2026},
     PAGES = {759--808},
      ISSN = {2429-7100,2270-518X},
   MRCLASS = {37C70 (35F21 37E40 53D12 57R17)},
  MRNUMBER = {5052827},
       DOI = {10.5802/jep.336},
       URL = {https://doi.org/10.5802/jep.336},
}

@misc{viterbo2019sheafquantizationlagrangiansfloer,
      title={Sheaf Quantization of {L}agrangians and {F}loer cohomology}, 
      author={Claude Viterbo},
      year={2019},
      eprint={1901.09440},
      archivePrefix={arXiv},
      primaryClass={math.SG},
      url={https://arxiv.org/abs/1901.09440}, 
}

@article {MR4683312,
    AUTHOR = {Asano, Tomohiro and Guillermou, St\'ephane and Humili\`ere,
              Vincent and Ike, Yuichi and Viterbo, Claude},
     TITLE = {The {$\gamma$}-support as a micro-support},
   JOURNAL = {C. R. Math. Acad. Sci. Paris},
  FJOURNAL = {Comptes Rendus Math\'ematique. Acad\'emie des Sciences. Paris},
    VOLUME = {361},
      YEAR = {2023},
     PAGES = {1333--1340},
      ISSN = {1631-073X,1778-3569},
   MRCLASS = {53D12},
  MRNUMBER = {4683312},
MRREVIEWER = {Andrea\ D'Agnolo},
       DOI = {10.5802/crmath.499},
       URL = {https://doi.org/10.5802/crmath.499},
}

@article {MR4768582,
    AUTHOR = {Guillermou, St\'ephane and Viterbo, Claude},
     TITLE = {The singular support of sheaves is {$\gamma $}-coisotropic},
   JOURNAL = {Geom. Funct. Anal.},
  FJOURNAL = {Geometric and Functional Analysis},
    VOLUME = {34},
      YEAR = {2024},
    NUMBER = {4},
     PAGES = {1052--1113},
      ISSN = {1016-443X,1420-8970},
   MRCLASS = {53D12 (18F20 37J99)},
  MRNUMBER = {4768582},
MRREVIEWER = {Andrea\ D'Agnolo},
       DOI = {10.1007/s00039-024-00682-x},
       URL = {https://doi.org/10.1007/s00039-024-00682-x},
}

@article {MR5021042,
    AUTHOR = {Asano, Tomohiro and Guillermou, St\'ephane and Ike, Yuichi and
              Viterbo, Claude},
     TITLE = {Regular {L}agrangians are smooth {L}agrangians},
   JOURNAL = {J. Math. Soc. Japan},
  FJOURNAL = {Journal of the Mathematical Society of Japan},
    VOLUME = {78},
      YEAR = {2026},
    NUMBER = {1},
     PAGES = {37--54},
      ISSN = {0025-5645,1881-1167},
   MRCLASS = {53D12 (35A27 37J11)},
  MRNUMBER = {5021042},
       DOI = {10.2969/jmsj/93799379},
       URL = {https://doi.org/10.2969/jmsj/93799379},
}

@book {MR1299726,
    AUTHOR = {Kashiwara, Masaki and Schapira, Pierre},
     TITLE = {Sheaves on manifolds},
    SERIES = {Grundlehren der mathematischen Wissenschaften [Fundamental
              Principles of Mathematical Sciences]},
    VOLUME = {292},
      NOTE = {With a chapter in French by Christian Houzel,
              Corrected reprint of the 1990 original},
 PUBLISHER = {Springer-Verlag, Berlin},
      YEAR = {1994},
     PAGES = {x+512},
      ISBN = {3-540-51861-4},
   MRCLASS = {58G07 (18F20 32C38 35A27)},
  MRNUMBER = {1299726},
}

@incollection {MR1736217,
    AUTHOR = {Po{\'z}niak, Marcin},
     TITLE = {Floer homology, {N}ovikov rings and clean intersections},
 BOOKTITLE = {Northern {C}alifornia {S}ymplectic {G}eometry {S}eminar},
    SERIES = {Amer. Math. Soc. Transl. Ser. 2},
    VOLUME = {196},
     PAGES = {119--181},
 PUBLISHER = {Amer. Math. Soc., Providence, RI},
      YEAR = {1999},
      ISBN = {0-8218-2075-3},
   MRCLASS = {53D40 (57R58)},
  MRNUMBER = {1736217},
MRREVIEWER = {David\ E.\ Hurtubise},
       DOI = {10.1090/trans2/196/08},
       URL = {https://doi.org/10.1090/trans2/196/08},
}

@book {MR4211770,
    AUTHOR = {Zhang, Jun},
     TITLE = {Quantitative {T}amarkin theory},
    SERIES = {CRM Short Courses},
 PUBLISHER = {Springer, Cham},
      YEAR = {[2020] \copyright 2020},
     PAGES = {x+146},
      ISBN = {978-3-030-37887-5; 978-3-030-37888-2},
   MRCLASS = {53Dxx (18F20 18G80 35A25 55N31 55U35)},
  MRNUMBER = {4211770},
MRREVIEWER = {Andrea\ D'Agnolo},
       DOI = {10.1007/978-3-030-37888-2},
       URL = {https://doi.org/10.1007/978-3-030-37888-2},
}

@misc{nakamura2020c0limitslegendriansubmanifolds,
      title={${C}^0$-limits of {L}egendrian Submanifolds}, 
      author={Lukas Nakamura},
      year={2020},
      eprint={2008.00924},
      archivePrefix={arXiv},
      primaryClass={math.SG},
      url={https://arxiv.org/abs/2008.00924}, 
}

@misc{asano2025c0rigiditylegendrianscoisotropicssheaf,
      title={${C}^0$-rigidity of {L}egendrians and coisotropics via sheaf quantization}, 
      author={Tomohiro Asano and Yuichi Ike and Christopher Kuo and Wenyuan Li},
      year={2025},
      eprint={2510.01746},
      archivePrefix={arXiv},
      primaryClass={math.SG},
      url={https://arxiv.org/abs/2510.01746}, 
}

@article {MR4876451,
    AUTHOR = {Dimitroglou Rizell, Georgios and Sullivan, Michael G.},
     TITLE = {{$C^0$}-limits of {L}egendrians and positive loops},
   JOURNAL = {Compos. Math.},
  FJOURNAL = {Compositio Mathematica},
    VOLUME = {160},
      YEAR = {2024},
    NUMBER = {12},
     PAGES = {2904--2915},
      ISSN = {0010-437X,1570-5846},
   MRCLASS = {53D10 (53D12)},
  MRNUMBER = {4876451},
MRREVIEWER = {Stanis\l aw\ Pawe\l\ Kasperczuk},
       DOI = {10.1112/S0010437X24007474},
       URL = {https://doi.org/10.1112/S0010437X24007474},
}

@article{viterbo2026supportshumilierecompletiongammacoisotropic, title={On the supports in the Humilière completion and $\gamma$-coisotropic sets}, volume={162}, DOI={10.1017/S0010437X26103212}, number={7}, journal={Compositio Mathematica}, author={Viterbo, Claude}, year={2026}, pages={1439–1486}}

@article {MR4334195,
    AUTHOR = {Membrez, Cedric and Opshtein, Emmanuel},
     TITLE = {{$\mathcal C^0$}-rigidity of {L}agrangian submanifolds and
              punctured holomorphic disks in the cotangent bundle},
   JOURNAL = {Compos. Math.},
  FJOURNAL = {Compositio Mathematica},
    VOLUME = {157},
      YEAR = {2021},
    NUMBER = {11},
     PAGES = {2433--2493},
      ISSN = {0010-437X,1570-5846},
   MRCLASS = {53D05 (53D12)},
  MRNUMBER = {4334195},
MRREVIEWER = {Chun-Gen\ Liu},
       DOI = {10.1112/S0010437X21007570},
       URL = {https://doi.org/10.1112/S0010437X21007570},
}

@article {MR3619673,
    AUTHOR = {Hungerb\"uhler, Norbert and Mettler, Thomas and Wasem, Micha},
     TITLE = {Convex integration and {L}egendrian approximation of curves},
   JOURNAL = {J. Convex Anal.},
  FJOURNAL = {Journal of Convex Analysis},
    VOLUME = {24},
      YEAR = {2017},
    NUMBER = {1},
     PAGES = {309--317},
      ISSN = {0944-6532,2363-6394},
   MRCLASS = {53D10},
  MRNUMBER = {3619673},
MRREVIEWER = {Mahuya\ Datta},
}

@book {MR864505,
    AUTHOR = {Gromov, Mikhael},
     TITLE = {Partial differential relations},
    SERIES = {Ergebnisse der Mathematik und ihrer Grenzgebiete (3) [Results
              in Mathematics and Related Areas (3)]},
    VOLUME = {9},
 PUBLISHER = {Springer-Verlag, Berlin},
      YEAR = {1986},
     PAGES = {x+363},
      ISBN = {3-540-12177-3},
   MRCLASS = {58G99 (35A99 35B99 53C42 58-02)},
  MRNUMBER = {864505},
MRREVIEWER = {Hung-Hsi\ Wu},
       DOI = {10.1007/978-3-662-02267-2},
       URL = {https://doi.org/10.1007/978-3-662-02267-2},
}

@book {MR3024860,
    AUTHOR = {Spring, David},
     TITLE = {Convex integration theory},
    SERIES = {Modern Birkh\"auser Classics},
      NOTE = {Solutions to the $h$-principle in geometry and topology,
              Reprint of the 1998 edition [MR1488424]},
 PUBLISHER = {Birkh\"auser/Springer Basel AG, Basel},
      YEAR = {2010},
     PAGES = {viii+213},
      ISBN = {978-3-0348-0059-4; 978-3-0348-0060-0},
   MRCLASS = {01A75},
  MRNUMBER = {3024860},
}

@article {MR4381219,
    AUTHOR = {Theilli\`ere, M\'elanie},
     TITLE = {Convex integration theory without integration},
   JOURNAL = {Math. Z.},
  FJOURNAL = {Mathematische Zeitschrift},
    VOLUME = {300},
      YEAR = {2022},
    NUMBER = {3},
     PAGES = {2737--2770},
      ISSN = {0025-5874,1432-1823},
   MRCLASS = {53C42 (53C21)},
  MRNUMBER = {4381219},
MRREVIEWER = {Mahuya\ Datta},
       DOI = {10.1007/s00209-021-02785-9},
       URL = {https://doi.org/10.1007/s00209-021-02785-9},
}

@book {MR2397738,
    AUTHOR = {Geiges, Hansj\"org},
     TITLE = {An introduction to contact topology},
    SERIES = {Cambridge Studies in Advanced Mathematics},
    VOLUME = {109},
 PUBLISHER = {Cambridge University Press, Cambridge},
      YEAR = {2008},
     PAGES = {xvi+440},
      ISBN = {978-0-521-86585-2},
   MRCLASS = {57R17 (53D35)},
  MRNUMBER = {2397738},
MRREVIEWER = {John\ B.\ Etnyre},
       DOI = {10.1017/CBO9780511611438},
       URL = {https://doi.org/10.1017/CBO9780511611438},
}

@incollection {MR2179261,
    AUTHOR = {Etnyre, John B.},
     TITLE = {Legendrian and transversal knots},
 BOOKTITLE = {Handbook of knot theory},
     PAGES = {105--185},
 PUBLISHER = {Elsevier B. V., Amsterdam},
      YEAR = {2005},
      ISBN = {0-444-51452-X},
   MRCLASS = {57R17 (53D35 57M25 57M27)},
  MRNUMBER = {2179261},
MRREVIEWER = {Lenhard\ L.\ Ng},
       DOI = {10.1016/B978-044451452-3/50004-6},
       URL = {https://doi.org/10.1016/B978-044451452-3/50004-6},
}

@article {MR2231385,
    AUTHOR = {Czarnecki, Marc-Olivier and Rifford, Ludovic},
     TITLE = {Approximation and regularization of {L}ipschitz functions:
              convergence of the gradients},
   JOURNAL = {Trans. Amer. Math. Soc.},
  FJOURNAL = {Transactions of the American Mathematical Society},
    VOLUME = {358},
      YEAR = {2006},
    NUMBER = {10},
     PAGES = {4467--4520},
      ISSN = {0002-9947,1088-6850},
   MRCLASS = {49J45 (49J52)},
  MRNUMBER = {2231385},
MRREVIEWER = {Adam\ B.\ Levy},
       DOI = {10.1090/S0002-9947-06-04103-1},
       URL = {https://doi.org/10.1090/S0002-9947-06-04103-1},
}

@article {MR3070514,
    AUTHOR = {Kragh, Thomas},
     TITLE = {Parametrized ring-spectra and the nearby {L}agrangian
              conjecture},
      NOTE = {With an appendix by Mohammed Abouzaid},
   JOURNAL = {Geom. Topol.},
  FJOURNAL = {Geometry \& Topology},
    VOLUME = {17},
      YEAR = {2013},
    NUMBER = {2},
     PAGES = {639--731},
      ISSN = {1465-3060,1364-0380},
   MRCLASS = {53D12 (53D40 55P42 55T10)},
  MRNUMBER = {3070514},
MRREVIEWER = {Michael\ J.\ Usher},
       DOI = {10.2140/gt.2013.17.639},
       URL = {https://doi.org/10.2140/gt.2013.17.639},
}

@article {MR3849284,
    AUTHOR = {Abouzaid, Mohammed and Kragh, Thomas},
     TITLE = {Simple homotopy equivalence of nearby {L}agrangians},
   JOURNAL = {Acta Math.},
  FJOURNAL = {Acta Mathematica},
    VOLUME = {220},
      YEAR = {2018},
    NUMBER = {2},
     PAGES = {207--237},
      ISSN = {0001-5962,1871-2509},
   MRCLASS = {53D12 (53D40 57R17 57R58)},
  MRNUMBER = {3849284},
MRREVIEWER = {Darko\ Milinkovi\'c},
       DOI = {10.4310/ACTA.2018.v220.n2.a1},
       URL = {https://doi.org/10.4310/ACTA.2018.v220.n2.a1},
}

@article {MR2415347,
    AUTHOR = {Humili\`ere, Vincent},
     TITLE = {On some completions of the space of {H}amiltonian maps},
   JOURNAL = {Bull. Soc. Math. France},
  FJOURNAL = {Bulletin de la Soci\'et\'e{} Math\'ematique de France},
    VOLUME = {136},
      YEAR = {2008},
    NUMBER = {3},
     PAGES = {373--404},
      ISSN = {0037-9484,2102-622X},
   MRCLASS = {53D35 (37J05 53D12 53D40)},
  MRNUMBER = {2415347},
MRREVIEWER = {Michael\ J.\ Usher},
       DOI = {10.24033/bsmf.2560},
       URL = {https://doi.org/10.24033/bsmf.2560},
}

@article {MR4719181,
    AUTHOR = {Chass\'e, Jean-Philippe},
     TITLE = {Hausdorff limits of submanifolds of symplectic and contact
              manifolds},
   JOURNAL = {Differential Geom. Appl.},
  FJOURNAL = {Differential Geometry and its Applications},
    VOLUME = {94},
      YEAR = {2024},
     PAGES = {Paper No. 102123, 22},
      ISSN = {0926-2245,1872-6984},
   MRCLASS = {53D12 (53C23 53D05)},
  MRNUMBER = {4719181},
       DOI = {10.1016/j.difgeo.2024.102123},
       URL = {https://doi.org/10.1016/j.difgeo.2024.102123},
}

@article {MR4263685,
    AUTHOR = {Buhovsky, Lev and Humili\`ere, Vincent and Seyfaddini, Sobhan},
     TITLE = {The action spectrum and {$C^0$} symplectic topology},
   JOURNAL = {Math. Ann.},
  FJOURNAL = {Mathematische Annalen},
    VOLUME = {380},
      YEAR = {2021},
    NUMBER = {1-2},
     PAGES = {293--316},
      ISSN = {0025-5831,1432-1807},
   MRCLASS = {53D05 (37J06)},
  MRNUMBER = {4263685},
MRREVIEWER = {St\'ephane\ Tchuiaga},
       DOI = {10.1007/s00208-021-02183-w},
       URL = {https://doi.org/10.1007/s00208-021-02183-w},
}

@article {MR1043418,
    AUTHOR = {Polterovich, Leonid},
     TITLE = {The second {B}irkhoff theorem for optical {H}amiltonian
              systems},
   JOURNAL = {Proc. Amer. Math. Soc.},
  FJOURNAL = {Proceedings of the American Mathematical Society},
    VOLUME = {113},
      YEAR = {1991},
    NUMBER = {2},
     PAGES = {513--516},
      ISSN = {0002-9939,1088-6826},
   MRCLASS = {58F05 (53C22 58F17)},
  MRNUMBER = {1043418},
MRREVIEWER = {Jean-Claude\ Sikorav},
       DOI = {10.2307/2048537},
       URL = {https://doi.org/10.2307/2048537},
}

\end{document}